%% file: master.tex
\documentclass[leqno,a4paper]{scrbook}
\usepackage{a4wide}

\usepackage{amssymb,amsfonts,amsthm,mathrsfs}

\usepackage[sumlimits,intlimits,namelimits]{amsmath}

\usepackage{graphicx}

\usepackage[automark]{scrpage2}

\usepackage[scaled=0.90]{helvet}
\usepackage{makeidx}

\renewcommand{\theequation}{\thesection.\arabic{equation}}

\input{hypnsmakros}

\makeindex 

\begin{document}
%\fontfamily{cmtt}
\pagestyle{scrheadings}

\frontmatter

\input{titlepage}

\input{preface}
\tableofcontents

\mainmatter

\setcounter{section}{-1}

\input{intro}

\input{ggraph}

\input{ode}

\input{pde}

\input{regularity}

\input{appendix}

\backmatter

\printindex

\bibliographystyle{abbrv}

\bibliography{simhal}

\input{abstract}

\input{cv}

\end{document}

%% file: hypnsmakros.tex
\newcommand{\unfavsupp}[2]{
	{\operatorname{U}}^{\mathrm #1}(#2) 
}

\newcommand{\ksupp}[1] {
	\ensuremath{\operatorname{K}^{#1}}
}

\newcommand{\slsupp}[1] {
	\ensuremath{\operatorname{S}^{#1}}
}

\newcommand{\domainset}[1]{
\ensuremath{\operatorname{D}^{#1}}
}

\newcommand{\m}[1]{
\ensuremath{\,{\mathrm d}{#1}}
}

\newcommand{\gnormalb}[1]{
\ensuremath{{\mathrm N}^{#1}}
}

\newcommand{\ran}[1]{
	 {\operatorname{Ran}}(#1)
}

\newcommand{\ch}[1]{
	 {\operatorname{ch}}(#1)
}

\newcommand{\ech}[1]{
	 {\operatorname{ess\,ch}}(#1)
}

\newcommand{\csub}{
 \Subset
}

\newcommand{\ode}{
\dot{\xi}(s) = a(s, \xi(s))\ \ a.e., \quad \xi(t)=x
}

\newcommand{\colode}{
\dot{\xi}(s) = A(s, \xi(s))\ \, \quad \xi(\widetilde{t})=\widetilde{x}
}

\newcommand{\odie}{
\dot{\xi}(s) \in A_{s, \xi(s)}\ \ a.e., \quad \xi(t)=x
}

\newcommand{\cp}[1]{
\ensuremath{ {\operatorname{CP}} \left(#1 \right)}
}

\newcommand{\colmod}[1]{
\ensuremath{\mathcal{E}_{M}\left(#1\right)}
}

\newcommand{\colneg}[1]{
\ensuremath{\mathcal{N}\left(#1\right)}
}

\newcommand{\wf}[1]{
\ensuremath{\operatorname{WF}(#1)}
}

\newcommand{\cf}[1]{
\ensuremath{\mathcal{G}\left(#1\right)}
}

\newcommand{\colmap}[2]{
\ensuremath{\mathcal{G}\!\left[#1,#2\right]}
}

\newcommand{\colcomp}[1]{
\ensuremath{\mathcal{G}_c\left(#1\right)}
}

\newcommand{\four}[1]{
\mathcal{F}(#1)
}

\newcommand{\ggraph}[1]{
 \ensuremath{ \operatorname{Graph}(#1) }
}

\newcommand{\graph}[1]{
 \ensuremath{ \operatorname{graph}(#1) }
}

\newcommand{\esup}{ \operatorname{ess\,sup} }

\newcommand{\AC}{\ensuremath{AC}}
\newcommand{\Con}{\ensuremath{C}}

\newcommand{\D}{\ensuremath{{\cal D}}}
\newcommand{\loc}{\ensuremath{\text{loc}}}

\newcommand{\mb}[1]{\ensuremath{\mathbb{#1}}}
\newcommand{\N}{\mb{N}}

\newcommand{\R}{\mb{R}}
\newcommand{\C}{\mb{C}}

\newcommand{\G}{\ensuremath{{\cal G}}}

\newcommand{\singsupp}[1]{\ensuremath{\mathrm{sing\,supp}(#1)}}

\renewcommand{\d}{\ensuremath{\partial}}
\newcommand{\diff}[1]{\frac{d}{d#1}}

\renewcommand{\div}{\ensuremath{\mbox{\rm div}\,}}

\newfont{\bl}{msbm10 scaled \magstep2}

\newtheorem{theorem}{Theorem}[section]
\newtheorem{lemma}[theorem]{Lemma}
\newtheorem{proposition}[theorem]{Proposition}
\newtheorem{definition}[theorem]{Definition}
\newtheorem{corollary}[theorem]{Corollary}

\theoremstyle{definition}
\newtheorem{remark}[theorem]{Remark}
\newtheorem{example}[theorem]{Example}

\newcommand{\beq}{\begin{equation}}
\newcommand{\eeq}{\end{equation}}
\newcommand{\col}{\colon}

\newcommand{\dis}[2]{\langle #1 , #2 \rangle}
\newcommand{\inp}[2]{\langle #1 | #2 \rangle}  %math mode
\newcommand{\notmid}{\mid\kern-0.5em\not\kern0.5em}

\newcommand{\norm}[2]{{\| #1 \|}_{#2}}

\newcommand{\Norm}[1]{\norm{#1}{}}

\newcommand{\al}{\alpha}

\newcommand{\eps}{\varepsilon}

\newcommand{\vphi}{\varphi}

\newcommand{\supp}{\mathop{\mathrm{supp}}}

\renewcommand{\Re}{\ensuremath{\mathop{\mathrm{Re}}}}

\newcommand{\ovl}[1]{\overline{#1}}

%% file: titlepage.tex
%%%%%%%%%%%%%%%%%%%%%%%%%%%%%%%%%%%%%%%%%%%%%%%%%%%%%%%%%%%%%%%%%%%%%%
%%  Titelseite fuer eine Diplomarbeit/Dissertation an der Uni Wien  %%
%%                     zur Benutzung mit LaTeX                      %%
%%%%%%%%%%%%%%%%%%%%%%%%%%%%%%%%%%%%%%%%%%%%%%%%%%%%%%%%%%%%%%%%%%%%%%

%%  Erstellt anhand der Vorlagendefinition von
%%  http://www.univie.ac.at/Psychologie/cgi-bin/dbman/uploads/download/
%%      51_infoblatt__titelblatt_wissenschaftliche_arbeit.pdf
%%
%% Einzubinden in der eigenen document-Umgebung mittels \input{thesistitle}
%%
%% Ueberschriften wie "Titel der Diplomarbeit" oder "Verfasserin" (etc.)
%% muessen genau so stehen bleiben. Nur der Titel der Arbeit oder die Namen
%% sind entsprechend beliebig.

% Stephan Paukner, 14.08.2007
% Obige Vorlagendefinition schlaegt Arial oder eine vergleichbare serifenlose
% Schriftart vor. Der serifenlose Font von LaTeX, CMSS, hat leider keine
% fette (bold) Variante. Arial ist allerdings ein kommerzieller Font und unter
% LaTeX standardmaessig nicht verfuegbar. Am naehesten kommt dem der Font
% Helvetica, einzubinden via \usepackage[scaled=0.90]{helvet} in der Praeambel.

\begin{titlepage}
\vspace*{-2cm}  % bei Verwendung von vmargin.sty
\begin{flushright}
    \includegraphics{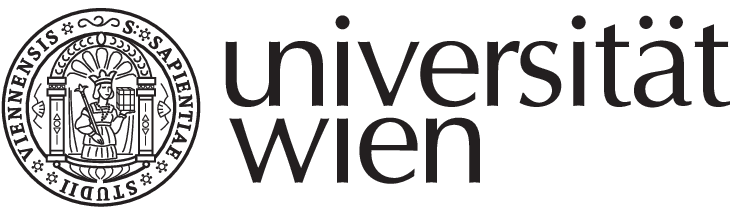}
\end{flushright}
\vspace{1cm}

\begin{center}  % Diplomarbeit ODER Magisterarbeit ODER Dissertation
    \Huge{\textbf{\textsf{\MakeUppercase{
        Dissertation
    }}}}
    \vspace{2cm}

    %\large{\textsf{  % Diplomarbeit ODER Magisterarbeit ODER Dissertation
    %                 % (Dies ist erst die Ueberschrift!)
    %    Dissertation
    %}}
    \vspace{.1cm}

    \LARGE{\textsf{  % Hier kommt der eigentliche Titel, bei Bedarf mit \\
                     % ACHTUNG: Deutsche Anfuehrungszeichen: ,,Titel``
                     %          English quotes:              ``title''
        % >>>>> BEGINN TITEL >>>>> 	
        %``Solution concepts for linear first-order hyperbolic differential equations with non-smooth coefficients and mircolocal analysis of generalized pullbacks''
        Generalized regularity and solution concepts for differential equations
        %``The Creation of a Title Page''
        % <<<<< ENDE TITEL <<<<<
    }}
    \vspace{3cm}

    \large{\textsf{  % Verfasserin ODER Verfasser (Ueberschrift)
        Verfasser
    }}

    \Large{\textsf{  % Vorname Nachname
        % >>>>> BEGINN VORNAME NACHNAME >>>>>
        Simon Haller
        % <<<<< ENDE VORNAME NACHNAME <<<<<
    }}
    \vspace{3cm}

    \large{\textsf{
        angestrebter akademischer Grad  % (Ueberschrift)
    }}

    \Large{\textsf{  % Magistra ODER Magister ODER Doktorin ODER Doktor
                     % ACHTUNG: Kuerzel "Mag.a" oder "Dr.in" nicht zulaessig
        Doktor der Naturwissenschaften (Dr.rer.nat)
    }}
\end{center}
\vspace{2cm}

\noindent\textsf{Wien, im Juni 2008}  % <<<<< ORT, MONAT UND JAHR EINTRAGEN
\vfill

\noindent\begin{tabular}{@{}ll}
\textsf{Studienkennzahl lt.\ Studienblatt:}
&
\textsf{A 091 405}  % <<<<< STUDIENKENNZAHL EINTRAGEN
\\
    % BEI DISSERTATIONEN:
\textsf{Dissertationsgebiet lt. Studienblatt:}
    % ANSONSTEN:
%\textsf{Studienrichtung lt.\ Studienblatt:}
&
\textsf{Mathematik}  % <<<<< DISSGEBIET/STUDIENRICHTUNG EINTRAGEN
\\
% Betreuerin ODER Betreuer:
\textsf{Betreuer:}
&
\textsf{Dr. G\"unther H\"ormann}  % <<<<< NAME EINTRAGEN
\end{tabular}

\end{titlepage}

%% file: preface.tex
\chapter{Preface}
The present doctoral thesis is the result of my research work\footnote{supported by FWF-grant Y237-N13} in the DIANA research group
under supervision of G\"unther H\"ormann.

As the title ``Generalized regularity and solution concepts for differential equations'' suggests, the main topic of my thesis
is the investigation of generalized solution concepts for differential equations, in particular 
first order hyperbolic partial differential equations with real-valued, non-smooth coefficients and their characteristic system of ordinary
differential equations.

In the case of smooth coefficients classical distribution theory offers a convenient framework for solving such partial differential equations.
If the coefficients are non-smooth (or even discontinuous) the well-known fact that the multiplication of distributions cannot be carried out unrestrictedly
limits the scope of distributional techniques.

Note that the product of two distributions can formally be written as the pullback of a tensor product of the two factors by the diagonal map $\delta:x \mapsto (x,x)$,
so another viewpoint of this limitation is that the pullback of a general distribution by a $C^{\infty}$-function, 
as defined in \cite[Theorem 8.2.4]{Hoermander:V1}, exists only if the normal bundle of the $C^{\infty}$-function
intersected with the wave front set of the distribution is empty.

The relation between problems arising from non-smooth pullbacks and multiplication of distributions, is emphasized by the fact that the pullback of the initial condition by the characteristic backward flow is a solution candidate for a homogeneous partial differential equation of first order.

Generalized functions in the sense of Colombeau extend distribution theory in a way that it becomes a differential algebra with 
a product that preserves the classical product $\cdot: C^{\infty} \times C^{\infty} \rightarrow C^{\infty}$.
In addition \cite[Proposition 1.2.8]{GKOS:01} states that the Colombeau algebra of generalized functions allows the definition of
a pullback by any c-bounded generalized function.

So it seems that the Colombeau algebra of generalized functions offers a suitable framework for dealing with the various problems caused by multiplication and pullbacks.

By means of regularization it is easy to carry over any partial differential equation with non-smooth coefficient from distribution theory to Colombeau theory.
In Colombeau theory there have been developed existence results (f.e. \cite[Theorem 1.5.2]{GKOS:01},\cite{LO:91}) that yield solutions for ordinary and partial differential equations beyond the scope of
classical approaches. Nevertheless this comes at the price of sacrificing regularity (in general a Colombeau solution may even lack a distributional shadow).
It is prevailing in the Colombeau setting that the question of mere existence of solutions is much easier to answer than to determine their regularity properties (i.e. if a distributional shadow exists and how regular it is).

This was my motivation for doing a microlocal analysis of the pullback of c-bounded Colombeau generalized functions, since the solution of the (homogeneous) partial differential equation can be written as a pullback of the initial condition by the characteristic backward flow. The results of these investigations have been collected in an article \cite{HalSim:07} and are presented in Chapter 5 and Appendix B in a slightly adapted form. The first section of Chapter 2 is also based on this article.

A further main topic of my thesis is to compare the Colombeau techniques for solving ordinary and partial differential equations to other generalized solution concepts, which has led to a joint article \cite{HalHoer:08}  with Dr. G\"unther H\"ormann. Chapter 4 contains a slightly modified version of this article. 

Chapter 3 contains generalized solution concepts for ordinary differential equations, which are of relevance when studying
the characteristic systems of ordinary differential equations.

Finally I want to point out the important role the generalized graph, as introduced in Chapter 2, plays throughout my thesis:
In Chapter 3 it enables us to give a precise characterization of Colombeau generalized solutions of ordinary differential equations and its relation
to other concepts as the Filippov generalized solutions. In Chapter 5 it serves as a technical tool, playing a crucial role in many of the proofs. 
Nevertheless I believe there are still further applications for the concept of the generalized graph that are yet to be discovered.

I would like to thank Dr. G\"unther H\"ormann for his excellent support.
\newpage

\begin{center}
{\Large{\textbf{to Petra}}}
 \end{center}

%% file: intro.tex
\chapter{Introduction}

\section{Basic notation and overview}

\paragraph{Notation for sets:}
By
$$
B_{\rho}(x) := \{y \in \R^n \mid |x-y| \le \rho \}
$$
we denote the closed ball around $x \in \R^n$ with radius $\rho >0$. If $K$ is a compact subset of $\R^n$ we define
$$
B_{\rho}(K) := \{y \in \R^n \mid \exists x\in K: |x-y| \le \rho \}.
$$
We write $S^{n-1}$ for the unit sphere $\{x \in R^n \mid |x| =1 \}$.

Let $\Omega$ denote an open subset of $\R^n$.
We use the notation $K \Subset \Omega$, if $K$ is a compact subset of $\Omega$.
The Lebesgue $\sigma$-algebra of $\Omega$ is denoted by $\mathcal{L}(\Omega)$ and the Borel $\sigma$-algebra of $\Omega$
is denoted by  $\mathcal{B}(\Omega)$.

\paragraph{Nets, subnets and clusterpoints:}

Let $K \csub \Omega$. If $(\lambda_{\varepsilon})_{\varepsilon \in ]0,1]}$ is a net with $\lambda_{\varepsilon} \in K \csub \R^n$, then we denote 
\begin{equation} \label{set_of_clusterpoints} \index{set of clusterpoints}
\cp{(\lambda_{\varepsilon})_{\varepsilon}}
\end{equation}
as the \emph{set of clusterpoints} of the net $(\lambda_{\varepsilon})_{\varepsilon}$. For sake of brevity we skip the index set $]0,1]$ in our notation. 

A \emph{subnet} of $(\lambda_{\varepsilon})_{\varepsilon}$ is defined by \index{subnet}
$$
(\lambda_{\tau(\varepsilon)})_{\varepsilon}
$$
for some map $\tau$ in the set 
\begin{equation} \label{big_tau}
\mathcal{T} := \left\{ \tau: ]0,1] \rightarrow ]0,1] \mid \lim_{\varepsilon \rightarrow 0} \tau(\varepsilon)=0 \right\}.
\end{equation}
This definition of a subnet is due to \cite[p.70]{kelley:55}.
It holds that
$\lambda \in \cp{(\lambda_{\varepsilon})_{\varepsilon}}$ if and only if there exists some 
$\tau \in \mathcal{T}$ with $\lambda = \lim_{\varepsilon\rightarrow 0} \lambda_{\tau(\varepsilon)}$.

Sometimes we will consider the subsequence of a net $(\lambda_{\varepsilon})_{\varepsilon}$, 
i.e. a sequence $(\lambda_{\varepsilon_j})_{j\in \N}$
where $\varepsilon_j \in ]0,1]$ tending to zero for $j\rightarrow \infty$.

According to \cite[Definition 2.2]{HO:04} a net $(\lambda_{\varepsilon})_{\varepsilon}$ in $\mathbb{R}$ 
is said to be \emph{slow scaled}\index{slow scaled}, if
\begin{eqnarray*}
\exists \varepsilon'\in]0,1]: \forall t \ge 0, \exists C_t>0 \rm{\ such\ that\ } |\lambda_{\varepsilon}|^t \le C_t \varepsilon^{-1}, \rm{\ for\ all\ } \varepsilon \in ]0,  \varepsilon']
\end{eqnarray*}
holds.

\paragraph{Convex sets and functions:} \label{convex_intro} 
A set $C \subseteq \R^n$ is called convex, if for all $x_1,x_2\in C$ it follows 
that $\lambda_1 x_1 +\lambda_2 x_2 \in C$, when $\lambda_1,\lambda_2 \ge 0$ and $\lambda_1+ \lambda_2 =1$.

Let $A$ be a closed subset of $\R^n$, then the set
$$
\ch{A} := \left\{ x \in \R^n \mid x= \sum_{k=1}^N \lambda_k a_k : a_k \in A, \lambda_k \ge 0, \sum_{k=1}^N \lambda_k = 1, N \in \N \right\} 
$$
is called \emph{convex hull} of $A$ \index{convex hull}. It is the smallest convex subset containing $A$.

A function $f: C \rightarrow \R \cup \{ +\infty\}$ is called convex, if 
$$
f(\lambda_1 x_1 +\lambda_2 x_2 ) \le \lambda_1 f(x_1) + \lambda_2 f(x_2),
$$
where $x_1,x_2 \in C$ and $\lambda_1,\lambda_2 \ge 0$ and $\lambda_1+ \lambda_2 =1$.

A function 
$h: \R^n \rightarrow \R \cup \{ +\infty\}$ with $h \not \equiv +\infty$,
which is convex, positively homogeneous and lower semi-continuous is called a \emph{convex supporting function}.\index{convex supporting function}
Due to \cite[Theorem 2.2.8]{Hoermander:94} or \cite[Theorem 13.2]{Rockafellar:70} it 
uniquely defines a non-empty, closed and convex set by
\begin{equation} \label{convex_set_by_supporting_function}
C := \{x \in \R^n \mid \forall w\in\R^n : \langle x,w \rangle \le h(w) \}
\end{equation}
and it holds that
$$
h(w) = \sup_{x \in C} \langle x, w \rangle,
$$
so there is a one-to-one correspondence between
the non-empty, convex and closed subsets of $\R^n$ and their convex supporting functions.

Some important examples:
The supporting function of the space $\R^n$ is defined by 
$h_1(0)=0$ and $h_1(w)=+\infty$ for $w\not\in 0$. 
If $x_0\in\R^n$, the supporting function of
the point set $\{x_0\}$ is defined by $h_2(w) := \langle x_0, w\rangle$. The closed unit ball $B_{\rho}(0)$
gives rise to the supporting function $h_3 (w) := \rho |w|$.

If $C_1$ and $C_2$ are convex subsets of $\R^n$ and $h_1, h_2$ their convex supporting functions,
the set 
$$
C_1 + C_2 =\{z \in \R^n \mid \exists x\in C_1, \exists y\in C_2 : z= x+y \}
$$ 
is convex and its supporting is defined by $h(w) := h_1(w) + h_2(w)$. It follows by the examples above, that the
closed ball around $x_0$ with radius $\rho$ has the supporting function $w \mapsto \langle x_0, w\rangle +\rho |w|$.

If $C_1, C_2$ are convex, closed subsets of $\R^n$ with the convex supporting functions $h_1,h_2$, then we have $C_1 \subseteq C_2$ if and only if 
$h_1(w) \le h_2(w)$ for all $w\in \R^n$. For a proof see \cite[Corollary 13.1.1]{Rockafellar:70}.

Let $A$ be some (non-empty) closed subset of $\R^n$. 
Then the function defined by $h(w) := \sup_{x \in A} \langle x, w \rangle$
is the supporting function of the convex hull ${\rm ch}{(A)}$, thus
$$
{\rm ch}{(A)} =  \{x \in \R^n \mid \forall w\in\R^n : \langle x,w \rangle \le h(w) \}.
$$

At this point we refer to the books \cite{Hoermander:94} or \cite{Rockafellar:70} for a detailed presentation of the theory of convex sets. 
See also the brief presentation in the beginning of Chapter 4.3 in \cite{Hoermander:V1}.

\paragraph{Function spaces and distributions:}
Let $\Omega$ denote an open subset of $\R^n$.

The letter $T$ will always be used for real number such that $T > 0$.
We often write $\Omega_T$ to mean $]0,T[ \times \R^n$ with closure $\ovl{\Omega_T} = [0,T] \times \R^n$.
The space $C^\infty(\ovl{\Omega})$ consists of smooth functions on $\Omega$ all whose derivatives have continuous extensions 
to $\ovl{\Omega}$. For any $s \in \R$ and $1 \leq p \leq \infty$ we have the Sobolev space $W^{s,p}(\R^n)$ 
(such that $W^{0,p} = L^p$), in particular $H^s(\R^n) = W^{s,2}(\R^n)$. 
Our notation for $H^s$-norms and inner products will be ${\Norm{.}}_s$ and $\dis{.}{.}_s$, 
in particular, this reads ${\Norm{.}}_0$ and $\dis{.}{.}_0$ for the standard $L^2$ notions.

We will also make use of the variants of Sobolev and $L^p$ spaces of functions on an interval 
$J \subseteq \R$ with values in a Banach space $E$, for which we will employ a notation as in $L^1(J;E)$, for example. 
(For a compendious treatment of the basic constructions we refer to \cite[Sections 24 and 39]{Treves:75}.) 
Furthermore, as usually the subscript 'loc' with such spaces will mean that upon multiplication by a smooth cutoff we have 
elements in the standard space. We occasionally write $AC(J;E)$ instead of $W^{1,1}_{\rm{loc}}(J;E)$ to emphasize the 
property of absolute continuity. 

The subspace of distributions of order $k$ on $\Omega$ ($k \in \N$, $k \geq 0$) will be denoted by $\D'^k(\Omega)$. 
We identify $\D'^0(\Omega)$ with the space of complex Radon measures $\mu$ on $\Omega$, i.e.,
 $\mu = \nu_+ - \nu_- + i(\eta_+ - \eta_-)$, where $\nu_\pm$ and $\eta_\pm$ are positive Radon measures on $\Omega$, i.e.,
 locally finite (regular) Borel measures. 

As an alternative regularity scale with real parameter $s$ we will often refer to the H\"older-Zygmund classes $C_\ast^s(\R^n)$  (cf.\ \cite[Section 8.6]{Hoermander:97}). In case $0 < s < 1$ the corresponding space comprises the continuous bounded functions $u$ such that there is $C > 0$ with the property that for all $x \neq y$ in $\R^n$ we have
$$
   \frac{|u(x) - u(y)|}{|x - y|^s} \leq C.
$$

\paragraph{Notation and basic outline of the theory of Colombeau algebras:}
Let us recall the basic construction: A Colomb\-eau 
(generalized) function on some open set $\Omega \subseteq \mathbb{R}^n$ is defined as equivalence class $(u_{\varepsilon})_{\varepsilon}$ of
nets of smooth functions $u_{\varepsilon} \in C^{\infty}(\Omega)$ and $\varepsilon\in]0,1]$ subjected to some asymptotic norm conditions (with respect to $\varepsilon$) for their derivatives on compact sets. We have the following:
\begin{enumerate}
\item \label{labelfirst} Moderate nets $\colmod{\Omega}$: $(u_{\varepsilon})_{\varepsilon}\in C^{\infty}(\Omega)^{]0,1]}$ such that, for all $K\csub\Omega$ and
$\alpha \in \mathbb{N}^n$, there exists $p\in \mathbb{R}$ such that
\begin{equation} \label{modest}
\sup_{x\in K} \|\partial^{\alpha} u_{\varepsilon}(x)\| = O(\varepsilon^{-p})\ \rm{as} \ \varepsilon\rightarrow 0.
\end{equation}
\item Negligible nets $\colneg{\Omega}$: $(u_{\varepsilon})_{\varepsilon}\in C^{\infty}(\Omega)^{]0,1]}$ such that, for all $K\csub\Omega$ and
all $q \in \mathbb{R}$ an estimate
\begin{equation*}
\sup_{x\in K} \|u_{\varepsilon}(x)\| = O(\varepsilon^{q})\ \rm{as} \ \varepsilon\rightarrow 0.
\end{equation*}
holds.
\item \label{labellast} $\colmod{\Omega}$ is a differential algebra with operations defined at fixed $\varepsilon$, $\colneg{\Omega}$ is an ideal
and  $\G{(\Omega)}:=\colmod{\Omega}/ \colneg{\Omega}$ is the special Colombeau algebra. 
\item If we replace the nets of smooth functions by nets of real numbers in (i)-(iii) we obtain the ring of generalized numbers $\widetilde{\mathbb{R}}$.
\item There are embeddings, $\sigma: C^{\infty}(\Omega) \hookrightarrow \G{(\Omega)}$ as subalgebra and \\ $\iota: \mathcal{D}'(\Omega) \hookrightarrow \G{(\Omega)}$ as linear space, commuting with partial derivatives.
\item $\Omega \rightarrow \G{(\Omega)}$ is a fine sheaf and $\colcomp{\Omega}$ denotes the subalgebra of elements with compact support; using a cut-off in a neighborhood of the support, one can always obtain representing nets with supports contained in a joint compact set.
\end{enumerate}

\subparagraph{Regular Colombeau functions:} The subalgebra $\mathcal{G}^{\infty}(\Omega)$ of regular Colomb\-eau (generalized) functions 
consists of those elements in $\G{(\Omega)}$ possessing representatives
such that the estimate (\ref{modest}) holds for a certain $m$ uniformly over all $\alpha \in \mathbb{N}^n$.
\subparagraph{Rapidly decreasing Colombeau functions:}
A Colombeau function in $\G{(\mathbb{R}^n)}$ is called \emph{rapidly decreasing}\index{rapidly decreasing} in the directions $\Gamma \subseteq S^{n-1}$, 
if it has a representative with the property that there exists a $N\in\mathbb{N}_0$, such that
for all $p \in \mathbb{N}_0$
\begin{equation*}
\sup_{(\lambda, \xi_1) \in \mathbb{R}^{+} \times \Gamma } (1+\lambda^2)^{p/2}|u_{\varepsilon}(\lambda \xi_1)| = O(\varepsilon^{-N}) \ \rm{as} \ \varepsilon\rightarrow 0
\end{equation*}
holds. 
\subparagraph{Fourier transform of a compactly supported Colombeau function:}
The \emph{Fourier transform}\index{Fourier transform} of $u \in \colcomp{\Omega}$ is the Colombeau function in $\G{(\mathbb{R}^n)}$ defined by 
\begin{equation*}
\mathcal{F}(u) := \left( \int_{\Omega} u_{\varepsilon}(x) e^{-i \langle x, \cdot \rangle} \m{x}\right)_{\varepsilon} + \colneg{\mathbb{R}^n},
\end{equation*}
where $(u_{\varepsilon})_{\varepsilon}$ is a representative of $u$ with joint compact support in $\Omega$.
\subparagraph{$\mathcal{G}^{\infty}$-wavefront set of a Colombeau function:}
If $v \in \colcomp{\Omega}$, we define the set $\Sigma(v) \subset S^{n-1}$ to be the complement of those points having open neighborhoods $\Gamma\subseteq S^{n-1}$ such that $\four{v}$ is rapidly decreasing in the directions $\Gamma$. $\Sigma(v)$ is a closed subset of $S^{n-1}$.
Now let $u \in \G{(\Omega)}$. Then we define the cone of irregular directions at $x_0$ by
\begin{eqnarray*}
\Sigma_{x_0}(u) = \bigcap_{\varphi \in C^{\infty}_c(\Omega), \varphi(x_0)\ne 0} \Sigma(\varphi \cdot u ). 
\end{eqnarray*}
Then the \emph{(generalized) wave front set}\index{wave front set, generalized} of $u$ is the set
\begin{eqnarray*}
\wf{u}:=\left\{ (x,\xi)\in \Omega \times S^{n-1} | \xi \in \Sigma_x(u) \right\}
\end{eqnarray*}
We denote the projection onto the first component by
\begin{eqnarray*}
\singsupp{u}:= \left\{ x\in \Omega | (x,\xi) \in WF(u)\right\}
\end{eqnarray*}
and call this \emph{(generalized) singular support}\index{singular support, generalized} of $u$.

\subparagraph{C-bounded Colombeau maps:}\label{colmap_cbounded}
Let $\Omega_1 \subseteq \R^n$ and $\Omega_2 \subseteq \R^m$ be open sets.
By $\colmap{\Omega_1}{\Omega_2}$ we denote the generalized maps $F\in\mathcal{G}(\Omega_1)^m$ with the property that 
$F$ is c-bounded on $\Omega_1$ (into $\Omega_2$), 
i.e. it possesses a representative $(F_{\varepsilon})_{\varepsilon}$ with $F_{\varepsilon} \in C^{\infty}(\Omega_1, \Omega_2)$ 
satisfying the condition
\begin{eqnarray*}
\forall K \csub \Omega_1, \exists K' \csub \Omega_2 : \textrm{such that }  F_{\varepsilon}(K ) \subseteq K' \textrm{ for all }\varepsilon \in ]0, 1].
\end{eqnarray*}

Our standard references for
the foundations and some applications of Colombeau's nonlinear theory of generalized functions 
are \cite{Colombeau:84,Colombeau:85,O:92,GKOS:01}. 
We will employ the so-called special variant of Colombeau algebras, denoted by $\G^s$ in \cite{GKOS:01},
although here we shall simply use the letter $\G$ instead.

For microlocal analysis in context of Colombeau theory we refer to \cite{DPS:98,NPS:98,Hoermann:99,HK:01,GH:05}.

\section{Set-valued Maps} \label{set_valued_map_section}

Here we introduce the notion of a \emph{set-valued map} (also referred to as \emph{correspondence}) and a few of their properties. 
This concept plays an important role in the field of mathematical economics and game theory. 
A main tool here is the generalization of Brower's theorem by Kakutani 
(see \cite{Kakutani:41}) for set-valued maps.
Nevertheless we are not intending to give a complete overview of the topic and restrict 
ourselves to a few basic results, which will be needed in the existence proof 
for ordinary differential inclusions.

The interested reader is referred to \cite{KleinThompson:84} and  \cite{AubinFrankowska:90}.

Let $\Omega, \Omega_1$ be arbitrary non-empty subsets of $\R^n$ and let $\Omega_2$ be a non-empty subset of $\R^m$ .

Throughout this section we define the projection maps of the product space of $\Omega_1 \times \Omega_2$
onto $\Omega_1$ resp. $\Omega_2$ by
$$
\pi_1: \Omega_1 \times \Omega_2 \rightarrow \Omega_1,  \  (x,y) \mapsto x \quad \text{resp.} \quad
\pi_2: \Omega_1 \times \Omega_2 \rightarrow \Omega_2,  \ (x,y) \mapsto y.
$$

\begin{definition} 
We introduce the following notions:
\begin{enumerate}
\item $\mathcal{P}_0(\Omega)$ denotes the set of non-empty subsets of $\Omega$,
\item $\mathcal{F}_0(\Omega)$ denotes the set of non-empty, closed subsets of $\Omega$, and
\item $\mathcal{K}_0(\Omega)$ denotes the set of non-empty, closed and convex subsets of $\Omega$.
\end{enumerate}

Note that for $M,N \in \mathcal{P}_0(\Omega)$ we have 
$$
\lambda \cdot M := \{y \in \Omega \mid \exists x\in M: y = \lambda \cdot x \} \in \mathcal{P}_0(\Omega)
$$ for $\lambda\in \R$, and
$$
M + N := \{ y\in \Omega \mid \exists x_1 \in M, x_2 \in N : y = x_1+ x_2 \} \in \mathcal{P}_0(\Omega).
$$
It follows that $\mathcal{P}_0(\Omega)$ is a linear space and $\mathcal{F}_0(\Omega)$ is a linear subspace of $\mathcal{P}_0(\Omega)$.
Furthermore $\mathcal{K}_0(\Omega)$ is a linear subspace of $\mathcal{F}_0(\Omega)$.

For any $M,N \in \mathcal{F}_0(\Omega)$ we put 
$$
\rho_{+}(M,N) := \inf{ \{ \rho \in [0, \infty] :  N \subseteq B_{\rho}(M) \}}, \quad \rho_{-}(M,N):= \rho_{+}(N,M),
$$
then we define the (extended) Hausdorff metric by
$$
d(M,N):= \max{(\rho_{+}(M,N),\rho_{-}(M,N))}.
$$ 
for all $M,N \in \mathcal{F}_0(\Omega)$. Note that in case of $\mathcal{P}_0(\Omega)$ the map $d:\mathcal{P}_0(\Omega)\times \mathcal{P}_0(\Omega)$ is no longer a metric, since
$d(B_1(0), B_1(0)^{\circ})=0$, but $B_1(0) \backslash B_1(0)^{\circ} = S^{n-1} \ne \emptyset$.

We put $|M| := d(M,\{0\}) = \sup_{x\in M} |x|$.
\end{definition}

\begin{proposition}
The linear space $(\mathcal{F}_0(\R^n), d)$ is a complete metric space. Its linear subspace 
$(\mathcal{K}_0(\R^n), d\mid_{\mathcal{K}_0(\R^n) \times \mathcal{K}_0(\R^n)} )$ is again complete.
\end{proposition}

\begin{proof}
We refer to \cite[Proposition 4.3.10]{KleinThompson:84} and \cite[Proposition 4.3.11]{KleinThompson:84}. 
\end{proof}

\begin{remark}
In the case of $\mathcal{K}_0(\Omega)$ we can write the Hausdorff metric in terms of convex supporting functions.
Recall from the introduction that each $M \in \mathcal{K}_0(\Omega)$ has a uniquely determined convex supporting function $h:\Omega \rightarrow \R\cup\{+\infty\}$, 
defined by
$$
h(w) = \sup_{a \in M} \langle a, w \rangle
$$ 
and the set $M$ satisfies 
$$
M = \{ a \in \Omega \mid  w\in \R^n : \langle a, w \rangle \le h(w) \}.
$$
Then we can write $\rho_{\pm}(M,N)$ and the Hausdorff metric $d(M,N)$ in terms of the convex supporting function by
$\rho_{+}(M,N) = \sup_{w\in\R^n} \max{\{h_{N}(w/|w|) - h_{M}(w/|w|),0\}}$ and $\rho_{-}(M,N) = \sup_{w\in\R^n} \max{\{ h_{M}(w/|w|) - h_{N}(w/|w|) ),0\}} $, thus 
$$
d(M,N) = \max{(\rho_{+}(M,N), \rho_{-}(M,N))} = \sup_{w\in \R^n}|h_{M}(w/|w|) - h_{N}(w/|w|)|. 
$$
and we have $|M|:= d(M,\{0\}) = \sup_{w\in \R^n}|h_{M}(w/|w|)|$.
\end{remark}

\begin{proposition}
Let $(M_{\iota})_{\iota \in ]0,1]}$ be a net in $\mathcal{F}_0(\Omega)$ such that for some $K\csub \Omega$,
it holds that $M_{\iota}\subseteq K$ for all $\iota\in ]0,1]$. If $M \in \mathcal{F}_0(\Omega)$, then the following statements hold:
\begin{itemize}
\item $\lim_{\iota \rightarrow 0} \rho_{+}(M_{\iota}, M) \rightarrow 0 \quad \Rightarrow \quad M \supseteq \limsup_{\iota \rightarrow 0} M_{\iota}$.
\item $\lim_{\iota \rightarrow 0} \rho_{-}(M_{\iota}, M) \rightarrow 0 \quad \Rightarrow \quad M \subseteq \liminf_{\iota \rightarrow 0} M_{\iota} $.
\item $\lim_{\iota \rightarrow 0} d(M_{\iota}, M) \rightarrow 0 \quad \Leftrightarrow \quad M=\liminf_{\iota \rightarrow 0} M_{\iota} = \limsup_{\iota \rightarrow 0} M_{\iota}$.
\end{itemize}
\end{proposition}
Recall that $\limsup_{\iota \rightarrow 0} M_{\iota}  :=\bigcap_{\iota' \in ]0,1]} \bigcup_{\iota \in ]0,\iota']} M_{\iota}$
resp. $\liminf_{\iota \rightarrow 0}M_{\iota} := \bigcup_{\iota' \in ]0,1]} \bigcap_{\iota \in ]0,\iota']} M_{\iota}$, denote the set-theoretic
limes superior resp. limes inferior.

\begin{proof}
Follows by \cite[Proposition 4.2.2]{KleinThompson:84} and \cite[Theorem 3.3.11]{KleinThompson:84}.
\end{proof}

\begin{definition} \label{set_valued_map_def} \index{set-valued map} 
A map defined by
$$
F: \Omega_1 \rightarrow \mathcal{P}_0(\Omega_2), \quad x \rightarrow F_{x},
$$
is called a \emph{set-valued map}. 
If $F$ maps into $\mathcal{F}_0(\Omega_2)$ resp. $\mathcal{K}_0(\Omega_2)$ we call $F$ \emph{closed-valued} resp. \emph{closed- and convex-valued}.

If $X\subseteq \Omega_1$ we use the notation $F_X := \bigcup_{x\in X} F_x$.

Putting
$$
\ggraph{F} := \{ (x,y) \in \Omega_1 \times \Omega_2 \mid y \in F_x \},
$$
we say $F$ is \emph{closed}, 
if $\ggraph{F}$ is closed in the relative topology of $\Omega_1 \times \Omega_2$. \index{graph!of a set-valued map}
Note that we use the upper-case spelling $\ggraph{F}$ to distinguish it from the classical graph $\graph{f}$ of a continuous function $f$.

$F$ is called \emph{proper}, if $\ggraph{F} \cap \Omega_1 \times K \csub \Omega_1 \times \Omega_2$ for all $K\csub \Omega_2$. \index{proper}

$F$ is called \emph{upper semi-continuous at} $x$, if for any $W \subseteq \Omega_2$ being an open neighborhood
of $F_x$ there exists some neighborhood $X \subseteq \Omega_1$ of $x$ such that
$F_X \subseteq W$ holds. 
We say $F$ is \emph{upper semi-continuous}, if it is upper semi-continuous for all $x\in \Omega_1$. \index{upper semi-continuous}

$F$ is called \emph{locally bounded}, if for all $x$ there exists a neighborhood $X\subseteq \Omega_1$ such that
$F_X$ is bounded. \index{set-valued map!locally bounded}

$F$ is called \emph{bounded}, if the \emph{range} of $F$ denoted
by $\ran{F} := F_{\Omega_1}$ is a bounded subset of $\Omega_2$. 
\end{definition}

\begin{definition} \label{selection} \index{selection}
Let $f:\Omega_1\rightarrow\Omega_2$ be a (single-valued) function, then we say $f$ is a \emph{selection} of 
the set-valued map $F$, if $f(x) \in F_x$ for all $x\in \Omega_1$.
\end{definition}

\begin{definition} \label{supporting_function} \index{supporting function!of a set-valued map}
Let $F: \Omega_1 \rightarrow \mathcal{F}_0(\Omega_2)$ be a set-valued map. Then the function
$$
H_F :  \Omega_1 \times \R^m \rightarrow \R \cup \{+\infty\} ,\quad (x,w) \mapsto h_{x}(w), 
$$ 
where $h_{x}(w)= \sup_{a \in F_x}\langle a, w\rangle$ is called the \emph{supporting function} of the set-valued map $F$.
It is the convex supporting function (cf. Introduction) of the convex hull of $F_x$. In particular, if $F$ is convex-valued
it holds that $h_x$ is the convex supporting function of $F_x$, which implies 
$$
F_x = \{ a\in \Omega_2 \mid \forall w\in \R^m : \langle a, w \rangle \le H_F(x,w)\}.
$$

If $F: \Omega_1 \rightarrow \mathcal{F}_0(\Omega_2)$ is a set-valued map (allowing non-convex values), then
we define the \emph{convex hull} of $F$ by        \index{convex hull!of a set-valued map}
$$
\ch{F}: \Omega_1 \rightarrow \mathcal{K}_0(\Omega_2), \quad x \mapsto \ch{F_x}. 
$$
It holds that $\ch{F}_x = \{ a \in \Omega_2 \mid \forall w\in \R^m : \langle a, w \rangle \le H_F(x,w)\}$.
\end{definition}

\begin{example} \label{supportfunction_heaviside_example}
Let $F: \R \mapsto \mathcal{K}_0(\R)$ be the set-valued map defined by  
\begin{equation*}
F_x := \left\{ \begin{array}{cl} 
\{0\} & x \in ]-\infty,0[ \\
\left[ \alpha, \beta \right] & x = 0\\
\{1\} & x\in ]0, +\infty[ \quad,
\end{array} 
\right.
\end{equation*}
then $F$ is upper semi-continuous if and only if $\alpha \le 0$ and $\beta \ge 1$.
Obviously $F$ is bounded. Its supporting function $H(x,w)$ is defined by
$$
H(x,+1) := \left\{ \begin{array}{cl} 
0 & x \in ]-\infty,0[ \\
\beta  & x = 0\\
1 & x\in ]0, +\infty[ 
\end{array} 
\right. 
$$
and 
$$
H(x,-1) := \left\{ \begin{array}{cl} 
0 & x \in ]-\infty,0[ \\
-\alpha  & x = 0\\
-1 & x\in ]0, +\infty[ \quad ,
\end{array} 
\right. 
$$
which are upper semi-continuous in $x$ if and only if $\alpha \le 0$ and $\beta \ge 1$.
\end{example}

\begin{lemma} \label{corr_approx_lemma}
Let $F: \Omega_1 \rightarrow \mathcal{F}_0(\Omega_2)$ be a closed set-valued map, then it holds that
$$
\bigcap_{\varepsilon \in ]0,1]} \bigcup_{y\in B_{\varepsilon}(x)\cap \Omega_1} F_{y} = F_x.
$$
\end{lemma}

\begin{proof}
We put $M:=\bigcap_{\varepsilon \in ]0,1]} \bigcup_{y\in B_{\varepsilon}(x)\cap \Omega_1} F_y \subseteq \Omega_2$,
 then $M \supseteq F_{x}$ follows. We are going to show that $M = F_{x}$. We prove by contradiction: 
Assume there exists some $a' \in M/ F_{x}$, 
then it follows that $(x,a')$ is contained in the open set  $(\Omega_1 \times \Omega_2) \backslash \ggraph{F}$. 
So there exists some neighborhood $X \subseteq \Omega_1$ of $x$ and $W \subseteq \Omega_2$ of $a'$ such that 
$(X\times W) \bigcap \ggraph{F} = \emptyset$.
In particular $a' \not \in F_y$ for $y\in X$, so $\bigcup_{y\in B_{\varepsilon}(x)\cap \Omega_1}  F_{y} \subseteq W^c$ for $\varepsilon$ small enough.
It immediately follows that 
$$
M = \bigcap_{\varepsilon \in ]0,1]} \bigcup_{y\in B_{\varepsilon}(x)\cap \Omega_1} F_{y}  \subseteq  W^c,
$$ 
a contradiction to $a' \in M$, so the statement follows.
\end{proof}

\begin{proposition} \label{closed_corr_proposition}
Let $F: \Omega_1 \rightarrow \mathcal{P}_0(\Omega_2)$ be a set-valued map, then the following statements
are equivalent:
\begin{enumerate}
\item $F$ is closed and locally bounded.
\item $\ggraph{F} \cap K \times \Omega_2 \csub \Omega_1 \times \Omega_2$ for all $K\csub \Omega_1$.
\end{enumerate}
\end{proposition}

\begin{proof}
$(i) \Rightarrow (ii):$ 
Let $K\csub \Omega_1$, then $\ggraph{F} \cap K \times \Omega_2 \subseteq K \times F_K$ is closed and bounded 
($F_K$ is bounded since $F$ is locally bounded). 
$(ii) \Rightarrow (i):$ Let $X$ be some bounded subset of $\Omega_1$, then there exists some 
compact set $K$ with $X\subseteq K$. It follows that $F_X =\pi_2(\ggraph{F} \cap X \times \Omega_2) \subseteq 
 \pi_2(\ggraph{F} \cap K \times \Omega_2)$ is bounded, thus $F$ is locally bounded. 
Let $(x_{\varepsilon},a_{\varepsilon}) \in \ggraph{F}$ be a net converging to some $(x,a) \in \Omega_1 \times \Omega_2$,
then there exists some companeighborhoodood $K$ of $x$. Since $\ggraph{F} \cap K \times \Omega_2 \csub \Omega_1 \times \Omega_2$
and $x_{\varepsilon} \in K$ for $\varepsilon$ small enough, we obtain $(x,a) \in \ggraph{F} \cap K \times \Omega_2$, 
thus $F$ is closed.
\end{proof}

\begin{proposition} \label{proper_corr_proposition}
Let $F: \Omega_1 \rightarrow \mathcal{P}_0(\Omega_2)$ be a proper set-valued map. Then for all $A \subseteq \Omega_1$ closed, we have that
$F_A$ is a closed subset of $\Omega_2$.
\end{proposition}

\begin{proof}
Since $F$ is proper we have that $\ggraph{F} \cap \Omega_1 \times K$ is compact for all $K\csub \Omega_1$,
thus $\pi_2 \mid_{\ggraph{F}}$ is proper and continuous, implying 
that $F_A = \pi_2(\ggraph{F} \cap A \times \Omega_2)$ is closed, if $A$ is a closed subset of $\Omega_1$.
\end{proof}

\begin{proposition} \label{usc_corr_proposition}
Let $F: \Omega_1 \rightarrow \mathcal{P}_0(\Omega_2)$ be a set-valued map, then $F$ is upper semi-continuous if and only
if $\pi_1( \ggraph{F} \cap (\Omega_1 \times A))$ is closed in $\Omega_1 \times \Omega_2$ for all closed subsets $A$ of $\Omega_2$.  
\end{proposition}

\begin{proof}
If $F$ is upper semi-continuous, it holds that for any open neighborhood $W$ of $F_x$, there exists some open neighborhood
$X$ of $x$, such that $F_y \subseteq W$. This is equivalent to the statement:
For all open sets $W \subseteq \Omega_2$ we have that $U(W):=\{x \in \Omega_1 \mid \forall a \in F_x: a \in W\}$ is open.  
Thus $\Omega_1 \backslash U(W) = \{x \in \Omega_1 \mid \exists a \in F_x: a \in W^{c} \} =
\pi_1(\ggraph{F} \cap (\Omega_1 \times W^c))$ is closed in $\Omega_1$, showing that $F$ being upper semi-continuous is equivalent to 
$\pi_1( \ggraph{F} \cap (\Omega_1 \times A))$  being closed, whenever $A \subseteq \Omega_2$ is closed. 
\end{proof}

\begin{example} \label{continuous_map_example}
Let $f:\Omega_1 \rightarrow \Omega_2$ be a (single-valued) function, then the set-valued map 
$F:\Omega_1 \mapsto \mathcal{K}_0(\Omega_2),\  x\mapsto \{f(x)\}$ is upper 
semi-continuous if and only if $f$ is continuous. Note that $\ggraph{F} =\graph{f}$.
Note that the equivalence in Proposition \ref{usc_corr_proposition} translates to the classical equivalence: $f$ is continuous if and only if $f^{-1}(A)$
is closed whenever $A$ is closed. Since $F_x=  \{f(x)\}$ we have
$\pi_1( \ggraph{F} \cap (\Omega_1 \times A)) = \{ x \in \Omega_1 \mid f(x) \in A \} = f^{-1}(A)$.
Furthermore we observe that $F_A := \pi_2(\ggraph{F} \cap A \times \Omega_2) =  f(A)$ is closed whenever $A\subseteq \Omega_1$ is closed,
if and only if $f$ is a closed map.
\end{example}

\begin{theorem} \label{set_valued_map_main_theorem}
Let $F: \Omega_1 \rightarrow \mathcal{F}_0(\Omega_2)$ be a set-valued map, then the following statements hold:
\begin{enumerate}
\item If $F$ is closed and locally bounded, then $F$ is upper semi-continuous and locally bounded.
\item If $F$ is upper semi-continuous, then $F$ is closed.
\item If $F$ is convex-valued, then $F$ is upper semi-continuous and locally bounded if and only if its supporting function 
$H$ has the property that $x\mapsto H(x,w)$ is upper semi-continuous for all $w\in \R^m$.
\end{enumerate}
\end{theorem}

\begin{proof}
$(i)$: 
Assume that $F$ is not upper semi-continuous at $x$. Then there exists an open neighborhood $W$ of $F_x$, 
such that any neighborhood $X$ of $x$ contains some $y$
with $F_y \cap W^c \ne \emptyset$. 
We prove by contradiction: Assume that $F$ is closed and locally bounded. 

Picking the family of neighborhoods $(B_{\varepsilon}(x) \cap \Omega_1)_{{\varepsilon} \in ]0,1]}$ 
we can find
$x_{\varepsilon} \in B_{\varepsilon}(x) \cap \Omega_1$ with $F_{x_\varepsilon}\cap W^c \ne \emptyset$, thus
we can find some net $(x_{\varepsilon},a_{\varepsilon})_{{\varepsilon}\in ]0,1]}$ 
with $a_{\varepsilon} \in F_{x_{\varepsilon}}\cap W^c$. 
By Proposition \ref{closed_corr_proposition} we have that 
$\bigcup_{y\in B_1(x)} \{y\} \times F_{y}$ is a compact subset of $\Omega_2$, implying that
$$
\emptyset \ne \cp{(x_{\varepsilon},a_{\varepsilon})_{{\varepsilon}\in ]0,1] }} \subseteq \left(\bigcup_{y\in B_1(x)} \{y\} \times F_{y}\right),
$$
and since $x_{\varepsilon} \rightarrow x$ we obtain $\cp{(a_{\varepsilon})_{{\varepsilon}\in ]0,1]}} \subseteq F_{x}$.
 But $a_{\varepsilon} \in W^c$ implies $\cp{(a_{\varepsilon})_{{\varepsilon}\in ]0,1]}} \subseteq W^c$, 
thus $F_{x}\cap W^c \supseteq \cp{(a_{\varepsilon})_{{\varepsilon}\in ]0,1]}} \ne \emptyset$ which contradicts 
$F_{x} \subseteq W$. So $F$ cannot be closed and locally bounded.

$(ii)$:
If $F$ is not closed in $\Omega_1 \times \Omega_2$ there exist a net 
 $\{(x_{\varepsilon}, a_{\varepsilon})_{\varepsilon} \} \subseteq \ggraph{F}$ 
converging to some $(x,a) \in \Omega_1 \times \Omega_2$ such that $a\not \in F_x$. 
Let $W$ be an open neighborhood of $F_x$
such that $a \not \in \overline{W}$.
We prove by contradiction: Assume that $F$ is upper semi-continuous at $x$, then there exists a neighborhood $X \subseteq \Omega_1$ of $x$ such that
$F_y \subseteq W$ for all $y\in X$.

It follows that $F_{x_{\varepsilon}} \subseteq W$ (for small $\varepsilon$) which implies 
$a=\lim_{\varepsilon \rightarrow 0} a_{\varepsilon} \in \overline{W}$, yielding a contradiction to the 
initial choice of $W$.

$(iii)$: Assume $F$ is upper semi-continuous and locally bounded. By (ii) $F$ closed.  
Let $w\in \R^n$ be arbitrary, then
$$
\limsup_{y\rightarrow x} H(y,w) = \inf_{\delta \in ]0,1]}{\sup_{y \in B_{\delta}(x)\cap \Omega_1} H(y,w)} =
\inf_{\delta \in ]0,1]} { \sup_{y \in B_{\delta}(x)\cap \Omega_1} \sup_{a \in F_{y}} \langle a, w \rangle}  
$$
and putting $K_{\delta}:= \bigcup_{y \in B_{\delta}(x)\cap \Omega_1} \{y\} \times F_{y}$, which is a descending family of sets, we can estimate
$$
\inf_{\delta \in ]0,1]} { \sup_{y \in B_{\delta}(x)\cap \Omega_1} \sup_{a \in F_{y}} \langle a, w \rangle} \le 
\inf_{\delta \in ]0,1]} \sup_{(z,a) \in K_{\delta} } \langle a,w \rangle.
$$
Lemma \ref{corr_approx_lemma} implies that $\bigcap_{\delta\in ]0,1]} K_{\delta} = \{x\} \times F_{x}$ and 
Lemma \ref{inf_descending_family_of_sets} yields
$$
\inf_{\delta \in ]0,1]} \sup_{(z,a) \in K_{\delta} } \langle a,w \rangle = \sup_{(z,a) \in \bigcap_{\delta \in ]0,1]}K_{\delta} } \langle a,w \rangle =\sup_{a \in F_{x}} \langle a,w \rangle = H(x,w)<\infty,
$$
which shows that $x\mapsto H(x,w)$ is upper semi-continuous for all $w\in \R^n$. 

Assume $x\mapsto H(x,w)$ is upper semi-continuous:
 If $\{(x_{\varepsilon}, a_{\varepsilon})_{\varepsilon}\} \subseteq \bigcup_{y\in \Omega_1} \{y\} \times F_{y}$ is a net converging to some $(x,a) \in \Omega_1 \times \Omega_2$, then
we have that $\langle a_{\varepsilon}, w\rangle \le H(x_{\varepsilon},w)$ for all $w\in\R^n$ and $\varepsilon\in]0,1]$. For $w\in \R^n$ fixed, it holds that
$$
\langle a, w \rangle = \lim_{\varepsilon \rightarrow 0} \langle a_{\varepsilon}, w\rangle 
\le \limsup_{\varepsilon\rightarrow 0} H(x_{\varepsilon},w) \le  H(x,w) <\infty,
$$
for all $w\in \R^n$ because $x \mapsto H(x,w)$ is upper semi-continuous, thus $a \in F_{x}$. It follows that 
$\bigcup_{x\in \Omega_1} \{x\} \times F_{x}$ is a closed subset of $\Omega_1 \times \Omega_2$. 
Let $K\csub \Omega_1$ then $$
\sup_{x \in K} |F_x| = \sup_{x \in K} \sup_{a \in F_x} |a| \le \sup_{x \in K} \sup_{w\in \R^m} H(x, w/|w|) \le \sup_{(x,w)\in K \times S^{m-1}} H(x, w) =  
H(x_0, w_0)<\infty,
$$
for some $(x_0,w_0)\in K \times S^{m-1}$. By \cite[Proposition]{Hoermander:94} the convexity of the map $w \mapsto H(x, w)$ implies its continuity, so $(x,w) \mapsto H(x, w)$ is an upper semi-continuous function and Lemma \ref{supremum_uppersemicontinuous_functions} states that an upper semi-continuous function attains its supremum on a compact set.
\end{proof}

\begin{example} \label{wavefront_map_example}
Let $u \in \mathcal{D}'(\Omega_1)$ be a distribution, then the wave front set of $u$
defines a closed set-valued map by
$$
\Sigma :\singsupp{u} \rightarrow \mathcal{K}_0(S_{n-1}), 
\quad x \mapsto \Sigma_x(u),
$$
where $\Sigma_x(u)$ denotes the cone of irregular directions of $u$ at $x$. 
Since locally boundedness is implied by the compact range of $\Sigma$ we obtain that $\Sigma$ is upper semi-continuous,
i.e. if $x\in \singsupp{u}$, then for any open neighborhood $W$ of $\Sigma_x$, 
there exists an open neighborhood $X \subseteq \Omega_1$ of $x$ such that $\Sigma_{X\cap \singsupp{u}} \subseteq W$.

In the same way the $\mathcal{G}^{\infty}-$wave front set of a Colombeau function $u \in \G{(\Omega_1)}$
defines an upper semi-continuous set-valued map $\Sigma^{\infty}:\singsupp{u} \rightarrow \mathcal{K}_0(S_{n-1})$.
\end{example}

\begin{definition} \label{set_valued_map_inverse} \index{invertible!set-valued map} \index{inverse map!set-valued map}
Let $F:\Omega_1 \mapsto \mathcal{P}_0(\Omega_2)$ be a set-valued map. 

We define the \emph{inverse graph} by 
$$
\ggraph{F}^{-1} = \{(y,x) \in \ran{F} \times \Omega_1 \mid (x,y) \in \ggraph{F}\}.
$$

% We call $F$ \emph{invertible} if 
% $\pi_2(\ggraph{F}^{-1} \cap(\{y\} \times \Omega_1)) \in \mathcal{F}_0(\Omega_1)$ for all $y\in \ran{F}$.

We define the \emph{inverse map} of $F$ by
$$
F^{-1}: \ran{F} \rightarrow \mathcal{P}_0(\Omega_1),
 \quad y \mapsto \pi_2(\ggraph{F}^{-1} \cap(\{y\} \times \Omega_1)) \in \mathcal{P}_0(\Omega_1).
$$

We say $F$ is \emph{$\mathcal{F}_0(\Omega)$-invertible}\index{invertible, set-valued map} if $\pi_2(\ggraph{F}^{-1} \cap(\{y\} \times \Omega_1)) \in \mathcal{F}_0(\Omega_1)$
for all $y\in \ran{F}$. If $F$ is $\mathcal{F}_0$-invertible it holds that $F^{-1}$ maps $\ran{F}$ into $\mathcal{F}_0(\Omega_1)$.
\end{definition}

\begin{theorem} \label{inverse_set_valued_theorem}
Let $F:\Omega_1 \mapsto \mathcal{F}_0(\Omega_2)$ be a set-valued map, then the following statements hold:
\begin{enumerate}
\item It holds that
\begin{eqnarray*}
\pi_1(\ggraph{F}^{-1} \cap (Y \times X)) =  \pi_2(\ggraph{F} \cap (X \times Y)) \\
\pi_2(\ggraph{F}^{-1} \cap (Y \times X)) =  \pi_1(\ggraph{F} \cap (X \times Y))
\end{eqnarray*}
for all $X\subseteq \Omega_1, Y \subseteq \Omega_2$ and $\ggraph{F^{-1}} = \ggraph{F}^{-1}$.

\item If $F$ is $\mathcal{F}_0$-invertible, then $F^{-1}$ is upper semi-continuous if and only if 
      $F_A$ is closed whenever $A$ is closed. 
\item If $F$ is upper semi-continuous, then $F$ is $\mathcal{F}_0$-invertible and $F^{-1}$ satisfies that $(F^{-1})_A$ is closed whenever $A$ is closed. 
\item If $F$ is upper semi-continuous and locally bounded, then $F$ is $\mathcal{F}_0$-invertible and $F^{-1}$ is proper.
\item If $F$ is proper, then $F$ is $\mathcal{F}_0$-invertible and $F^{-1}$ is upper semi-continuous and locally bounded.
\end{enumerate}
\end{theorem}

\begin{proof}
(i) 
Observe that 
\begin{multline*}
\pi_i(\ggraph{F}^{-1}\cap (Y \times X)) =  \pi_i(\{(y,x) \in \Omega_2 \times \Omega_1 \mid (x,y) \in \ggraph{F}\cap (X \times Y)\} )\\ =
\pi_j(\ggraph{F}\cap (X \times Y)), \quad i,j \in \{1,2\}, i\ne j
\end{multline*}
for all $X\subseteq \Omega_1$ and $Y\subseteq \Omega_2$.

If $(y,x) \in \ggraph{F^{-1}}$ it holds that 
$$
x\in \pi_2(\ggraph{F}^{-1} \cap(\{y\} \times \Omega_1)),
$$
thus  $(y,x) \in \ggraph{F}^{-1}$ implying $(x,y) \in \ggraph{F}$.

If $(x,y) \in \ggraph{F}$ we have $y\in \pi_2(\ggraph{F}\cap \{x\} \times \Omega_2) = \pi_1(\ggraph{F}^{-1}\cap \Omega_2 \times \{x\})$, thus $x\in \pi_2( \ggraph{F}^{-1} \cap \{y\} \times \Omega_2 )=(F^{-1})_y$ and $(y,x) \in \ggraph{F}^{-1}$.

(ii) 
Since $F$ is $\mathcal{F}_0$-invertible it follows by (i) that
$$
\pi_1(\ggraph{F^{-1}} \cap (\Omega_2 \times A)) = 
\pi_2(\ggraph{F} \cap A \times \Omega_1) = F_A
$$
and Proposition \ref{usc_corr_proposition} yields that $F^{-1}$ is upper semi-continuous if and only if 
$F_A$ is closed whenever $A$ is closed.

(iii) Proposition \ref{usc_corr_proposition}
yields that $F_A = \pi_1(\ggraph{F} \cap (\Omega_1 \times A))$ is closed, whenever $A \subseteq \Omega_2$ is closed.
By (i) we conclude that $\pi_2(\ggraph{F}^{-1} \cap (A \times \Omega_1)) = \pi_1(\ggraph{F} \cap (\Omega_2 \times A))$
is closed whenever $A$ is closed.
We obtain $(F^{-1})_y =\pi_2(\ggraph{F}^{-1} \cap (\{y\} \times \Omega_1)) \in \mathcal{F}_0(\Omega_1)$ for all $y\in \ran{A}$, so
$A$ is $\mathcal{F}_0$-invertible. Furthermore $(F^{-1})_A$ is closed whenever $A\subseteq \Omega_2$ is closed.

(iv) By (iii) we obtain that $F$ is $\mathcal{F}_0$-invertible and since 
$\pi_2(\ggraph{F^{-1}} \cap ( K \cap \Omega_2)) \csub \Omega_1 \times \Omega_2$ for all $K\csub \Omega_1$ by Proposition \ref{closed_corr_proposition}
we follow by (i) that
$$
\pi_1(\ggraph{F^{-1}} \cap (\Omega_2 \times K)) = \pi_2(\ggraph{F^{-1}} \cap ( K \cap \Omega_2)) \csub \Omega_1 \times \Omega_2,
$$
so $F^{-1}$ is proper.

(v) If $F$ is proper, then we have $\pi_1(\ggraph{F} \cap (\Omega_1 \times K)) \csub \Omega_1 \times \Omega_2$ 
for all $K\csub \Omega_2$, thus
$$
\pi_2(\ggraph{F^{-1}}  \cap (K \times \Omega_2))=\pi_1(\ggraph{F} \cap (\Omega_1 \times K))
$$
is compact. In particular if we set $K=\{y\}$. So $(F^{-1})_y \in \mathcal{F}_0(\Omega_1)$
and Proposition \ref{closed_corr_proposition}
implies that $F^{-1}$ is closed and locally bounded. 
\end{proof}

\begin{example} \label{inverse_image_example}
Let us reconsider Example \ref{continuous_map_example}, where the set-valued map $F: x \mapsto \{f(x)\}$ was obtained from a 
continuous function $f: \Omega_1 \rightarrow \Omega_2$. 
Then the inverse map $F^{-1}$ is defined by the inverse images of $f$, i.e.  $(F^{-1})_y = f^{-1}(\{y\})$. 
Theorem \ref{inverse_set_valued_theorem} yields that $F^{-1}$ is a set-valued map such that $(F^{-1})_A$ is closed whenever 
$A$ is closed. In addition $F^{-1}$ is upper semi-continuous if and only if $f$ is a closed function, i.e. $f(A)\subseteq \Omega_2$ is closed for all $A\subseteq \Omega_1$ closed.
\end{example}

\begin{theorem} \label{set_valued_composition_theorem}
Let $F:\Omega_1 \rightarrow \mathcal{F}_0(\Omega_2)$
and $G:\Omega \rightarrow \mathcal{F}_0(\Omega_1)$ be set-valued maps, such that one of 
the following properties holds:
\begin{itemize}
\item  $G$ is single-valued.
\item  $F$ is upper semi-continuous and locally bounded, and $G$ is locally bounded.
\item $F_A$ is closed whenever $A \subseteq \Omega$ is closed.
\end{itemize}
then we define the \emph{composition map}\index{composition, of set-valued maps}
$$
(F\circ G):\Omega \rightarrow \mathcal{F}_0(\Omega_2), \quad x \mapsto (F\circ G)_x. 
$$
where $(F\circ G)_x  := (F)_{G_x}$. It satisfies the following properties:
\begin{enumerate}
\item If $F_A$ is closed whenever $A \subseteq \Omega_1$ is closed and 
	  $G_B$ is closed whenever $B \subseteq \Omega$ is closed, then
	  $(F\circ G)_B$ is closed whenever $B \subseteq \Omega$ is closed.
\item If $F$ and $G$ are upper semi-continuous, then $(F\circ G)$ is upper semi-continuous.
\item If $F$ and $G$ are locally bounded, then $(F\circ G)$ is locally bounded.
\end{enumerate}
\end{theorem}

\begin{proof}
Since $G_x \in \mathcal{F}_0(\Omega_1)$ is non-empty, there exists some $y\in G_x$ and since $F_y \in \mathcal{F}_0(\Omega_2)$ 
we have that $(F\circ G)_x $ is a non-empty subset of $\Omega_2$. 

If $G$ is single-valued it immediately follows that $(F\circ G)_x = (F)_{G_x} \in \mathcal{F}_0(\Omega_2)$. 

In the case where $F$ is upper semi-continuous it 
follows by Theorem \ref{set_valued_map_main_theorem} (ii) that $F$ is closed. 
Since $G$ is locally bounded it follows that
$G_x$ is compact.
By Proposition \ref{closed_corr_proposition} (ii) we have that $\pi_2(\ggraph{F} \cap G_x \times \Omega_2) = \bigcup_{y\in G_x} F_y$ is a compact subset of $\Omega_2$. 
This yields $(F\circ G)_x \in \mathcal{F}_0(\Omega_2)$. 

In the case where $F_A$ is closed whenever $A\subseteq \Omega_1$ is closed, we have that $G_x \in \mathcal{F}_0(\Omega_1)$
implies that $(F\circ G)_x = (F)_{G_x} \in \mathcal{F}_0(\Omega_2)$.  

(i) $(F\circ G)_B = (F)_{G_B}$ is closed for all $B\subseteq \Omega$ closed, since $G_B$ is closed and $F_A$ is closed 
for arbitrary $A \subseteq \Omega_1$ closed. 

(ii) Let $W$ be some open neighborhood of $(F\circ G)_x$, then $W$ is an open neighborhood for 
all $F_y$ with $y\in G_x$. Due to the upper semi-continuity of $F$ we can find neighborhoods $Z_y$
for all $y\in G_x$, such that $F_z \subseteq W$ for $z\in Z_y$.
%Since $(Z_y)_{y\in B_x}$ is an open covering
%of $B_x$ and since $B_x$ is a compact subset of $\Omega_1$ we can choose a finite open subcovering $(Y_{b_k})_{k=1}^N$ of $B_x$.
Put $Z:= \bigcup_{y \in G_x} Z_{y}$, then $Z$ is an open neighborhood of $G_x$ and since 
$G$ is upper semi-continuous we can find an open neighborhood $X$ of $x$ such that $G_w \subseteq Z$ for all $w\in X$.
It follows that $\bigcup_{y\in G_w} F_y \subseteq W$ for all $y\in X$, thus $F\circ B$ is upper semi-continuous.

(iii) Let $X$ be some bounded subset $\Omega$, then
$$
\bigcup_{x \in X}(F\circ G)_x = \bigcup_{x \in X} \bigcup_{y\in G_x} F_y = \bigcup_{y\in C} F_y
$$
where $C:=\bigcup_{x\in X} G_x$. Since $G$ is locally bounded it holds that $C$ is a bounded subset of $\Omega_1$. Due
to the locally boundedness of $F$ we conclude that $\bigcup_{y\in C} F_y$ is a bounded subset of $\Omega_2$, thus
$F\circ G$ is locally bounded. 
\end{proof}

\begin{remark}
Note that if the set-valued map $F:\Omega_1 \rightarrow \mathcal{F}_0(\Omega_2)$ is neither upper semi-continuous
nor closed and $G:\Omega \rightarrow \mathcal{F}_0(\Omega_1)$ is a locally bounded set-valued map which is not single-valued,
the composition $(F\circ G)$ need not define a set-valued map $\Omega \rightarrow \mathcal{F}_0(\Omega_2)$. 
Consider the following example: Let $F$ be defined by
$$
F_x :=  \left\{ \begin{array}{cl}
\{0\} & x \le 0 \\
\{1+x\} & x > 0 \quad ,
\end{array}
\right. 
$$
Observe that it is not upper semi-continuous at $0$ and $F_{[0,+\infty]} = \{0\}\cup ]1,\infty]$. 
Let $G_x:=[-1,1]$ for all $x\in\R$, then  $(F\circ G)_x = \{0\} \cup ]1, 2] \not \in \mathcal{F}_0(\R)$ for any $x\in \R$.

% Nevertheless if  $F:\Omega_1 \rightarrow \mathcal{F}_0(\Omega_2)$ is a set-valued map (not necessarily upper semi-continuous)
% and $f:\Omega \rightarrow \Omega_1$ is a (single-valued) map, then
% the \emph{composition map} $F\circ f$ maps $\Omega$ into $\mathcal{F}_0(\Omega_2)$ and $(F\circ f)_x = F_{f(x)}$.
\end{remark}

\begin{corollary} \label{composition_single_valued_function}
Let  $F:\Omega_1 \rightarrow \mathcal{F}_0(\Omega_2)$ be an upper semi-continuous, locally bounded set-valued map.
If $g:\Omega \rightarrow \Omega_1$ is a continuous map, 
then the composition $F\circ g$ is an upper semi-continuous, locally bounded set-valued map 
$\Omega \rightarrow \mathcal{F}_0(\Omega_2)$. 

In the case where $F$ is convex-valued, we obtain that $F\circ g$ is again convex-valued. If $H$ is
the supporting function of $F$, we have that $(x,w) \rightarrow H(g(x),w)$ is the supporting function of $F\circ g$.
\end{corollary}

\begin{proof}
Follows by Theorem \ref{set_valued_composition_theorem} and Definition \ref{supporting_function}.
\end{proof}

For the remaining part of this Section we assume that $\Omega_1$ is an open or closed subset of $\R^m$, such that
the set of all Lebesgue measurable subsets of $\Omega_1$ is complete with respect to the Lebesgue measure.

\begin{definition} \index{set-valued map!Lebesgue measurable}
Let $F$ be a set-valued map from $\Omega_1$ into $\R^m$. 
Then we call $F$ \emph{Lebesgue measurable}, if the set
$$
\{ x \in \Omega_1 \mid  F_x \cap \Omega \ne \emptyset \}
$$
is a Lebesgue measurable set for all open sets $\Omega \subseteq \R^m$.

We denote the set of Lebesgue measurable set-valued maps by $\mathcal{M}(\Omega_1; \mathcal{F}_0(\R^m))$.
\end{definition}

\begin{theorem}
Let $F:\Omega_1 \rightarrow \mathcal{F}_0(\R^m)$ be a set-valued map. Then the following properties are equivalent:
\begin{enumerate}
\item $F$ is Lebesgue measurable.
\item $\ggraph{F}$ belongs to $\mathcal{L}(\Omega_1) \otimes \mathcal{B}(\R^m)$.
\item for all $w\in \R^m$ the map $x\mapsto \inf_{a\in F_x} |a-w|$ is Lebesgue measurable.
\end{enumerate}
\end{theorem}

\begin{proof}
See \cite[Theorem 8.1.4]{AubinFrankowska:90}.
\end{proof}

\begin{proposition}
Let $F:\Omega_1 \rightarrow \mathcal{F}_0(\R^m)$ be a set-valued map. If $F$ is measurable, then 
its support function $H$ has the property that $x \mapsto H(x,w)$ is measurable for all $w\in \R^n$. 

If $F$ is convex-valued and locally bounded, then $F$ is measurable if and only if its supporting function $H$ has the property
that $x \mapsto H(x,w)$ is measurable for all $w\in \R^m$. 
\end{proposition}

\begin{proof}
We refer to \cite[Theorem 8.2.10]{AubinFrankowska:90}.
\end{proof}

\begin{definition} \label{set_valued_integral} \index{set-valued map!integral}
Let $K\csub \Omega_1$ and $F\in\mathcal{M}(\Omega_1; \mathcal{K}_0(\R^m))$, 
then we define the integral over $K$ by
$$
\int_{K} F_y \m{y} := \left\{ b\in \R^m \mid \forall w\in \R^m : \langle b, w\rangle \le \int_K H(y, w) \m{y} \right\}.
$$ 
If $F\in\mathcal{M}(J; \mathcal{K}_0(\R^m))$ where $J$ is some subinterval of $\R$, we 
use the notation $\int_t^s F_\tau \m{\tau}  := \int_{[t,s]} F_\tau \m{\tau}$
and $\int_s^t F_\tau \m{\tau}  := -\int_{[t,s]} F_{\tau} \m{\tau}$ for all $s,t\in J$ with $t\le s$.
\end{definition}

\begin{remark}
Note that the integral defined above is only available for set-valued maps with convex values. We do not investigate its
relation to the \emph{Aumann integral} (see \cite[Definition 8.6.1.]{AubinFrankowska:90} or \cite[Definition 17.1.1]{KleinThompson:84}), which
is available (if the domain of integration is compact) for any $F\in \mathcal{M}(\Omega_1; \mathcal{F}_0(\Omega_2))$. 
\end{remark}

\begin{remark}
We define now the space of (Lebesgue) integrable functions by 
$$
\mathcal{L}^1(\Omega_1; \mathcal{K}_0(\Omega_2)) = \{ x\mapsto F_x \in \mathcal{M}(\Omega_1; \mathcal{K}_0(\Omega_2)) \mid \exists C\in \R: \forall K \csub J, \int_K F_y \m{y} \le C\}.
$$
This space can be equipped with a pseudo-metric
$$
d_{\mathcal{L}}(F,G) :=  \sup_{K \csub \Omega_1} \int_K d(F_y,G_y) \m{y} =  
   \int_{\Omega_1} \sup_{w\in\R^n} | H_F(y,w/|w|)-H_G(y,w/|w|) | \m{y} 
$$
where $H_F, H_G$ denote the supporting functions of $F,G \in \mathcal{M}(\Omega_1; \mathcal{K}_0(\Omega_2))$. 
We obtain an equivalence relation by putting $F \sim G$ if $d_{\mathcal{L}}(F,G)=0$. 
By a standard argument the space of equivalence classes 
$L^1(\Omega_1; \mathcal{K}_0(\R^m)) := \mathcal{L}^1(\Omega_1; \mathcal{K}_0(\R^m)) / \sim$ 
is a metric space (with metric $d_{\mathcal{L}}$).

It is straight-forward now to define $L^p(\Omega_1; \mathcal{K}_0(\Omega_2))$ and 
its local version $L^p_{\rm loc}(\Omega_1; \mathcal{K}_0(\Omega_2))$ for all $p\in \N$. 
\end{remark}

\begin{theorem} \label{primitive_set-valued map} \index{set-valued map!selection}
Let $J$ be some subinterval of $\R$. 
Furthermore let $F \in L^1_{\rm loc}(J; \mathcal{K}_0(\R^m))$ and $t\in J$ fixed and $C_0 \in \mathcal{K}_0(\R^m)$ . Then we have that 
$$
s \mapsto C_0 + \int_t^s F_{\tau} \m{\tau} 
$$
defines an upper semi-continuous, locally bounded and convex-valued set-valued map from $J$ to $\R^m$ called the \emph{primitive}
of $F$. Furthermore for any $c\in C_0$ there exists an absolutely continuous selection 
$f : J \rightarrow \R^m$ 
satisfying 
$f(s) \in C_0 + \int_t^s F_{\tau} \m{\tau}$ with $f(t)=c$. It has the property that 
$$
|f(s)-f(r)| \le \int_{\min{(r,s)}}^{\max{(r,s)}} \sup_{w\in \R^m} H(\tau,w/|w|) \m{\tau}.
$$
\end{theorem}

\begin{proof}
Let $H$ denote the supporting function of $F$ and let $h_0$ be the convex supporting function of the set $C_0$.
We observe that for all $b \in C_0 + \int_t^s F_{\tau} \m{\tau}$ it holds that 
$$
|b| \le \sup_{w\in \R^m}\left| h_0(w/|w|) + \int_t^s H(\tau,w/|w|) \m{\tau} \right| < \infty,
$$
thus 
$$
\sup_{s\in M} |C_0 + \int_t^s F_{\tau} \m{\tau}| \le 
\sup_{s\in M} \sup_{w\in \R^m} |h_0(w/|w|)+\int_t^s H(\tau,w/|w|) \m{\tau}| <  \infty
$$
for any bounded $M \subseteq \Omega_1$. Since $s\mapsto h_0(w)+\int_t^s H(\tau,w) \m{\tau}$ is upper semi-continuous for fixed $w\in\R^n$ it follows by Theorem \ref{set_valued_map_main_theorem} that $s\mapsto C_0 + \int_t^s F_\tau \m{\tau}$ is upper semi-continuous.

It is sufficient to construct a continuous selection $f$ of $\int_t^s F_\tau \m{\tau}$ with $f(t)=0$. Putting
$f_c(s) := f(s) + c$ for some $c\in C_0$ immediately yields a continuous selection of $C_0 + \int_t^s F_\tau \m{\tau}$ 
with $f_c(t)=c$. So without loss of generality we may assume $C_0=\{0\}$.
  
Let us first consider some interval $[t,T_+] \csub J$ for some $T_+ \in J$ with $T_+ > t$.
For $k\in\N$ we put $\lambda_k:=\frac{T^{+}-t}{k}$ and $t_{k,j}:= t+ j \lambda_k$, so $t_{k,0}=t$ and $t_{k,k}=T_{+}$.
We obtain a decomposition of the interval $J^+$ in $k$ subintervals $[t_{k,j-1}, t_{k,j}]$ for $1\le j \le k$. 

We are going to construct a sequence of continuous functions $(f_k)_{k\in\N}$ such that
$$
\langle f_{k}(s),w\rangle \le \int_t^s H(\tau ,w) \m{\tau} 
$$ 
holds for all $w\in\R^n$ and $k\in\N$.

Put $B_{t,s}:=  \int_t^s F_{\tau} \m{\tau}$.
Any $b_{t,s} \in B_{t,s}$ has the property that 
$\langle b_{t,s} , w \rangle \le \int_t^s H(\tau,w)\m{\tau} = -\int_s^t H(\tau,w)\m{\tau}$
and $-\langle b_{s,t} , w \rangle \ge - \int_t^s H(\tau,w)\m{\tau}$ for all $w\in \R^m$, thus
$$
\langle b_{t,s} , w \rangle \le -\langle b_{s,t} , w \rangle  ,\quad w\in \R^n,
$$
which immediately implies $b_{s,r} = - b_{r,s}$. It follows directly that $b_{s,s} =0$ and $B_{s,s} = \{ 0\}$
and $B_{t,s} = -B_{s,t}$ for any $s,t \in J$. Furthermore we observe that 
$$
\int_t^r H(\tau,w)\m{\tau} + \int_r^s H(\tau,w)\m{\tau}  =  \int_t^s H(\tau,w)\m{\tau} 
$$
for all $w\in\R^m$ which implies $ B_{t,r}+B_{r,s} = B_{t,s}$.

Put $f_{k}(s) = 0$  for all $s\in [t, t+ \lambda_k]$, then
we define iteratively
$$
f_k \mid_{[t+j \lambda_k, t+ (j+1)\lambda_k]} (s) = f_k(t+ j\lambda_k ) + b_{ t+ j \lambda_k,s}, \quad j=1,\dots,k
$$ 
for some $b_{ t+ j \lambda_k,s}\in B_{t+ (j-1)\lambda_k,s}$. 
It follows immediately that 
$$
f_k(s) = \sum_{j=1}^{l}  b_{t+ (j-1) \lambda_k , t+ j \lambda_k}+ b_{t+ l \lambda_k , s} \in B_{t,s}
$$
for $s\in [t+ l \lambda_k,t+ (l+1) \lambda_k]$, which implies that $(f_k)_{k\in\N}$ is an
equi-bounded family.

Assume $s,r \in [t,T_+]$  then we have
\begin{multline} 
\langle f_k(s) - f_k(r), w \rangle =  
\langle \sum_{j=1}^{l_1}  b_{t+ (j-1) \lambda_k , t+ j \lambda_k}+ b_{t+ l_1 \lambda_k , s} -
\sum_{j=1}^{l_2}  b_{t+ (j-1) \lambda_k , t+ j \lambda_k}- b_{t+ l_2 \lambda_k , r},w \rangle \\
 =
 \langle \text{sign}(s-r) \left(-b_{t+ (\min{(l_1,l_2)}-1) \lambda_k, \min{(s,r)}} + \sum_{j=\min{(l_1,l_2)}}^{\max{(l_1,l_2)}}  b_{t+ (j-1) \lambda_k , t+ j \lambda_k}+ b_{t+ \max{(l_1,l_2)} \lambda_k, \max{(s,r)}} \right),w \rangle 
\end{multline}
and since
$-b_{t+ (\min{(l_1,l_2)}-1) \lambda_k, \min{(s,r)}} + \sum_{j=\min{(l_1,l_2)}}^{\max{(l_1,l_2)}}  b_{t+ j \lambda_k , t+ j \lambda_k}+ b_{t+ \max{(l_1,l_2)} \lambda_k }\in B_{\min{(r,s)},\max{(r,s)}}$ we have the estimate
$$
\langle f_k(s) - f_k(r), w \rangle \le  \int_{\min{(r,s)}}^{\max{(r,s)}} H(\tau,w) \m{\tau} 
$$
implying
$$
| f_k(s) - f_k(r)| \le  \int_{\min{(r,s)}}^{\max{(r,s)}} \sup_{w\in \R^m} H(\tau,w/|w|) \m{\tau} 
$$
which show that $(f_k)_{k\in\N}$ is equi-continuous family of continuous functions.

By the theorem of Arzela-Ascoli we can find a convergent subsequence $(f_{k_j})_{j\in\N}$
converging to some continuous function $f \in C([t,T_+])$. It has the property that
$$
\langle f(s),w \rangle = \lim_{j\rightarrow \infty} \langle f_{k_j}(s),w \rangle 
\le \int_t^s H(\tau,w )\m{\tau}, 
$$
so $f$ is a continuous selection of the set-valued map $s\mapsto\int_t^s F_\tau \m{\tau}$ for all $s\in [t,T^{+}]$. 
Absolute continuity of $f$ follows by 
$$
|f(s)-f(r)| \le |f(s)-f_{k_j}(s)|+|f_{k_j}(s)-f_{k_j}(r)|+|f_{k_j}(r)-f(r)|
\stackrel{j\rightarrow \infty}{\rightarrow} \int_{\min{(r,s)}}^{\max{(r,s)}} \sup_{w\in \R^m} H(\tau,w/|w|) \m{\tau}.
$$
The proof for the backward direction $[T_{-},t] \csub J$ for some $T_- \in J$ with $T_- < t$ is analogous.
In the case of a non-compact interval $J$ the extension of the selection 
to the whole interval is straight-forward.
\end{proof}

%% file: ggraph.tex
\chapter{The generalized graph} \label{generalized_graph}

In this section we introduce the concept of a generalized graph for a c-bounded generalized map. 
The generalized graph extends the classical graph of a continuous map, in 
the sense that the generalized graph of the embedded map (which is a c-bounded generalized map) coincides with the classical graph. Furthermore the generalized graph is
closed in the product topology of the domain and the image space of the generalized map.
% It turns out that the concept of a set-valued map from Section \ref{set_valued_map_section} is suitable for
% describing the generalized graph.   

%This concept will enable us to define the normal bundle of a generalized c-bounded map and the transformation of a wave front set by a generalized c-bounded map (this is %carried out in Section \ref{sectionwfset}).

Throughout the entire Section we let $\Omega_1, \Omega_2$ be open subsets of $\mathbb{R}^n$ resp. $\mathbb{R}^m$. We define the projection maps of the product space of $\Omega_1 \times \Omega_2$
onto $\Omega_1$ resp. $\Omega_2$ by
$$
\pi_1: \Omega_1 \times \Omega_2 \rightarrow \Omega_1,  \  (x,y) \mapsto x \quad \text{resp.} \quad
\pi_2: \Omega_1 \times \Omega_2 \rightarrow \Omega_2,  \ (x,y) \mapsto y.
$$

\section{The generalized graph of a c-bounded Colombeau map}

\begin{definition}
Let $F\in\colmap{\Omega_1}{\Omega_2}$ be a c-bounded generalized map. Then the set 
\begin{equation}\label{generalized_graph_definition}
\begin{split}
\ggraph{F} := & \left\{(x,y) \in \Omega_1 \times \Omega_2 \mid \exists \text{\ a \ net \ }(x_{\varepsilon})_{\varepsilon} 
\text{\ in\ } \Omega_1: \lim_{\varepsilon\rightarrow 0} x_{\varepsilon}=x \in \Omega_1 \right. \\
&\left. \text{ and } y\in\cp{(F_{\tau(\varepsilon)}(x_{\varepsilon}))_{\varepsilon}}  
 \text{for\ some\  map}\ \tau \in \mathcal{T} \right\} 
\end{split}
\end{equation} is called the \emph{generalized graph} of $F$. By an abuse of notation we
use the symbol $\ggraph{F}$ from Definition \ref{set_valued_map_def} for the generalized graph of a Colombeau map.
This is motivated by the observation made further below (cf. Theorem \ref{generalized_graph_theorem}) that the generalized 
graph actually defines a set-valued map by $x\mapsto \pi_2(\ggraph{F}\cap (\{x\} \times \Omega_2))$. 

\end{definition}
Observe that the definition for the generalized graph does not depend on the representative of $F$: Let
$(F_{\varepsilon})_{\varepsilon},(\widetilde{F}_{\varepsilon})_{\varepsilon}$ be two representatives of $F$, then $n_{\varepsilon} := F_{\varepsilon} - \widetilde{F}_{\varepsilon} $ is negligible and
$$
|F_{\tau(\varepsilon)}(x_{\varepsilon})-\widetilde{F}_{\tau(\varepsilon)}(x_{\varepsilon})| \le \sup_{z\in X} |n_{\tau(\varepsilon)}(z)| = O(\tau(\varepsilon))  \rightarrow 0
$$ 
as $\varepsilon \rightarrow 0$ (since $\tau \in \mathcal{T}$), where $X$ is some compact set containing the converging net $x_{\varepsilon}$ for small $\varepsilon$.

% Note that we do not consider nets $(x_{\varepsilon})_{\varepsilon}$ in $\Omega_1$ that
% converge to the boundary $\partial \Omega_1$. 

We use the notation $\ggraph{F}_{x} := \pi_2(\ggraph{F} \cap \{x\} \times \Omega_2 )$, thus
$$
\ggraph{F} = \bigcup_{x \in \Omega_1} \{x\} \times \ggraph{F}_x.
$$

\begin{definition}
Let $F \in \colmap{\Omega_1}{\Omega_2}$  be a c-bounded generalized map. We say $F$ is of \emph{natural type},
if it has a representative such that 
$$
\varepsilon \mapsto F_{\varepsilon}(x)
$$
is continuous on $]0,1]$ for all $x\in\Omega_1$ fixed. Note that a Colombeau map which is the embedding of a distribution
is of \emph{natural type}. 
\end{definition}

\begin{example}
Consider the c-bounded Colombeau function $F$ of natural type defined by the representative $(F_{\varepsilon})_{\varepsilon}$ with 
$$
F_{\varepsilon} := \sin{(x/\varepsilon)},
$$
then it is straight-forward to calculate its generalized graph
$$
\ggraph{F} = \mathbb{R} \times [-1,1].
$$
\end{example}

\begin{example} \label{heaviside_ggraph_example}
Consider the Colombeau function $F:=\iota(H)$, where $H$ is the Heaviside function. If we set $g:= H \ast \rho$, where $\rho$ is
the mollifier of the embedding $\iota$ with $\int \rho(y)\m{y}=1$ and $\int y^k \rho(y) \m{y} = 0$ for all $k \ge 1$, we have that $F_{\varepsilon}(x):=g(x/\varepsilon)$ defines a representative of $F$.
For all nets $(x_{\varepsilon})_{\varepsilon}$ tending to $x_0 \ne 0$, we obtain that $F_{\varepsilon}(x_{\varepsilon}) \rightarrow \rm{sign}(x_0)$ as $\varepsilon \rightarrow 0$.
Let us consider the nets $(x_{\varepsilon})_{\varepsilon}$ tending to $0$: Since we can find zero-nets $(x_{\varepsilon})_{\varepsilon}$ such that $(x_{\varepsilon} / \varepsilon)$ tends to any point in $\mathbb{R} \cup \{\pm \infty\}$ and $g$ is continuous and bounded, we have that
\begin{equation*}
\ggraph{F} := \left(\{x\in\mathbb{R}\mid x<0\} \times \{0\}\right) \cup \left(\{0\} \times \overline{\ran{g}} \right) \cup \left(\{x\in\mathbb{R}\mid x>0\} \times \{1\}\right).
\end{equation*}
Note that $\overline{\ran{g}}$ is a closed interval that contains $[0,1]$.
\end{example}

\begin{definition}
Let $F\in\colmap{\Omega_1}{\Omega_2}$ be a c-bounded generalized map. 
$F$ is said to be equi-continuous at $x_0$, if there exists a representative $(F_{\varepsilon})_{\varepsilon}$ which is equi-continuous in $x_0$: For all $\gamma>0$,
there exists some $\delta>0$ such that
\begin{equation*}
|F_{\varepsilon}(x) - F_{\varepsilon}(x_0) | < \gamma
\end{equation*}
holds for all $x \in B_{\delta}(x_0)$ and $\varepsilon \in ]0,1]$.
The generalized map $F$ is called \emph{equi-continuous} on some subset $X \subseteq \Omega_1$, if it has a representative which
is equi-continuous in each point $x \in X$.
\end{definition}
% Note that if $F$ is equi-continuous at $x_0$, it follows that any representative $(F_{\varepsilon})_{\varepsilon}$ is an 
% equi-continuous family at $x_0$. Let
% $F_{\varepsilon}, \widetilde{F}_{\varepsilon}$ be two representatives of $F$, then $n_{\varepsilon} := F_{\varepsilon} - \widetilde{F}_{\varepsilon} $ is negligible and
% \begin{multline*}
% ||\widetilde{F}_{\varepsilon}(x) - \widetilde{F}_{\varepsilon}(x_0)| -|F_{\varepsilon}(x)-F_{\varepsilon}(x_0)|| \le  
% |\widetilde{F}_{\varepsilon}(x) - F_{\varepsilon}(x) + F_{\varepsilon}(x_0) - \widetilde{F}_{\varepsilon}(x_0) |\\
% =   2 \sup_{z\in X} |n_{\varepsilon}(z)| = O(\varepsilon)  \rightarrow 0
% \end{multline*}
% as $\varepsilon \rightarrow 0$ uniformly for all $x\in X$, where $X$ is a suitable compact neighborhood of $x_0$.

\begin{proposition}\label{generalized_graph_equicontinuous_proposition}
Let $F\in\colmap{\Omega_1}{\Omega_2}$ be a c-bounded  generalized map. 
If $F$ is equi-continuous on ${\Omega_1}$,
then it follows that the generalized graph of $F$ can be determined pointwise by
\begin{equation}\label{generalized_graph_equicontinuous}
\ggraph{F} =  \left\{(x,y) \in \Omega_1 \times \Omega_2 \mid y \in \cp{ (F_{\varepsilon}(x))_{\varepsilon}}\ \right\}.
\end{equation}
\end{proposition}
\begin{proof}
We can choose a equi-continuous representative $(F_{\varepsilon})_{\varepsilon}$ and proceed along a standard argument:
Let $(x_{\varepsilon})_{\varepsilon},(x'_{\varepsilon})_{\varepsilon}$ be two nets tending to $x \in \Omega_1$. 
Then we have, using the local equi-continuity of $(F_{\varepsilon})_{\varepsilon}$ on $\Omega_1$, 
that for all $\gamma>0$ we can find some $\delta>0$,
such that the distance is bounded by
\begin{eqnarray*}
|F_{\tau(\varepsilon)}(x_{\varepsilon}) - F_{\tau(\varepsilon)}(x_{\varepsilon}')| < \gamma 
\end{eqnarray*}
for all $x_{\varepsilon}, x'_{\varepsilon} \in B_{\delta/2}(x)$ uniformly for all $\tau\in\mathcal{T}$. 
Since $\lim_{\varepsilon \rightarrow 0} x_{\varepsilon} =
\lim_{\varepsilon \rightarrow 0} x_{\varepsilon}^{\prime} =x$ there exists $\varepsilon' \in ]0,1]$ such that
$|x_{\varepsilon} - x'_{\varepsilon}|<\delta$ holds for all $\varepsilon < \varepsilon'$.
Thus $\lim_{\varepsilon\rightarrow0} (F_{\tau(\varepsilon)}(x_{\varepsilon})-F_{\tau(\varepsilon)}(x_{\varepsilon}')) =0$ 
for any $\tau \in \mathcal{T}$ and it suffices to consider the constant net $x_{\varepsilon}:=x$ in 
Definition \ref{generalized_graph_definition}.
As $\tau\in \mathcal{T}$ we have that
 $(F_{\tau(\varepsilon)}(x))_{\varepsilon}$ is a subnet of $(F_{\varepsilon}(x))_{\varepsilon}$ and
$$
\cp{(F_{\tau(\varepsilon)}(x))_{\varepsilon} } \subseteq \cp{ (F_{\varepsilon}(x))_{\varepsilon} },
$$
which yields the statement (\ref{generalized_graph_equicontinuous}).
\end{proof}

\begin{lemma} \label{generalized_graph_lemma}
Let $F\in\colmap{\Omega_1}{\Omega_2}$ be a Colombeau generalized map. 
If $(x,y) \in (\Omega_1 \times \Omega_2) \backslash \ggraph{F}$, then
there exists neighborhoods $Y \subseteq \Omega_2$ of $y$, $X \subseteq \Omega_1$ of $x$ and $\varepsilon' \in ]0,1]$ such that
\begin{equation*}
Y \cap F_{\varepsilon}(X) = \emptyset 
\end{equation*}
for all $\varepsilon <\varepsilon'$.
\end{lemma}
\begin{proof}
Let $(x,y) \in  (\Omega_1 \times \Omega_2) \backslash \ggraph{F}$. 
The proof proceeds by contradiction: 
Assume that for all neighborhoods $Y' \subseteq \Omega_2$ of $y$ and $X' \subseteq \Omega_1$ of $x$ and for
all $\varepsilon'>0$ there exists some $\widetilde{\tau}: \varepsilon' \mapsto \tau(X',Y',\varepsilon') \le \varepsilon'$ with
\begin{equation*}
Y' \cap F_{\widetilde{\tau}(\varepsilon')}(X') \ne \emptyset.
\end{equation*}
By setting $X':=B_{\varepsilon}(x) \cap \Omega_1$, $Y':=B_{\varepsilon}(y) \cap \Omega_2$ and $\varepsilon':=\varepsilon$
we obtain a map $\tau: \varepsilon \mapsto \widetilde{\tau}(B_{\varepsilon}(x) \cap \Omega_1 ,B_{\varepsilon}(y) \cap \Omega_2,\varepsilon)$.
As $\tau(\varepsilon) \le \varepsilon$ for all $\varepsilon \in ]0,1]$ we have $\tau \in \mathcal{T}$.
We choose a net $x_{\varepsilon} \in B_{\varepsilon}(x) \cap \Omega_1$ satisfying $F_{\tau(\varepsilon)}(x_{\varepsilon}) \in B_{\varepsilon}(y) \cap \Omega_2$ for all $\varepsilon \in ]0,1]$.

Obviously $(x_{\varepsilon})_{\varepsilon}$ is a convergent net in $\Omega_1$ with $\lim_{\varepsilon \rightarrow 0} x_{\varepsilon}=x \in \Omega_1$
and the net $(F_{\tau(\varepsilon)}(x_{\varepsilon}))_{\varepsilon}$ converges to $y$. Since $\tau \in \mathcal{T}$ we obtain
$y \in \cp{(F_{\tau(\varepsilon)} (x_{\varepsilon}))_{\varepsilon}}$, so $(x, y) \in \ggraph{F}$ contradicting the assumption.
\end{proof}

\begin{theorem}[Properties of the generalized graph]\label{generalized_graph_theorem}
Let $F\in\colmap{\Omega_1}{\Omega_2}$ be a c-bounded generalized map,
then the generalized graph $\ggraph{F}$ has the following properties:
\begin{enumerate}
\item It extends the classical notion of a graph (of a continuous function) in the following sense: 
If $F$ is the embedding of a continuous function $G := (G_1, ..., G_m) \in C(\Omega_1,\Omega_2)$, i.e. $F=(\iota(G_1),...,\iota(G_m))$,
then the generalized graph of $F$ coincides with the graph of the continuous function $G$.
\item It defines a closed and locally bounded set-valued map by
$$
\Omega_1 \rightarrow \mathcal{F}_0(\Omega_2), \quad x \mapsto \ggraph{F}_x.
$$
\end{enumerate}
\end{theorem}
\begin{proof}
\begin{trivlist}
\item (i) follows by Proposition \ref{generalized_graph_equicontinuous_proposition} and the fact that the embedding of a continuous function is an equi-continuous generalized function.
\item (ii) 
$\ggraph{F}_x \ne \emptyset$ for all $x\in \Omega_1$ is trivial.

We proceed showing that $\ggraph{F}$ is closed in $\Omega_1\times \Omega_2$.
Let $(x,y) \in (\Omega_1 \times \Omega_2) \backslash \ggraph{F}$, then we know by Lemma \ref{generalized_graph_lemma} that there exists $X \subseteq \Omega_1, Y \subseteq \Omega_2$ open neighborhoods of $x$ resp. $y$ and $\varepsilon'>0$, such that
\begin{equation*}
F_{\varepsilon}(X) \cap Y = \emptyset
\end{equation*}
for all $\varepsilon < \varepsilon'$.
Now for any point $y \in X$ and any net $(y_{\varepsilon})_{\varepsilon}$ in $\Omega_1$ converging to $y$, there
exists some $\varepsilon''>0$ such that $y_{\varepsilon} \in X$ for all $\varepsilon < \varepsilon''$.
It follows that
\begin{equation*}
\ggraph{F} \cap ( X \times Y ) = \emptyset,
\end{equation*}
so $(\Omega_1 \times \Omega_2) / \ggraph{F}$ is open in  $\Omega_1 \times \Omega_2$ and hence $\ggraph{F}$ is closed in the product topology of $\Omega_1 \times \Omega_2$.

In the next step we proove that $x \mapsto \ggraph{F}_x$ is locally bounded.
Let $B \subseteq \Omega_1$ be bounded, then there exists some $K \csub \Omega_1$ with $B \subseteq K$. Furthermore let 
$K_1$ be some compact neighborhood for all points in $K$. Since $F$ is c-bounded we can find some $\varepsilon'>0$ and 
$K_2 \csub \Omega_2$,
such that 
\begin{equation*}
F_{\varepsilon}(K_1) \subseteq K_2 
\end{equation*}
for all $\varepsilon < \varepsilon'$. So $ \ggraph{F} \cap K \times \Omega_2 \subset \ggraph{F} \cap K \times K_2$, which implies that 
$\ggraph{F} \cap K \times \Omega_2$ is a bounded subset of $\Omega_1 \times \Omega_2$. So $x \mapsto \ggraph{F}_x$ 
is locally bounded.
 
Putting $K=\{x\}$ we obtain that $\ggraph{F}_x = \pi_2 (\ggraph{F} \cap (\{x\}\times \Omega_2))$ is closed, 
since it is the continuous image of a bounded, closed set. So $\ggraph{F}_x \in \mathcal{F}_0(\Omega_2)$ for all $x\in \Omega_1$.
\end{trivlist}
\end{proof} 

\begin{corollary} \label{generalized_graph_usc_corollary}
Since the generalized graph can be seen as a closed and locally bounded set-valued map $x\mapsto \ggraph{F}_x$, it follows by
Theorem \ref{set_valued_map_main_theorem} that it is upper semi-continuous, i.e.
for any open neighboorhood $W$ of $\ggraph{F}_x$, there exists an open neighboorhood $X$ of $x$, such that
$\ggraph{F}_X \subseteq W$. 
\end{corollary}

\begin{proposition} \label{generalized_graph_connected_prop}
Let $F\in\colmap{\Omega_1}{\Omega_2}$ be a c-bounded generalized map of natural type, then the generalized graph 
$\ggraph{F}$ has the property that $\ggraph{F}_x$ is connected for all $x\in \R^n$.
\end{proposition}
\begin{proof}
We prove by contradiction: Assume that $\ggraph{F}_x$ is not connected, then there exist $V',W'$ open subsets
of $\R^m$ with $V:=V' \cap \ggraph{F}_x, W:=W' \cap \ggraph{F}_x \ne \emptyset$ such that $V \cap W = \emptyset$ and $V \cup W= \ggraph{F}_x$.
Choose $a\in V, b \in W$, then we can find 
a net $(x_{\varepsilon})_{\varepsilon},(x_{\varepsilon}')_{\varepsilon}$ with $x_{\varepsilon},x'_{\varepsilon} \rightarrow x$ and $\tau,\tau' \in \mathcal{T}$ such that 
$y_{\varepsilon}:=F_{\tau(\varepsilon)}(x_{\varepsilon}) \rightarrow a$ and $y'_{\varepsilon}:=F_{\tau'(\varepsilon)}(x_{\varepsilon}') \rightarrow b$. 
Since $V,W$ are neighboorhoods of $a$ resp. $b$ there exists some $\widetilde{\varepsilon} \in ]0,1]$ such that
$$
F_{\tau(\varepsilon)}(x_{\varepsilon}) \in V, \quad F_{\tau'(\varepsilon)}(x'_{\varepsilon}) \in W
$$
for all $\varepsilon \in ]0, \widetilde{\varepsilon}]$. 
Observe that since $F$ is of natural type we have that 
$$
g_{\varepsilon}: \alpha \mapsto F_{\alpha \tau(\varepsilon) +(1-\alpha)\tau'(\varepsilon) }(\alpha x_{\varepsilon}+ (1-\alpha) x'_{\varepsilon}) 
$$ 
is a continuous map with $g_{\varepsilon}(0)=y_{\varepsilon}$ and $g_{\varepsilon}(1)=y_{\varepsilon}'$. The intermediate value theorem then yields that $g_{\varepsilon}([0,1])$ is connected for fixed $\varepsilon$. 
Since $V \cap W =\emptyset$ and for $\varepsilon\in]0,\widetilde{\varepsilon}]$ we have that $V \cap  g_{\varepsilon}([0,1]),W \cap  g_{\varepsilon}([0,1])$ are non-empty, open sets in the relative topology of $\ran{g_{\varepsilon}}\subseteq \Omega_2$ 
it follows that $(V\cup W) \cap g_{\varepsilon}([0,1]) \ne g_{\varepsilon}([0,1])$
(equality would contradict the connectedness of $g_{\varepsilon}([0,1])$), thus  
there exists some $\alpha_{\varepsilon} \in [0,1]$ with 
$g_{\varepsilon}(\alpha_{\varepsilon}) \in (V\cup W)^c$.

Putting $\mu(\varepsilon):=\alpha_{\varepsilon} \tau(\varepsilon) +(1-\alpha_{\varepsilon})\tau'(\varepsilon)$ and 
$y_{\varepsilon}:=\alpha_{\varepsilon} x_{\varepsilon}+ (1-\alpha_{\varepsilon}) x'_{\varepsilon}$ 
we observe that $\mu \in \mathcal{T}$ and $y_{\varepsilon} \rightarrow x$.
Since $g_{\varepsilon}(\alpha_{\varepsilon}) = F_{\mu(\varepsilon)}(y_{\varepsilon}) \in (V \cup W)^c$ 
it follows that $\emptyset \ne \cp{(F_{\mu(\varepsilon)}(y_{\varepsilon}) )_{\varepsilon} }
 \subseteq (V\cup W)^c \cap  \ggraph{F}_x$ which is a contradiction to $(V \cup W) \cap \ggraph{F}_x= \ggraph{F}_x$.
\end{proof}

\begin{remark} 
If $F\in \colmap{\Omega_1}{\R}$ is of natural type, then Proposition \ref{generalized_graph_connected_prop} yields that $\ggraph{F}_x$ 
is a convex set for all $x\in\R^n$.

For the higher dimensional case ($m>1$), the convexity of $\ggraph{F}_x$ does not hold for arbitrary 
$F \in \colmap{\Omega_1}{\R^m}$ of natural type. 
Consider the Colombeau map in $\colmap{\R}{\R^2}$ defined by 
$$
F_{\varepsilon}(x) := (\sin^{1/\varepsilon}(x/\varepsilon),\cos^{1/\varepsilon}(x/\varepsilon)),
$$
then it is straight-forward to verify $\ggraph{F}_{0} = \bigcup_{\alpha \in [-1,1]} ((\alpha,0) \cup (0,\alpha))$, which is 
connected but not convex.

If $F\in \colmap{\Omega_1}{\R}$ is not of natural type, then $\ggraph{F}_x$ is in general neither convex nor connected. Just consider the simple Colombeau map $F\in\colmap{\R}{\R}$ defined by the representative
$$
F_{\varepsilon}(x) := \left\{ \begin{array}{ll}
1 & \varepsilon \in \bigcup_{j\in\N} ]2^{-2j-1}, 2^{-2j}] \\
0 & \varepsilon \in \bigcup_{j\in\N} ]2^{-2(j+1)}, 2^{-2j-1}]. 
\end{array}
\right. 
$$
Its generalized graph satisfies $\ggraph{F}_{x} = \{0\}\cup\{1\}$ for all $x\in \R$. 
\end{remark}

% \begin{remark}
% Under the assumption that $F\in \colmap{\Omega_1}{\Omega_2}$ is of natural type:
% Due to the convexity and compactness of $\ggraph{F}_x$ it follows by Theorem \cite[4.3.2]{Hoermander:V1} that the 
% supporting function 
% $$
% H_F: \Omega_1 \times \R^m \rightarrow \R, \quad (x,w) \mapsto \max_{y\in \ggraph{F}_x} \langle y, w \rangle
% $$ 
% has the property that it is convex, positively homogeneous and lower semi-continuous in $w$ and
% $$
% \ggraph{F}_x = \{ y\in \Omega_2 \mid \forall  \R^n : \langle y,w  \rangle \le H_F(x,w) \}.
% $$ 
% Then the closedness of $\ggraph{F}$ in $\Omega_1 \times \Omega_2$ implies that $x \mapsto H_F(x,w)$ is upper semi-continuous on $\Omega_1$ for any fixed $w\in\R^n$.
% This is a consequence of the more general Theorem \ref{colombeau_supporting_function_proposition} (which proves the last conclusion in detail).
% \end{remark}

\begin{lemma} \label{generalized_graph_approximation_property}
Let $F\in \colmap{\Omega_1}{\Omega_2}$ be a c-bounded generalized map. Furthermore,
let $K_1 \csub \Omega_1$ and $K_2 := \pi_2(\ggraph{F} \cap ( K_1 \times \Omega_2))$.
If $Y$ is a neighborhood of $K_2$, then there exists a neighborhood $X \subseteq \Omega_1$ of $K_1$ and $\varepsilon'\in ]0,1]$, such that
\begin{equation*}
F_{\varepsilon}(X) \subseteq Y
\end{equation*}
for all $\varepsilon \le \varepsilon'$. 
\end{lemma}
\begin{proof} 
First we proof the Lemma for $K_1:=\{x\}$ containing only a single point $x\in\Omega_1$.
We note that $K_2$ is a compact subset of $\Omega_2$ by Theorem \ref{generalized_graph_theorem}. 
Assume that the statement of the Lemma does not hold, that is for all neighborhoods $X'$ of $x$ and for
all $\varepsilon'\in]0,1]$ there exists some $\widetilde{\tau}(\varepsilon',X') \le \varepsilon'$ and $x_{\varepsilon'} \in X'$ with
\begin{equation*}
F_{\widetilde{\tau}(\varepsilon')}(x_{\varepsilon'}) \not \in Y.
\end{equation*}
Then (by setting $X'=B_{\varepsilon}(x) \cap \Omega_1$ and $\varepsilon'=\varepsilon$) we
define $\tau: \varepsilon  \mapsto \widetilde{\tau}(\varepsilon, B_{\varepsilon}(x) \cap \Omega_1)$
and obtain some net $x_{\varepsilon} \in B_{\varepsilon}(x)\cap \Omega_1$ such that 
$F_{\tau(\varepsilon)}(x_{\varepsilon}) \not \in Y$. Note that $\tau(\varepsilon) \le \varepsilon$ for all 
$\varepsilon \in ]0,1]$ implies $\tau \in \mathcal{T}$.

Since $Y$ is a neighborhood of $K_2$, we have  that $\overline{Y^c} \cap K_2 = \emptyset$.
The net $(x_{\varepsilon})_{\varepsilon}$  converges to $x$ and the net $(F_{\tau(\varepsilon)}(x_{\varepsilon}))_{\varepsilon} \in Y^c$
has the property that the set of clusterpoints $\cp{(F_{\tau(\varepsilon)}(x_{\varepsilon}))_{\varepsilon}}$
is contained in the closure of $Y^c$. 
By definition of the generalized graph $\ggraph{F}$ we have that $\cp{ (F_{\tau(\varepsilon)}(x_{\varepsilon}))_{\varepsilon} } \subseteq K_2$,
which contradicts the fact that $\overline{Y^c} \cap K_2 = \emptyset$. 

Now we consider the general case, when $K_1$ is an arbitrary compact set. Then $Y$ is a neighborhood of $\pi_2(C_f \cap K_1 \times \Omega_2)$ and
it follows that $Y$ is a neighborhood of each set $\pi_2(\{ z\} \times \ggraph{F}_z)$ for $z\in K_1$. We can apply the first part of the proof for each point $z$ and obtain (open) neighborhoods $X_z$ of each point $z$ such that $F_{\varepsilon}(X_z) \subseteq Z$ for all $\varepsilon<\varepsilon_z$,  where $\varepsilon_z\in ]0,1]$ depends on $z$.
Since $K_1$ is compact and $(X_z)_{z\in K_1}$ is an open covering we can choose some finite subcovering $(X_{z_k})_{k=1}^l$. Then $X:=\bigcup_{k=1}^l X_{z_k}$
is a neighborhood of $K_1$ such that $f_{\varepsilon}(X) \subseteq Y$ holds for $\varepsilon < \varepsilon':=\min_{k=1}^l \varepsilon_{z_k}$.
\end{proof}

\section{The generalized graph as a set-valued map}

By Theorem \ref{generalized_graph_theorem} it turns out that we think of the generalized graph of a c-bounded Colombeau map 
$F\in\colmap{\Omega_1}{\Omega_2}$ as a set-valued map
$$
\Omega_1 \rightarrow \mathcal{F}_0(\Omega_2), \quad x \mapsto \ggraph{F}_x.
$$
Since the generalized graph $\ggraph{F}$ is closed in $\Omega_1 \times \Omega_2$ 
and $\ggraph{F}_x$ is compact for all $x\in \Omega_1$ it follows by Proposition \ref{closed_corr_proposition} and 
Theorem \ref{set_valued_map_main_theorem} that $x\mapsto \ggraph{F}_x$ 
is an upper semi-continuous and locally bounded set-valued map.

It turns out that the convex hull of the generalized graph is an important tool to characterize the shadows of 
Colombeau solutions of ordinary differential equations, 
where the right-hand side $A:=[(A_{\varepsilon})_{\varepsilon}]$ is a Colombeau map satisfying certain bounds.

\begin{theorem} \label{generalized_graph_composition}
Let $F \in \colmap{\Omega_1}{\Omega_2}$ and $G \in \colmap{\Omega}{\Omega_1}$, then the composition
$F\circ G \in \colmap{\Omega}{\Omega_2}$ defined by $(F_{\varepsilon} \circ G_{\varepsilon})_{\varepsilon}$ has the property
that 
$$
\ggraph{F \circ G}_x \subseteq (\ggraph{F} \circ \ggraph{G})_x \quad \text{for all } x\in\Omega,
$$ 
where the composition on the right-hand side denotes the composition of set-valued maps (cf. Theorem \ref{set_valued_composition_theorem}).
\end{theorem}
\begin{proof}
Note that $(\ggraph{F} \circ \ggraph{G})_x$ is again an upper semi-continuous, locally bounded set-valued map due to Theorem \ref{set_valued_composition_theorem}.

Let $a\in \ggraph{F \circ G}_x$ then there exists some $\tau \in \mathcal{T}$ and a net $(x_{\varepsilon})_{\varepsilon}$ with $x_{\varepsilon} \rightarrow x$, such that
$$
a = \lim_{\varepsilon \rightarrow 0} F_{\tau(\varepsilon)}(G_{\tau(\varepsilon)}(x_{\varepsilon})).
$$
Without loss of generality we can assume that $y_{\varepsilon}=G_{\tau(\varepsilon)}(x_{\varepsilon})$ converges to some $y\in \ggraph{G}_x$.
It immediately follows that $a\in \ggraph{F}_y$ and we have 
$\ggraph{F \circ G}_x \subseteq \bigcup_{y \in \ggraph{G}_x} \ggraph{F}_y = (\ggraph{F} \circ \ggraph{G})_x$.
\end{proof}

\begin{proposition} \label{generalized_graph_supporting_function_proposition}
Assume $F  \in \G(\Omega_1 , \Omega_2)$ is a c-bounded Colombeau map with representative 
$(F_{\varepsilon})_{\varepsilon}$.
Then the supporting function of the generalized graph $\ggraph{F}$ defined by
$$
H_F(x,w) := \sup_{a \in \ggraph{F}_x} \langle a,w \rangle
$$
can be obtained by
\begin{equation} \label{generalized_graph_supporting_function} \index{generalized graph!supporting function}
H_F(x,w) = \lim_{\delta \rightarrow 0} \sup_{y\in B_{\delta}(x)\cap \Omega_1, \mu \in ]0,\delta]} \langle F_{\mu}(y), w \rangle.
\end{equation}
\end{proposition}

\begin{proof}

Let $w\in\R^m$ be fixed. We put $g(\mu,y):=\langle F_{\mu}(y), w \rangle$ for fixed $w\in\R^n$ and $M_{\delta}(x):=]0,\delta] \times  B_{\delta}(x)$
(which is a descending family of sets), then Lemma \ref{supremum_approximating_sequence} yields
$$
h_{\delta}(x,w):= \sup_{(\mu,y) \in  M_{\delta}(x)} g(\mu,y) = \lim_{\gamma \rightarrow 0} g(\mu_{\delta,\gamma}, x_{\delta,\gamma}) 
$$
for some converging net $(\mu_{\delta,\gamma}, x_{\delta,\gamma})_{\gamma \in ]0,1]}$ in $M_{\delta}(x)$.

For all $\rho,\delta>0$ there exists a $\tau_1(\rho, \delta) \in ]0,\delta]$ such 
that $|h_{\delta}(x,w) -  g(\mu_{\gamma,\delta},y_{\gamma ,\delta})| < \rho/2$
for all $\gamma  \in ]0, \tau_1(\rho,\delta)]$.

Since the right-hand side of (\ref{generalized_graph_supporting_function}) is defined by 
$$
h(x,w)=\lim_{\delta\rightarrow 0} h_\delta(x,w)
$$
there exists for all $\rho > 0$ some $\tau_2(\rho)\in ]0,\rho]$ such that
$|h(x,w) - h_{\delta}(x,w)| < \rho/2$ for all $\delta \le \tau_2(\rho)$. Setting $\psi(\varepsilon):= \mu_{\tau_1(\varepsilon, \tau_2(\varepsilon) ),\tau_1(\varepsilon)}$ (which implies $\psi(\varepsilon) \le \varepsilon$ since $\mu_{\gamma,\delta}\le \delta$, thus $\psi \in \mathcal{T}$) and $x_{\varepsilon} := y_{\tau_1(\varepsilon,\tau_2(\varepsilon)),\tau_2(\varepsilon)}$
we obtain the estimate
\begin{multline*}
|h(x,w) - \langle F_{\psi(\varepsilon)}(x_{\varepsilon}),w \rangle| 
\le
|h(x,w) - h_{\tau_2(\varepsilon)}(x,w)|+
|h_{\tau_2(\varepsilon)}(x,w) - g(\psi(\varepsilon),x_{\varepsilon}))| \\
=  |h(x,w) - h_{\tau_2(\varepsilon)}(x,w)| +    |h_{\tau_2(\varepsilon)}(x,w) -  g(\mu_{\tau_1(\varepsilon, \tau_2(\varepsilon) ),\tau_2(\varepsilon)},x_{\tau_1(\varepsilon, \tau_2(\varepsilon) ) ,\tau_2(\varepsilon)})|
\le \varepsilon/2 + \varepsilon/2 = \varepsilon.
\end{multline*}
We conclude $h(x,w) = \lim_{\varepsilon \rightarrow 0} \langle F_{\psi(\varepsilon)}(x_{\varepsilon}),w \rangle$.
Note that any clusterpoint $a \in \cp{ (F_{\psi(\varepsilon)}(x_{\varepsilon}))_{\varepsilon} } \subseteq\ggraph{F}_x$ 
has the property that $h(x,w)=\langle a, w \rangle \le H_F(x,w)$ for $w\in\R^m$ fixed. By repeating the steps of the proof for all $w\in\R^m$ the bound
$$
h(x,w) \le H_F(x,w)
$$
is obtained for all $x\in \Omega_1$ and $w\in\R^m$.

It remains to prove $H_F(t,x) \le h(t,x)$: Since $H_F(x,w)= \langle a_w, w\rangle$
for fixed $x\in \R^n$, where $a_w = \lim_{\varepsilon \rightarrow 0 } F_{\tau(\varepsilon)}(x_{\varepsilon})$ 
for some map $\tau \in \mathcal{T}$ and $x_{\varepsilon}\rightarrow x$. Without loss of generality we assume $\tau(\varepsilon) \le \varepsilon$.
It follows that 
$$
H_F(x,w) = \lim_{\varepsilon \rightarrow 0} \langle F_{\tau(\varepsilon)} (x_{\varepsilon}),w \rangle.
$$
For all $\delta>0$ we can find some $\varepsilon'(\delta) \in]0,\delta]$ such that $|x_{\varepsilon}-x| <\delta$ for $\varepsilon \in ]0,\varepsilon'(\delta)]$ and $\tau(\varepsilon) \le \varepsilon \le \varepsilon'(\delta) \le \delta$ for all $\varepsilon \in ]0,\varepsilon'(\delta)]$. It follows that
\begin{multline*}
H_F(x,w) = \lim_{\varepsilon \rightarrow 0}  \langle F_{\tau(\varepsilon)} (x_{\varepsilon}),w \rangle
=\lim_{\delta \rightarrow 0} \sup_{\varepsilon \in ]0, \varepsilon'(\delta)]}  \langle F_{\tau(\varepsilon)} (x_{\varepsilon}),w \rangle \\
\le \lim_{\delta \rightarrow 0} \sup_{\mu \in ]0, \delta], y\in B_{\delta}(x)}  \langle F_{\mu} (y),w \rangle = \lim_{\delta \rightarrow 0}h_{\delta}(x,w) = h(x,w).
\end{multline*}
\end{proof}

\begin{example}
Recall Example \ref{heaviside_ggraph_example}, where we considered the Colombeau embedding of the 
Heaviside function $F=\iota(\Theta)$.
If we put $g:= \Theta \ast \rho$, where $\rho$ is
the mollifier of the embedding $\iota$ with $\int \rho(y) \m{y}=1$ and $\int y^k \rho(y) \m{y} = 0$ for all $k \ge 1$, we have that $F_{\varepsilon}(x):=g(x/\varepsilon)$ is a representative of $F$. 
Note that $\lim_{y \rightarrow -\infty} g(y) =0$ and $\lim_{y \rightarrow +\infty} g(y) =1$.
According to Proposition \ref{generalized_graph_supporting_function_proposition} the supporting function $H_F$ can be obtained by
$$
H_F(x, +1) = \lim_{\delta \rightarrow 0} \sup_{y\in B_{\delta}(x)\cap \Omega_1, \mu \in ]0,\delta]} g(y/\mu) = \lim_{\delta \rightarrow 0} \sup_{z\in \delta^{-1} \cdot B_\delta(x)} g(z) = \left\{ 
\begin{array}{ll} 
0 & x > 0 \\
\sup_{y\in \R} g(y) & x=0 \\
1 & x < 0
\end{array}
\right.
$$
and
$$
H_F(x, -1) = \lim_{\delta \rightarrow 0} \sup_{y\in B_{\delta}(x)\cap \Omega_1, \mu \in ]0,\delta]} -g(y/\mu) = - \lim_{\delta \rightarrow 0} \inf_{z\in \delta^{-1}\cdot B_\delta(x)} g(z) = \left\{ 
\begin{array}{ll} 
0 & x > 0 \\
-\inf_{y\in \R} g(y) & x=0 \\
-1 & x < 0 .
\end{array}
\right. 
$$ 
Since $F \in \colmap{\R}{\R}$ (space dimension $m=1$) it follows that $\ggraph{F}_x$ is convex for all $x\in\R$, thus the supporting function $H_F$
generates the graph by
\begin{multline*}
\ggraph{F} = \{a\in\R \mid \forall w\in\R: \langle a, w\rangle \le H_F(x,w) \} \\
= \left(\{x\in\mathbb{R}\mid x<0\} \times \{0\}\right) \cup \left(\{0\} \times \overline{{\rm Im}(g)} \right) \cup \left(\{x\in\mathbb{R}\mid x>0\} \times \{1\}\right),
\end{multline*}
which is consistent with the result of Example \ref{heaviside_ggraph_example}.
\end{example}

\begin{theorem} \label{colombeau_supporting_function_proposition}
Let $F  \in \G(J \times \R^n)^n$ be a Colombeau map with a representative $(F_{\varepsilon})_{\varepsilon}$ such that
there exists a positive $\beta\in L^1_{\rm loc}(J)$ with
$\sup_{y\in\R^n}{|F_{\varepsilon}(t,y)|} \le \beta(t)$ for $t\in J$. Then the generalized graph $\ggraph{F(t,\cdot)}$
exists for almost all $t\in J$ fixed. Its supporting function defined by
\begin{equation*} 
H_F(t,x,w) = \lim_{\delta \rightarrow 0} \sup_{y\in B_{\delta}(x), \mu \in ]0,\delta]} \langle F_{\mu}(t,y), w \rangle,
\end{equation*}
for almost all $t\in J$ satisfies the following properties:
\begin{enumerate}
\item $t \mapsto H_F(t,x,w)$ is Lebesgue measurable on $J$ for all $x,w \in \R^n$,
\item $x \mapsto H_F(t,x,w)$ is a upper semi-continuous for almost all $t \in J$ and $w \in \R^n$,
\item there exists a positive function $\beta \in L_{\rm loc}^1(J)$ such that 
 $\sup_{x\in\mathbb{R}^n} |H_F(t,x,w/|w|)| \le \beta(t)$ for almost all $t\in J$ and all $w\in \R^n$.
\end{enumerate}
\end{theorem}

\begin{proof}
Observe that $\sup_{y\in\R^n}{|F_{\varepsilon}(t,y)|} \le \beta(t)$ implies that $x \mapsto F_{\varepsilon}(t,x)$ defines a c-bounded Colombeau map for 
almost all $t\in J$, so the generalized graph $\ggraph{F(t,\cdot)}$ is defined for almost all $t\in J$. 
Thus property (ii) follows since $x \mapsto \ggraph{F(t,x)}$ is a closed, locally bounded set-valued map for almost all $t\in J$, implying that
its supporting function $x \mapsto H_F(t,x,w)$ is upper semi-continuos.

Proposition \ref{generalized_graph_supporting_function_proposition} immediately yields that the supporting function of $\ggraph{F(t,\cdot)}$ can be obtained by
$$
H_F(t,x,w) = \lim_{\delta \rightarrow 0} \sup_{y\in B_{\delta}(x), \mu \in ]0,\delta]} \langle F_{\mu}(t,y), w \rangle.
$$
Property (iii) then follows from the estimate 
$$
\sup_{y\in \R^n, \mu \in ]0,1]}  \langle F_{\mu}(t,y), w/|w| \rangle \le \sup_{y\in \R^n, \mu \in ]0,1]} |F_{\mu}(t,y)| \le \beta(t)
$$
for almost all $t\in J$.

Due to Lemma \ref{supremum_approximating_sequence} we have $\sup_{y\in B_{\delta}(x), \mu \in ]0,\delta]} \langle F_{\mu}(t,y), w \rangle = 
\lim_{\gamma \rightarrow 0} \langle F_{\mu_{\delta, \gamma}}(t,x_{\delta, \gamma}), w \rangle$ for some sequence $(\mu_{\delta, \gamma},x_{\delta, \gamma} )_{\gamma \in ]0,1]}$
with $\mu_{\delta, \gamma} \in ]0,\delta]$ and $x_{\delta, \gamma} \in B_{\delta}(x)$.

It follows that $\sup_{y\in B_{\delta}(x), \mu \in ]0,\delta]} \langle F_{\mu}(t,y), w \rangle$ is Lebesgue measurable as the pointwise limit 
of a net of Lebesgue measurable functions. 
Finally the same argument yields that $t\mapsto H_F(t,x,w)$ is Lebesgue measurable.
\end{proof}

\begin{theorem}[Approximation property of the supporting function $H_F$] \label{supporting_function_approximation_property}
Let $F  \in \G(J \times \R^n , \R^n)$ be a Colombeau map as in Theorem
\ref{colombeau_supporting_function_proposition}, then $H_F$ denotes the supporting function of $\ggraph{F(t,\cdot)}$ 
for almost all $t\in J$.
If $(\xi_{\varepsilon})_{\varepsilon}$ is a net of smooth functions converging uniformly on compact sets
to some $\xi \in \Con(J)^n$, then 
$$
\limsup_{\varepsilon \rightarrow 0} \langle F_{\varepsilon}(s, \xi_{\varepsilon}(s)),w \rangle \le 
H_F(s, \xi(s), w)
$$
for all $w\in \R^m$ and almost all $s\in J$.
\end{theorem}

\begin{proof}
First of all we observe that
$$
s \mapsto \langle F_{\varepsilon}(s, \xi_{\varepsilon}(s)),w \rangle
$$
is a family of functions in $L^1_{\rm loc}(J)$ and  
$$
\langle F_{\varepsilon}(s, \xi_{\varepsilon}(s)),w \rangle  \le 
\sup_{y\in \R^n} |F_{\varepsilon}(s, y)| \le \beta(s)
$$ for almost all $s\in J$ and $\varepsilon \in ]0,1]$. 
Thus
$$
\limsup_{\varepsilon \rightarrow 0} \langle F_{\varepsilon}(s, \xi_{\varepsilon}(s)),w \rangle 
$$
exists for almost all $s\in J$. 

Let $K\csub J$ be fixed, then there exists some $\widetilde{\mu} \in \mathcal{T}$ (without loss of generality we assume
 $\widetilde{\mu}(\varepsilon) \le \varepsilon$) such that 
$\sup_{s\in K}|\xi_{\mu}(s)-\xi(s)|< \varepsilon$ for all 
$\mu\in]0,\widetilde{\mu}(\varepsilon)]$.
We have that
$$
\left(\sup_{\mu \in ]0,\widetilde{\mu}(\varepsilon)]} \langle F_{\mu}(s, \xi_{\mu}(s)),w \rangle \right)_{\varepsilon}
$$
is a subnet of the net $(\sup_{\mu \in ]0,\varepsilon]} \langle F_{\mu}(s, \xi_{\mu}(s)),w \rangle)_{\varepsilon}$, 
which converges for almost all $s\in K$, thus
$$
\lim_{\varepsilon \rightarrow 0} \sup_{\mu \in ]0,\varepsilon]} \langle F_{\mu}(s, \xi_{\mu}(s)),w \rangle
=
\lim_{\varepsilon \rightarrow 0} \sup_{\mu \in ]0,\widetilde{\mu}(\varepsilon)]} \langle F_{\mu}(s, \xi_{\mu}(s)),w \rangle.
$$
We conclude
\begin{multline*}
\limsup_{\varepsilon \rightarrow 0} \langle F_{\varepsilon}(s, \xi_{\varepsilon}(s)),w \rangle =
\lim_{\varepsilon \rightarrow 0} \sup_{\mu \in ]0,\varepsilon]} \langle F_{\mu}(s, \xi_{\mu}(s)),w \rangle
=
\lim_{\varepsilon \rightarrow 0} \sup_{\mu \in ]0,\widetilde{\mu}(\varepsilon)]} \langle F_{\mu}(s, \xi_{\mu}(s)),w \rangle
\\\le
\lim_{\varepsilon \rightarrow 0} \sup_{\mu \in ]0,\widetilde{\mu}(\varepsilon)], y\in B_{\varepsilon}(\xi(s))} \langle F_{\mu}(s, y),w \rangle 
\le
\lim_{\varepsilon \rightarrow 0} \sup_{\mu \in ]0,\varepsilon], y\in B_{\varepsilon}(\xi(s))} \langle F_{\mu}(s, y),w \rangle
= H_F(s,\xi(s),w)
\end{multline*}
for almost all $s\in K$.
\end{proof}

%% file: ode.tex
\chapter{Generalized solutions of ordinary differential equations}

Throughout this chapter $J$ will be some subinterval of $\R$. Starting point of our overview of generalized solution concepts
for ordinary differential equations is the initial value problem
$$
\ode
$$
where the right-hand side $a(t,x)=(a_1(t,x), \dots, a_n(t,x))$ is a map from $J \times \R^n$ to $\R^n$. 
Occasionally we will indicate dependence on the initial conditions by writing $\xi(\cdot;t,x)$ for a solution $\xi$ with $\xi(t)=x$.
 
Let us first consider the well-posedness: The left-hand side of the equation requires a possible solution $\xi$ 
to be at least absolutely continuous in order to ensure differentiability 
almost everywhere on $J$. 
Furthermore it is apparent that the right-hand side of the equation 
has to be in $L^1_{loc}(J)^n$. 
Nevertheless in general it is not sufficient to have $a \in L^1_{\rm loc}(J \times \R^n)^n$ 
in order to have the composition $s \mapsto a(s, f(s))$ in $L^1_{\rm loc}(J)^n$ 
for $f \in AC(J)^n$. 

In the first section we investigate the classical Caratheodory theory, 
which requires the right-hand side $a$ to be in the space 
$L^1_{\rm loc}(J, L^{\infty}(\R^n))^n$ with $a(\cdot, x) \in L^1_{\rm loc}(J)^n$ 
for all $x\in \R^n$ and $a(s,\cdot) \in C(J)^n$ for almost all $s \in J$ 
(referred to as \emph{Caratheodory conditions}), then the existence of 
a solution $\xi \in AC(J)^n$ can be proved.

Of course the continuity required in the space-variable of the 
right-hand side restricts the applicability of the Caratheodory theory. 
Nevertheless if we want to drop this requirement we cannot retain 
on the classical concept of a solution. 

A natural generalization of an ordinary differential equation is the differential inclusion
$$
\odie,
$$
where $A_{t,x}$ is a subset of $\R^n$ for each $(t,x) \in J \times \R^n$. The map
$A:(t,x) \mapsto A_{t,x}$ is called a \emph{set-valued map} (cf. Section \ref{set_valued_map_section}). 
As in the case of Caratheodory theory we have to impose certain conditions (Filippov conditions) on the set-valued map 
$A$ in order to obtain an existence result for the differential inclusion. These results will be presented in the second section.
 
The concept of a differential inclusion can then be utilized to obtain a generalized solution concept 
for ordinary differential equations, when the right-hand side fails continuity in the space variable, 
but is still in $L^1_{\rm loc}(J; L^{\infty}(\R^n))^n$. It is obtained by mapping the right-hand side $a$ to some $A$ 
in the space of set-valued functions satisfying the above-mentioned Filippov conditions
in a way that $A_{t,x} =\{a(t,x)\}$ whenever $a$ fulfills the classical Caratheodory conditions. 
Then the absolutely continuous solution of the differential inclusion is called a  Filippov generalized solution.

Another approach to generalize the classical solution concept for ordinary differential equations 
is the regularization of the right-hand side $a$ by means of convolution with a net of smooth mollifiers, such that
the resulting net of smooth functions $(A_{\varepsilon})_{\varepsilon}$ has the property
$\lim_{\varepsilon \rightarrow 0} \langle A_{\varepsilon}, \varphi \rangle = \langle a, \varphi\rangle$ for all 
$\varphi \in \mathcal{D}'(\R^n)$.
Then a net of solutions $(\xi_{\varepsilon})_{\varepsilon}$ can be obtained by solving the differential equation for fixed 
$\varepsilon \in ]0,1]$ using the classical existence results. If the net of solutions $(\xi_{\varepsilon})_{\varepsilon}$
possesses a subnet converging to some $\xi \in AC(J)^n$, i.e. there exists some $\tau \in \mathcal{T}$ such that
$\xi_{\tau(\varepsilon)}(s) \rightarrow \xi(s)$ uniformly on compact subsets of $J$, then the function $\xi$ 
is a generalized solution in the sense that it satisfies the classical ordinary differential equation 
whenever the initial right-hand side $a$ satisfies the Caratheodory conditions. 

It is remarkable that the latter approach is closely related to the former, since the generalized solution obtained by the
regularization procedure satisfies a certain differential inclusion. The concept of a generalized graph (cf. Chapter \ref{generalized_graph}) 
plays a crucial role in understanding this relation.
 
Another viewpoint of the regularization approach is that the net $(A_{\varepsilon})_{\varepsilon}$ 
can be considered as an element in the Colombeau algebra of generalized functions. 
We put $A := [(A_{\varepsilon})_{\varepsilon}]\in \cf{J\times \R^n}^n$. Note that in general $A=(A_1, \dots, A_n) \ne (\iota(a_1), \cdots, \iota(a_n))$, 
where $\iota: \mathcal{D}'(\R^n) \rightarrow \cf{\R^n}$ denotes the standard embedding, 
since we did not specify the net of smooth mollifier any further.
Occasionally it may be convenient to use mollifiers with non-vanishing moments 
(f.e. positive or compactly supported mollifiers), which is not possible when using $\iota$ to embed the right-hand side $a$.

If the ordinary differential equation
$$
\colode
$$
can be solved in the setting of Colombeau generalized functions by some 
$\xi := [(\xi_{\varepsilon})_{\varepsilon}] \in \cf{J}^n$, we call $\xi$ a Colombeau solution with inital condition 
$(\widetilde{t},\widetilde{x}) \in \widetilde{J} \times \widetilde{\R}^n$.

It generalizes the classical solution concept whenever there exists some $\tau \in \mathcal{T}$ 
such that $(\xi_{\tau(\varepsilon)})_{\varepsilon}$ converges to a function $\zeta \in AC(J)^n$ 
uniformly on compact sets of $J$. 
We say $\xi$ has an absolutely continuous \emph{sub-shadow} $\zeta$\index{sub-shadow}. This notion is justified
by the fact that $\zeta$ is an absolutely continuous shadow of $\xi$, 
if $(\xi_{\varepsilon})_{\varepsilon}$ converges to $\zeta$ uniformly on compact subsets of $J$.

$\zeta$ is a generalized solution in the sense that it satisfies the 
classical ordinary differential equation whenever the initial right-hand side fulfills the 
Caratheodory conditions.

\section{Solutions of ordinary differential equations according to Caratheodory} \label{caratheodory_solution_concept}

In this section we will review the classical Caratheodory theory for ordinary differential equations
\begin{equation} \label{ode}
\ode
\end{equation}
with non-smooth right-hand side $a: J \times \R^n \rightarrow \R^n$ or equivalently the corresponding integral equation
\begin{equation} \label{int_eq}
\xi(s) = x + \int_t^s a(\tau,\xi(\tau)) \m{\tau}.
\end{equation}

We follow closely the introduction from the book \cite{Filippov:88}. 
Throughout the entire section the right-hand side $a=(a_1,\ldots, a_n)$ satisfies the conditions
\begin{enumerate} 
\item $a(t,x)$ is Lebesgue measurable in $t$ for all fixed $x\in \mathbb{R}^n$,
\item $a(t,x)$ is continuous in $x$ for almost all $t \in J$, and
\item $\sup_{x\in\mathbb{R}^n} |a(t,x)| \le \beta(t)$ almost everywhere for some positive function $\beta \in L_{\rm loc}^1(J)$,
\end{enumerate}
referred to as \emph{Caratheodory conditions} (CC). \index{Caratheodory conditions}

Note that the first two Caratheodory conditions ensure Lebesgue measurability of the composition $s\mapsto a(s,f(s))$ 
for all $f\in {C}(J)^n$ (due to Lemma \ref{measurability_of_composition_lemma}), 
while the third condition is crucial in the existence proof.

\begin{theorem}[Existence theorem for ordinary differential equations] \label{caratheodory_ode_existence}
Let $J$ be some subinterval of $\R$ and assume that $a=(a_1,\ldots, a_n)$ satisfies (CC) on $J\times\mathbb{R}^n$. 
Let $(t,x) \in J \times \mathbb{R}^n$, then there exists an absolutely continuous solution $\xi=(\xi_1, \ldots, \xi_n)$ to
the ordinary differential equation
$$
\ode
$$
\end{theorem}

\begin{proof}
For sake of completeness we include this classical proof, following Filippov (\cite[Ch. 1,Thm. 1]{Filippov:88}).
Without loss of generality we assume that $t$ lies in the interior of $J$. 

Note that the Caratheodory condition (i) and (ii) imply (due to Lemma \ref{measurability_of_composition_lemma}) 
that for any $f \in C(J)^n$ the composition $s \mapsto a(s,f(s))$ is in $L^1_{\rm loc}(J)$.
Let us consider a compact subset of $J$ in the forward direction, i.e. set $J^{+}:= J \cap [t,T^{+}]$ for some 
$T_{+} \in J$ with $T^{+}>t$.
For any $k\in\N$ we put $\lambda_k:=\frac{T^{+}-t}{k}$ in order to obtain a decomposition
of the interval $J^+$ in $k$ subintervals by $J_l^{+} := [t+(l-1) \lambda_k, t+ l\lambda_k]$ for $1\le l \le k$, such that
$$
J^{+} := \bigcup_{l=1}^k J^+_{l}.
$$ 
We construct a function $\xi_k \in AC(J^+)$ inductively by setting
$$
\xi_k(s) := x +  \int_{t}^{\max{(t,s-\lambda_k)}} a(\tau+ \lambda_k, \xi_k(\tau)) \m{\tau} 
$$
for all $s\in J^+$. Whenever $s\in J_l^+$ for $l\ge 2$, 
the integral expression on the right-hand side of the above definition depends only on $\xi_k$ restricted to $J^+_{l-1}$.
For $s \in  J_1^+$ we have $\xi_k(s) =x$.

Caratheodory condition (iii) implies 
\begin{eqnarray*} 
|\xi_k(s) - \xi_k(r)| = \left|\, \int_{\max{(t,r-\lambda_k)}}^{\max{(t,s-\lambda_k)}} a(\tau+ \lambda_k, \xi_k(\tau)) \m{\tau}\right| 
\le \int^{\max{(r,s)}}_{\min{(r,s)}} \beta(\tau) \m{\tau} \quad \text{ for } s,r \ge t
\end{eqnarray*} 
and $\xi_k(s) \in B_{\rho}(x)$ for all $s\in J^+$, where $\rho := \int^{T_{+}}_t \beta(\tau) \m{\tau}$.
Obviously $(\xi_k)_k$ is a equi-continuous and equi-bounded family in $C(J^+)^n$. The Theorem of Arzela-Ascoli yields an uniformly convergent subsequence 
$\xi_{k_j} \stackrel{j\rightarrow \infty}{\rightarrow} \xi$, where $\xi \in C(J^+)^n$ with $\xi(t)=x$.

Observe that
\begin{multline*}
|\xi_{k_j}(s-\lambda_{k_j}) - \xi(s)| \le |\xi_{k_j}(s-\lambda_{k_j}) - \xi_{k_j}(s)|+|\xi_{k_j}(s) - \xi(s)| \\
= \int_{s- \lambda_{k_j}}^{s} \beta(\tau) \m{\tau} + |\xi_{k_j}(s) - \xi(s)| \stackrel{j \rightarrow \infty}{\rightarrow} 0
\end{multline*}
holds. So we finally obtain by the theorem of dominated convergence that
\begin{multline*}
\xi(s) = x + \lim_{j\rightarrow \infty} \int_{t}^{\max{(t,s-\lambda_{k_j})}} a(\tau+ \lambda_{k_j}, \xi_{k_j}(\tau)) \m{\tau} 
=  
x + \lim_{j\rightarrow \infty} \int_{t}^{\max{(t,s)}} a(\tau, \xi_{k_j}(\tau- \lambda_{k_j})) \m{\tau}  \\=
x +  \int_{t}^{\max{(t,s)}} \lim_{j\rightarrow \infty} a(\tau, \xi_{k_j}(\tau- \lambda_{k_j})) \m{\tau} =
x +  \int_{t}^{\max{(t,s)}} a(\tau, \xi(\tau) ) \m{\tau}.
\end{multline*}
Equivalently $\xi$ satisfies the ordinary differential equation $\dot{\xi}(s)= a(s, \xi(s)),\ \xi(t)=x$ for almost all $s \in J^+$. The proof for backward problem on some $J^{-}:= J \cap [T^{-}, t]$ for some $T^{-} \in J$ with $T^{-}<t$ is analogously.

In the case where $J$ is an open interval or unbounded, it is straight-forward to decompose the domain in countable many compact subintervals. Then the local solutions for each subinterval are continuously glued together (by choosing appropriate initial conditions for the local solutions) to obtain a global solution.
\end{proof}

\begin{definition}\label{solution_set} \index{solution set}
Let $\Omega$ be a non-empty subset of $J \times \R^n$, then we define
$$
\Xi(\Omega) := \{ \xi \in AC(J)^n \mid  \exists (t,x)\in \Omega: \dot{\xi}(s) = a(s,\xi(s)) \text{ a.e., }\ \xi(t)=x \}
$$
denoted as (Caratheodory) solution set with initial conditions in $\Omega$. 
If we put $\Xi_{t,x} := \Xi(\{(t,x)\})$ it obviously holds that 
$$
\Xi(\Omega) = \bigcup_{(t,x) \in \Omega} \Xi_{t,x}.
$$

The solution set $\Xi(\Omega) $ is called forward (resp. backward) unique, if $\xi_1, \xi_2 \in \Xi(\Omega)$
with $\xi_1(t) = \xi_2(t)$ implies $\xi_1(s)=\xi_2(s)$ for all $s\ge t$ (resp. $s\le t$).

In the case where $J= \R$ and $\Omega'$ is a non-empty subset of $\R^n$, we refer to a solution set $\Xi(\R\times \Omega')$ as autonomous, if $\xi\in \Xi(\R \times \Omega')$
implies $\xi(\cdot+r) \in \Xi(\R \times \Omega')$ for any $r\in \R$. 
\end{definition}

\begin{theorem}[Properties of the solution set] \label{caratheodory_solution_set_properties}
Let $\Omega$ be a non-empty subset of $\R^n$, then the solution set $\Xi(J \times \Omega)$ has the following properties:
\begin{enumerate} 
\item For every $(t,x) \in J \times \Omega$, there exists $\xi \in \Xi(J \times \Omega)$ with $\xi(t)=x$.
\item If $t\in J$ and $\Omega$ is closed, then $\Xi(\{t\} \times \Omega)$ is a closed subset of $C(J)^n$. 
\item $(\xi)_{\xi \in \Xi(J \times \Omega)}$ is an equi-continuous family in $C(J)^n$.
\item If $\Omega \csub \R^n$, then $\Xi(\{t\} \times \Omega)$ is a compact subset of $C(J)^n$.
\end{enumerate}
If the right-hand side is time-independent, we put $J=\R$ and the solution set $\Xi(\R \times \Omega)$ is autonomous.
\end{theorem}

\begin{proof}
(i) By Theorem \ref{caratheodory_ode_existence} the set $\Xi(J \times \Omega)$ is a non-empty subset of 
$C(J)^n$, such that for every $(t,x) \in J \times \Omega$ there exists a $\xi \in \Xi(J \times \Omega)$ with $\xi(t)=x$. 

(ii) First we prove that $\Xi(\{t\} \times \Omega)$ is a closed set in $C(J)^n$ if $\Omega$ is closed:
Let $(\xi_k)_{k\in\N}$  be sequence in $\Xi(\{t\} \times \Omega)$ uniformly converging (on compact sets)
to some $\xi \in C(J)^n$. 
In particular $x_k=\xi_k(t) \rightarrow \xi(t)=:x \in \Omega$, since $\Omega$ is closed.  
It follows that  
\begin{multline*}
\xi(s) =\lim_{k\rightarrow\infty} \xi_k(s) =
 \lim_{k\rightarrow\infty} \left( x_k+\int_{t}^s a(\tau, \xi_k(\tau), w) \m{\tau} \right) \\
= x +  \int_{t}^s \lim_{k \rightarrow \infty} a(\tau, \xi_k(\tau)) \m{\tau} 
= x + \int_{t}^s a(\tau, \xi(\tau)) \m{\tau}, 
\end{multline*}
which implies $\xi \in \Xi(\{t\} \times \Omega)$ satisfying the initial condition $\xi(t)=x \in \Omega$, 
so $\Xi(\{t\} \times \Omega)$ is a closed subset of $C(J)^n$.

(iii) Using the integral equation (\ref{int_eq}) we have for all $s,t \in J$ that 
\begin{equation} \label{caratheodory_equi_continuous_estimate}
|\xi(s) -\xi(t)| \le  \int_{\min{(t,s)}}^{\max{(t,s)}} \beta(\tau) \m{\tau}
\end{equation}
holds uniformly for all $\xi \in \Xi(J \times \Omega)$, thus $\Xi(J \times \Omega)$ is an equi-continuous family of continuous
functions.

(iv) Since $K\csub \R^n$ the estimate
\begin{eqnarray*}
\sup_{\xi \in \Xi(\{t\} \times K)} |\xi(s)| \le \sup_{x \in K} |x| + 
\sup_{\xi \in \Xi(\{t\} \times K)} \left| \int_t^s a(\tau, \xi(\tau)) \m{\tau} \right| \le \sup_{x \in K} |x| + 
\int_{\min{(t,s)}}^{\max{(t,s)}} \beta(\tau) \m{\tau} < \infty
\end{eqnarray*}
holds for all $s \in J$, thus $(\xi )_{\xi \in \Xi(\{t\} \times K)}$ is an equi-bounded family. 

By the Theorem of Arzela-Ascoli it follows that $\{ \xi : \xi \in \Xi(\{t\} \times K) \}$ is pre-compact in 
$C(J)^n$. The closedness of $\{ \xi :\xi \in \Xi(\{t\} \times K) \}$ in $C(J)^n$ was already shown in (ii), thus
$\Xi(\{t\} \times K ) \csub C(J)^n$.

If the right-hand side of (\ref{ode}) is time-independent, it is  obvious that we can put $J=\R$. For any $\xi \in \Xi( \R \times \Omega)$
with $\xi(t) =x \in \Omega$ we can define $\widetilde{\xi} = \xi(\cdot - r)$ for some $r\in \R$.
Upon change of variable in the integral on the right-hand side of (\ref{int_eq}) it follows directly that 
$\widetilde{\xi}$ again satisfies the ordinary differential equation with modified initial condition 
$\widetilde{\xi}(t+r)=x \in \Omega$, thus $\widetilde{\xi} \in \Xi(\R \times \Omega)$.
\end{proof}

\begin{theorem}[Unique solution sets] \label{flow_map}
If there exists some $\alpha \in L^{1}_{\rm loc}(J)$ such that the right-hand side $a$ satisfies
the forward uniqueness condition
\begin{equation} \label{forward_uniqueness_condition} \index{forward uniqueness condition}
\langle x-y,  a(s,x)-a(s,y) \rangle  \le \alpha(s) |x-y|^2 
\end{equation}
resp. the backward uniqueness condition
\begin{equation} \label{backward_uniqueness_condition} \index{backward uniqueness condition}
 \langle x-y,  a(s,x)-a(s,y) \rangle  \ge - \alpha(s) |x-y|^2
\end{equation}
for almost all $s \in J$ and all $x,y\in \R^n$, then the solution set $\Xi(J \times \R^n)$ is forward (resp. backward) unique. 
\end{theorem}

\begin{proof}
We prove only the forward uniqueness (as the proof of backward uniqueness is analogous):
Let $\xi_1, \xi_2 \in \Xi(J \times \Omega)$ with $\xi_1(t)=\xi_2(t)=x$. 
Setting $\zeta(s) :=|\xi_1(s)- \xi_2(s)|^2$ we observe $\zeta(t)=  |\xi_1(t) -\xi_2(t)|^2 = 0$ and using that 
$\xi_1, \xi_2$ satisfy the ordinary differential equation, we derive that the estimate
\begin{eqnarray*}
\partial_s \zeta(s)  &=& 2 \langle \xi_1(s) -\xi_2(s), \dot{\xi_1}(s)-\dot{\xi_2}(s)  \rangle 
= 2 \langle \xi_1(s) -\xi_2(s) , a(s,\xi_1(s)) - a(s,\xi_2(s)) \rangle \\
&\le &  2 \alpha(s) \zeta(s)
\end{eqnarray*}
holds for almost all $s\in J$, thus Gronwall's inequality gives
\begin{eqnarray*}
\zeta(s) \le \zeta(t) \exp{\left(2\int_t^s \alpha(r) \m{r}\right)} = 0
\end{eqnarray*}
for all $s \ge t$. It follows that $\zeta(s) = |\xi_1(s)- \xi_2(s)|^2 \equiv 0$ for all $s\ge t$, 
so $\Xi(J\times \Omega)$ is forward unique. 
\end{proof}

\begin{theorem}[Forward unique flow map] \label{forward_flow_map} \index{flow!forward unique}
Let $\Xi(J\times \R^n)$ be a forward unique solution set. Put $J_t^+:=[t,\infty[ \cap J$, then the map 
defined by
\begin{equation*}
\begin{array}{rll}
 \chi_t: J_t^+ \times \R^n & \rightarrow & \R^n \\
 (s,x) &\mapsto& \xi(s;t,x) 
\end{array}
\end{equation*}
for some $\xi(\cdot;t,x) \in \Xi(J\times \R^n)$ with $\xi(t;t,x)=x$ is uniquely defined,
called \emph{forward flow map}.

It has the properties that
\begin{enumerate}
\item $\chi_t$ is a continuous map, 
\item $\chi_t(s, \cdot)$ is proper and onto for all $s\in J_{t}^+$ fixed, and
\item $\chi_r(s,\chi_t(r,x)) = \chi_t(s,x)$ for $s\ge r \ge t$.
\end{enumerate}

If the right-hand side of (\ref{ode}) is time-independent, we have the additional flow property
$\chi_t(s, x) =\chi_0(s-t, x)$ for all $s \ge t$, so property (iii) can be written in the well-known form
$$
\chi_{t_1}(s, \chi_{t_2}(r,x)) = \chi_{t_1+t_2}(s+ r ,x) 
$$
for $s\in J_{t_1}^+, r\in J_{t_2}^+$ with $t_1,t_2 \in J$.
\end{theorem}

\begin{proof}
First of all we note that $\chi_t$ is uniquely defined, as the solution set $\Xi(J \times \R^n)$ is forward unique.

(i) Considering the continuity of $\chi_t$, we first prove that $\chi_t(s,x)$ is separately continuous. 
The continuity in $s$ is obvious by the definition of $\chi_t$.

Continuity in $x$ we shall prove by contradiction: Assume $\chi_t(s,x)$ is not continuous in $x$ at fixed $s$. Then
there exists a sequence $(x_k)_{k\in\mathbb{N}}$ converging to $x$, $\gamma >0$ and $k_0\in \N$ such that
$|\chi_t(s,x) - \chi_t(s,x_k)|>\gamma$ for all $k\ge k_0$.
Choose $K\csub \R^n$ such that $x_k \in K$ for all $k\in\N$.

For each $x_k$ we can find a $\xi_k\in\Xi(\{t\} \times K)$ such that $\xi_k(t) = x_k$ and $\chi_t(s,x_k)=\xi_k(s)$ for all $s\in J_t^+$. 
Furthermore there exists a $\xi\in\Xi(\{t\} \times K)$ such that $\xi(t) = x$ and $\chi_t(s,x)=\xi(s)$  for all $s\in J_t^+$.
The compactness of $\Xi(\{t\} \times K))$ in $C(J)^n$ yields that, there exists a locally uniformly converging subsequence 
$\xi_{k_l} \stackrel{l \rightarrow \infty}{\rightarrow} \tilde{\xi} \in \Xi(\{t\} \times K))$. 
The forward uniqueness of $\Xi(J \times \R^n)$ gives that $\widetilde{\xi}(s) = \xi(s)$ for all $s\in J_t^+$, thus $\lim_{l\rightarrow \infty} |\chi_t(s,x)-\chi_t(s, x_{k_l})| = 0 < \gamma$ contradicting the assumption.

The equi-continuity of $\Xi(J \times \R^n)$ yields the estimate
\begin{multline*}
|\chi_t(s,x) - \chi_t(s',x')| \le |\chi_t(s,x) - \chi_t(s,x')| + |\chi_t(s,x') - \chi_t(s',x')| \le
|\chi_t(s,x) - \chi_t(s,x')| \\+ \sup_{\xi \in \Xi(J \times \R^n)} |\xi(s) - \xi(s')|,
\end{multline*}
which, as $\chi_t$ was already shown to be separately continuous, implies the joint continuity of $\chi_t$.

(ii)
There exists some $\xi \in \Xi_{t,x}$ with $\chi_t(s,x) = \xi(s)$, thus 
$$
 |\chi_t(s,x) -x| = |\xi(s)-x| = \left|\int_t^s a(\tau,\xi(\tau)) \m{\tau} \right|\le \int_t^s \beta(\tau) \m{\tau}
$$
holds for all $s \ge t$. For any compact set $K\csub \R^n$ we have that 
$\chi_t^{-1}(\{s\}\times K) \subseteq K':=\left\{x\in \R^n \mid \exists x_0 \in K : |x-x_0| \le  \int_t^s \beta(\tau) \m{\tau} \right\}$,
implying the properness of $x \mapsto \chi_t(s,x)$.

To show that $x\mapsto \chi_t(s,x)$ is onto for all $s\in J$ fixed, we have to show $\chi_t(s, \R^n)=\R^n$. 
We prove by contradiction: Assume there exists some $y\not\in \chi_t(s,\R^n)$, then by Theorem
\ref{caratheodory_solution_set_properties} (i) there exists a solution $\xi \in \Xi_{s,y}$ with $\xi(s)=y$. Putting $x:=\xi(t)$ we observe that $\xi \in \Xi_{t,x}$ and
forward-uniqueness of $\Xi(J\times\R^n)$ yields that $\xi(r) = \chi_t(r,x)$ for all $r\ge t$. In particular we have $y =\xi(s)=\chi_t(s,x) \in \chi_t(s,\R^n)$ which is a contradiction. 

(iii)
For fixed $x\in\mathbb{R}^n$ we have that $\chi_t(s,x) = \xi_1(s), s\in J_t^+$ for some $\xi_1 \in \Xi(J \times \R^n)$
with $\xi_1(t)=x$.
Fixing $r\in J_t^{+}$ we may choose $\xi_2 \in \Xi(J \times \R^n)$ with 
$\xi_2(r) = \chi_t(r,x)$ and $\xi_2(s) =\chi_r(s, \chi_t(r,x))$ for $s\ge r$. Since $\xi_1, \xi_2 \in \Xi(J \times \R^n)$ with $\xi_1(r) =  \chi_t(r,x)$ and $\xi_2(r) =  \chi_t(r,x)$ it follows by the forward uniqueness of $\Xi(J \times \R^n)$ that $\xi_1(s)=\xi_2(s)$ for all $s\ge r$, thus
$$
 \chi_r(s,\chi_t(r,x)) = \chi_t(s,x) 
$$
for $s\in J_t^{+} \cap J_r^{+}$ and $r\in J_t^{+}$. 

In the case where the right-hand side of \ref{ode} is time-independent, we put $J=\R$ and have that $\Xi(\R^{n+1})$ is autonomous
due to Theorem \ref{solution_set}. As $\chi_0(s, x) = \xi_1(s)$ for some $\xi_1 \in \Xi(\R^{n+1})$ with $\xi_1(0)=x$
and $\chi_t(s,x) = \xi_2(s)$ for some $\xi_2 \in \Xi(\R^{n+1})$ with $\xi_2(t)=x$. Due to the autonomy of the solution set 
it follows that $\widetilde{\xi_2}(s) := \xi_2(s-t)$ is again in $\Xi(\R^{n+1})$ . As $\widetilde{\xi_2}(t) = \xi_1(t) =x$ the forward uniqueness of $\Xi(\R^{n+1})$ immediately yields $\widetilde{\xi_2}(s) = \xi_1(s)$ for all $s \in J_t^+$, thus
$\chi_t(s,x) = \chi_0(s-t,x)$ for all $s\in J_t^+$. Using this relation and replacing $r$ by $r+t-t_2$ and $s$ by $s+r+t-t_2-t_1$ we can rewrite property (iii) in order to obtain
$$
\chi_{t_1}(s, \chi_{t_2}(r,x)) = \chi_{t_1+t_2}(s+ r ,x) 
$$
for all $r\in J_{t_1}^+, s \in J_{t_2}^+$ with $t_1, t_2 \in \R$.
\end{proof}

\begin{theorem}[Backward unique flow maps]\label{backward_flow_map} \index{flow!backward unique}
Let $\Xi(J\times \R^n)$ be a backward unique solution set. Put $J_t^-:=]-\infty,t] \cap J$, then the map 
defined by
\begin{equation*}
\begin{array}{rll}
 \Phi_t: J_t^+ \times \R^n & \rightarrow & \R^n \\
 (s,x) &\mapsto& \xi(s;t,x) 
\end{array}
\end{equation*}
for some $\xi(\cdot;t,x) \in \Xi(J\times \R^n)$ with $\xi(t;t,x)=x$ is uniquely defined,
called \emph{backward flow map}.

It has the properties that
\begin{enumerate}
\item $\Phi_t$ is a continuous map, 
\item $\Phi_t(s, \cdot)$ is proper for all $s\in J_{t}^-$ fixed, and
\item $\Phi_r(s,\Phi_t(r,x)) = \Phi_t(s,x)$ for $s\le r \le t$.
\end{enumerate}

If the right-hand side of (\ref{ode}) is time-independent, we have the additional property
$ \Phi_t(s, x) =\Phi_0(s-t, x)$ for all $s \le t$, so property (iii) can be written in the well-known form
$$
\Phi_{t_1}(s, \Phi_{t_2}(r,x)) = \Phi_{t_1+t_2}(s+ r ,x) 
$$
for $s\in J_{t_1}^-, r\in J_{t_2}^-$ with $t_1,t_2 \in \R$.
\end{theorem}

\begin{proof}
Analogous to the proof of Theorem \ref{forward_flow_map}.
\end{proof}

\begin{remark}[Discontinuous flow maps] \index{flow!discontinuous}
Due to Theorem \ref{caratheodory_ode_existence} it is of course possible (regardless of any uniqueness property of the solution set $\Xi(J\times \R^n)$) to define a "global" flow map $\Theta_t: J \times \R^n \rightarrow \R^n, \quad (s,x) \mapsto \xi(s)$ where $\xi \in \Xi(J\times \R^n)$ with $\xi(t) =x$. But such a map is not uniquely defined and
 $\Theta_t(s,\cdot)$ is in general neither continuous, nor onto $\R^n$ for fixed $s,t \in J$.
 Nevertheless at least $(\Theta_t(\cdot, x))_{x \in \R^n}$ is equi-continuous family in $C(J)^n$ and
$\Theta_t(s, \cdot)$ is a proper map, i.e. for all $K \csub \R^n$ there exists some $K' \csub \R^n$ such that
the inverse image $\Theta_t^{-1}(\{s\} \times K)$ is contained in $K'$. 
In the case where $\Xi(J \times \R^n)$ is forward (resp. backward) unique, we have that
 $$
 \Theta_t \mid_{J_t^+} = \chi_t \quad (\textrm{resp. } \Theta_t \mid_{J_t^-} = \Phi_t )
 $$
holds. In the case when the solution set is unique, the globally defined flow map $\Theta_t$ is uniquely defined by 
the forward flow on $J_t^+$ and the backward flow on $J_t^-$, as we will see.
\end{remark}

\begin{theorem}[Unique flow maps] \label{unique_flow_map} \index{flow!properties}
Let $\Xi(J \times \R^n)$ be a unique solution set, then it generates a unique, global (defined on the whole interval J) flow map by
$$
\Theta_t(s,x) := \left\{\begin{array}{ll}
\chi_t(s,x) & s \ge t \\
\Phi_t(s,x) & s < t
\end{array} \right.
$$ for $t\in J$ fixed. 
It satisfies the properties 
\begin{enumerate}
\item $\Theta_t$ is a continuous map,
\item $\Theta_t(s, \cdot)$ is proper for all $s \in J$ fixed, and
\item $\Theta_s(r,\Theta_t(s,x)) = \Theta_t(r,x)$ for $s,r,t \in J$ with the special case (setting $r=t$)\\
 $\Theta_s(t,\Theta_t(s,\cdot)) = {\rm id}_{\R^n}$ for $s,t \in J$.
\end{enumerate}
We can write property (iii) in terms of the forward and backward flow as
$$
\Phi_s(t, \cdot) = (\chi_t(s, \cdot))^{-1} 
$$
for all $s,t\in J$ with $s\ge t$. 

In the case where the right-hand side of (\ref{ode}) is time independent we again obtain
$\Theta_t(s,x) = \Theta_0(s-t,x)$ for all $s,t \in J$ and $x\in \R^n$.
\end{theorem}

\begin{proof}
The uniqueness of $\Theta_t$ follows immediately by combining the uniqueness results from Theorem \ref{forward_flow_map} and Theorem \ref{backward_flow_map}. It is globally defined since $J_t^{-} \bigcup J_t^+ = J$.

(i) The continuity of $\Theta_t$ is obvious by Definition, as $\chi_t(t,\cdot) = \Phi_t(t,\cdot) = {\rm id}_{\R^n}$, $\chi_t$ is continuous on $J_{t}^+ \times \R^n$ and $\Phi_t$ is continuous on $J_t^{-} \times \R^n$ due to Theorem \ref{forward_flow_map} and \ref{backward_flow_map}.

(ii) $\Theta_t(s, \cdot)$ is proper follows immediately from the properness of $\chi_t(s, \cdot)$ and $\Phi_t(s, \dot)$ 
on the respective domains as proved in Theorem \ref{forward_flow_map} and \ref{backward_flow_map}.

(iii) 
The case where $r\le s \le t$ or $t \le r \le s$ follows again directly from Theorem \ref{forward_flow_map} and \ref{backward_flow_map}.
The $r\le s \le t$, $t \le r \le s$, $r \le t \le s$ and $s \le t \le r$ need some additional work, but the proof is analogous
to the respective proofs in Theorem \ref{forward_flow_map} and \ref{backward_flow_map}. 
To illustrate the proving technique one more time, we consider the case $r\le s \le t$: We have $\Theta_s(r, \Theta_t(s,x)= \Phi_s(r, \chi_t(s,x)) = \xi_1(r)$ for some $\xi_1 \in \Xi(J\times\R^n)$ with $\xi_1(s) =\chi_t(s,x)$ and $\Theta_t(r,x) = \chi_t(r,x)=\xi_2(r)$ for some $\xi_1 \in \Xi(J\times\R^n)$ with $\xi_2(t) = x$. If $r=s$ we have $\xi_1(s) = \xi_2(s)$, thus uniqueness of $\Xi(J\times \R^n)$ implies that $\xi_1 \equiv \xi_2$ on the whole interval $J$.

The relation $\Theta_t(s,x) = \Theta_0(s-t, x)$ for all $s,t \in J$ and $x \in \R^n$ follows directly from
Theorem \ref{forward_flow_map} and \ref{backward_flow_map}.
\end{proof}

\begin{remark}[$C^1$ regularity of the flow] \label{ode_derivatives}
Let $a$ satisfy the Caratheodory condition on $J \times \mathbb{R}^{n}$ and assume that $a(s, \cdot) \in \Con^1(\R^n)^n$ for almost all $s\in J$.
Then a the solution set $\Xi(J)$ is unique and the
uniquely generated flow map $\Theta_t$ satisfies the properties 
\begin{eqnarray*}
(d_x \Theta_t) (s,x) &=& \exp{\left(\int_t^s d_x a(\tau, \Theta_t(\tau,x)) \m{\tau} \right)}  
\end{eqnarray*}
and 
\begin{eqnarray*}
\det(d_x \Theta_t) (s,x) &=& \exp{\left(\int_t^s \div{a}(\tau, \Theta_t(\tau,x)) \m{\tau} \right)}. 
\end{eqnarray*}
for all $(s,x) \in J \times \mathbb{R}^n$. For a proof we refer to \cite[Theorem 1.2.5.]{Hoermander:97}.
\end{remark}

\section{Differential inclusions and integral inequalities} \label{differential_inclusion_equation_section}

In this section we introduce a natural generalization of the ordinary differential equation, namely the \emph{differential inclusion} which is of the form
\begin{equation} \label{odie}
\odie,
\end{equation}
where $A_{t,x}$ is a non-empty, closed and convex subset of $\R^n$ for all $(t,x) \in J \times \R^n$. 

It turns out that the notion of a \emph{set-valued map} (cf. Section \ref{set_valued_map_section}) is suitable for the description of such 
generalized right-hand sides of the form $(t,x)\mapsto A_{t,x}$. 

As in the case of Caratheodory theory we have to impose certain conditions on the set-valued map $A$
in order to obtain an existence result for the differential inclusions.  

Throughout this section we assume that the set-valued map $A$ satisfies
\begin{enumerate} \label{filippov_conditions}
\item $t \mapsto A_{t,x}$ is Lebesgue measurable on $J$ for all fixed $x \in \mathbb{R}^n$,
\item $x \mapsto A_{t,x}$ is a upper semi-continuous for almost all $t \in J$,        
\item there exists a positive function $\beta \in L_{\rm loc}^1(J)$ such that $\sup_{x\in\mathbb{R}^n} |A_{t,x}| \le \beta(t)$ almost everywhere on $J$,
\end{enumerate}
referred to as \emph{Filippov conditions} (FC). \index{Filippov conditions}

It should be noted that $\sup_{x\in \R^n}|A_{t,x}| \le \beta(t)$ implies that $x\mapsto A_{t,x}$
is bounded for almost all $t\in J$. Due to Theorem \ref{set_valued_map_main_theorem} condition (ii)
is equivalent to have $x\mapsto  A_{t,x}$ being a closed set-valued map for almost all $t\in J$.  

We should think of $t \mapsto A_{t,x}$ as a set-valued map in $L^1_{\rm loc}(J; \mathcal{K}_0(\R^n))$.
This means that the Filippov conditions do not depend on the chosen representative of $t \mapsto A_{t,x}$ for $x\in \R^n$ fixed, so
without loss of generality we can assume that $x\mapsto A_{t,x}$ is upper semi-continuous (or equivalently - closed) for all $t\in J$.  

Equivalently we can formulate these conditions for the 
supporting $H$ (according to Theorem \ref{set_valued_map_main_theorem}) by
\begin{enumerate}
\item $t \mapsto H(t,x,w)$ is Lebesgue measurable on $J$ for all $x,w \in \R^n$,
\item $x \mapsto H(t,x,w)$ is upper semi-continuous for almost all $t \in J$ and $w \in \R^n$,        
\item there exists a positive function $\beta \in L_{\rm loc}^1(J)$ such that 
 $\sup_{x\in\mathbb{R}^n} |H(t,x,w/|w|)| \le \beta(t)$ for almost all $t\in J$ and all $w\in \R^n$.
\end{enumerate}

Furthermore the differential inclusion (\ref{odie}) is equivalent to the following integral inequality \index{integral inequality}
\begin{equation}
\langle \xi(s), w\rangle \le \langle x,w\rangle + \int_t^s H(\tau, \xi(\tau), w) \m{\tau} ,\quad \forall w\in \R^n.
\end{equation}

\begin{theorem}[Existence theorem for differential inclusions] \label{filippov_existence_theorem}
Let $A$ be set-valued map satisfying (FC). Let $(t,x) \in J \times \R^n$ be the initial value. Then there exists a solution
$\xi \in AC(J)^n$ of the differential inclusion
\begin{eqnarray*}
\dot{\xi}(s) \in A_{s, \xi(s)} \  a.e., \quad \xi(t)=x.
\end{eqnarray*} 
\end{theorem}

\begin{proof}
First of all we notice that for any $f\in C(J)^n$ it holds that the composition $t \mapsto A_{t, f(t)}$ is a Lebesgue measurable set-valued map from $J$ to $\R^n$.
By Definition \ref{set_valued_integral} and Corollary \ref{composition_single_valued_function} the measurability of the set-valued map $t \mapsto A_{t, f(t)}$ is equivalent to the measurability its supporting function
$t\mapsto H(t, f(t), w)$ for all $w\in\R^n$, which follows by Lemma \ref{measurability_of_composition_lemma} 
since $H$ satisfies the Filippov condition (i) and (ii).

We construct a sequence of continuous functions $(\xi_{k})_{k \in \N}$
which approximates the solution of the differential inclusion.

Put $\xi_0(s)=x$ for all $s\in J$, then using Theorem \ref{primitive_set-valued map} we iteratively construct 
a sequence of set-valued maps $(X_k)_{k \in \N}$ with $X_k: J \rightarrow \mathcal{K}_0(\R^n)$ by
$$
X_{k+1}(s) = \{x\} + \int_{t}^{s} A_{\tau,\xi_k(\tau)} \m{\tau}
$$
Furthermore Theorem  \ref{primitive_set-valued map} yields an absolutely continuous 
selection $\xi_{k+1}$ of $X_{k+1}$ with $\xi_{k+1}(t)=x$, so we obtain a family $(\xi_{k})_{k \in \N}$ of
absolutely continuous functions with the property
\begin{equation} \label{approx_solution}
\langle \xi_k(s) - \xi_k(r), w \rangle \le \int_r^s  H(\tau,\xi_k(\tau), w ) \m{\tau}
\end{equation}
for $s,r \in J$ and $w\in\R^n$. 

We observe that the (FC) (iii) yields the estimate
$$ 
| \xi_k(s) - \xi_k(r)| \le \int_{\min{(r,s)}}^{\max{(r,s)}}  \sup_{w\in\R^n}{H(\tau,\xi_{k}(\tau), w/|w| )} \m{\tau} 
\le \int_{\min{(r,s)}}^{\max{(r,s)}} \beta(\tau) \m{\tau} \quad \text{ for } s,r \in J.
$$
Thus $(\xi_k)_{k\in\N}$ is an equi-continuous family. The estimate
$$
|\xi_k(s)| \le |x| + \int_r^s  \sup_{w\in\R^n }H(\tau,\xi_k(\tau), w/|w| ) \m{\tau}
$$
yields the equi-boundedness of $(\xi_k)_{k\in\N}$.

By the Theorem of Arzela Ascoli there exists an uniformly convergent (on compact subsets of $J$) subsequence $\xi_{k_j} \stackrel{j\rightarrow \infty}{\rightarrow} \xi$, where $\xi \in C(J)^n$ with $\xi(t)=x$.

We have for all $s,r\in J$,
\begin{multline} \label{inclusion_equi_cont}
|\xi(s) - \xi(r)| \le |\xi(s) - \xi_{k_j}(s)|+ |\xi_{k_j}(s)-\xi_{k_j}(r)| +|\xi_{k_j}(r) - \xi(r)|  \\
\le |\xi(s) - \xi_{k_j}(s)| +\int_{\min{(r,s)}}^{\max{(r,s)}} \beta(\tau) \m{\tau} +|\xi_{k_j}(r) - \xi(r)| \stackrel{j \rightarrow \infty}{\rightarrow}\int_{\min{(r,s)}}^{\max{(r,s)}} \beta(\tau) \m{\tau} , 
\end{multline}
hence $\xi$ is absolutely continuous in $J$.

By the upper semi-continuity of $x \mapsto H(t,x,w)$, it follows that
$$
\limsup_{j \rightarrow \infty}  H(t, \xi_{k_j}(\tau), w )  \le 
 H(t, \xi(t),w ).  
$$
The theorem of dominated convergence \cite[Theorem 12.24]{HewStr:65} yields
$$
\langle \xi(s),w \rangle  \le \limsup_{j\rightarrow \infty} \langle \xi_{k_j}(s) ,w \rangle = \langle x, w \rangle  + \limsup_{j\rightarrow \infty} \int_t^s H(\tau, \xi_{k_j}(\tau), w )  \m{\tau} \le \langle x, w \rangle + \int_{t}^s H(\tau, \xi(\tau),w) \m{\tau} 
$$
Differentiation with respect to $s$ yields  
$$
\langle \dot{\xi}(s) , w \rangle \le H(s,\xi(s), w)
$$
for all $w\in \R^n$, thus $\dot{\xi}(s) \in A_{s, \xi(s)}$ for almost all $s\in J$ and $\xi(s)=x$.   
\end{proof}

\begin{definition}\label{odi_solution_set} \index{inclusion solution set}
Let $\Omega$ be a non-empty subset of $\R^n$. Analogous to the case of classical ordinary differential 
equations, we define the \emph{inclusion solution set} by
$$
\Xi_I(J \times \Omega) := \{ \xi \in AC(J)^n \mid  \exists (t,x)\in J\times \Omega : 
\dot{\xi}(s) \in A_{s,\xi(s)} \text{ a.e., }\ \xi(t)=x \}.
$$
If we put $\Xi_{t,x} := \Xi_I(\{(t,x)\})$ it obviously holds that 
$$
\Xi_I(J \times \Omega) = \bigcup_{(t,x) \in J \times \Omega} \Xi_{t,x}.
$$

The solution set $\Xi_I(J \times \Omega) $ is called \emph{forward (resp. backward) unique}, if $\xi_1, \xi_2 \in \Xi_I(J \times \Omega)$
with $\xi_1(t) = \xi_2(t)$ implies $\xi_1(s)=\xi_2(s)$ for all $s\ge t$ (resp. $s\le t$).

In the case where $J= \R$, we refer to a solution set $\Xi_I(\R\times \Omega)$ as \emph{autonomous}, if $\xi\in \Xi_I(\R \times \Omega)$
implies $\xi(\cdot+r) \in \Xi_I(\R \times \Omega)$ for any $r\in \R$. 
\end{definition}

\begin{proposition}[Properties of the inclusion solution set] \label{filippov_solution_set_prop} 
Let $\Omega$ be a non-empty subset of $\R^n$, then the inclusion solution set $\Xi_I(J \times \Omega)$ has the following properties:
\begin{enumerate} 
\item For every $(t,x) \in J \times \Omega$, there exists $\xi \in \Xi_I(J \times \Omega)$ with $\xi(t)=x$.
\item If $t\in J$ and $\Omega$ is closed, then $\Xi_I(\{t\} \times \Omega)$ is a closed subset of $C(J)^n$. 
\item $(\xi)_{\xi \in \Xi_I(J \times \Omega)}$ is an equi-continuous family in $C(J)^n$.
\item If $\Omega \csub \R^n$, then $\Xi_I(\{t\} \times \Omega)$ is a compact subset of $C(J)^n$.
\end{enumerate}
If the right-hand side of (\ref{odie}) is time-independent, we can put $J=\R$ and the (Filippov) solution set $\Xi_I (\R \times \Omega)$ is autonomous.
\end{proposition}

\begin{proof}
(i) By Theorem \ref{filippov_existence_theorem} the set $\Xi_I(J \times\Omega)$ is a non-empty subset of 
$C(J)^n$, such that for every $(t,x) \in \Omega$ there exists a $\xi \in \Xi_I(J \times\Omega)$ with $\xi(t)=x$. 

(ii) First we prove that $\Xi_I(\{t\} \times \Omega)$ is a closed set in $C(J)^n$ if $\Omega$ is closed:
Let $(\xi_k)_{k\in\N}$  be sequence in $\Xi_I(\{t\} \times \Omega)$ uniformly converging (on compact sets)
to some $\xi \in C(J)^n$. 
In particular $x_k=\xi_k(t) \rightarrow \xi(t)=:x \in \Omega$, since $\Omega$ is closed.  
It follows that  
\begin{multline*}
\langle \xi(s), w\rangle =\lim_{k\rightarrow\infty} \langle \xi_k(s), w\rangle = \lim_{k\rightarrow\infty} 
\left( \langle x_k, w\rangle + \int_{t}^s H(\tau, \xi_k(\tau), w) \m{\tau} \right) \\
\le \langle x,w \rangle + \int_{t}^s \limsup_{k \rightarrow \infty} H(\tau, \xi_k(\tau), w) \m{\tau} 
\le \langle x,w \rangle 
+ \int_{t}^s H(\tau, \xi(\tau), w) \m{\tau} 
\end{multline*}

which implies $\xi \in \Xi_I(\{t\} \times \Omega)$ satisfying the initial condition $\xi(t)=x \in \Omega$, 
so $\Xi_I(\{t\} \times \Omega)$ is a closed subset of $C(J)^n$.

(iii) For all $t\in J$, there exist some compact neighborhood $K_t \csub J$, such that by (\ref{inclusion_equi_cont})
\begin{equation} \label{filippov_equi_continuous_estimate}
|\xi(s) -\xi(r)| \le  \int_{\min{(r,s)}}^{\max{(r,s)}} \beta(\tau) \m{\tau}
\end{equation}
holds uniformly for all $\xi \in \Xi_I(J \times \Omega)$, thus $\Xi_I(J \times \Omega)$ is a equi-continuous family.

(iv) Since $K\csub \R^n$ the estimate
\begin{multline*}
\sup_{\xi \in \Xi_I(\{t\} \times K)} |\xi(s)| \le \sup_{x \in K} |x| + 
\sup_{\xi \in \Xi_I(\{t\} \times K)} \left| \int_t^s \sup_{w\in \R^n} H(\tau, \xi(\tau),w/|w|) \m{\tau} \right| \\
\le \sup_{x \in K} |x| + \int_{\min{(t,s)}}^{\max{(t,s)}} \beta(\tau) \m{\tau} < \infty
\end{multline*}
holds for all $s \in J$, thus $(\xi )_{\xi \in \Xi_I(\{t\} \times K)}$ is an equi-bounded family. 

By the Theorem of Arzela-Ascoli it follows that $\{ \xi : \xi \in \Xi_I(\{t\} \times K) \}$ is pre-compact in 
$C(J)^n$. The closedness of $\{ \xi :\xi \in \Xi_I(\{t\} \times K) \}$ in $C(J)^n$ was already shown in (ii), thus
$\Xi_I(\{t\} \times K ) \csub C(J)^n$.

If the rightside of (\ref{odie}) is time-independent, it is  obvious that we can put $J=\R$. For any $\xi \in \Xi_I( \R \times \Omega)$
with $\xi(t) =x \in \Omega$ we can define $\widetilde{\xi} = \xi(\cdot - r)$ for some $r\in \R$. 
By integral substitution on the right-hand side of (\ref{int_eq}) it follows directly 
that $\widetilde{\xi}$ again satisfies the ordinary differential equation with modified initial condition 
$\widetilde{\xi}(t+r)=x \in \Omega$, thus $\widetilde{\xi} \in \Xi_I(\R \times \Omega)$ .
\end{proof}

\begin{theorem}[Properties of the inclusion solution set] 
Let $A$ satisfy conditions (FC) and assume there exists some positive 
$\alpha(t)\in L^{1}_{\rm loc}(J)$ such that the supporting function $H$ of $A$ satisfies
\begin{eqnarray} \label{set_valued_function_forward_uniqueness_condition}
H(s,x,x-y)+H(s,y,y-x) \le \alpha(s) |x-y|^2 
\end{eqnarray}
resp. 
\begin{eqnarray} \label{set_valued_function_backward_uniqueness_condition}
 H(s,x,x-y)+H(s,y,y-x)\ge - \alpha(s) |x-y|^2)
\end{eqnarray} 
for all $x,y\in\R^n$ and almost all $s \in J$, then the solution set $\Xi_I(J\times \R^n)$ is a forward (resp. a backward) unique 
transport set. 
\end{theorem}

\begin{proof}
We prove only the forward uniqueness (the proof of backward uniqueness is done analogously):
Let $\xi_1, \xi_2 \in \Xi_I(J \times \R^n)$ with $\xi_1(t)=\xi_2(t)=x$. 
Setting $\zeta(s) :=|\xi_1(s)- \xi_2(s)|^2$ we observe $\zeta(t)=  |\xi_1(t) -\xi_2(t)|^2 = 0$ 
and using the definition of $\Xi_I(J\times \R^n)$ and the forward uniqueness 
condition (\ref{set_valued_function_forward_uniqueness_condition}), we obtain the estimate
\begin{eqnarray*}
\partial_s \zeta(s)  &=& 2 \langle \xi_1(s) -\xi_2(s), \dot{\xi_1}(s)-\dot{\xi_2}(s)  \rangle 
\le 2 ( H(s, \xi_1(s) , \xi_1(s) -\xi_2(s) ) + H(s, \xi_2(s) , \xi_2(s) -\xi_1(s) )\\
&\le &  2 \alpha(s) \zeta(s)
\end{eqnarray*}
for almost all $s\in J$, thus Gronwall's inequality gives
\begin{eqnarray*}
\zeta(s) \le \zeta(t) \exp{\left(2\int_t^s \alpha(\tau) \m{\tau}\right)} = 0.
\end{eqnarray*}
It follows that $\zeta(s) = |\xi_1(s)- \xi_2(s)|^2 \equiv 0$ for all $s\ge t$ if $\xi_1(t)=\xi_2(t)$, 
so $\Xi_I(J\times \R^n)$ is a forward unique transport set.
\end{proof}

\begin{theorem}[Forward unique (inclusion) flow map] \label{filippov_forward_flow_map} \index{inclusion flow!forward unique}
Let $\Xi_I(J\times \R^n)$ be a forward unique solution set. Put $J_t^+:=[t,\infty[ \cap J$, then the map 
defined by
\begin{equation*}
\begin{array}{rll}
 \chi_t: J_t^+ \times \R^n & \rightarrow & \R^n \\
 (s,x) &\mapsto& \xi(s;t,x) 
\end{array}
\end{equation*}
for some $\xi(\cdot;t,x) \in \Xi_I(J\times \R^n)$ with $\xi(t;t,x)=x$ is uniquely defined,
called \emph{(inclusion) forward flow map}.

It satisfies the following properties:
\begin{enumerate}
\item $\chi_t$ is a continuous map, 
\item $\chi_t(s, \cdot)$ is proper and onto for all $s\in J_{t}^+$ fixed, and
\item $\chi_r(s,\chi_t(r,x)) = \chi_t(s,x)$ for $s\ge r \ge t$.
\end{enumerate}

If the right-hand side of (\ref{odie}) is time-independent, we have the additional flow property
$\chi_t(s, x) =\chi_0(s-t, x)$ for all $s \ge t$, so property (iii) can be written in the well-known form
$$
\chi_{t_1}(s, \chi_{t_2}(r,x)) = \chi_{t_1+t_2}(s+ r ,x) 
$$
for $s\in J_{t_1}^+, r\in J_{t_2}^+$ with $t_1,t_2 \in \R$.
\end{theorem}

\begin{proof}
First of all we note that $\chi_t$ is uniquely defined, as the solution set $\Xi_I(J \times \R^n)$ is forward unique.

(i) Considering the continuity of $\chi_t$, we first prove that $\chi_t(s,x)$ is separately continuous. 
The continuity in $s$ is obvious by the definition of $\chi_t$.

Continuity in $x$ we shall prove by contradiction: Assume $\chi_t(s,x)$ is not continuous in $x$ at fixed $s$. Then
there exists a sequence $(x_k)_{k\in\mathbb{N}}$ converging to $x$, $\gamma >0$ and $k_0\in \N$ such that
$|\chi_t(s,x) - \chi_t(s,x_k)|>\gamma$ for all $k\ge k_0$.
Choose $K\csub \R^n$ such that $x_k \in K$ for all $k\in\N$.

For each $x_k$ we can find a $\xi_k\in\Xi_I(\{t\} \times K)$ such that $\xi_k(t) = x_k$ and $\chi_t(s,x_k)=\xi_k(s)$ for all $s\in J_t^+$. 
Furthermore there exists a $\xi\in\Xi_I(\{t\} \times K)$ such that $\xi(t) = x$ and $\chi_t(s,x)=\xi(s)$  for all $s\in J_t^+$.
The compactness of $\Xi_I(\{t\} \times K))$ in $C(J)^n$ yields that, there exists a locally uniformly converging subsequence 
$\xi_{k_l} \stackrel{l \rightarrow \infty}{\rightarrow} \tilde{\xi} \in \Xi_I(\{t\} \times K))$. 
The forward uniqueness of $\Xi_I(J \times \R^n)$ gives that $\widetilde{\xi}(s) = \xi(s)$ for all $s\in J_t^+$, 
thus $\lim_{l\rightarrow \infty} |\chi_t(s,x)-\chi_t(s, x_{k_l})| = 0 < \gamma$ contradicting the assumption.

The equi-continuity of $\Xi_I(J \times \R^n)$ yields the estimate
\begin{multline*}
|\chi_t(s,x) - \chi_t(s',x')| \le |\chi_t(s,x) - \chi_t(s,x')| + |\chi_t(s,x') - \chi_t(s',x')| \le
|\chi_t(s,x) - \chi_t(s,x')| \\+ \sup_{\xi \in \Xi_I(J \times \R^n)} |\xi(s) - \xi(s')|,
\end{multline*}
which, as $\chi_t$ was already shown to be separately continuous, implies the joint continuity of $\chi_t$.

(ii)
There exists some $\xi \in \Xi_{t,x}$ with $\chi_t(s,x) = \xi(s)$, thus 
$$
 |\chi_t(s,x) -x| = \sup_{w\in \R^n} \langle \chi_t(s,x) -x,w  \rangle  \le \int_t^s \sup_{w \in \R^n} \langle H(\tau,\xi(\tau),w/|w|)  \m{\tau} 
\le \int_t^s \beta(\tau) \m{\tau}
$$
holds for all $s \ge t$. For any compact set $K\csub \R^n$ we have that 
$\chi_t^{-1}(\{s\}\times K) \subseteq K':=\left\{x\in \R^n \mid \exists x_0 \in K : |x-x_0| \le  \int_t^s \beta(\tau) \m{\tau} \right\}$,
implying the properness of $x \mapsto \chi_t(s,x)$.

To show that $x\mapsto \chi_t(s,x)$ is onto for all $s\in J$ fixed, we have to show $\chi_t(s, \R^n)=\R^n$. 
We prove by contradiction: Assume there exists some $y\not\in \chi_t(s,\R^n)$, then by Theorem
\ref{filippov_solution_set_prop} (i) there exists a solution $\xi \in \Xi_{s,y}$ with $\xi(s)=y$. Putting $x:=\xi(t)$ we observe that $\xi \in \Xi_{t,x}$ and
forward-uniqueness of $\Xi_I(J\times\R^n)$ yields that $\xi(r) = \chi_t(r,x)$ for all $r\ge t$. In particular we have $y =\xi(s)=\chi_t(s,x) \in \chi_t(s,\R^n)$ which is a contradiction. 

(iii)
For fixed $x\in\mathbb{R}^n$ we have that $\chi_t(s,x) = \xi_1(s), s\in J_t^+$ for some $\xi_1 \in \Xi_I(J \times \R^n)$
with $\xi_1(t)=x$.
Fixing $r\in J_t^{+}$ we may choose $\xi_2 \in \Xi_I(J \times \R^n)$ with 
$\xi_2(r) = \chi_t(r,x)$ and $\xi_2(s) =\chi_r(s, \chi_t(r,x))$ for $s\ge r$. Since $\xi_1, \xi_2 \in \Xi_I(J \times \R^n)$ with $\xi_1(r) =  \chi_t(r,x)$ and $\xi_2(r) =  \chi_t(r,x)$ it follows by the forward uniqueness of $\Xi_I(J \times \R^n)$ that $\xi_1(s)=\xi_2(s)$ for all $s\ge r$, thus
$$
 \chi_r(s,\chi_t(r,x)) = \chi_t(s,x) 
$$
for $s\in J_t^{+} \cap J_r^{+}$ and $r\in J_t^{+}$. 

In the case where the right-hand side of \ref{odie} is time-independent, we put $J=\R$ and have that $\Xi_I(\R^{n+1})$ is autonomous
due to Theorem \ref{solution_set}. As $\chi_0(s, x) = \xi_1(s)$ for some $\xi_1 \in \Xi_I(\R^{n+1})$ with $\xi_1(0)=x$
and $\chi_t(s,x) = \xi_2(s)$ for some $\xi_2 \in \Xi_I(\R^{n+1})$ with $\xi_2(t)=x$. Due to the autonomy of the solution set 
it follows that $\widetilde{\xi_2}(s) := \xi_2(s-t)$ is again in $\Xi_I(\R^{n+1})$ . As $\widetilde{\xi_2}(t) = \xi_1(t) =x$ the forward uniqueness of $\Xi_I(\R^{n+1})$ immediately yields $\widetilde{\xi_2}(s) = \xi_1(s)$ for all $s \in J_t^+$, thus
$\chi_t(s,x) = \chi_0(s-t,x)$ for all $s\in J_t^+$. Using this relation and replacing $r$ by $r+t-t_2$ and $s$ by $s+r+t-t_2-t_1$ we can rewrite property (iii) in order to obtain
$$
\chi_{t_1}(s, \chi_{t_2}(r,x)) = \chi_{t_1+t_2}(s+ r ,x) 
$$
for all $r\in J_{t_1}^+, s \in J_{t_2}^+$ with $t_1, t_2 \in \R$.
\end{proof}

\begin{theorem}[Backward unique flow maps]\label{filippov_backward_flow_map} \index{inclusion flow!backward unique}
Let $\Xi_I(J\times \R^n)$ be a backward unique solution set. Put $J_t^-:=]-\infty,t] \cap J$, then the map 
defined by
\begin{equation*}
\begin{array}{rll}
 \Phi_t: J_t^+ \times \R^n & \rightarrow & \R^n \\
 (s,x) &\mapsto& \xi(s;t,x) 
\end{array}
\end{equation*}
for some $\xi(\cdot;t,x) \in \Xi_I(J\times \R^n)$ with $\xi(t;t,x)=x$ is uniquely defined,
called \emph{(inclusion) backward flow map}.

It has the properties that
\begin{enumerate}
\item $\Phi_t$ is a continuous map, 
\item $\Phi_t(s, \cdot)$ is proper for all $s\in J_{t}^-$ fixed, and
\item $\Phi_r(s,\Phi_t(r,x)) = \Phi_t(s,x)$ for $s\le r \le t$.
\end{enumerate}

If the right-hand side of \ref{odie} is time-independent, we have the additional property
$ \Phi_t(s, x) =\Phi_0(s-t, x)$ for all $s \le t$, so property (iii) can be written in the well-known form
$$
\Phi_{t_1}(s, \Phi_{t_2}(r,x)) = \Phi_{t_1+t_2}(s+ r ,x) 
$$
for $s\in J_{t_1}^-, r\in J_{t_2}^-$ with $t_1,t_2 \in \R$.
\end{theorem}

\begin{proof}
Analogous to the proof of Theorem \ref{filippov_forward_flow_map}.
\end{proof}

\section{Generalized solution concepts for ordinary differential equations} \label{section_generalized_solution_concepts_ode}

It is possible to generalize the classical solution concept for 
ordinary differential equations with discontinuous right-hand side $a\in L^1_{\rm loc}(J; L^{\infty}(\R^n))^n$,
by replacing $a$ with a suitable set-valued map $A$  satisfying (FC) 
and solving the resulting differential inclusion by means of Theorem \ref{filippov_existence_theorem}.

Generalization of the classical solution concept means that whenever a right-hand side 
$a$ satisfies the Caratheodory conditions the graph of the set-valued map $A$ 
and the graph of $a\in L^1_{\rm loc}(J; L^{\infty}(\R^n))^n$ are equal.
One way of mapping functions $a\in L^1_{\rm loc}(J; L^{\infty}(\R^n))^n$ into set-valued maps 
satisfying (FC) is by means of the essential convex hull. 

Another method to generalize the classical solution concept, is to regularize 
the right-hand side $a$ to obtain a net of smooth functions $(A_{\varepsilon})_{\varepsilon \in ]0,1]}$ 
with $\lim_{\varepsilon \rightarrow 0}\langle A_{\varepsilon}, \varphi \rangle = \langle a, \varphi \rangle$ 
for all $\varphi \in\mathcal{D}(\R^n)$.
Then a net of solutions $(\xi_{\varepsilon})_{\varepsilon}$ is obtained by solving
the family of ordinary differential equations
\begin{eqnarray} \label{ode_net}
\dot{\xi_{\varepsilon}}(s) = A_{\varepsilon}(s, \xi_{\varepsilon}(s)), \ \xi_{\varepsilon}(t_{\varepsilon}) = x_{\varepsilon},
\end{eqnarray}
where $(t_{\varepsilon}, x_\varepsilon)_{\varepsilon}$ is converging net with $\lim_{\varepsilon \rightarrow 0} (t_{\varepsilon}, x_{\varepsilon}) = (t,x)$.

Such nets of solutions $(\xi_{\varepsilon})_{\varepsilon}$ have convergent subnets $(\xi_{\tau(\varepsilon)})_{\varepsilon}$ with $\tau \in \mathcal{T}$ and $\xi_{\tau(\varepsilon)} \stackrel{\varepsilon \rightarrow 0}{\rightarrow} \xi \in C(J)^n$ uniformly on compact subsets of $J$, when the net of smooth functions $(A_{\varepsilon})_{\varepsilon}$ satisfies
certain boundedness assumptions. The generalized solution $\xi$ does in general not solve any differential equation, 
but it will turn out that $\xi$ solves a differential inclusion. So obviously these two solution concepts are closely related.

\subsection{The Filippov solution concept} \label{filippov_solution_concept}
If $a\in L^1_{\rm loc}(J; L^{\infty}(\R^n))^n$ we can consider its
essential convex hull $A_{t,x}$ for almost all $t\in J$ and all $x\in \R^n$.
We show that $A: (t,x) \mapsto A_{t,x}$ determines a set-valued map, satisfying 
the conditions (FC).
Furthermore the essential convex hull $A_{t,x}$ is equal to $\{a(t,x)\}$ for almost
all $t\in J$ and all $x\in \R^n$, if $a$ satisfies the Caratheodory conditions.

By Theorem \ref{filippov_existence_theorem} there exists a solution $\xi$ of differential
inclusion
$$
\odie,
$$
which we call \emph{Filippov (generalized) solution} of the ordinary \index{Filippov generalized solution}
differential equation $\dot{\xi} = a(s, \xi(s))$ when $(t,x) \mapsto A_{t,x}$ denotes the essential convex hull of $a$.

This type of generalized solution concept for ordinary differential equations with discontinuous right-hand side
was presented in \cite{Filippov:60}, \cite[Chapter 2]{AubinCellina:84}, and \cite[Chapter 1.4]{Hoermander:97}.

\begin{definition} \label{essential_convex_hull} \index{essential convex hull}
Let $a \in L^{1}_{\rm loc}(J; L^{\infty}(\R^n))^n$ and define the essential supporting function for $a$ by
$$
H_a(t,x, w) := \lim_{\delta \rightarrow 0} \esup_{y \in B_{\delta}(x)}{\langle a(t,y), w \rangle}. 
$$
which is defined for almost all $t \in J$ and all $x,w\in \R^n$.

The generated set-valued map $A: (t,x)\mapsto A_{t,x}$ obtained by
$$
A_{t,x} := \{ y\in \R^n \mid \langle y, w \rangle \le H_a(t,x,w), w \in \R^n \}.
$$
is called \emph{essential convex hull} of $a$.
\end{definition}

\begin{proposition}
The essential convex hull of any function $a \in L^{1}_{\rm loc}(J; L^{\infty}(\R^n))^n$ satisfies (FC).
\end{proposition}

\begin{proof}
By Definition \ref{essential_convex_hull} it is obvious that $w \mapsto H_a(t,x,w)$ 
is positively homogeneous and convex in $w$ for almost all $t\in J$ and all $x\in\R^n$. 

We put 
$$
h_{\delta}(t,x) := \esup_{y\in B_{\delta}(x)} \langle a(t,y),w \rangle 
$$
By \cite[Theorem 20.14]{HewStr:65} there exists a null set $N$, such that
$$
h_{\delta}(t,x) = \sup_{y\in B_{\delta}(x)/N} \langle a(t,y),w \rangle.
$$ 
Due to Lemma \ref{supremum_approximating_sequence} we can find a sequence 
$(x_k)_{k\in \N}$ in $B_{\delta}(x) / N$ such that
$$
h_{\delta}(t,x,w) = \limsup_{k \rightarrow \infty} \langle a(t,x_k),w \rangle,
$$
thus $(t,x) \mapsto h_{\delta}(t,x,w)$ is measurable as the pointwise 
limes superior of a family of measurable functions. 

Observe that $h_{\delta_1}(t,x,w) \le h_{\delta_2}(t,x,w)$ for $\delta_2 \ge \delta_1$,
so $(h_{\delta}(t,x,w))_{\delta \in ]0,1]}$ is a decreasing net, bounded for almost all $t\in J$.
For $y\in B_{\delta}(x)/N$ it holds that
$h_{\delta}(t,y,w) \ge \langle a(t,x),w \rangle$, thus $h_{\delta}(t,y,w)$ is bounded from below for almost all $t\in J$.
It follows that
$$
\lim_{\delta \rightarrow 0} h_{\delta}(t,x,w) = \inf_{\delta \in ]0,1]} h_{\delta}(t,x,w) 
$$
and since the pointwise infimum of a net of upper semi-continuous functions is again
upper semi-continuous (see \cite[Exercise 7.23 (a)]{HewStr:65}), we have that
$x \mapsto H_a(t,x,w)$ is upper semi-continuous for almost all $t\in J$ and all $w\in \R^n$.

We have
$$
H_a(t,x,w/|w|) \le h_1(t,x,w/|w|) \le \esup_{y \in \R^n } |a(t,y)| \in L^1_{\rm loc}(J)
$$
since $a \in  L^1_{\rm loc}(J; L^\infty(\R^n))^n$.
\end{proof}

\begin{theorem}[Approximation property of the essential convex hull] \label{filippov_approximation_property}
Let $a \in L^{1}_{\rm loc}(J; L^{\infty}(\R^n))^n$ and let $\rho$ be a positive mollifier in $\mathcal{D}(\R^n)$ with
$\int \rho(z) \m{z}=1$. Put $\rho_{\varepsilon}(x) = \varepsilon^{-n} \rho(x/\varepsilon)$.

Then the approximating net $a_{\varepsilon}(s,x) = a(s,\cdot) \ast \rho_{\varepsilon}$ has the property, that if $(\xi_{\varepsilon})_{\varepsilon}$ is a net 
of smooth functions converging to some $\xi \in C(J)^n$ uniformly on compact sets, then it holds
that 
$$
\limsup_{\varepsilon \rightarrow 0} 
\langle a_{\varepsilon}(s,\xi_{\varepsilon}(s)), w \rangle \le  H_a(s,\xi(s),w)
$$
for almost all $s\in J$.
\end{theorem}

\begin{proof}
Let $R > 0$ be such that $\supp{\rho} \subseteq B_{R}(0)$.  
Then we observe that 
$$
\sup_{\mu \in ]0, \varepsilon]} \langle a_{\mu}(s,\xi_{\mu}(s)), w \rangle 
= \sup_{\mu \in ]0, \varepsilon]} \int_{\R^n} \langle a(s,y) ,w \rangle \rho_{\mu}(\xi_{\mu}(s)-y) \m{y} 
\le \sup_{\mu \in ]0, \varepsilon]} \sup_{y \in B_{R \cdot \mu}(\xi_{\mu}(s))} \langle a(s,y) ,w \rangle 
 < \infty.
$$
for almost all $s\in J$.

Let $K\csub J$ fixed.

Then for all $\varepsilon > 0$ there exists 
$\mu'(\varepsilon)\in ]0,\varepsilon/(2R)]$  such 
that $|\xi_{\mu}(s)-\xi(s)|< \varepsilon/2$ for all $s\in K$ and 
$\mu \in]0,\mu'(\varepsilon)]$.
It follows that 
$$B_{R \cdot \mu }(\xi_{\mu}(s)) \subseteq 
B_{R \cdot \mu }(\xi(s)) + B_{\varepsilon/2}(\xi(s)) =
B_{R \cdot \mu+\varepsilon/2 }(\xi(s)) \subseteq B_{\varepsilon}(\xi(s)) 
$$
for all $s\in K$ and $\mu \in ]0, \mu'(\varepsilon)]$, implying 
$$
\sup_{\mu \in ]0, \mu'(\varepsilon)]} 
\sup_{y \in B_{R \cdot \mu}(\xi_{\mu}(s))} \langle a(s,y) ,w \rangle
\le 
\sup_{y \in B_{\varepsilon}(\xi(s))} \langle a(s,y) ,w \rangle. 
$$
Observe that 
$
\left(\sup_{\mu \in ]0, \varepsilon]} \langle a_{\mu}(s,\xi_{\mu}(s)), w \rangle \right)_{\varepsilon}
$
is a decreasing, bounded net for almost all $s\in J$, thus converging.
We conclude that any subnet of the net 
$(\sup_{\mu \in ]0,\varepsilon]} \langle a_{\mu}(s, \xi_{\mu}(s)),w \rangle)_{\varepsilon}$, 
has the same limit, thus
$$
\lim_{\varepsilon \rightarrow 0} \sup_{\mu \in ]0,\varepsilon]} \langle a_{\mu}(s, \xi_{\mu}(s)),w \rangle
=
\lim_{\varepsilon \rightarrow 0} \sup_{\mu \in ]0,\mu'(\varepsilon)]} \langle a_{\mu}(s, \xi_{\mu}(s)),w \rangle.
$$
Finally we obtain
\begin{multline*}
\limsup_{\varepsilon \rightarrow 0} \langle a_{\varepsilon}(s, \xi_{\varepsilon}(s)),w \rangle =
\lim_{\varepsilon \rightarrow 0} \sup_{\mu \in ]0,\varepsilon]} \langle a_{\mu}(s, \xi_{\mu}(s)),w \rangle
=
\lim_{\varepsilon \rightarrow 0} \sup_{\mu \in ]0,\mu'(\varepsilon)]} \langle a_{\mu}(s, \xi_{\mu}(s)),w \rangle
\\\le
\lim_{\varepsilon \rightarrow 0} \sup_{y \in B_{\varepsilon}(\xi(s))} \langle a(s,y) ,w \rangle 
=
H_a(s,\xi(s),w)
\end{multline*}
for almost all $s\in K$ and all $w\in \R^n$.
\end{proof}

\begin{remark}
Theorem \ref{filippov_approximation_property} shows that a Filippov generalized solution for an ordinary differential equation with
discontinuous right-hand side  $a\in L^{1}_{\rm loc}(J; L^{\infty}(\R^n))^n$ can be obtained by regularizing 
the right-hand side by convolution with a positive, smooth and compactly supported mollifier (instead of applying the existence theorem \ref{filippov_existence_theorem} for differential inclusions).
Such an approach for constructing Filippov generalized solutions is pursued in \cite[Chapter 1.4]{Hoermander:97}.

Putting $a_{\varepsilon}(s,x) = a(s,\cdot) \ast \rho_{\varepsilon}$ we have that any
net of regular functions $(\xi_{\varepsilon})_{\varepsilon}$ solving
$$
\dot{\xi_{\varepsilon}} = a_{\varepsilon}(s,\xi_{\varepsilon}(s)),\quad \xi_{\varepsilon}(t) = x
$$
has the property that 
\begin{multline*}
\limsup_{\varepsilon \rightarrow 0} \langle \xi_{\varepsilon}(s), w\rangle 
= \langle x,w\rangle + \limsup_{\varepsilon \rightarrow 0} \int_{t}^s \langle a_{\varepsilon}(\tau,\xi_{\varepsilon}(\tau)), w \rangle \m{\tau}  \\
\le
 \langle x,w\rangle + \int_{t}^s \limsup_{\varepsilon \rightarrow 0} \langle a_{\varepsilon}(\tau,\xi_{\varepsilon}(\tau)), w \rangle \m{\tau} 
\le  \langle x,w\rangle +\int_t^s H_a(\tau,\limsup_{\varepsilon \rightarrow 0}  \xi_{\varepsilon}(\tau),w) \m{\tau}
\end{multline*}
for all $w\in \R^n$, which implies that any subnet $(\xi_{\tau(\varepsilon)})_{\varepsilon}$ (with $\tau\in\mathcal{T}$)
converging uniformly on compact subsets to some $\xi \in C(J)^n$, satisfies the differential inclusion
$$
\dot{\xi}(s) \in A_{s,\xi(s)}, \quad \xi(t)=x,
$$
where $A$ denotes the essential convex hull. 
\end{remark}

\begin{theorem} \label{filippov_forward_uniqueness}
Let $a \in L^{1}_{\rm loc}(J; L^{\infty}(\R^n))^n$ satisfy the forward uniqueness condition
\begin{equation} 
\langle x-y,  a(s,x) - a(s,y) \rangle  \le \alpha(s) |x-y|^2 
\end{equation}
resp. the backward uniqueness condition
\begin{equation}
 \langle x-y,  a(s,x) - a(s,y) \rangle  \ge - \alpha(s) |x-y|^2,
\end{equation}
for some positive $\alpha \in L^1_{\rm loc}(J)$,
then the supporting function of the essential convex hull $H_f$
satisfies the forward uniqueness condition (\ref{set_valued_function_forward_uniqueness_condition})
resp. backward uniqueness condition (\ref{set_valued_function_backward_uniqueness_condition}).
\end{theorem}

\begin{proof}
We prove only the case where $a$ satisfies the forward-uniqueness condition.
\begin{multline*}
\langle a(s,z_1), x-y \rangle - \langle a(s,z_2), x-y \rangle
=
\langle a(s,z_1), z_1-z_2 \rangle + \langle a(s,z_2), z_2-z_1 \rangle +
2 \beta(s) (|z_1 - x|+|z_2 - y|)
\\
\le 
|z_1-z_2|^2 \alpha(s) +  4 \delta \beta(s) 
\end{multline*}
for all $z_1 \in B_{\delta}(x), z_2 \in B_{\delta}(y)$ and almost all $s\in J$. It follows that 
\begin{multline*}
\esup_{z\in B_{\delta}(x)} \langle a(s,z), x-y \rangle + \esup_{z\in B_{\delta}(y)} \langle a(s,z), y-x \rangle
\le \sup_{(z_1,z_2)\in B_{\delta}(x) \times  B_{\delta}(y)} |z_1-z_2|^2 \alpha(s) +  4 \delta \beta(s) \\
= |x-y|^2 \alpha(s) + 2 \delta \alpha(s) +  4 \delta \beta(s), 
\end{multline*}
thus 
\begin{multline*}
 H_a(s,x, x-y) + H_a(s,y, y-x) = 
\lim_{\delta \rightarrow 0}
(\esup_{z\in B_{\delta}(x)} \langle a(s,z), x-y \rangle + \esup_{z\in B_{\delta}(y)} \langle a(s,z), y-x \rangle)
\\ \le |x-y|^2 \alpha(s)
\end{multline*}
holds for almost all $s\in J$.
\end{proof}

\subsection{Colombeau solutions for ordinary differential equations}

Finally we consider generalized solutions for ordinary differential equation
in the setting of Colombeau generalized functions. 

Let $(\widetilde{t},\widetilde{x}) \in \widetilde{J \times \R^n}$ be the initial condition, then
we study the ordinary differential equation
\begin{eqnarray*}
\dot{\xi}(s) &=& A(s,\xi(s)), \quad \xi(\widetilde{t})=\widetilde{x},
\end{eqnarray*}
where $A:=[(A_{\varepsilon})_{\varepsilon}] \in \colmap{J\times \R^n}{\R^n}$ such that there exists some representative satisfying the bound
\begin{eqnarray} \label{colombeau_existence_bound}
\sup_{x\in\mathbb{R}^n}|A_{\varepsilon}(t,x)| \le  \beta(t) ,\ \ \varepsilon \in\, ]0,1], \ \text{almost everywhere in}\ t\in J
\end{eqnarray}
for some positive function $\beta \in L^1_{\rm loc}(J)$.
It will turn out that the generalized solution $\xi:=[(\xi_{\varepsilon})_{\varepsilon}]$ then allows 
for picking a subnet $(\xi_{\tau(\varepsilon)})_{\varepsilon}$ converging uniformly on compact subsets of 
$J$ to some $\zeta \in AC(J)^n$ with $\zeta(t)=x$.

\begin{theorem}[Existence] \label{colombeau_existence_theorem} \index{Colombeau generalized solution}
Let $A \in \colmap{J \times \R^n}{\R^n}$ be Colombeau map of natural type 
with a representative $(A_{\varepsilon})_{\varepsilon}$ satisfying (\ref{colombeau_existence_bound}).

If $(\widetilde{t},\widetilde{x}) \in \widetilde{J \times \R^n}$, then there exists a c-bounded solution $\xi \in \G(J)^n$ to the initial value problem
\begin{eqnarray*}
\dot{\xi}(s) &=& A(s,\xi(s)), \quad \xi(\widetilde{t})=\widetilde{x}.
\end{eqnarray*}
Furthermore, there exists some $(t,x) \in \overline{J}\times \Omega$, $\zeta \in AC(J)^n$ such that for any representative $(\xi_{\varepsilon})_{\varepsilon}$ of $\xi$, $(\widetilde{t}_\varepsilon, \widetilde{x}_{\varepsilon})_{\varepsilon}$ of $(\widetilde{t},\widetilde{x})$ there exists subnets $(\widetilde{t}_{\tau(\varepsilon)}$, $\widetilde{x}_{\tau(\varepsilon)})_{\varepsilon}$,$(\xi_{\tau(\varepsilon)})_{\varepsilon}$ with $\tau \in \mathcal{T}$ and
$\lim_{\varepsilon \rightarrow 0} (\widetilde{t}_{\tau(\varepsilon)}, \widetilde{x}_{\tau(\varepsilon)}) =(t,x)$
 and $\xi_{\tau(\varepsilon)} \stackrel{\varepsilon \rightarrow 0}{\rightarrow} \zeta$ uniformly on compact sets of $J$.

It satisfies the integral inequality
\begin{equation} \label{colombeau_solution_pseudo_shadow} 
\langle \zeta(s),w \rangle \le \langle x, w\rangle + \int_t^s H_A(\tau, \zeta(\tau), w) \m{\tau}
\end{equation}
for all $w\in \R^n$, where $(x,w) \mapsto H_A(t,x,w)$ denotes the supporting function of 
the generalized graph $\ggraph{A(t,\cdot)}$.
\end{theorem}

\begin{proof}
Let $(A_{\varepsilon})_{\varepsilon}$ be a representative of $A$ and let $(t_{\varepsilon}, x_{\varepsilon})_{\varepsilon}$ be a representative of $(\widetilde{t},\widetilde{x})$.
By classical existence and uniqueness we obtain a solution $\xi_{\varepsilon}$ for each $\varepsilon \in ]0,1]$ such that
\begin{eqnarray*}
\xi_{\varepsilon}(s) &=& x_{\varepsilon} + \int_{t_{\varepsilon}}^s A_{\varepsilon}(\tau,\xi_{\varepsilon}(\tau)) \m{\tau} 
\end{eqnarray*}
holds.  (\ref{colombeau_existence_bound}) yields $|\xi_{\varepsilon}(s)| \le |x_{\varepsilon}| + \int_{t_\varepsilon}^s \beta(\tau) \m{\tau}$ for all
$s\in J$, hence $(\xi_{\varepsilon})_{\varepsilon}$ is bounded on compact subsets  of $J$. 

Since $(t_{\varepsilon}, x_{\varepsilon})_{\varepsilon}$ is bounded we have that for all $s\in J$ there exists some $C>0$ such that $|\xi_{\varepsilon}(s)| \le C$ for all $\varepsilon \in ]0,1]$, thus $(\xi_{\varepsilon})_{\varepsilon}$ is an equi-bounded family. The estimate
\begin{eqnarray*}
|\xi_{\varepsilon}(s) - \xi_{\varepsilon}(r)| \le  \int_{\min{(r,s)}}^{\max{(r,s)}} \beta(\tau) \m{\tau} \quad s,r \in J, \varepsilon \in ]0,1].
\end{eqnarray*}
yields that the family $(\xi_{\varepsilon})_{\varepsilon}$  is  equi-continuous.

The theorem of Arzela-Ascoli yields a subnet $(\xi_{\tau(\varepsilon)})_{\varepsilon}$ with $\tau \in \mathcal{T}$ converging to some $\zeta \in C(J)^n$ uniformly on compact sets.
Without loss of generality we can assume that $(t_{\tau(\varepsilon)},x_{\tau(\varepsilon)})_{\varepsilon}$ 
converges to some $(t,x) \in \overline{J} \times \R^n$ with $\zeta(t)=x$.

We have for all $s,r\in J$,
\begin{multline*}
|\zeta(s) - \zeta(r)| \le |\zeta(s) - \xi_{ \tau(\varepsilon) }(s)|+ |\xi_{ \tau(\varepsilon) }(s)-\xi_{ \tau(\varepsilon) }(r)| +|\xi_{ \tau(\varepsilon) }(r) - \zeta(r)|  \\
\le |\zeta(s) - \xi_{ \tau(\varepsilon)  }(s)| +\int_{s}^{r} \beta(\tau) \m{\tau} +|\xi_{\tau(\varepsilon)} (r) - \zeta(r)| \stackrel{\varepsilon \rightarrow 0}{\rightarrow}\int_{s}^{r} \beta(\tau) \m{\tau} , 
\end{multline*}
hence $\zeta$ is absolutely continuous on $J$.

Note that the equi-continuity of $(\xi_{\varepsilon})_{\varepsilon}$ implies its c-boundedness. Then
the moderateness of $(\xi_{\varepsilon})_{\varepsilon}$ follows by
the same argument as in the proof of \cite{GKOS:01}[Proposition 1.2.8], since
the derivative $\dot{\xi_{\varepsilon}}(s)$ may be written as the composition of the moderate net $A_{\varepsilon}$
and the c-bounded net $(s, \xi_{\varepsilon}(s))$.

By Theorem \ref{colombeau_supporting_function_proposition} we obtain that
$A$ induces a supporting function satisfying (FC). By the theorem of dominated convergence
\cite[Theorem 12.24]{HewStr:65} and Theorem
\ref{supporting_function_approximation_property} it follows that 
\begin{multline*}
\langle \zeta(s), w \rangle = \lim_{j \rightarrow \infty} \langle \xi_{\tau(\varepsilon)}, w \rangle \le 
 \langle x,w\rangle + \limsup_{j \rightarrow \infty} \int_t^s A_{\tau(\varepsilon)}(\tau, \xi_{\tau(\varepsilon)}(\tau)) \m{\tau} \\
\le 
\langle x,w\rangle + \int_t^s \limsup_{j \rightarrow \infty}  A_{\tau(\varepsilon)}(\tau, \xi_{\tau(\varepsilon)}(\tau)) \m{\tau} \\
\le 
\langle x,w\rangle + \int_t^s H_{A}(\tau, \zeta(\tau), w ) \m{\tau} 
\end{multline*}
for all $w\in\R^n$.
\end{proof}

\begin{remark} \label{colombeau_solution_remark}
By Proposition \ref{generalized_graph_equicontinuous_proposition} the equi-continuity of 
$(\xi_{\varepsilon})_{\varepsilon}$ yields that
$$
\ggraph{\xi} = \bigcup_{s\in J} \{s\} \times \cp{(\xi_{\varepsilon}(s))_{\varepsilon}}.
$$
Then $\zeta \in AC(J)^n$ as in (\ref{colombeau_solution_pseudo_shadow})
is a selection of the set-valued map $s \mapsto \ggraph{\xi}_s$, i.e. $\zeta(s) \in \ggraph{\xi}_s$ for all $s\in J$.

If $(r,z) \in \ggraph{\xi}$ we can find some $\tau \in \mathcal{T}$ such that 
$\lim_{\varepsilon \rightarrow 0} \xi_{\tau(\varepsilon)}(r) = z$. 

Then the Theorem of Arzela-Ascoli yields that it is possible to
choose a subnet of $(\xi_{\tau(\varepsilon)})_{\varepsilon}$ 
converging uniformly on compact sets to some $\zeta \in AC(J)^n$ with $\zeta(r)=z$ and $\zeta(t)=x$.

This shows that the $\ggraph{\xi}$ is the union of $\graph{\zeta}$ of all possible $\zeta\in AC(J)^n$ as in
Theorem \ref{colombeau_existence_theorem}. We can reformulate
property (\ref{colombeau_solution_pseudo_shadow}) to obtain
\begin{equation} \label{hs_graph}
\ggraph{\xi}_s \subseteq \{x\}+ \int_t^s \ch{\ggraph{A} \circ ( \{\tau\} \times \ggraph{\xi}_{\tau} )} \m{\tau}
\end{equation}
where "$\circ$" denotes the composition of set-valued maps (cf. Theorem \ref{set_valued_composition_theorem})
and the integral is according to Definition \ref{set_valued_integral}.
\end{remark}

\begin{example}[Hurd-Sattinger characteristic flow] \label{hs_flow_example} \index{Hurd-Sattinger!characteristic flow}
Consider the ordinary differential equation: 
$$
\dot{\xi}(s) = \Theta(-\xi(s)) ,\quad \xi(t)=x,
$$
where $\Theta = [(\Theta_{\varepsilon})_{\varepsilon}]\in \cf{\R}$  is defined by $\Theta_{\varepsilon}(x) =  \rho_{\varepsilon} \ast H(-\cdot)$
with $\rho_{\varepsilon} = \frac{1}{\varepsilon}\rho (\cdot /  \varepsilon)$  
for some $\rho \in \mathcal{S}(\mathbb{R})$ with $\int \rho(x) \m{x}=1$. 

We put $g(x):= \int_x^{\infty} \rho(z) \m{z}$ and denote $\overline{\ran{g}} = [\alpha_{-},\alpha_{+}]$ for some $\alpha_{-},\alpha_{+} \in \R$.
Note that $\lim_{x \rightarrow -\infty}g(x) = 1$ and $\lim_{x \rightarrow \infty}g(x) = 0$ implies $\alpha_{-}\le 0$ and 
$\alpha_{+}\ge 1$. It holds that $\alpha_{-}=0$ and $\alpha_{+}=1$ if $\rho \ge 0$.

In our example we allow $\rho$ to have non-vanishing moments, 
so $\Theta$ is not necessarily the embedding of the Heaviside function.

Nevertheless $(\Theta_{\varepsilon})_{\varepsilon}$ converges to the Heaviside function in $L^1_{\rm loc}(\R)$
and since $\sup_{x\in\R} |\Theta_{\varepsilon}(-x)| < \sup_{x\in\R} |g(x/\varepsilon)| \le \alpha_+$ it satisfies
(\ref{colombeau_existence_bound}). 

Due to Theorem \ref{colombeau_existence_theorem} there exists a c-bounded Colombeau solution $\xi = [(\xi_{\varepsilon})_{\varepsilon}]$, such that
there exists a converging subnet with $\lim_{\varepsilon \rightarrow 0} \xi_{\tau(\varepsilon)} \rightarrow \zeta$
uniformly on compact set for some $\zeta \in \AC(\R)^n$. 

As in Example \ref{heaviside_ggraph_example} we obtain that the generalized graph of $\Theta$ is
$$
\ggraph{\Theta(-\cdot)} := \left(\{x\in\mathbb{R}\mid x<0\} \times \{1\}\right) \cup \left(\{0\} \times \overline{\ran{g}} \right) \cup \left(\{x\in\mathbb{R}\mid x>0\} \times \{0\}\right).
$$
Using its support function (Example \ref{supportfunction_heaviside_example}) the property 
(\ref{colombeau_solution_pseudo_shadow}) of $\zeta$ translates to 
\begin{eqnarray*}
x + \int_t^s 1_{\zeta(\tau) < 0}(\tau) \m{\tau} \le \zeta(s) \le x + \int_t^s 1_{\zeta(\tau) < 0}(\tau) \m{\tau}.
\end{eqnarray*}
Then it is straight-forward to derive that 
\begin{equation} \label{hs_shadow}
\zeta(s;t,x) \in \left\{
\begin{array}{ll}
\{\min{(x+ s-t,0)}\} & x < 0 \\
\{0\} & x = 0,  s \ge  t \\  
\left[s-t,0\right] & x = 0, s \le  t \\  
\{x\} & x > 0.
\end{array}\right.
\end{equation}
holds, which implies that $\zeta$ is forward unique. It follows that 
$$
\lim_{\varepsilon \rightarrow }\sup_{s\in K\cap[t,+\infty[} |\xi_{\varepsilon}(s)-\zeta(s)|=0
$$
for all $K \csub \R$. 

Note that the support function of $\ggraph{\Theta(-\cdot)}$ does not satisfy the forward uniqueness condition 
(\ref{set_valued_function_forward_uniqueness_condition}) if $\alpha_{+}>1$ or $\alpha_{-}<0$, nevertheless 
$\zeta$ is forward unique due to (\ref{hs_shadow}) regardless of the values $\alpha_{-},\alpha_{+}$.

By Remark \ref{colombeau_solution_remark} we obtain that the Colombeau solution $\xi$ of the Hurd-Sattinger equation
has the property that 
$$
\ggraph{\xi}_s \subseteq \{x\} + \int_t^s \ch{\ggraph{\Theta} \circ (\{\tau\} \times \ggraph{\xi}_\tau)} \m{\tau} =
\{x\} + \int_t^s 1_{\ggraph{\xi}_\tau<0} (\tau)  \m{\tau} 
$$
which in (\ref{hs_shadow}) implies
$$
\ggraph{\xi}_s \subseteq \left\{
\begin{array}{ll}
\{\min{(x+ s-t,0)}\} & x < 0 \\
\{0\} & x = 0,  s \ge  t \\  
\left[s-t,0\right] & x = 0, s \le  t \\  
\{x\} & x > 0.
\end{array}\right.
$$
\end{example}

\begin{theorem}[Forward uniqueness of the distributional shadow] \label{colombeau_uniqueness_theorem} \index{distributional shadow!forward uniqueness}
Let $A  \in \G(\overline{\Omega_T})^n$ with representative $(A_{\varepsilon})_{\varepsilon}$ 
satisfying (\ref{colombeau_existence_bound}) and in addition the forward uniqueness condition
\begin{eqnarray} \label{colombeau_forwarduniqueness_condition}
\langle A_{\varepsilon}(s,x)-A_{\varepsilon'}(s,y), x-y \rangle \le \alpha(s) |x-y|^2 + \gamma_{\varepsilon, \varepsilon'}(s)  
\end{eqnarray}
resp. the backward uniqueness condition
\begin{eqnarray} \label{colombeau_backwarduniqueness_condition}
\langle A_{\varepsilon}(s,x)-A_{\varepsilon'}(s,y), x-y \rangle \ge -\alpha(s) |x-y|^2 - \gamma_{\varepsilon, \varepsilon'}(s)  
\end{eqnarray}
holds for some positive functions $\alpha,\gamma_{\varepsilon, \varepsilon'}\in L^1([0,T])$ such that 
$\int_0^T \gamma_{\varepsilon, \varepsilon'}(\tau) \m{\tau} =O(\varepsilon+\varepsilon')$ as 
$\varepsilon,\varepsilon' \rightarrow 0$. 

Then the supporting function $H_A$ of the generalized graph $\ggraph{A(s,\cdot)}$ satisfies the forward uniqueness property
$$
H_A(s,x,x-y)+H_A(s,y,y-x) \le \alpha(s) |x-y|^2 
$$
resp.
$$ 
H_A(s,x,x-y)+H_A(s,y,y-x)\ge - \alpha(s) |x-y|^2).
$$
\end{theorem}

\begin{proof}
By the proof of Theorem \ref{generalized_graph_supporting_function_proposition} we can write
$$
H_A(s,x,x-y) = \lim_{\varepsilon \rightarrow 0} 
\langle A_{\tau_1(\varepsilon)}(x_{\varepsilon}), x-y \rangle
$$
and 
$$
H_A(s,y,y-x) = \lim_{\varepsilon \rightarrow 0} \langle 
A_{\tau_2(\varepsilon)}(y_{\varepsilon}), y-x \rangle
$$
for fixed $x,y \in \R^n$ and
for some $\tau_1, \tau_2 \in\mathcal{T}$ and nets 
$(x_{\varepsilon})_{\varepsilon},(y_{\varepsilon})_{\varepsilon}$
with $x_{\varepsilon} \rightarrow x$, $y_{\varepsilon} \rightarrow y$. 
For almost all $s\in J$ and all $x,y \in \R^n$ we have
\begin{multline*}
H_A(s,x,x-y)+H_A(s,y,y-x) =  \lim_{\varepsilon \rightarrow 0}  
(\langle A_{\tau_1(\varepsilon)}(s,x_{\varepsilon}), x-y \rangle + \langle A_{\tau_2(\varepsilon)}(s,y_{\varepsilon}), y-x \rangle) 
\\
\le
 \lim_{\varepsilon \rightarrow 0}  
(\langle A_{\tau_1(\varepsilon)}(s,x_{\varepsilon}), x_{\varepsilon} - y_{\varepsilon} \rangle 
- \langle A_{\tau_2(\varepsilon)}(s,y_{\varepsilon}),x_{\varepsilon} - y_{\varepsilon} \rangle 
+\beta(s)  ( |x- x_{\varepsilon}|+ |y- y_{\varepsilon}|) ) 
\\
\le
  \lim_{\varepsilon \rightarrow 0} (\alpha(s) |x_{\varepsilon} - y_{\varepsilon}|^2 +
\beta(s)  (|x- x_{\varepsilon}|+ |y- y_{\varepsilon}|)  
+ \gamma_{\tau_1(\varepsilon), \tau_2(\varepsilon)}(t) )  
\le
\alpha(t) |x - y|^2.
\end{multline*}
The proof for the backward uniqueness condition is analogous.
\end{proof}

\begin{remark}
By means of Theorems \ref{colombeau_uniqueness_theorem} it immediately follows that the 'sub-shadow' $\zeta$
of the Colombeau solution $\xi = [(\xi_{\varepsilon})_{\varepsilon}] \in  \cf{J}^n$ (as obtained by Theorem \ref{colombeau_existence_theorem}), is forward unique resp. backward unique if
the right-hand side $A:=[(A_{\varepsilon})_{\varepsilon}] \in \colmap{J \times \R^n}{\R^n}$
has the property that the supporting function of $\ggraph{A(t,\cdot)}$ satisfies the forward uniqueness condition \ref{colombeau_forwarduniqueness_condition}
resp. the backward uniqueness condition \ref{colombeau_backwarduniqueness_condition}.
\end{remark}

\begin{example}\label{colombeau_solution_uniqueness_example}
Let $\widetilde{a}\in {L^1(\R, L^{\infty}(\mathbb{R}^n))}^n$ 
satisfy the classical forward uniqueness condition 
\begin{eqnarray} \label{limit_solution_forward_uniqueness_example}
\langle \widetilde{a}(t,x) - \widetilde{a}(t,y), x-y  \rangle \le 
\widetilde{\alpha}(t) |x-y|^2 \ \ \text{for almost all}\ (t,x), (t,y) \in \R^{n+1}
\end{eqnarray} 
for some positive $\widetilde{\alpha} \in L^1(\R)$.

Assume $\rho \in\mathcal{D}(\mathbb{R}^{n+1})$ positive, $\int_{\mathbb{R}^n} \rho(z) \m{z} = 1$ and define $\rho_{\varepsilon}(t,x):=\varepsilon^{-n-1} \rho(t /\varepsilon,x/\varepsilon)$.  
Consider the Colombeau function $A:=[(A_{\varepsilon})_{\varepsilon}] \in \colmap{\R^{n+1}}{\R^n}$ defined by
\begin{eqnarray*} 
A_{\varepsilon} := \widetilde{a} \ast \rho_{\varepsilon},
\end{eqnarray*}
then $[(A_{\varepsilon})_{\varepsilon}]$ satisfies the condition (\ref{colombeau_existence_bound}) for the existence 
theorem and the forward uniqueness condition (\ref{colombeau_forwarduniqueness_condition}).
\end{example}

\begin{proof}
 Since $\widetilde{a} \in L^{1}(\mathbb{R}, L^{\infty}(\mathbb{R}^n))^n$ it holds that for almost all $t \in \mathbb{R}$, we have
 $\widetilde{\beta}(t):= \sup_{z\in \mathbb{R}^n} |\widetilde{a}(t,z)|$ where $\widetilde{\beta}$ is some positive function in $L^1(\mathbb{R})$. 
 First, observe that
 \begin{multline} \label{colombeau_example_1}
  \sup_{z\in \mathbb{R}^n}  |A_{\varepsilon}(t, z)|   =   
  \sup_{z \in \mathbb{R}^n} \left|\int_{\mathbb{R}}\int_{\mathbb{R}^n}   a(t-\varepsilon \tau, z-\varepsilon y) \rho(\tau,y) \m{y}  \m{\tau}\right|  \\
  \le  \int_{\mathbb{R}} \widetilde{\beta}(t-\varepsilon \tau)   
\left(  \int_{\mathbb{R}^n} \rho(\tau,y) \m{y} \right) \m{\tau} 
 \le \beta(t) \ \ \text{for almost all}\  t\in\R
 \end{multline}
 where $\beta(t):= \sup_{\varepsilon\in]0,1]}\int_{\mathbb{R}} \widetilde{\beta}(t-\varepsilon s) \left(  \int_{\mathbb{R}^n} \rho(\tau,y) \m{y} \right) \m{\tau}$ is a positive function in $L^1(\R)$ and set $\varphi(s) := \int_{\mathbb{R}^n} \rho(s,y) \m{y} $. 
Hence $A_{\varepsilon}(t,x)$ satisfies (\ref{colombeau_existence_bound}).
 Using (\ref{limit_solution_forward_uniqueness_example}) 
and writing $A_{\varepsilon}$ explicitly as convolution integral, a 
 straight-forward calculation and estimation gives
 \begin{multline} \label{expression_colombeau_forwarduniqueness}
 \langle A_{\varepsilon}(t,x) - A_{\varepsilon'}(t,y), x-y\rangle \\ \le 2 \widetilde{\alpha}(t)  |x-y|^2  + 2 C_1 \widetilde{\alpha}(t) (\varepsilon+\varepsilon')^2 + 2 C_2  \sup_{z\in\mathbb{R}^n}{|\widetilde{a}(t,z)|} (\varepsilon+\varepsilon') + r_{\varepsilon, \varepsilon'}(t) \frac{1}{2}(1+|x-y|^2)
 \end{multline}
 where $C_1, C_2 $ are positive constants and 
$r_{\varepsilon,\varepsilon'}(t) := \lambda_{\varepsilon}(t)-\lambda_{\varepsilon'}(t)$ with
 \begin{eqnarray*}
 \lambda_{\varepsilon}(t) = \frac{1}{\varepsilon|x-y|} \int_{\mathbb{R}^n} \int_{\mathbb{R}} \langle \widetilde{a}(s, x-\varepsilon z) - \widetilde{a}(t,x-\varepsilon z) , x-y \rangle \rho(t-s/\varepsilon,z)  \m{s} \m{z} .
 \end{eqnarray*} 
 One easily verifies that
 \begin{eqnarray} \label{expression_lambda_estimate}
 \int_0^T \lambda_{\varepsilon}(t) \m{t} = O(\varepsilon),
 \end{eqnarray}
 hence $\int_{0}^T r_{\varepsilon,\varepsilon'}(\tau) \m{\tau} 
\le \int_{0}^T \lambda_{\varepsilon}(\tau) \m{\tau}+
\int_{0}^T \lambda_{\varepsilon'}(\tau) \m{\tau} = O(\varepsilon+\varepsilon')$. 
  
  The constants $C_1$ and $C_2$ are derived 
  from the estimates
  \begin{eqnarray*}
  \int \int |\varepsilon z-\varepsilon' z'|^2 \phi(z) \phi(z') \m{z} \m{z'} 
  &\le& \int \int ((\varepsilon z)^2 -2 \varepsilon \varepsilon' \langle z',  z \rangle +(\varepsilon' z')^2)  \phi(z) \phi(z') \m{z} \m{z'} \\
  &\le&
  \int \int ((\varepsilon z)^2 + \varepsilon \varepsilon' ((z')^2+ (z)^2) +(\varepsilon' z')^2)  \phi(z) \phi(z') \m{z} \m{z'}\\
  &=&
   2(\varepsilon +\varepsilon')^2  \int z^2 \rho(z) \m{z} =:  2(\varepsilon +\varepsilon')^2  C_1
  \end{eqnarray*}
  and 
  \begin{eqnarray*}
  \int \int |\varepsilon z-\varepsilon' z'| \phi(z) \phi(z') \m{z} \m{z'} 
  &\le& \int \int (|\varepsilon z| +|\varepsilon' z'|)  \phi(z) \phi(z') \m{z} \m{z'} \\
  &\le& (\varepsilon +\varepsilon')  \int |z| \rho(z) \m{z}  =:  (\varepsilon +\varepsilon')  C_2.
  \end{eqnarray*}
  Finally we can put (\ref{expression_colombeau_forwarduniqueness}) and (\ref{expression_lambda_estimate}) 
together to obtain
  \begin{multline*}
  \langle A_{\varepsilon}(t,x) - A_{\varepsilon'}(t,y), x-y\rangle \\ \le
  (2 \widetilde{\alpha}(t)+ \frac{1}{2}\gamma_{\varepsilon, \varepsilon'}(t)) |x-y|^2  + 2 C_1 \widetilde{\alpha}(t) (\varepsilon+\varepsilon')^2 + 2 C_2  \sup_{z\in\mathbb{R}^n}{|\widetilde{a}(t,z)|} (\varepsilon+\varepsilon') + \frac{1}{2} \gamma_{\varepsilon, \varepsilon'}(t) 
  \end{multline*}
 Setting
 \begin{eqnarray*}
 \alpha(t) &:=& (2 \widetilde{\alpha}(t)+ \frac{1}{2} \sup_{\varepsilon, \varepsilon' \in ]0,1]}r_{\varepsilon, \varepsilon'}(t))\\
 \gamma_{\varepsilon, \varepsilon'}(t)&:=& 2 C_1 \widetilde{\alpha}(t) (\varepsilon+\varepsilon')^2 + 2 C_2  \sup_{z\in\mathbb{R}^n}{|\widetilde{a}(t,z)|} (\varepsilon+\varepsilon') + \frac{1}{2} r_{\varepsilon, \varepsilon'}(t). 
 \end{eqnarray*}
 we obtain that $A_{\varepsilon}(t,x)$ satisfies (\ref{colombeau_forwarduniqueness_condition}). 
So the requirements for Theorem 
 \ref{colombeau_existence_theorem} and \ref{colombeau_uniqueness_theorem} are fulfilled.
\end{proof}

%% file: pde.tex
\chapter{A comparative study of solution concepts for first order hyperbolic partial differential equations with non-smooth coefficients}

This chapter contains a slightly adapted version of the joint (with G\"unther H\"ormann) article \cite{HalHoer:08}.

According to Hurd and Sattinger in \cite{HS:68}
the issue of a systematic investigation of hyperbolic partial differential equations with discontinuous coefficients as a research topic has been raised by Gelfand in 1959. Here, we attempt a comparative study of some of the theories on that subject which have been put forward since. More precisely, we focus on techniques and concepts that build either on the geometric picture of propagation along \emph{characteristics} or on the functional analytic aspects of \emph{energy estimates}. 

In order to produce a set-up which makes the various methods comparable at all, we had to stay with the special situation of a \emph{scalar partial differential equation with real coefficients}. As a consequence, for example, we do not give full justice to theories whose strengths lie in the application to systems rather than to a single equation. 
A further limitation in our choices comes from the restriction to concepts, hypotheses and mathematical structures which (we were able to) directly relate to distribution theoretic or measure theoretic notions.

To illustrate the basic problem in a simplified lower dimensional situation for a linear conservation law, we consider the following formal differential equation for a density function (or distribution, or generalized function) $u$ depending on time $t$ and spatial position $x$ 
$$
   \d_t u(t,x) + \d_x(a(t,x) u(t,x)) = 0.
$$
Here, $a$ is supposed to be a \emph{real} function (or distribution, or generalized function) and the derivatives shall be interpreted in the distributional or weak sense. This requires either to clarify the meaning of the product $a \cdot u$ or to avoid the strict meaning of ``being a solution''.  

An enormous progress has been made in research on nonlinear conservation laws (cf., e.g.\ \cite{Hoermander:97,BGS:07} and references therein)  of the form
$$
  \d_t u(t,x) + \d_x( g(u(t,x)) ) = 0,  
$$
where $g$ is a (sufficiently) smooth function and $u$ is such that $g(u)$ can be defined in a suitable Banach space of distributions. Note however, that this equation  does not include linear operators of the form described above as long as the nonlinearity $g$ does not include
additional dependence on $(t,x)$ as independent variables (i.e., is not of the more general form $g(t,x,u(t,x))$). Therefore the theories for linear equations described in the present paper are typically not mere corollaries of the nonlinear theories.
Essentially for the same reason we have also not included methods based on Young measures (cf.\ \cite[Chapter V]{Hoermander:97}).

Further omissions in our current paper concern hyperbolic equations of second order. For advanced theories on these we refer to the energy method developed by Colombini-Lerner in \cite{CL:95}. An overview and illustration of non-solvability or non-uniqueness effects with wave equations and remedies using Gevrey classes can be found in \cite{Spagnolo:87}.

Of course, also the case of first-order equations formally ``of principal type'' with non-smooth complex coefficients is of great interest. It seems that the borderline between solvability and non-solvability is essentially around Lipschitz continuity of the coefficients (cf.\ \cite{Jacobowitz:92,Hounie:91,HounieMelo:95}). Moreover, the question of uniqueness of solutions in the first-order case has been addressed at impressive depth in \cite{CL:02}.

Our descriptive tour with examples consists of two parts: Section 4.1 describes concepts and theories extending the classical method of characteristics, while Section 4.2 is devoted to theories built on energy estimates. All but two of the theories or results (namely, in Subsections 4.1.3 and 4.2.3.2) we discuss and summarize are not ours. However, we have put some effort into unifying the language and the set-up, took care to find as simple as possible examples which are still capable of distinguishing certain features, and have occasionally streamlined or refined the original or well-known paths in certain details. 

In more detail, Subsection 4.1.1 starts with Caratheodory's theory of generalized solutions to first-order systems of (nonlinear) ordinary differential equations and adds a more distribution theoretic view to it. In Subsection 4.1.2 we present the generalization in terms of Filippov flows and the application to transport equations according to Poupaud-Rascle. Subsection 4.1.3 provides a further generalization of the characteristic flow as Colombeau generalized map with nice compatibility properties when compared to the Filippov flow. In Subsection 4.1.4 we highlight some aspects or examples of semigroups of operators on Banach spaces stemming from underlying generalized characteristic flows on the space-time domain. We also describe a slightly exotic concept involving the measure theoretic adjustment of coefficients to prescribed characteristics  for $(1+1)$-dimensional equations according to Bouchut-James in Subsection 4.1.5.

Subsection 4.2.1 presents a derivation of energy estimates under very low regularity assumptions on the coefficients and also discusses at some length the functional analytic machinery to produce a solution and a related weak solution concept for the Cauchy problem. Subsection 4.2.2 then compares those three theories, namely by Hurd-Sattinger, Di Perna-Lions, and Lafon-Oberguggenberger,  which are based on regularization techniques combined with energy estimates. Finally, Subsection 4.2.3 briefly describes two related results obtained by paradifferential calculus, the first concerning energy estimates and the solution of the Cauchy problem for a restricted class of operators, the second is a method to reduce equations to equivalent ones with improved regularity of the source term.

As it turns out in summary, none of the solution concepts for the hyperbolic partial differential equation is contained in any of the others in a strict logical sense. However, there is one feature of the Colombeau theoretic approach: it is always possible to model the coefficients and initial data considered in any of the other theories (by suitable convolution regularization) in such a way that the corresponding Cauchy problem becomes uniquely solvable in Colombeau's generalized function algebra.  In many cases the Colombeau generalized solution can be shown to have the appropriate distributional aspect in the sense of heuristically reasonable solution candidates.

\section{Solution concepts based on the characteristic flow}

In this section we introduce solution concepts for first order partial differential equations,
which are based on solving the system of ordinary differential equations for the characteristics and using the resulting characteristic flow to define a solution.

To illustrate the basic notions we consider the following special case of the Cauchy problem in conservative form
$$
  L u :=\partial_t u + \sum_{k=1}^n \d_{x_k}( a_k(t,x)  u) = 0,
  \quad u(0) = u_0 \in\mathcal{D}'(\mathbb{R}^n),
$$ 
where the coefficients $a_k$ are real-valued bounded smooth functions. 
The associated system of ordinary differential equations for the characteristic curves reads
\begin{eqnarray*}
  \dot{\xi_k}(s) = a_k(s, \xi(s)) , \ \ \xi_k(t)=x_k \qquad 
  (k=1,\ldots,n).
\end{eqnarray*}
We use the notation $\xi(s;t,x) = (\xi_1(s;t,x),... ,\xi_n(s;t,x))$, where the variables after the semicolon indicate the initial conditions $x= (x_1, ..., x_n)$ at $t$.
%, or $\xi(s)$ to simplify.
We define the smooth characteristic forward flow
$$ 
    \chi: [0,T] \times \mathbb{R}^n \rightarrow \mathbb{R}^n. 
    \quad  (s,x) \mapsto \xi(s;0,x)
$$
% $\chi(s,\chi(r,x)) = \chi(s+r,x)$ for all $s,r,s+r \in [0,T]$. 
Note that $\chi$ satisfies the relation ($d_x$ denoting the Jacobian with respect to the $x$ variables)
$$
  \d_t \chi(t,x) =  d_x \chi(t,x) \cdot a(t,x) \qquad 
    \forall (t,x) \in [0,T] \times \R^n,
$$
which follows upon differentiation of the characteristic differential equations and the initial data with respect to $t$ and $x_k$ ($k=1,\ldots,n$).
Using this relation a straightforward calculation shows that the distributional solution $u  \in  C^\infty([0,T];\mathcal{D}'(\mathbb{R}^{n}))$ to 
$$
  L u  = 0 , \quad u(0) = u_0 \in\mathcal{D}'(\mathbb{R}^n)
$$
is given by
$$
  \dis{u(t)}{\psi} := \dis{u_0}{\psi(\chi(t,.))} \qquad 
    \forall \psi\in \D(\R^n), 0 \leq t \leq T.
$$
If there is a further zero order term $b \cdot u$ in the differential operator $L$, then the above solution formula is modified by an additional factor involving $b$ and $\chi$ accordingly.

In a physical interpretation the characteristic curves correspond to the trajectories of point particles. This provides an idea for introducing a generalized solution concept when the partial differential operator has non-smooth coefficients:
As long as a continuous flow can be defined, the right-hand side in the above definition of $u$ is still meaningful when we 
assume $u_0 \in \mathcal{D}^{\prime 0}(\mathbb{R}^n)$. The distribution $u$ defined in such a way belongs to $AC([0,T]; \mathcal{D}^{'0}(\mathbb{R}^{n}))$ and will be called a \emph{measure solution}\index{measure solution}.

This approach is not limited to classical solutions of the characteristic system of ordinary differential equations, but can be extended to  more general solution concepts in ODE theory (for example, solutions in the sense of Filippov). Although such a generalized solution will lose the property of solving the partial differential equation in a distributional sense it is a useful generalization with regard to the physical picture.

\subsection{Caratheodory theory}

Let $T>0$ and $\Omega_T = ]0,T[ \times \R^n$. 
Classical Caratheodory theory (cf. Section \ref{caratheodory_solution_concept}) requires the coefficient $a=(a_1, ... , a_n)$ to satisfy
\begin{enumerate}
\item $a(t,x)$ is continuous in $x$ for almost all $t \in [0,T]$,
\item $a(t,x)$ is measurable in $t$ for all fixed $x\in \mathbb{R}^n$ and
\item $\sup_{x\in\mathbb{R}^n} |a(t,x)| \le \beta(t)$ almost everywhere for some positive function $\beta \in L^1([0,T])$.
\end{enumerate}
Then the existence of an absolutely continuous characteristic curve $\xi=(\xi_1, ..., \xi_n)$, which fulfills the ODE almost everywhere, is guaranteed.
Note that the first two Caratheodory conditions ensure Lebesgue measurability of the composition $s\mapsto a(s,f(s))$ for all $f\in {AC}([0,T])^n$, while the
third condition is crucial in the existence proof.\\
A sufficient condition for forward uniqueness of the characteristic system 
is the existence of a positive $\alpha \in L^1([0,T])$,
such that ($\langle . , .\rangle$ denoting the standard inner product on $\mathbb{R}^n$)
\begin{eqnarray*}
  \langle a(t,x) - a(t,y) , x-y \rangle \le \alpha(t) |x-y|^2 
\end{eqnarray*}
for almost all $(t,x), (t,y) \in \ovl{\Omega_T}$ (cf.\ \cite[Theorem 3.2.2]{AgarLak:93}). As well-known from classical ODE theory, forward uniqueness of the characteristic curves yields a continuous forward flow
$$
  \chi: [0,T] \times \mathbb{R}^n \rightarrow \mathbb{R}^n.
  \quad (s,x) \mapsto  \xi(s;0,x)
$$
It is a proper map and for fixed time $\chi(t,.)$ is onto. For the sake of simplicity we assume $a \in C([0,T] \times \mathbb{R}^n)^n$
and $b \in C([0,T] \times \mathbb{R}^n)$.

Let
$$
h_b(t,x) := \exp{ \left( -\int_0^t  b(\tau,\chi(\tau,x)) \m{\tau} \right) },
$$
then $u\in \mathcal{D}^{'}(\Omega_T)$ defined by 
\begin{eqnarray} \label{carflowsol}
\langle u, \varphi \rangle_{\mathcal{D}^{'}(\Omega_T)}  := \int_0^T \langle u_0 , \varphi(t,\chi(t,\cdot)) h_b(t,\cdot) \rangle_{\mathcal{D}^{'0}(\mathbb{R}^n)} \m{t}
\end{eqnarray}
(note that $u$ can be regarded as element in $AC([0,T];\mathcal{D}^{'0}(\mathbb{R}^n))$, so the restriction $u(0)$ is well-defined and equal to $u_0\in\mathcal{D}^{'0}(\mathbb{R}^n)$)
solves the initial value problem 
$$
  Lu:= \partial_t u + 
  \sum_{k=1}^n  \partial_{x_k} (a_k \cdot u) +b u = 0, \ \ u(0)=u_0 
$$
on $\Omega_T$, where $a_k \cdot u$ and $b \cdot u$ denotes the distributional product defined by
\begin{eqnarray*}
\cdot: C(\Omega_T)  \times \mathcal{D}^{'0}(\Omega_T) &\rightarrow& \mathcal{D}^{'0}(\Omega_T)\\
  		(f,u) & \mapsto & (\varphi \mapsto \langle u, f \cdot \varphi \rangle_{\mathcal{D}^{'0}(\Omega_T)}).  
\end{eqnarray*}
%We will use the notation $({\rm id}\times\chi)^{\ast}g) (t,x)=g(t,\chi(t,x))$ for any continuous function $g\in C_c(\Omega)$.

Applying $L$ on $u$ we obtain
\begin{multline*}
\langle L u, \varphi \rangle_{\mathcal{D}^{'}(\Omega_T)} = \langle u, -\partial_t \varphi - \sum_{k=1}^n a_k \partial_{x_k} \varphi + b \varphi \rangle_{\mathcal{D}^{'}(\Omega_T)} \\
=
 \int_0^T \langle u_0 , (-\partial_t \varphi - \sum_{k=1}^n a_k \partial_{x_k} \varphi + b \varphi)(t,\chi(t,\cdot)) h_b(t,\cdot) \rangle_{\mathcal{D}^{'0}(\mathbb{R}^n)} \m{t}.
\end{multline*}
Set $\phi(t,x):= \varphi(t,\chi(t,x))$ and $\psi(t,x) := \phi(t,x) \cdot h_b(t,x)$, then we have
$$
\partial_t \phi(t,x) = \partial_t \varphi(t,\chi(t,x)) = (\partial_t \varphi + \sum_{k=1}^n a_k(t,x) \partial_{x_k} \varphi)(t,\chi(t,x)),
$$ and
\begin{multline*}
\partial_t \psi(t,x)  = \partial_t \phi(t,x) h_b(t,x) + \phi(t,x) \partial_t h_b(t,x)  \\
= 
(\partial_t \varphi + \sum_{k=1}^n a_k(t,x) \partial_{x_k} \varphi)(t,\chi(t,x)) \cdot  h_b(t,x)- \varphi(t,\chi(t,x)) b(t,\chi(t,x)) h_b(t,x) \\
=(\partial_t \varphi + \sum_{k=1}^n a_k(t,x) \partial_{x_k} \varphi - b \varphi)(t,\chi(t,x)) \cdot h_b(t,x),  
\end{multline*}
thus
$$
\langle L u, \varphi \rangle_{\mathcal{D}^{'}(\Omega_T)} = -\int_0^T \langle u_0 , \partial_t \psi(t,\cdot) \rangle_{\mathcal{D}^{'0}(\mathbb{R}^n)} \m{t} = 
-\int_0^T \partial_t \langle u_0 ,  \psi(t,\cdot) \rangle_{\mathcal{D}^{'0}(\mathbb{R}^n)} \m{t} =0.
$$
for all $\varphi \in \mathcal{D}(\Omega_T)$. The initial condition $u(0)=u_0$ is satisfied, since $\chi(0,x)=x$, thus $h_b(0,x)=1$.

\begin{remark}
In this sense, we can obtain a distributional solution for the Cauchy problem 
\begin{eqnarray*}
P v:= \partial_t v +\sum_{k=1}^n a_k \partial_{x_k} v + c v = 0, \ \ v(0)=v_0,
\end{eqnarray*}
whenever $a \in C([0,T]\times \mathbb{R}^n)^n$ and $c \in \mathcal{D}'([0,T] \times\mathbb{R}^n)$, such that $-{\rm div}(a) + c \in C([0,T] \times \mathbb{R}^n)$ and
$v_0 \in \mathcal{D}^{'0}(\mathbb{R}^n)$. We simply set $b:= -{\rm div}(a) + c$ and construct the solution as above. In other words, such a solution
solves the equation in a generalized sense, relying on the definition of the action of $Q:=\sum_{k=1}^n a_k \partial_k +c$ on a distribution of order $0$ by
\begin{eqnarray*}
\langle Q v, \varphi \rangle_{\mathcal{D}^{'}(\Omega_T)} := -\langle v , \sum_{k=1}^n a_k \partial_{x_k} \varphi \rangle_{\mathcal{D}^{'0}(\Omega_T)} - \langle v, (-{\rm div}(a) +c) \varphi \rangle_{\mathcal{D}^{'0}(\Omega_T)}. 
\end{eqnarray*}
In case where ${\rm div}(a)$ and $c$ are both continuous, we can define the operator $Q$ classically by using the product $\cdot:
\mathcal{D}^{'0}(\Omega_T) \times C(\Omega_T) \rightarrow \mathcal{D}^{'0}(\Omega_T) $ as above.
\end{remark}

\subsection{Filippov generalized characteristic flow}

As we have seen in the previous subsection, forward unique characteristics give rise to a continuous forward flow. But
in order to solve the characteristic differential equation in the sense of Caratheodory, we needed continuity of the coefficient $a$ in the space variables for almost all $t$. 
In case of more general coefficients $a\in L^1_{\rm loc}(\mathbb{R}, L^{\infty}(\mathbb{R}^n))^n$ we can employ
the notion of Filippov characteristics, which replaces the ordinary system of differential equations by a system of differential inclusions 
(cf. \ref{filippov_solution_concept}). 
The generalized solutions are still absolutely continuous functions.
Again, the forward-uniqueness condition on the coefficient $a$ 
\begin{eqnarray} 
\langle a(s,x) - a(s,y) , x-y \rangle \le \alpha(s) |x-y|^2,\quad x,y\in \R^n 
\end{eqnarray}
for almost all $s\in \R$ yields unique solutions in the Filippov generalized sense. Then the generated Filippov flow is again continuous and will
enable us to define measure-valued solutions of the PDE (cf. \cite{PoupaudRascle:97}), as before.

In the Filippov solution concept the coefficient is replaced by a set-valued function $A:(t,x) \rightarrow A_{t,x}$ where $A_{t,x}$ are non-empty, closed and convex subsets of $\R^n$. It has to satisfy some basic properties (FC) (as introduced in Section \ref{filippov_conditions}) which imply by means of Theorem \ref{filippov_existence_theorem} the solvability of the resulting system of differential inclusions 
\begin{eqnarray*}
\dot{\xi}(s) \in A_{s, \xi(s)}, \ a.e., \ \ \xi(t)=x,
\end{eqnarray*} 
with $\xi \in AC([0,\infty[)^n$. These basic conditions (cf. section \ref{differential_inclusion_equation_section}) are
\begin{enumerate} \label{filippov_conditions_2}
\item $s \mapsto A_{s,x}$ is Lebesgue measurable on $\R$ for all fixed $x \in \mathbb{R}^n$,
\item $x \mapsto A_{s,x}$ is upper semi-continuous for almost all $s \in \R$,        
\item there exists a positive function $\beta \in L_{\rm loc}^1(\R)$ such that $\sup_{x\in\mathbb{R}^n} |A_{s,x}| \le \beta(s)$ almost everywhere.
\end{enumerate}

There are several ways to obtain such a set-valued function $A$ from a coefficient $a\in L^1([0,T]; L^{\infty}(\mathbb{R}^n) )^n$, such that
the classical theory is extended in a compatible way, i.e. the set-valued function $A$ should fulfill $A_{t_0,x_0}:=\{a(t_0,x_0)\}$ whenever $a$ is continuous at $(t_0,x_0)\in [0,\infty[ \times \mathbb{R}^n$.

One way obtaining a set-valued function corresponding to a $a \in L^{1}_{\rm loc}(\mathbb{R}; L^{\infty}(\mathbb{R}^n))^{n}$ is
by means of the essential convex hull $\ech{a}$.
According to Definition \ref{essential_convex_hull} its supporting function is
$$
H_a(t,x, w) = \lim_{\delta \rightarrow 0} \esup_{y \in B_{\delta}(x)}{\langle a(t,y), w \rangle}
$$
for almost all $t\in J$ and all $x,w \in \R^n$, so 
$$
\ech{a}_{t,x} = \{a\in \R^n \mid \forall w\in\R^n: \langle a,w \rangle \le H_a(t,x,w) \}.
$$

% Another way is to use a mollifier $\rho \in \mathcal{S}(\mathbb{R}^n)$ with $\int \rho(x) dx =1$, put $\rho_{\varepsilon}(x) = \varepsilon^{-n} \rho(\varepsilon^{-1} x)$ and
% $A_{\varepsilon} := a \ast \rho_{\varepsilon}\mid_{[0,T] \times \mathbb{R}^n}$. 
% Then we can use the concept the of a generalized graph $\ggraph{[(A_{\varepsilon})_{\varepsilon}]}$ as defined in Section \ref{} a which also satisfies the above basic properties.

\subsubsection{Measure solutions according to Poupaud-Rascle} \label{section_poupaud_rascle}

Let $\Omega_{\infty} := \,]0,\infty[ \times \mathbb{R}^n$.
We assume $a \in L^1_{\rm loc}(\mathbb{R}_{+}; L^{\infty}(\mathbb{R}^n))^n$ to be a coefficient satisfying the forward uniqueness criterion \eqref{forward_uniqueness_condition}.
Let $L u := \partial_t u + \sum_{i=1}^n \partial_{x_i} (a_i u)$ and 
$\xi$ be the forward unique solution to
\begin{eqnarray}
\dot{\xi}(s) & \in & \ech{a}_{s, \xi(s)}, \ \ \xi(t) = x.
\end{eqnarray}
 The map
$$
    \chi: \mathbb{R}_{+} \times \mathbb{R}^n \rightarrow 
    \mathbb{R}^n, \quad (t,x) \mapsto \xi(t;0,x)
$$
is the continuous Filippov (forward) flow.

\begin{definition}[Solution concept according to Poupaud-Rascle]
Let $u_0\in \mathcal{M}_b(\mathbb{R})^n$ be a bounded Borel measure, then the image measure at $t\in[0,\infty[$ is 
\begin{eqnarray} \label{filflowsol}
u(t)(B) :=  \int_{\mathbb{R}^n} 1_{B}(\chi(t,x)) \m{u_0(x)}, 
\end{eqnarray}
where $B\subseteq \mathbb{R}^n$ is some Borel set. The map $u \col [0,\infty[ \to \mathcal{M}_b(\mathbb{R}^n))$ belongs to $C([0,\infty[; \mathcal{M}_b(\mathbb{R}^n))$ and is called
a measure solution in the sense of \emph{Poupaud-Rascle}\index{Poupaud-Rascle!solution concept of} of the initial value problem
\begin{equation} \label{pdo2}
L u:= \partial_t u + \sum_{k=1}^n \partial_{x_k} (a_k \cdot u) = 0, \ \ u(0) = u_0.
\end{equation}
Note that $u$ defines a distribution of order $0$ in $\mathcal{D}'(\Omega_{\infty})$ by
\begin{eqnarray*}
\langle u, \varphi \rangle_{\mathcal{D}'(\Omega_{\infty})} := \int_0^{\infty} \langle u_0,  \varphi(t,\chi(t,x)) \rangle_{\mathcal{D}^{'0}(\mathbb{R}^n)} \m{t}, \quad \quad \forall \varphi\in \mathcal{D}'(\Omega_{\infty}). 
\end{eqnarray*}
\end{definition}

The solution concept of Poupaud-Rascle does not directly solve the partial differential equation in a distributional sense,
but it still reflects the physical picture of a "transport process" as imposed by
the properties of the Filippov characteristics. Nevertheless, in the cited paper of Poupaud-Rascle(\cite{PoupaudRascle:97}) the authors present an a posteriori definition of the particular product $a\cdot u$, which restore the validity of the PDE in a somewhat artificial way. We investigate this in the sequel in some detail.

\begin{definition}[A posteriori definition of a distributional product in the sense of Poupaud-Rascle]
Let $u\in \mathcal{D}^{\prime}(\mathbb{R}^n)$ be a distribution of order $0$ and $a\in {L^1_{\rm loc}([0,\infty[, L^{\infty}(\mathbb{R}^n))}^n$,
satisfying the forward uniqueness condition (\ref{forward_uniqueness_condition}), such that there exists a continuous Filippov flow $\chi$. Furthermore we assume that $u$ is a generalized solution of the initial value problem
as defined in (\ref{filflowsol}). Then we define the product $a \bullet u =(a_k\cdot u)_k$ in $\mathcal{D}'(]0,\infty[\times \mathbb{R}^n)^n$ by
\begin{eqnarray*}
\langle a \bullet u, \varphi \rangle_{\mathcal{D}'(\Omega_{\infty})} &:=& \langle u_0, \int_0^{\infty} \partial_t \chi(t,x) \varphi(t, \chi(t,x) ) \m{t} \rangle_{\mathcal{D}^{'0}(\mathbb{R}^n)}, \ \ \varphi \in \mathcal{D}(\Omega_{\infty}).
\end{eqnarray*}
\end{definition}

\begin{remark}
Note that the product $a\cdot u$ is defined only for distributions $u$ that are generalized solutions (according to Poupaud-Rascle) of the initial value problem 
(\ref{pdo2}) with the coefficient $a$. The domain of the product map $(a,u) \mapsto a\bullet u$, as subspace of $\mathcal{D}^{'0}(\R^n) \times \mathcal{D}^{'0}(\R^n)$  has a complicated structure: Just note that the property to generate a continuous characteristic Filippov flow $\chi$
is not conserved when the sign of the coefficient $a$ changes, as we have seen for the coefficient $a(x)={\rm sign}(x)$.
\end{remark}

\begin{example} \label{example_poupaud_rascle_1}
Consider problem (\ref{pdo2}) with the coefficient $a(x):= -{\rm sign}(x)$ subject to the initial condition $u_0= 1$. Then the continuous Filippov flow
 is given by 
$$
 \chi(t,x) = -(t+x)_{-} H(-x) + (x-t)_{+} H(x).
$$
We have $\chi(t,0)= t_{+}- (-t)_{-} = 0$
and
\begin{eqnarray*}
\partial_t \chi(t,x) = - H(-t-x) H(-x) - H(x-t) H(x) \ \ \text{for almost all } \ t\in[0, \infty[. 
\end{eqnarray*}
The generalized solution $u$ is defined by $\langle u, \varphi \rangle: = \int_0^{\infty} \langle u_0, \phi(t, x) \rangle \m{t}$,
where $\phi(t,x) := \varphi(t, \chi(t,x))$. We have that
\begin{eqnarray*}
\phi(t,x) := \left\{  \begin{array}{ll}
\varphi(t,x+t) & x\le 0, 0\le t \le -x \\
\varphi(t,0) & t \ge |x| \\
\varphi(t,x-t) & x\ge 0, 0\le t \le x ,\\
\end{array}
\right. 
\end{eqnarray*}
thus 
\begin{multline*}
\langle u, \varphi \rangle_{\mathcal{D}^{'0}(\Omega_\infty)} :=\int_0^{\infty} \langle u_0, \phi(t, x) \rangle \m{t}
= \int_0^{\infty} \int_{-\infty}^{\infty}  \phi(t, x) \m{x} \m{t} \\
= 
2 \int_0^{\infty} \varphi(t, 0) t \m{t}  +  \int_0^{\infty} \int_{-\infty}^{-t} \varphi(t,x+t) \m{x} \m{t} +
\int_0^{\infty} \int_{t}^{\infty} \varphi(t,x-t)  \m{x}\m{t} \\
= 2 \int_0^{\infty} \varphi(t, 0) t \m{t}  +  \int_0^{\infty} \int_{-\infty}^{0} \varphi(t,z) \m{z} \m{t} +
\int_0^{\infty} \int_{0}^{\infty} \varphi(t,z)  \m{z}\m{t} 
= \langle 1+ 2t \delta, \varphi(t,\cdot) \rangle_{\mathcal{D}^{'0}(\mathbb{R}^n)} 
\end{multline*}
This generalized solution gives rise to the following product
\begin{eqnarray*}
\langle (-{\rm sign}(x)) \bullet (1+2t \delta(x)), \varphi \rangle_{\mathcal{D}'(\Omega_{\infty})} &:=&
 \langle 1 , \int_0^\infty  \partial_t \xi(t,x) \varphi(t, \xi(t,x)) \m{t} \rangle_{\mathcal{D}^{'0}(\mathbb{R}^n)}
\end{eqnarray*}
in $\mathcal{D}'(\Omega_{\infty})$. Evaluating the right-hand side we obtain 
\begin{multline*}
\langle 1 , \int_0^\infty  \partial_t \chi(t,x) \varphi(t, \chi(t,x)) \m{t} \rangle 
=
\int_{-\infty}^{\infty}  \left( -\int_{0}^{\infty} H(-x)H(-x-t) \varphi(t, x+t) \m{t} \right. \\
\left.
-\int_{0}^{\infty} H(x)H(x-t) \varphi(t, x-t) \m{t} \right) \m{x}.
\end{multline*}
Since $H(-x)H(-x-t) =H(-x-t)$ and $H(x) H(x-t) = H(x-t)$ for $t\ge 0$ the latter gives upon substitution
\begin{eqnarray*}
 \langle 1 , \int_0^\infty  \partial_t \chi(t,x) \varphi(t, \chi(t,x)) \m{t} \rangle 
&=& -\int_{-\infty}^{\infty} {\rm sign}(z)  \int_{0}^{\infty} \varphi(t, z) \m{t}  \m{z},
\end{eqnarray*}
hence
\begin{eqnarray*}
(-{\rm sign}(x)) \bullet (1+2t \delta(x)) = -{\rm sign}(x).
\end{eqnarray*}
However, we cannot define the product if  $-{\rm sign}(x)$ is replaced by $+{\rm sign}(x)$, since
the Filippov characteristics $\xi(t;0,x)$ are no longer forward unique and thus do not generate a continuous  Filippov flow $\chi$.
\end{example}

\begin{example}
We consider the same coefficient $a(x):=- {\rm sign}(x)$ as before, but now we set $u_0:=\delta$. We obtain the generalized solution
\begin{eqnarray*}
\langle u, \varphi \rangle_{\mathcal{D}^{'}(\Omega_{\infty})} &:=&  {\langle 1 \otimes \delta, \varphi(t, \chi(t,x)) \rangle}_{\mathcal{D}^{'}(\Omega_{\infty})} = \int_0^\infty  \varphi(t, \chi(t,0))  \m{t} 
\end{eqnarray*}
This enables us to calculate the product
\begin{eqnarray*}
&&\langle (- {\rm sign}(x)) \bullet \delta(x),\varphi \rangle = -\langle \delta, \int_0^{\infty} \partial_t \chi(t,x) \varphi(t, \xi(t,x)) \m{t} \rangle.
\end{eqnarray*}
Putting $\psi(x)=\int_0^{\infty} \partial_t \xi(t,x) \varphi(t, \xi(t,x)) \m{t}$ and observe that
\begin{eqnarray*}
\psi(x) &:=&\int_0^{\infty} \partial_t \chi(t,x)  \varphi(t, \chi(t,x)) \m{t} = 
\int_0^{-x} \varphi(t, x+t ) \m{t} , \ \text{if}\ x<0
\end{eqnarray*}
and
\begin{eqnarray*}
\psi(x) &=&\int_0^{\infty} \partial_t \chi(t,x)  \varphi(t, \chi(t,x)) \m{t} = 
-\int_0^{x}  \varphi(t,x-t) \m{t},\ \text{if}\ x>0.
\end{eqnarray*}
At $x=0$ we obtain $\psi(0)=\lim_{x\rightarrow 0_{-}} \psi(x)=\lim_{x\rightarrow 0_{+}} \psi(x)=0$, so it follows that
$(- {\rm sign}) \bullet \delta = 0$.
\end{example}

\begin{example}\label{test1}
Let $a(t,x):= 2 H(-x)$, so that the Filippov flow is given by
\begin{eqnarray*}
\chi(t,x) = -(x+2 t)_{-} H(-x) + x H(x). 
\end{eqnarray*}
We have $\chi(t,0)= -2t_{-} =0$ and 
\begin{eqnarray*}
\partial_t \chi(t,x):=2 H(-x-2t) H(-x).
\end{eqnarray*}
Hence $\partial_t \chi(t,0) = 0$ for almost all $t \in [0, \infty[$.
If $u_0=1$ the generalized solution is
\begin{eqnarray*}
\langle u, \varphi \rangle_{\mathcal{D}^{'0}(\Omega_{\infty})} := \int_0^{\infty} \int_{-\infty}^{\infty} \phi(t, x) \m{x} \m{t}, 
\end{eqnarray*}
where $\phi(t,x) = \varphi(t, \chi(t,x))$. Since
\begin{eqnarray*}
\phi(t,x)\mid_{\{x< -2t\}} &=& \varphi(t, x+2t) \\
\phi(t,x)\mid_{\{-2t \le x \le 0 \}} &=& \varphi(t, 0) \\
\phi(t,x)\mid_{\{0 < x   \}} &=& \varphi(t, x) ,
\end{eqnarray*}
we obtain
\begin{multline*}
\langle u, \varphi \rangle_{\mathcal{D}^{'0}(\Omega_{\infty})} 
= \int_0^{\infty}
 \left( \int_{-\infty}^{-2t} \varphi(t, x+2t) \m{x}  + 2t \varphi(t, 0)  + \int^{\infty}_{0} \varphi(t, x) \m{x} \right) \m{t} 
= \langle 1+ t \delta, \varphi(t,\cdot) \rangle_{\mathcal{D}^{'0}(\mathbb{R}^n)},
\end{multline*}

hence $u= 1+2 t \delta(x)$.
Again we determine the product $(2H(-x)) \bullet  (1+2 t \delta(x))$ by
\begin{multline*}
\langle 2 H(-x) \bullet  (1+2t \delta(x)), \varphi \rangle_{\mathcal{D}'(\Omega_{\infty})} = 
2 \int_{-\infty}^{\infty} \int_0^{\infty} H(-x) H(-x-2t) \varphi(t, x+2t) \m{t} \m{x} \\
=
2 \int_{-\infty}^{\infty} \int_0^{\infty}  H(-x-2t) \varphi(t, x+2t) \m{t} \m{x} 
= 2 \int_0^{\infty} \int_{-\infty}^{\infty} H(-z) \varphi(t,z)  \m{z} \m{t} = \langle 1\otimes 2 H(-\cdot), \varphi \rangle_{\mathcal{D}'(\Omega_{\infty})}.
\end{multline*}
We obtain $(2H(-x)) \bullet  (1+2 t \delta(x)) = 2 H(-x)$. 
Observe that together with the result in Example (\ref{example_poupaud_rascle_1}) $(-{\rm sign}(x)) \bullet (1+2 t \delta(x))=(2 H(-x) - 1 ) \bullet (1+2 t \delta(x))$ we 
can conclude that either $(-1) \bullet  (1+2 t \delta(x)) $ is not defined or the product $\bullet$ is not distributive.
In fact , it is not difficult to see that $(-1) \bullet  (1+2 t \delta(x)) $ cannot be defined in this way, neither can $1 \bullet  (1+2 t \delta(x))$.
\end{example}

\begin{example}[Generalization of Example \ref{example_poupaud_rascle_1}] \label{poupaud_rascle_jump} 
Let $c_1 \ge c_2$ be two constants, and $\alpha \in [c_1, c_2]$.
Consider the $a(t,x):=c_1 H(\alpha t- x) + c_2 H(x- \alpha t)$. We set $t_1(x):= \frac{-x}{c_1-\alpha}$ if $x<0$ and $t_2(x):= \frac{x}{\alpha-c_2}$  for $x>0$. 
The unique Filippov flow is given by
\begin{eqnarray*}
\chi(t,x) =\left\{ 
\begin{array}{ll}
c_1 t + x  &  x < 0, t < t_1(x)\\
\alpha t & x<0, t\ge t_1(x) \\
\alpha t & x=0, \\
c_2 t +x & x>0, t\le t_2(x) \\
 \alpha t & x>0 , t\ge t_2(x)
\end{array}\right.
\end{eqnarray*}
The generalized solution of the initial value problem $Lu:=\partial_t u + \partial_x (a \cdot u)=0, \ \ u(0)=u_0\in L^1_{\rm loc}(\mathbb{R})$,
according to Poupaud-Rascle is given by 
\begin{multline*}
\langle u, \varphi\rangle_{\mathcal{D}^{'0}(\Omega_T)}=
\int_0^T \langle u_0, \varphi(t, \chi(t,\cdot)) \rangle_{\mathcal{D}^{'0}(\mathbb{R})} \m{t} = \int_{-\infty}^0 \int_{0}^{t_1(x)} u_0(x)\varphi(t,c_1 t+x ) \m{t} \m{x} \\
+
\int_{-\infty}^0 \int_{t_1(x)}^{T} u_0(x)\varphi(t,\alpha t ) \m{t} \m{x} + 
\int_{0}^{\infty} \int_{0}^{t_2(x)} u_0(x)\varphi(t,c_2 t+x ) \m{t} \m{x} 
+
\int_{0}^{\infty} \int_{t_2(x)}^{T} u_0(x) \varphi(t,\alpha t ) \m{t} \m{x}  \\
=
  \int_{0}^{T}\int_{-\infty}^{-t(c_1-\alpha)} u_0(x)\varphi(t,c_1 t+x )  \m{x} \m{t} 
+
\int_{0}^{T}  \int_{-t(c_2 -\alpha)}^{\infty} u_0(x)\varphi(t,c_2 t+x )  \m{x} \m{t} \\ +  \int_{0}^T \left(\int_{-t(c_1-\alpha)}^{t(\alpha-c_2)} u_0(x)  \m{x} \right) \varphi(t,\alpha t ) \m{t},
\end{multline*}
hence 
\begin{eqnarray*}
u:= u_0(x-c_1 t) H(\alpha t-x) + u_0(x-c_2 t) H(x-\alpha t) + \left(\int_{-t(c_1-\alpha)}^{t(\alpha-c_2)} u_0(x) \m{x}\right) \delta(x-\alpha t).
\end{eqnarray*}
\end{example}

\subsection{Semi-groups defined by characteristic flows}

Let $X$ be a Banach space and $(\Sigma_t)_{t\in [0,\infty[}$ be a family of bounded operators $\Sigma_t$ on $X$.
Consider the following conditions: 
\begin{enumerate}
\item $\Sigma_0 = {\rm id}$ %as $\chi(0,x)=x$,
\item $\Sigma_s \circ \Sigma_t = \Sigma_{s+t}$ for all $s,t \in [0,\infty[$ and %as $\chi(s,\chi(t,x))=\chi(s+t,x)$ for all $s,t,s+t \in [0,T]$ and
\item the orbit maps 
\begin{eqnarray*}
\sigma_{u_0}: [0,\infty[ &\rightarrow& X \\ %C_{0}(\mathbb{R}^n) \\
t &\mapsto& \Sigma_t(u_0)
\end{eqnarray*}
are continuous for every $u_0 \in X$. % //C_0(\mathbb{R}^n)$. 
\end{enumerate}
If (i) and (ii) are satisfied, then we call $(\Sigma_t)$ a \emph{semi-group}\index{semi-group} acting on $X$.
If  in addition property (iii) holds, we say  $(\Sigma_t)_{t\in [0,\infty[}$ is a semi-group of type $C_0$.

We briefly investigate how the solution concepts discussed in Subsections 1.1 and 1.2 fit into the picture
of semi-group theory when the coefficient $a$ is time-independent. First we return to the classical Caratheodory case: 
Let $a\in C(\mathbb{R}^n)^n$ and assume that $a$ suffices the forward uniqueness condition (\ref{forward_uniqueness_condition}).
This implies that the characteristic flow $\chi : \overline{\Omega_{\infty}} \rightarrow \mathbb{R}^n$ is continuous and $\chi(t,\cdot)$ is onto $\mathbb{R}^n$ for fixed $t\in[0,\infty[$. Furthermore we have $\chi(s,\chi(r,x)) = \chi(s+r,x)$ for all $ x \in \mathbb{R}^n$ and $r,s \in [0,T] $ with $s+r \in [0,T]$, since $a$ is time independent.

Consider the inital value problem $P u = \partial_t + \sum_{k=1}^n a_k \partial_{x_k} u =0$ with initial condition
$u(0) = u_0\in C_{0}(\mathbb{R}^n)$ (i.e. vanishes at infinity). It is easy to verify that
\begin{eqnarray*}
\Sigma_t : C_0(\mathbb{R}^n) &\rightarrow& C_0(\mathbb{R}^n)\\
           u_0 &\mapsto& \chi^{\ast} u_0 
\end{eqnarray*}
defines $C_0$ semigroup on the Banach space $C_0(\overline{\Omega_{\infty}})$:
% using the flow property of $\chi$ and the fact $\chi(t,\cdot)$ is onto $\mathbb{R}^n$ for fixed $t\in[0,T]$.\todo{check the use of phrase "onto"}
Note that $\Sigma_t$ is a bounded operator on $C_0(\mathbb{R}^n)$ for each $t\in [0,\infty[$, as $\chi(t,\mathbb{R}^n) = \mathbb{R}^n$, so
\begin{eqnarray*}
\| \Sigma_t (u_0) \|_{\infty} &=& \sup_{x\in \mathbb{R}^n} \|u_0(\chi(t,x)) \| = \sup_{x\in \mathbb{R}^n} \|u_0(\chi(t,x)) \| = \sup_{x\in \mathbb{R}^n} \|u_0(x)\| =\|u_0\|_{\infty}.
\end{eqnarray*}
We have that $\| \Sigma_t  \| = 1$ for all $t\in[0,\infty[$.
Condition (i) and (ii) follow directly from the flow properties of $\chi$. The continuity condition (iii), which is equivalent to
\begin{eqnarray*}
\lim_{t  \rightarrow 0_{+}}\| \Sigma_t(u_0) - u_0\|_{\infty} = \lim_{t  \rightarrow 0_{+}} \sup_{x\in \mathbb{R}^n} \|u_0(\chi(t,x)) - u_0(x)\| = 0,
\end{eqnarray*}
holds, since $(\chi(t,x))_{x\in\mathbb{R}^n}$ is an equicontinuous family and $u_0$ vanishes at infinity.

\begin{remark}
For a coefficient $a$ in $L^{\infty}(\mathbb{R}^n)^n$ we can also define a semi-group on $C_0(\mathbb{R}^n)$  
by $\Sigma_t(u_0):= u_0(\chi(t,x))$, where $\chi$ is the generalized Filippov flow as introduced earlier. This is due to the fact,
that the Filippov flow has almost the same properties as the Caratheodory flow. 
\end{remark}

It seems natural to understand the  solution concepts as defined by (\ref{carflowsol}) and (\ref{filflowsol})
as action of the dual semigroup $(\Sigma_t^{\ast})$ on 
\begin{eqnarray*}
 \langle u(t), \varphi \rangle_{\mathcal{D}^{'0}(\mathbb{R}^n)} = \langle \Sigma_t^{\ast} u_0, \varphi \rangle_{\mathcal{D}^{'0}(\mathbb{R}^n)} 
=\langle u_0, \Sigma_t (\varphi) \rangle_{\mathcal{D}^{'0}(\mathbb{R}^n)}. 
\end{eqnarray*}
%The dual semi-group $(\Sigma_t^{\ast})_t$ acts on
 the Banach space of finite complex Radon measures, the dual space of $C_0(\mathbb{R}^n)$ 
%(we defined the solution concepts for more general initial conditions in $\mathcal{D}^{'0}(\mathbb{R}^n)$).
%This is a common approach in classical semi group theory 
(cf. \cite[Chapter 4]{Davies:80}, \cite[Chapter 1.10]{Pazy:83} or \cite[Chapter IX.13]{Yosida:80} for the general setting).
However, in general the dual semi-group is not of class $C_0$ (cf. \cite[Example 1.31]{Davies:80}). This is only guaranteed if we start from a $C_0$ semi-group  defined on a reflexive Banach space. 
% In the above case the underlying space is the non-reflexive Banach space $C_0(\mathbb{R}^n)$), so 
% the dual semigroup $(\Sigma_t^{\ast})_{t\in[0,\infty[}$ (coming from the characteristic flow) acting on the finite complex measures 
% need not to be of class $C_0$. 
% For example: Consider the constant coefficient case $a(x)=1$, where the solution coming from the characteristic flow
% is obtained by translation of the initial condition $\langle u(t), \varphi \rangle_{\mathcal{D}^{'0}(\mathbb{R}^n)} =  \langle u_0, \varphi(t+\cdot) \rangle_{\mathcal{D}^{'0}(\mathbb{R}^n)}$. For $u_0= \delta$, we obtain that $\Sigma_t^{\ast}(\delta) = \varphi(t)$ and it is quite obvious that
% $\Sigma_t^{\ast}(\delta) = \delta(\cdot -t)$ does not converge to $\delta$ strongly as $t \rightarrow 0$. Check that
% the bounded variation $| \!\! | \!\! | \delta(\cdot -t)-\delta  | \!\! | \!\! |=2$ for all $t\in[0,\infty[$. This example was presented in .

Nevertheless the solution concepts in (\ref{carflowsol}) and (\ref{filflowsol}) still yield the semi group properties (i) and (ii) with 
weak-$\ast$ continuity replacing the strong continuity property (iii).

The situation is much easier with Hilbert spaces, of course. We conclude with an example involving a discontinuous coefficient.

%Finally we are going to consider an example, where we obtain a $C_0$ semi-group $(\Sigma_t)_{t\in[0,\infty[}$ of class $C_0$ acting on
%the Hilbert space $L^2(\mathbb{R})$ (Note: one-dimensional in time). 

\begin{example} \label{semigroup_L2_example}
Let $a\in L^{\infty}(\mathbb{R})$ such that there exist $c_0,c_1>0$ such that
$c_1 < a(x) < c_2$ almost everywhere. We want to solve the initial value problem
\begin{eqnarray*}
Pu=\partial_t u + a(x) \partial_x u =0, \ \ u(0)=u_0 \in L^2(\mathbb{R})
\end{eqnarray*}
for $u \in AC([0,T]; L^2(\mathbb{R})) \cap L^1([0,T]; H^1(\mathbb{R}))$.

Let $A(x) = \int_{0}^{x} a(y)^{-1} \m{y}$, which is Lipschitz continuous and strictly increasing (thus globally invertible) 
and observe that $\chi(t,x) = A^{-1}(t+A(x))$ defines the (forward)characteristic flow that solves
$$
\chi(t,x) = x + \int_0^{t} a(\chi(\tau,x)) \m{\tau}.
$$
%\todo{maybe remark on the measureability of the right-hand side}
Let $Q:= - a(x) \partial_x$ with domain $D(Q):=  H^1(\mathbb{R})$.
%We start with investigating the spectral properties of $Q$. The characteristic equation
The resolvent of $Q$ for $\Re(\mu) > 0$ is obtained from the equation
\begin{eqnarray*}
(-Q + \mu ) v = f, \ \ f\in L^2(\mathbb{R}).
\end{eqnarray*}
Upon division by  $a$ we deduce
\begin{eqnarray} \label{characteristic_equation}
\partial_x v +\frac{\mu}{a}{v} = \frac{f}{a}.
\end{eqnarray}

Let us first consider uniqueness: Let  $w \in H^1(\R)$ satisfy
\begin{equation}  \label{semigroup_L2_homogenous}
\partial_x w + \frac{\mu}{a} w = 0.
\end{equation}
Since $w$ is absolutely continuous we have
$$
w(x) = C \exp{\left(- 2 {\rm Re}(\mu) \int_{-\infty}^{x} \frac{1}{a(z) }  \m{z} \right)}
$$
for some constant $C$. But $w \in L^2(\mathbb{R})$ if and only if $C=0$, thus $w = 0$.
% 
% Assume $u\in L^2(\mathbb{R})$  solves the homogenous equation
% \begin{eqnarray}
% \partial_x u(x) + \frac{\mu}{a(x)} u(x) =0,
% \end{eqnarray}
% in $L^2(\mathbb{R})$, so
% \begin{eqnarray*}
% |\partial_x u(x) |^2 + \frac{2 {\rm Re}(\mu)}{a(x)} \partial_x (|u(x)|^2) + |\mu|^2 \frac{|u(x)|^2}{a(x)^2} =0 
% \end{eqnarray*}
% holds almost everywhere in $x$. The fact that $u$ is a $L^2$-solution implies $u \in H^1(\mathbb{R}) \subseteq AC(\mathbb{R})$.
% Rewriting the equation
% \begin{eqnarray*}
%   \partial_x (|u(x)|^2)  =- \frac{a(x)}{2 {\rm Re}(\mu)} |\partial_x u(x) |^2 - \frac{|\mu|^2}{{2\rm Re}(\mu)} \frac{|u(x)|^2}{a(x)}
% \end{eqnarray*}
% we observe that
% \begin{eqnarray*}
% \partial_x (|u(x)|^2)  < 0
% \end{eqnarray*}
% holds almost everywhere, since ${\rm Re}(\mu) > 0$ and $a(x)\ge c_0 >0$, and $\partial_x (|u(x)|^2 \in L^1(\mathbb{R})$.  We conclude
% that $|u(x)|^2$ can be written as
% \begin{eqnarray*}
% |u(x)|^2 = n_0 + \int_{-{\infty}}^{x}  \partial_z (|u(z)|^2)  \,d z
% \end{eqnarray*}
% where $n_0 \in\mathbb{R}$ some constant. The boundedness of $\int_{\mathbb{R}} |u(x)|^2 \,dx$ forces $n_0$ less or equal zero. On the other hand, $|u(x)|^2$ is positive, so $|u(x)|^2 =0$ for all $x\in\mathbb{R}$ ($u\in AC(\mathbb{R})$).
% 
% This implies that $u$ is the trivial solution $0$, so $u=R(\mu) f$ is the unique $L^2$ solution depending continuously on $f$. $R(\mu)= (a \partial_x + \mu)^{-1}$ for 
% ${\rm Re}(\mu)>0$.
% 

Existence:
One easily verifies that
\begin{multline*}
v(x)=(R(\mu) f)(x) := \int_{-\infty}^{x} \exp{\left(- \mu \int_{y}^x a(z)^{-1} \m{z} \right)} \frac{f(y)}{a(y)} \m{y}  
= \int_{-\infty}^{x} \exp{(- \mu (A(x)-A(y)) )} \frac{f(y)}{a(y)} \m{y}    
\end{multline*} 
is a solution of (\ref{characteristic_equation}) in $AC(\mathbb{R})$. 
Upon substitution $y \mapsto z = A(x) - A(y)$ in the right-most integral we obtain
\begin{eqnarray} \label{laplace_transform}
(R(\mu) f)(x) = \int_{0}^{\infty} \exp{(-\mu z)} f(\chi(-z,x))  \m{z},
\end{eqnarray}
which is the Laplace transform of $f(\chi(-.,x))$.

We denote the kernel of the integral operator $R(\mu)$ by
\begin{eqnarray*}
M(x,y) &:=& H(x-y) \exp{(- \mu (A(x)-A(y))} a(y)^{-1}.
\end{eqnarray*} 
%It follows that $R(\mu): L^2(\mathbb{R}) \rightarrow L^2(\mathbb{R})$ is a continuous operator by applying Schur's Lemma.
%\todo{Schur's Lemma, H\"ormander 3}
%We check that the integral kernel $M$ fulfills
%\begin{eqnarray*}
%\sup_{x\in\mathbb{R}} \int_{\mathbb{R}}|M(x,y)| \,dy &=& \sup_{x\in\mathbb{R}}  \int_{\mathbb{R}} H(x-y) \exp{(- {\rm %Re}(\mu) (A(x)-A(y)))} a(y)^{-1} \,dy \\%
%&=& \int_{-\infty}^{x} \exp{(- \mu (A(x)-A(y)) )} \frac{1}{a(y)} dy  \\
%&=& \frac{1}{{\rm Re}(\mu)} \int_{-\infty}^{0} \exp{(z)}  dz  
%\end{eqnarray*}
We briefly sketch the derivation of $L^2$ estimates for the operator powers $R(\mu)^k$ for ${\rm Re}(\mu)>0$:
Note that $R(\mu)^k$ is an iterated integral operator of the form
\begin{multline*}
 R(\mu)^k f (x) :=R(\mu)^{k-1} \left(\int_{\mathbb{R}} M(\cdot,z_1) f(z_1) \m{z_1} \right) (x)\\= 
 \int_{\mathbb{R}} \cdots  \int_{\mathbb{R}} M(x, z_{k}) M(z_k, z_{k-1}) \ldots M(z_2,z_1) f(z_1) \m{z_{k-1}} \ldots \m{z_2} \m{z_1}. 
\end{multline*}
To simplify notation let $z=(z_1, \dots, z_k)$, ${\mathrm d}^k z = \m{z_1} \dots \m{z_k}$, $h(z):=\exp{ (-\mu \sum_{l=1}^k z_l)}$, and 
$g(x,z): = f(\chi(-\sum_{l=1}^k z_l, x))$. Using the flow property of $\chi$ we obtain that
\begin{eqnarray*}
R(\mu)^k f (x) = \int_{[0, \infty[^k} h(z) g(z,x) \,{\mathrm d}^k z 
\end{eqnarray*}
holds, hence by by the integral Minkowski inequality
\begin{multline} \label{semigroup_L2_existence}
  \|R(\mu)^k f \|_{L^2} \leq \left( 
  \int_{\mathbb{R}}\left(\int_{[0, \infty[^k} |h(z)| |g(z,x)| \,{\mathrm d}^k z \right)^2  \m{x} \right)^{1/2}\\
  \leq 
 \int_{[0, \infty[^k}  |h(z)|     
\left(\int_{\mathbb{R}}  |g(z,x)|^2  \m{x} \right)^{1/2} \,{\mathrm d}^k z.
\end{multline}
Since
$$
  \int_{\mathbb{R}}  |g(z,x)|^2 \m{x}  
  = \int_{\mathbb{R}} |f(\chi (-\sum_{l=1}^k z_l, x) )|^2  \m{x}
    = \int_{\mathbb{R}} |f(y)|^2  
    \left|\frac{a( y) }{a(\chi (\sum_{l=1}^k z_l, y ))}\right| \m{y} 
\le 
\frac{c_1}{c_0} \, \| f\|_{L^2}^2,
$$
we conclude
$$
  \| R(\mu)^k f \|_{L^2} \le  
  \sqrt{\frac{c_1}{c_0}}\cdot \| f\|_{L^2} \cdot
  \int_{]-\infty,0]^k}  |h(z)|    
 \, {\mathrm d}^k z = 
   \frac{\sqrt{\frac{c_1}{c_0}}}{{\rm Re(\mu)}^{k}}\cdot \| f\|_{L^2} . 
%\frac{c_1}{c_0} \| f\|_{L^2}^2 \left(\int_{]-\infty,0]^k} h(z)  d^k\,z\right)^2 =
%\frac{c_1}{c_0} \| f\|_{L^2}^2 \left(\int_{]-\infty,0]^k} \exp{({\rm Re}(\mu ) \sum_{l=1}^k z_l)}  d^k\,z\right)^2 \\
%\le
%\frac{c_1}{c_0} \| f\|_{L^2}^2 \left(\int_{]-\infty,0]^k} \exp{({\rm Re(\mu )} y)}  \,d y\right)^{2k} 
%=
$$
The  Hille-Yosida theorem 
(\cite[Theorem 5.2]{Pazy:83})
yields that $Q$ generates the $C_0$ semigroup 
\begin{eqnarray*}
 \Sigma_t : L^2(\mathbb{R}) &\rightarrow& L^2(\mathbb{R})\\
     	u_0 &\mapsto& u_0(\chi(-t,x)).
\end{eqnarray*}
The resolvent operator $\mu \mapsto R(\mu)$ (defined for ${\rm Re}(\mu)> 0$) is the Laplace transform of the semigroup $t \rightarrow  \Sigma_t$
as indicated in (\ref{laplace_transform}). 

\end{example}

\begin{remark}
Since $L^2(\mathbb{R}^n)$ is reflexive the dual semigroup is $C_0$ as well and has as its generator the adjoint operator $Q^*$.
\end{remark}

\begin{remark}
If we assume additional regularity on the coefficient, e.g.\ $a\in C^{\sigma}_{\ast}(\mathbb{R})$ with $\sigma>0$,  in Example \ref{semigroup_L2_example}, then we obtain a $C_0$ semigroup $(\Sigma_t)$ acting on the
Hilbert space $H^s(\mathbb{R})$ with $0 \leq s < \sigma$. We may then use the fact that the (square of the) Sobolev norm  $\norm{v}{s}^2$  is equivalent to the following expression  (cf.\ \cite[Equation (7.9.4)]{Hoermander:V1})
$$
   \int |v(x)|^2 \m{x} + C_s \int \int \frac{|v(x) - v(y)|^2}{|x-y|^{n + 2 s}} \m{x} \m{y},
$$
where the constant $C_s$ depends only on the dimension $n$ and $s$. 
From this it can be shown that we may have $D(Q)=H^{s+1}(\mathbb{R})$ as domain of $Q$ (this also corresponds to the special case of the mapping properties stated in \cite[Chapter 2]{Taylor:91}).
Clearly, uniqueness in the characteristic equation (\ref{semigroup_L2_homogenous}) is still valid. A corresponding variant of the estimate (\ref{semigroup_L2_existence}) for the powers of the resolvent operator $R(\mu)^k$ on $H^s(\mathbb{R})$ is obtained by  the following calculation (with the notation $h$ and $g$ as in Example \ref{semigroup_L2_example}):
\begin{multline*}
 \|R(\mu)^k f \|_{s}^2 = \|R(\mu)^k f \|_{0}^2 +   
  \int_{\mathbb{R}} \int_\mathbb{R} \frac{|(R(\mu)^k f)(x)-(R(\mu)^k f)(y)|^2}{|x-y|^{1+2s}} \m{x} \m{y} \\
  \le   
 \frac{\frac{c_1}{c_0}}{{\rm Re(\mu)}^{2k}}\cdot \| f\|_{0}^2  +
%\left( 
\int_{\mathbb{R}} \int_\mathbb{R} \left(\, \int_{[0,\infty[^k} |h(z)| \frac{|g(x,z) - g(y,z)|}{|x-y|^{\frac{1+2s}{2}}} \m{z} \right)^2  \m{x} \m{y} 
%\right)^{1/2} 
\\
  \le \frac{\frac{c_1}{c_0}}{{\rm Re(\mu)}^{2k}}\cdot \| f\|_{0}^2
    + \left( \, 
\int_{[0,\infty[^k}  |h(z)| \left( \int_{\mathbb{R}} \int_\mathbb{R}  \frac{|g(x,z) - g(y,z)|^2}{|x-y|^{1+2s}} \m{x} \m{y}   \right)^{1/2}   \,{\mathrm d}^k z \right)^2.
\end{multline*}
To carry out the $x$ and $y$ integrations we use the substitutions $x' = \chi(-\sum_{l=1}^k z_l, x)$, $y' = \chi(-\sum_{l=1}^k z_l, y)$ to obtain
\begin{multline*}
\int_{\mathbb{R}} \int_\mathbb{R}  \frac{|g(x,z) - g(y,z)|^2}{|x-y|^{1+2s}} \m{x} \m{y}
 =
\int_{\mathbb{R}} \int_\mathbb{R}  \frac{| f( \chi(-\sum_{l=1}^k z_l, x) )-   f(\chi(-\sum_{l=1}^k z_l, y) )|^2}{|x-y|^{1+2s}} \m{x} \m{y} \\
 =
\int_{\mathbb{R}} \int_\mathbb{R}  \frac{| f( x') -  f(y')|^2}{|\chi(\sum_{l=1}^k z_l, x')) -\chi(\sum_{l=1}^k z_l, y') |^{1+2s}} \left|\frac{a(\chi(\sum_{l=1}^k z_l, x'))}{a(x')}\right| \left|\frac{a(\chi(\sum_{l=1}^k z_l, y'))}{a(y')}\right| \m{x'} \m{y'}. 
\end{multline*}
Now,  by the mean value theorem we have $|\chi(\sum_{l=1}^k z_l, x)) -\chi(\sum_{l=1}^k z_l, y)| \ge \frac{c_0}{c_1} |x-y|$ 
and the assumed bounds for $a$ give $\left|\frac{a(\chi(\sum_{l=1}^k z_l, \cdot))}{a(\cdot)}\right|\le \frac{c_1}{c_0}$, 
thus we arrive at
$$
\int_{\mathbb{R}} \int_\mathbb{R}  
  \frac{|g(x,z) - g(y,z)|^2}{|x-y|^{1+2s}} \m{x} \m{y}
 \leq
  \left(\frac{c_1}{c_0}\right)^{3+2s} 
  \int_{\mathbb{R}} \int_\mathbb{R}  \frac{| f( x) -  f(y)|^2}{|x-y|^{1+2s}} \m{x} \m{y}
%  \leq \left(\frac{c_1}{c_0}\right)^{3+2s} \norm{f}{s}^2.
$$
Again by $\int_{[0,\infty[^k}  |h(z)| \,{\mathrm d}^k z = \frac{1}{{\rm Re}(\mu)^k}$  we conclude 
$$
 \|R(\mu)^k f \|_{s}^2 \le \frac{\frac{c_1}{c_0}}{{\rm Re(\mu)}^{2 k}} \cdot \| f\|_{0}^2   +
\frac{\left(\frac{c_1}{c_0}\right)^{3+2s}  }{{\rm Re(\mu)}^{2 k}} \int_{\mathbb{R}} \int_\mathbb{R}  \frac{| f( x') -  f(y')|^2}{|x'-y'|^{1+2s}} \m{x'}  \m{y'} 
\le
  \frac{\left(\frac{c_1}{c_0}\right)^{{3+2s}}  }{{\rm Re(\mu)}^{2 k}} 
    \|f\|_s^2,
$$
i.e., $\displaystyle{\|R(\mu)^k f \|_{s} \leq \frac{\left(\frac{c_1}{c_0}\right)^{(3+2s)/2}  }{{\rm Re(\mu)}^{k}}} \|f\|_s$.
\end{remark}

%\todo{delete after here - }
%hence only the second term needs consideration. Using the notation from the previous example, we note
%\begin{multline*}
%\left( \int_{\mathbb{R}} \int_\mathbb{R} \frac{|(R(\mu)^k f)(x)-(R(\mu)^k f)(y)|^2}{|x-y|^{1+2s}} \,dx \,dy \right)^{1/2} \\
%=
%\left( \int_{\mathbb{R}} \int_\mathbb{R} \frac{ \left| \int_{]-\infty,0]^k} (h(z) g(z,x) - h(z) g(z,y)) \,d^k z  \right|^2 %}{|x-y|^{1+2 s}} \,dx\, dy\right)^{1/2} \\
%\le  \left( \int_{]-\infty,0]^k} |h(z)|^2 \left(\int_{\mathbb{R}} \int_\mathbb{R}  \frac{|g(z,x) - g(z,y)|^2  }{|x-y|^{1+2 %s}} \,dx\, dy \right) \,d^k z \right)^{1/2} \\
%\end{multline*}
%and using
%\begin{eqnarray*}
%\int_\mathbb{R}\int_\mathbb{R}  \frac{|f(\chi(z, x)) - f(\chi(z, y))|^2  }{|x-y|^{1+2 s}} \,dx\, dy  \\
%=
%\int_\mathbb{R}\int_\mathbb{R}  \frac{|f(x') - f(y')|^2  }{|\chi(-z,x')-\chi(-z,y')|^{1+2 s}} \frac{a(x')}{a(\chi(-z,x'))} 
%\frac{a(y')}{a(\chi(-z,y'))} \,dx'\, dy'  \\
%\le
% \left(\frac{c_1}{c_0}\right)^{1-2 s} \int_\mathbb{R}\int_\mathbb{R}  \frac{|f(x') - f(y')|^2  }{|x'-y'|^{1+2 s}}  \,dx'\, dy' 
%\end{eqnarray*}

%%%%%%%%%%%%%%%%%%%%%%%%%%%%%%%%%%%%%%%%%%%%%%%%

\subsection{Measurable coefficients with prescribed characteristics}

This subsection discusses a solution concept according to Bouchut-James (\cite{BochutJames:98}), which is settled in one space dimension and --- from the distribution theoretic point of view --- can be considered as exotic.
The basic idea is to interpret the multiplication $a \cdot u$ occurring in the partial differential equation as a product of a (locally finite) Borel measure $u$ and a function $a$ from the set $\mathcal{B}^{\infty}$ of real bounded and Borel measurable functions.
%\begin{eqnarray*}
%$ \mathcal{B}^{\infty}(\mathbb{R}) := \{f:\mathbb{R} \mapsto \mathbb{R}\mid f %\text{ is } \}$.
% \\
%\mathcal{M}(\mathbb{R})&:=& \text{ the set of locally finite Borel measures on %} \mathbb{R}.
%$
%\end{eqnarray*}

\paragraph{Multiplication of Radon measures by bounded Borel functions:}
We may identify locally finite Borel measures on $\mathbb{R}$ with (positive) Radon-measures, that is the non-negative linear functionals on $C_c(\mathbb{R})$ (\cite[Remark 19.49]{HewStr:65}). Moreover, the space $\mathcal{D}^{'0}(\mathbb{R})$ is the space of complex Radon-measures, which allows for a decomposition of any
$u\in \mathcal{D}^{'0}(\mathbb{R})$ in the form $u= \nu_{+} -\nu_{-} + i (\eta_{+} - \eta_{-})$, where $\nu_{+},\nu_{-},\eta_{+},\eta_{+}$ are positive Radon-measures.

The product of a bounded Borel function $a\in  \mathcal{B}^{\infty}(\mathbb{R})$  with a positive Radon measure $\mu$ is defined to be the
measure given by
\begin{eqnarray*}
(a\odot \mu)(B) := \int_{\mathbb{R}^n} 1_B(x) a(x) \m{\mu(x)} , \ \
\end{eqnarray*} 
for all Borel sets $B$ in $\mathbb{R}$. Clearly, $a\odot \mu$ is again a locally finite Borel measure.
% /We note that any distribution of order zero can be represented by two Radon measures by
% \begin{eqnarray*}
% \langle u, \varphi \rangle_{\mathcal{D}^{'0}(\mathbb{R})} = \langle u_{+}, \varphi \rangle_{\mathcal{D}^{'0}(\mathbb{R})} - \langle u_{-}, \varphi \rangle_{\mathcal{D}^{'0}(\mathbb{R})} 
% \end{eqnarray*}
% where $u_{\pm}$ are the positive distributions of order zero defined by 
% \begin{eqnarray*}
% \langle u_{+}, \varphi \rangle_{\mathcal{D}^{'0}( \mathbb{R} )} &:=& \langle u, |\varphi|_{+}  \rangle_{\mathcal{D}^{'0}( \mathbb{R})}\\
% \langle u_{-}, \varphi \rangle_{\mathcal{D}^{'0}( \mathbb{R} )} &:=& \langle u, |\varphi|_{-}  \rangle_{\mathcal{D}^{'0}( \mathbb{R})},
% \end{eqnarray*}
% where $|\varphi(x)|_{+} = \max{(0, \varphi(x))}$ and $|\varphi(x)|_{-} = \max{(0, -\varphi(x))}$. This decomposition is unique modulo some diagonal element
% ${\rm Diag (\mathcal{M}(\mathbb{R}))} := \{(\mu_0,\mu_0) \mid \mu_0 \in \mathcal{M}(\mathbb{R}) \}$. We have the following isomorphism
% \begin{eqnarray*}
% \mathcal{I}:  (\mathcal{M}(\mathbb{R}) \times \mathcal{M}(\mathbb{R})) / {\rm Diag (\mathcal{M}(\mathbb{R}))} &\rightarrow& \mathcal{D}^{'0}(\mathbb{R}) \\ 
% 	(\mu_{+},\mu_{-}) + [ (\mu_0, \mu_0) ] &\mapsto& (\varphi\mapsto (\langle u_{+}, \varphi \rangle_{\mathcal{D}^{'0}(\mathbb{R})} - \langle u_{-}, \varphi \rangle_{\mathcal{D}^{'0}(\mathbb{R})}) )
% \end{eqnarray*}
% of vector-spaces. 

The product employed in \cite{BochutJames:98} is the extension of $\odot$ to $ \mathcal{B}^{\infty}(\mathbb{R}) \times \mathcal{D}^{'0}(\mathbb{R})$ in a bilinear way, i.e.
\begin{eqnarray*}
\diamond: \mathcal{B}^{\infty}(\mathbb{R}) \times \mathcal{D}^{'0}(\mathbb{R}) &\rightarrow& \mathcal{D}^{'0}(\mathbb{R})\\
		(a, u) &\mapsto&   
a_{+}\odot \nu_{+} + a_{-}\odot \nu_{-} - ( a_{-}\odot \nu_{+}+a_{+}\odot \nu_{-}) \\
&&+i (a_{+}\odot\eta_{+} + a_{-}\odot \eta_{-} - ( a_{-}\odot \eta_{+}+a_{+}\odot \eta_{-}) ).
\end{eqnarray*}

% Note that this is not a distributional product in the sense  of(\todo{cite appendix, general properties of a product}), as 
% $\diamond$ is not defined on a subset (in the sense that there exist no injective inclusion map) of $\mathcal{D}'(\mathbb{R}) \times \mathcal{D}'(\mathbb{R})$. 
Consider the following sequence of maps:
$$
C_b(\mathbb{R}) \stackrel{\iota_1}{\hookrightarrow} \mathcal{B}^{\infty}(\mathbb{R}) \stackrel{\lambda}{\rightarrow} L^{\infty}_{\rm loc}(\mathbb{R}) \stackrel{\iota_2}{\hookrightarrow} \mathcal{D}^{(0)\prime}(\mathbb{R})
$$
where $\iota_1, \iota_2$ are the standard embeddings and $\lambda$ sends bounded Borel functions to the corresponding classes modulo functions vanishing almost everywhere in
the sense of the Lebesgue measure.
% 
% and we denote $\iota=  \iota_2 \circ \lambda \circ \iota_1$. Then $\iota_1, \iota_2$ and $\iota$ are (injective) inclusion maps, where $\iota_1$ is trivial, since $C_b(\mathbb{R})$ is contained in $\mathcal{B}^{\infty}(\mathbb{R})$ and $\iota_2 (f):= f+ [0]$, where $[0]$ is the class of all Lebesgue-measurable functions that are $0$  almost everywhere.
% 
% Note that since $\iota_1, \iota_2$ are injective and ${\rm Im}(\iota_1) \cap {\rm ker}(\lambda)=\{0\}$ it follows that $\iota$ is injective.
Although we may identify $C_b(\mathbb{R})$ and $L^{\infty}_{\rm loc}$ with subspaces of $\mathcal{D}^{\prime 0}(\mathbb{R})$ this is not true of $\mathcal{B}^{\infty}(\mathbb{R})$, since $\lambda$ is not injective. Note that $\iota_2 \circ \lambda \circ \iota_1$ is injective though. The following example illustrates some consequences of the non-injectivity of the map $\lambda$ for the properties of the product $\diamond$.

\begin{example}
Let $\alpha\in\mathbb{R}$ and $a_{\alpha}(x):= 1,  x\ne 0$ and $a_{\alpha}(0):=\alpha$ and $u=\delta \in \mathcal{D}^{\prime0}(\mathbb{R})$.
 Note that $\lambda\circ \iota_1 (a_\alpha) =1$ as a distribution and the standard distributional product gives $\lambda\circ \iota_1(a_{\alpha}) \cdot u = 1 \cdot u = \delta$ for all $\alpha\in\mathbb{R}$. On the other hand $a_{\alpha}\diamond \delta = \alpha \delta$. 
\end{example}

%\label{diamond_product_remark}
The product $\diamond$ will be used in the solution concept for transport equations on $\overline{\Omega_T} = [0,T] \times \mathbb{R}$ with coefficient $a$
in $\mathcal{B}^{\infty}(\ovl{\Omega_T})$ and solution $u\in \mathcal{B}^{\infty}([0,T];\mathcal{D}^{'0}(\mathbb{R}))$, i.e., $u$ is a family of distributions $(u(t))_{t\in[0,T]}$ such
that $\langle u(t), \varphi\rangle_{\mathcal{D}^{'0}(\mathbb{R})}$ is a bounded Borel function on $[0,T]$ for all $\varphi \in C_c(\mathbb{R})$. The
extension of the product $\diamond$ to this space causes no difficulty.

\paragraph{The solution concept according to Bouchut and James:} A key ingredient for the solution concept according to Bouchut-James, is to stick to a particular representative of the coefficient (in the $L^{\infty}$ sense), by prescribing the value of the coefficient $a$ at curves of discontinuity. We refer to the following requirements on the coefficient $a\in \mathcal{B}^{\infty}(\Omega_T)$ as \emph{Bouchut-James conditions}
\index{Bouchut-James conditions}: Assume there exists a decomposition $\Omega_T=\mathcal{C} \cup \mathcal{D} \cup \mathcal{S}$ such that
\begin{enumerate} \label{bochut_james_conditions}
\item $\mathcal{S}$ is a discrete subset $\Omega_T$,
\item $\mathcal{C}$ is open, $a$ is continuous on $\mathcal{C}$,
\item $\mathcal{D}$ is a one-dimensional $C^1$-submanifold  of ${\Omega_T}$,
i.e., for each $(t_0,x_0) \in \mathcal{D}$ there exists a neighborhood $V$ of $(t_0,x_0)$ and a $C^1$ parametrization of the form $t \mapsto (t,\xi(t))$ in $\D \cap V$.
Furthermore, $a$ has limit values for each $(t,x) \in \mathcal{D}$ from both sides in $\mathcal{C} \setminus \mathcal{D}$. These limits are denoted by $a_{+}(t,x)$ and $a_{-}(t,x)$.
\item $a(t,x) \in [a_{-}(t,x), a_{+}(t,x)]$ for all $(t,x) \in \mathcal{D}$,
%\item if $(dt)_{\mathcal{D}} \ne 0$ and $(\frac{dx}{dt})_{\mathcal{D}}(t_0,x_0) \in [a_{-}(t,x), a_{+}(t,x)]$ then $a(t_0,x_0)=(\frac{dx}{dt})_{\mathcal{D}}(t_0,x_0)$ 
\item for any point $(t_0,x_0) \in \mathcal{D}$ with neighborhood $V$ and local parametrization $\xi$ as in (iii), we have 
 $\dot{\xi}(t) = a(t,\xi(t)) $.
\end{enumerate}
%if $(dt)_{\mathcal{D}} \ne 0$ and $(\frac{dx}{dt})_{\mathcal{D}}(t_0,x_0) \in [a_{-}(t,x), a_{+}(t,x)]$ then $a(t_0,x_0)=(\frac{dx}{dt})_{\mathcal{D}}(t_0,x_0)$ 
%\\end{enumerate}
Condition (v) prescribes the values of the coefficient $a(t,x)$ on the curves of discontinuity in such a way that
the characteristic differential equation holds. In this sense,  a coefficient satisfying $(i)-(v)$ is a piecewise continuous bounded function, where the (non-intersecting) curves of discontinuity can be parametrized as regular $C^1$ curves.

The Bouchut-James solution concept interprets hyperbolic Cauchy problems in $(1+1)$ dimension as
\begin{eqnarray} \label{pdo_P_bochut_james}
P u :=\partial_t u +  a \diamond \partial_{x}  u = 0, \ \ u(0)=u_0 \in {BV}_{loc}(\mathbb{R})  
\end{eqnarray}
and 
\begin{eqnarray} \label{pdo_L_bochut_james}
L u:= \partial_t u + \partial_{x} (a \diamond u)=0, \ \ u(0)=u_0 \in \mathcal{D}'^{ 0}(\mathbb{R}).
\end{eqnarray}

Note that $P$ (resp.\ L) is well-defined on the set $\mathcal{B}^{\infty}([0,T];{BV}_{loc}(\mathbb{R}))$ (resp.\ $\mathcal{B}^{\infty}([0,T]; \mathcal{D}^{\prime 0}(\mathbb{R}))$).

Now the main results of Bouchut-James \cite{BochutJames:98} are:

\begin{theorem}\cite[Theorem 3.4]{BochutJames:98}\label{BJexistence}
Assume that $a$ satisfies the Bouchut-James conditions (i)-(v).
For any $u_0\in BV_{\rm loc}(\mathbb{R})$ there exists $u\in {\rm Lip}([0,T]; L^1_{\rm loc}(\mathbb{R})) \cap \mathcal{B}^{\infty}([0,T]; BV_{\rm loc}(\mathbb{R}))$
solving (\ref{pdo_P_bochut_james}) and such that for any $x_1<x_2$ we have for all $t\in[0,T]$
\begin{eqnarray*}
{\rm Var}_I (u(t,\cdot)) &\le& {\rm Var}_J(u_0),\\
\|u(t,\cdot)\|_{L^{\infty}(I)} &\le& \|u_0\|_{L^{\infty}(J)}
\end{eqnarray*}
where $I:=]x_1,x_2[$ and $J:=]x_1-\|a\|_{\infty} t, x_2 + \|a\|_{\infty} t [$. If in addition the coefficient $a$ 
satisfies the one-sided Lipschitz condition
\begin{eqnarray*}
\langle a(t,x) -a(t,y), x-y \rangle \le \alpha(t) |x-y|^2 \ \ \text{for almost all} \ (t,x), (t,y) \in \Omega_T,
\end{eqnarray*}
where $\alpha \in L^1([0,T])$, then the solution $u$ is unique.
\end{theorem}

\begin{theorem}\cite[Theorem 3.6]{BochutJames:98} 
\label{BJexistence}
Assume that $a$ satisfies the Bouchut-James conditions (i)-(v).
Then it follows that for any $u_0\in \mathcal{D}^{'0}(\mathbb{R})$ there exists $u \in C([0,T]; \mathcal{D}^{'0}(\mathbb{R}))$
solving (\ref{pdo_L_bochut_james}). %and such that for any $x_1<x_2$ and $t\in[0,T]$
% \begin{eqnarray*}
% \int_I |\,d \mu_t(x)| &\le& \int_J |\,d \mu_0(x)|,\\
% \end{eqnarray*}
% where $I:=]x_1,x_2[$ and $J:=]x_1-\|a\|_{\infty} t, x_2+ \|a\|_{\infty} t [$ for any $x_1,x_2\in\mathbb{R}$ with $x_1 < x_2$. 
%If  \begin{eqnarray*}
% \int_I |\,d \mu_t(x)| &\le& \int_J |\,d \mu_0(x)|,\\
% \end{eqnarray*}
% where $I:=]x_1,x_2[$ and $J:=]x_1-\|a\|_{\infty} t, x_2+ \|a\|_{\infty} t [$ for any $x_1,x_2\in\mathbb{R}$ with $x_1 < x_2$. Iin addition the coeffcient $a$ 
If $a$ satisfies in addition the one-sided Lipschitz condition
\begin{eqnarray*}
\langle a(t,x) -a(t,y), x-y \rangle \le \alpha(t) |x-y|^2 \ \ \text{for almost all} \ (t,x), (t,y) \in \Omega_T,
\end{eqnarray*}
where $\alpha \in L^1([0,T])$, then the solution $u$ is unique.
\end{theorem}

We compare the solution concept of Bouchut-James with the generalized solutions according to Poupaud-Rascle.
\begin{example} \label{BJ_jump} 
We come back to Example \ref{poupaud_rascle_jump}, where $a(t,x):=c_1 H(\alpha t- x) + c_2 H(x- \alpha t)$ with $c_2 < c_1$ and $\alpha\in [c_2,c_1]$.
Let $\lambda \in L^1([0,T])$ such that $c_1 \ge \lambda(t) \ge c_2$. 
Consider a representative $a$ in $\mathcal{B}^{\infty}(\overline{\Omega_T})$ of the coefficient $a$ given by
\begin{eqnarray*}
\widetilde{a}(t,x) =\left\{ 
\begin{array}{ll}
c_1 & x< \alpha t \\
\lambda(t) & x=\alpha t \\
c_2 &  x > \alpha t
\end{array}\right. .
\end{eqnarray*}

We investigate wether the distribution $u$ given in Example \ref{poupaud_rascle_jump} solves (\ref{pdo_L_bochut_james}) in the sense of Bouchut-James.

Let us consider the case $u_0\equiv 1$. Then we obtain the solution $u= 1+t (c_1-c_2) \delta(x-\alpha t)$.
Note that the requirement that $\widetilde{a}$ fulfills the Bouchut-James conditions (i)-(v) forces $\lambda(t) := \alpha$ for all $t\in [0,T]$. Thus we have 
$$
  \widetilde{a} \diamond v = 
   c_1  H(\alpha t-x)  + \alpha \delta(x-\alpha t)  t ( c_1 -c_2) + 
   c_2 H(x - \alpha t) 
$$
and therefore
$\partial_x (\widetilde{a} \diamond v) =   \alpha  t (c_1-c_2) \delta'(x-\alpha t) - (c_1-c_2) \delta(x-\alpha t)$. Since
$ \partial_t v =  (c_1-c_2) \delta(x-\alpha t) - \alpha t (c_1-c_2) \delta'(x-\alpha t)$ we deduce that
%\alpha \delta(\alpha t - x) - (c_2-c_1) \delta(x-\alpha t ) + \alpha t (c_2-c_1) \delta'(x-\alpha t) - \alpha \delta(x-\alpha t)\\
%&=&- (c_2-c_1) \delta(x-\alpha t ) + \alpha t (c_2-c_1) \delta'(x-\alpha t),
$u$ solves the Cauchy problem in the sense of Bouchut-James.

Finally, we check whether the differential equation is fulfilled, if we employ the model product (cf.\ \cite[Chapter 7]{O:92} or the introduction) instead of $\diamond$. 
Let $[a\cdot v]$ denote the model product. We have that
\begin{multline*}
  [\widetilde{a} \cdot v] = 
  [(c_1  H(\alpha t -x)+ c_2  H(x-\alpha t)) \cdot u] \\   
  = c_1 H(\alpha t -x) + c_2 H(x-\alpha t) + 
  c_1 t (c_1 -c_2) [H(\alpha t -x)   \cdot \delta(x-\alpha t)] \\
  = c_2 t (c_1 -c_2) [H(x-\alpha t) \cdot \delta(x-\alpha t)] 
+ c_1 H(\alpha t -x) + c_2 H(x-\alpha t) +  \frac{t}{2} (c_1+c_2) (c_1 -c_2)  \delta(x-\alpha t), 
\end{multline*}
hence
\begin{eqnarray*}
\partial_x [a \cdot v] &=& (c_2-c_1) \delta(x-\alpha t) - \frac{t}{2} (c_1+c_2) (c_1 -c_2)  \delta'(x-\alpha t),\\
\partial_t v &=&   (c_1-c_2) \delta(x-\alpha t) - \alpha t (c_1-c_2) \delta'(x-\alpha t).
\end{eqnarray*}
Therefore $v$ solves the initial value problem
\begin{eqnarray*}
\partial_t v + \partial_x [a\cdot v ] =0,\ \ v(0)=1,
\end{eqnarray*}
if $\alpha=\frac{1}{2}(c_1+c_2)$.

Note that the coefficients $H(-x)$ (with $c_1=1,c_2=0, \alpha=0$), $-H(x)$ (with $c_1=0,c_2=-1, \alpha=0$), and $-{\rm sign}(x)$ (when  $c_1=1,c_2=-1, \alpha=0$) are included as special cases of the example
presented here. In case the coefficient reads
$$
  a(x):= H(-x)
$$ 
the unique solution in the sense of Bouchut-James is given by 
$$ 
   u = 1+t \delta.
$$
It has been shown in \cite[Theorem 5]{HdH:01} that no distributional solution  exists in this case when the model product is employed.
\end{example}

\section{Solutions from energy estimates}

\subsection{Direct energy estimates}
We briefly review the standard techniques of energy estimates for the initial value problem 
\begin{equation} \label{initial}
\begin{split}
P u &:= \partial_t u +\sum_{j=1}^n a_j \,\partial_{x_j} u + c\, u = f \qquad \text{in } ]0,T[ \times \R^n, \\ 
%\in L^{2}([0,T], H^{s}(\mathbb{R}^n)) \\
u(0)&= u_0 \in L^2(\R^n). 
% \in H^{s}(\mathbb{R}^n),
\end{split}
\end{equation}
%where for any $\tau \in \R$ the map $\iota_\tau \col \R^n \to \R^{n+1}$ denotes the embedding $x \mapsto (\tau, x)$. 

%\begin{proposition} \label{propenergyest}
Let $q \in [2,\infty]$. We assume that $f \in L^1([0,T];L^2(\R^n))$, 
$a = (a_1, ..., a_n) \in L^1([0,T]; W^{1,q}(\mathbb{R}^n))^n$ with real components,
$c \in  L^1([0,T]; L^q(\mathbb{R}^n))$
and in addition
\begin{equation} \label{div_cond} 
  \frac{1}{2} \, {\rm div}_x (a) -c \quad \in 
    L^1([0,T];L^{\infty}(\mathbb{R}^n)).
\end{equation}

\subsubsection{Example derivation of an energy estimate} \index{energy estimate}

We browse through the typical steps that lead to an estimate in the norm of $L^{\infty}([0,T]; L^2(\mathbb{R}^n))$ for any 
\begin{eqnarray*}
u\in AC([0,T]; L^2(\mathbb{R}^n)) \cap L^{\infty}([0,T]; W^{1,p}(\mathbb{R}^n))
\end{eqnarray*}
with $p \in [2,\infty]$ such that $\frac{1}{q} + \frac{1}{p} = \frac{1}{2}$ in terms of corresponding norms for $u(0)$ and $Pu$.
%\\ \todo{$s=1$, weil $W^{r_1,q} \cdot W^{s-1,p} \subseteq W^{0,2}$ braucht, dass $s=1$ und $1/p + 1/q = 1/2$ ist.}

%\end{proposition}

%\begin{proof}

We write $P = \d_t + Q$ with $Q:=\sum_{k=1}^n a_k(t,x) \partial_{x_k} + c(t,x)$ and observe that 
$$
  P u \in  L^{1}([0,T];L^2(\mathbb{R}^n))
$$ 
holds since $\partial_t u \in L^{1}([0,T];L^2(\mathbb{R}^n))$ and
$Qu \in L^{1}([0,T];L^2(\mathbb{R}^n))$ (the latter follows from the facts that $\d_{x_j} u(t,.) \in L^2$ and $L^q \cdot L^p \subseteq L^2$ when $2/p + 2/q = 1$).
%\\ \todo{hier gen\"ugt $L^q \cdot L^p \subseteq L^2 \Leftrightarrow 2/p + 2/q = 1$}
Hence
$r \mapsto {\rm Re}(\langle (P u)(r) , u(r) \rangle_0)$ is defined and in $L^1([0,T])$. Furthermore, the map $t \mapsto \| u(t,\cdot)\|_0$ is continuous. 

We put 
$$
  h(r):= \|\frac{1}{2} \div_x (a(r,\cdot)) - 
    c(r,\cdot)\|_{\infty} \quad \text{ and } \quad
   \lambda(r) := 2 \int_0^r h(s) \m{s} \geq 0 
   \qquad (r \in [0,T]).
$$ 
By assumption, $h \in L^1([0,T])$ and $\lambda \in AC([0,T])$. 

The standard integration by parts argument gives the G\r{a}rding-type inequality
\begin{equation} \label{Garding}
   \frac{1}{2} (\langle Qu(\tau),u(\tau) \rangle_0 + 
   \langle u(\tau), Qu(\tau) \rangle_0 ) =  
   {\rm Re}(\langle Qu(\tau), u(\tau) \rangle_0)  
   \ge - h(\tau) \| u(\tau) \|_0^2,
\end{equation}
and thus
\begin{eqnarray*}
  \lefteqn{\int_0^\tau e^{-\lambda(r)} 
     {\rm Re}(\langle (P u)(r) , u(r) \rangle_0)\, \m{r}} \\
   &=& \frac{1}{2} \int_0^\tau e^{-\lambda(r)} 
       \diff{r} \| u(r) \|_0^2 \m{r} + 
     \int_0^\tau e^{-\lambda(r)} 
     {\rm Re} \langle  (Q u)(r), u(r) \rangle_0 \m{r} \\
   &\ge& \frac{1}{2} e^{-\lambda(\tau)} 
     \| u(\tau) \|_0^2 - \frac{1}{2} \| u(0) \|_0^2 - 
     \int_0^\tau  \underbrace{\left(h(r) - 
     \frac{\dot{\lambda}(r)}{2}\right)}_{=0} 
     e^{-\lambda(r)} \| u(r) \|_0^2 \m{r}.
\end{eqnarray*}
Therefore
\begin{eqnarray*}
  e^{-\lambda(\tau)} \|  u(\tau) \|_0^2 &\le&  
  \| u(0) \|_0^2  + 
    2 \int_0^\tau e^{-\lambda(r)} \| (P u)(r) \|_0  
    \| u(r) \|_0 \m{r} \\
  &\leq& \| u(0) \|_0^2 +  
    2 \!\!\sup_{r\in[0,\tau]}{ \Big( e^{-\lambda(r)/2} 
    \|   u(r) \|_0 \Big)}  \int_0^\tau e^{-\lambda(r)/2} 
    \| (P u)(r) \|_0 \m{r},
\end{eqnarray*}
where we may take the supremum over $\tau \in [0,t]$ on the left-hand side and thus replace $\tau$ by $t$ on the right-hand upper bound. A simple algebraic manipulation then gives
$$
  \left( \sup_{r\in[0,t]} \| e^{-\lambda(r)/2} 
     u(r)\|_0- \int_0^t e^{-\lambda(r)/2}  
     \|(P u)(r)\|_0   \m{r} \right)^2  
  \le  \left(\| u(0) \|_0+    \int_0^t  
    e^{-\lambda(r)/2} \| (P u)(r) \|_0  \m{r}\right)^2.
$$

Upon removing the squares and multiplying by $\exp(\lambda(t)/2)$ we obtain the following basic inequality.
\paragraph{Energy estimate:}
\begin{multline}\label{energyest}
  \sup_{r\in[0,t]}\| u(r) \|_0 \le
   \exp (\int_0^t h(\sigma) \m{\sigma}) \cdot 
   \| u(0) \|_0 +  2   \exp(\int_0^t h(\sigma) \m{\sigma})  
   \cdot \int_0^t  \| (P u)(r) \|_0 \m{r} \\
   = \exp (\int_0^t h(\sigma) \m{\sigma}) \left( 
   \| u(0) \|_0 + 2 \int_0^t  \| (P u)(r) \|_0 \m{r} \right). 
\end{multline}
We recall that the exponential factor depends explicitly on the coefficients $a$ and $c$ via $h(r) = 
\|\frac{1}{2} \div_x (a(r,\cdot)) - c(r,\cdot)\|_{\infty}$.
%\end{proof}

Note that this derivation of an energy estimate relied on the G\r{a}rding inequality \eqref{Garding}. 

\begin{example}[Failure of the G\r{a}rding-inequality \eqref{Garding}]
Let $\alpha\in \,]1/2,1[$ and define $a \col \R \to \R$ by $a(x) := 1+x_{+}^{\alpha}$ when $x\le 1$, and $a(x) := 2$ when $ x > 1$. 
 We have $a \in C_{\ast}^{\alpha}(\mathbb{R}) \setminus {\rm Lip}(\mathbb{R})$. 

Let $Q \col H^1(\R) \to L^2(\R)$ be the operator defined by $(Q v) (x):= a(x) v'(x)$ for all $v \in H^1(\R)$. Note that compared to the general form of the  operator $Q$ in the derivation of the energy estimate above we have here $c = 0$, $a \in C^{\infty}([0,T]; W^{1,2}(\R))$ but $\div a /2 - c = a'/2 \not\in L^{\infty}(\R)$.

Since $Q$ is time independent, inequality \eqref{Garding} with some $h \in L^1([0,T])$ (not necessarily of the form given above) would imply 
$$
  \exists C \in \R, \; 
  \forall v \in C^{\infty}_{\textrm{c}}(\R): \quad 
  {\rm Re} (\langle Q v, v \rangle_0)  
   \ge - C \| v \|_0^2.
$$ 
We will show that there is no constant $C \in \R$ such that the latter holds. Thus \eqref{Garding} cannot hold for $Q$ (for any $h \in L^1([0,T])$).

Let $\rho\in C^1(\mathbb{R})$ be symmetric, non-negative, with
support in $[-1,1]$, $\|\rho\|_0 =1$,  and such that $\rho'(x)<0$ when $0 < x < 1$.
We define $v_{\varepsilon}(x):= \varepsilon^{-1/2} \rho(x/\varepsilon)$  ($x \in \R$, $\eps > 0$). Then clearly $v_{\varepsilon} \in C^{\infty}_{\textrm{c}}(\R) \subseteq H^1(\mathbb{R})$ and $\|v_\varepsilon \|_0 = 1$ for all $\eps > 0$, but 
\begin{eqnarray*}
 \lefteqn{\langle Q v_{\varepsilon}, v_{\varepsilon} \rangle_0 =  
  \int a(x) v_{\varepsilon}'(x) v_{\varepsilon}(x) \m{x}} \\
  &=& \underbrace{\int_{-\infty}^{\infty} v_{\varepsilon}'(x) 
  v_{\varepsilon}(x) \m{x}}_{= 0} + \int_0^1 x^{\alpha} 
  v_{\varepsilon}'(x) v_{\varepsilon}(x) \m{x} + 
   \underbrace{\int_1^{\infty} v_{\varepsilon}'(x) 
  v_{\varepsilon}(x) \m{x}}_{=0} \\
  &=& \varepsilon^{\alpha - 1} \int_0^{1} 
  z^{\alpha} \rho'(z) \rho(z) \m{z} \rightarrow - \infty 
  \qquad (\eps \to 0).
\end{eqnarray*}

We remark that even for $a \in C_{\ast}^{1}(\R) \setminus {\rm Lip}(\mathbb{R})$ the G\r{a}rding inequality may fail as well: for example, with $a(x):= - x \log{|x|} \rho(x)$ we have $a \in C_{\ast}^{1}(\R) \cap W^{1,q}(\R)$ for all $q \in [1,\infty[$, but 
$$
  \langle Q v_{\varepsilon}, v_{\varepsilon} \rangle_0 = 
   - 2 \int_0^1 \rho(\eps z) z \log |\varepsilon z|  
   \rho'(z) \rho(z) \m{z}  
   \leq  2 |{\log{\varepsilon}}| \int_0^1 z  
   \rho(\eps z) \rho'(z) \rho(z) \m{z}  \rightarrow - \infty,
$$
since $\lim_{\eps \to 0} \int_0^1 z \rho(\eps z) \rho'(z) \rho(z) \m{z} = \rho(0) \int_0^1 z \rho'(z) \rho(z) \m{z} < 0$.

\end{example}

\begin{remark}\label{adj_rem} (i) Let $Q^{\ast}$ denote the formal adjoint of $Q$ with respect to the $L^2$ inner product (on $x$-space). Due to our regularity assumptions on $a$ and $c$ we have for any $\vphi \in H^1$  (since $a$ is real)
$$
  Q^{\ast} \vphi = \sum_{j=1}^n (- a_j \d_{x_j} \vphi) 
  + (\bar{c} - \div_x(a)) \vphi,
$$
where the new coefficients $-a$, respectively $\bar{c} - \div_x(a)$, in place of $a$, respectively $c$, satisfy the exact same regularity assumptions, including the condition
$$
  \frac{1}{2} \div_x(- a) - (\bar{c} - \div_x(a)) =
  \ovl{\frac{\div_x(a) }{2} - c} \quad \in
   L^1([0,T];L^{\infty}(\mathbb{R}^n)).
$$
Thus the basic energy estimate \eqref{energyest} applies to $ \pm \d_t + Q^\ast$ as well. In particular, the function $h$ in the exponential factor occurring in the energy estimates is the same for $Q$ and $Q^\ast$.

(ii) Although the method of derivation discussed above relied on a G\r{a}rding-type inequality, it seems that in essence energy estimates are, in a vague sense, a necessary condition for a hyperbolic equation to hold in any meaningful context of ``suitable Banach spaces of distributions''. In other words, whenever a hyperbolic differential equation can be interpreted directly in terms of such Banach spaces it allows to draw consequences on combinations of corresponding norms of any solution. For example, if the operator $Q$ above generates a strongly continuous evolution system on some Banach space, then basic norm estimates for solutions follow from general principles of that theory (cf.\ \cite{Pazy:83,Tanabe:79}).
 \end{remark}

On the other hand, energy estimates are widely used to establish existence of solutions to \eqref{initial} by duality and an application of the Hahn-Banach theorem. We recall the basic steps of such method in the following.

\subsubsection{Existence proof based on the energy estimate}

%\begin{theorem}[Existence]\label{thmex}
%\end{theorem}

%\begin{proof} 
Let $R_T:=\{(t,x)\in \mathbb{R}^{n+1} \mid  t < T\}$. By abuse of notation we denote the trivial extension of a function $v\in C_c^{\infty}(R_T)$ by zero for $t\ge T$
again by $v$. Then $\mathcal{L}:= \{ f \in C^{\infty}([0,T]\times \R^n) \mid  \exists v\in C_c^{\infty}(R_T) \text{ with }
f = (-\partial_t v+ Q^{\ast} v) \mid_{[0,T] \times \R^n} \}$. For $0  \le t \leq T$ and $v\in C_c^{\infty}(R_T)$ we use the notation 
%$S(t):=\partial_t + Q^{\ast}(T-t)$,  
$w(t):= v(T-t)$ and $g(t):= (-\partial_t v + Q^{\ast}v)(t)$.
Then we have
\begin{eqnarray*}
  (\partial_t + Q^{\ast}(T-t)) w(t) &=& g(T-t)\\
  w(0) &=&  0
\end{eqnarray*}
and an application of \eqref{energyest} (with $Q^*$ in place of $Q$; cf.\ Remark \ref{adj_rem}(i) above) yields 
\begin{multline*}
\sup_{r\in[0,T]} \| w(r) \|_{0} 
  \le  2 \exp{\left(\int_0^T h( \sigma ) \m{\sigma}\right)} 
  \int_0^T  
  \| (-\partial_t+ Q^{\ast} v)(T-r) \|_{0} \m{r} 
  = C_h \int_0^T \|g(r)\|_{0} \m{r}.
\end{multline*}
We may deduce that for $f \in L^1([0,T]; L^2(\mathbb{R}^n))$ and $v\in C_c^{\infty}(R_T)$
\begin{multline*}
   \int_0^T \langle f(r), v(r) \rangle_0 \m{r} + 
   \langle u_0, v(0) \rangle_0 \le 
   \int_0^T \|f(r)\|_{0} \|v(r)\|_{0} \m{r}
   + \|u_0\|_{0}\, \|v(0)\|_{0}  \\
    \le C \sup_{r\in[0,T]}\|w(r)\|_{0} 
    \le C C_h \int_0^T \|g(r)\|_{0} \m{r},  
\end{multline*}
where $C$ depends on $f$ and $u_0$. Therefore the assignment
$g = (-\partial_t v + Q^{\ast} v) \mid_{[0,T] \times \R^n}\mapsto 
\int_0^T \langle f(r), v(r) \rangle_0 \m{r} + \langle u_0, v(0) \rangle_0$ defines a conjugate-linear functional $\nu \col \mathcal{L} \to \mathbb{C}$ on the  subspace $\mathcal{L}$ of $L^1([0,T]; L^2(\mathbb{R}^n))$ such that
$|\nu(g)| \le \sup_{0 \leq r \leq T} \|g(r)\|_{0}$.  Hahn-Banach extension of $\nu$ yields a conjugate-linear functional
$\nu': L^1([0,T];L^2(\mathbb{R}^n)) \to \mathbb{C}$ with the same norm estimate. 

Since $L^1([0,T]; L^2(\mathbb{R}^n))' \cong L^{\infty}([0,T]; L^2(\mathbb{R}^n))$ there is $u \in L^{\infty}([0,T]; L^2(\mathbb{R}^n))$ such that 
$\nu'(g) = \langle u, g \rangle$ for all $g\in L^1([0,T]; L^2(\mathbb{R}^n))$. When applied to $g = (-\partial_t v + Q^{\ast} v) \mid_{[0,T] \times \R^n}$  with $v\in C_c^{\infty}(R_T)$ we obtain
%\begin{multline*}
%  \int_0^T \langle (Pu)(r), v(r) \rangle_0 \, dr =
\begin{multline} \label{DE_orig}
  \int_0^T \langle u(t),  -\partial_t v(t) + 
  ( Q^{\ast} v )(t) \rangle_0  \m{t} =
  \dis{u}{((-\d_t + Q^{\ast}) v) \mid_{[0,T] \times \R^n}}  \\
  = \int_0^T \langle f(t), v(t,.)\rangle_0 \m{t} + 
  \langle u_0, v(0) \rangle_0. 
\end{multline}

\subsubsection{Model discussion of the weak solution concept}

\paragraph{Case of smooth symbol:}
If the coefficients of $Q$ (and thus of $Q^{\ast}$) are $C^\infty$ then the above identity implies that $u$ is a distributional solution to the 
partial differential equation $P u = f$ in $]0,T[ \times \R^n$.
In fact, with $\vphi \in C_c^{\infty}(]0,T[ \times \R^n)$ in place of $v$ we have
$$
  \dis{(\d_t + Q)u}{\vphi} = \dis{u}{(-\d_t + Q^{\ast}) \vphi} 
  = \int_0^T \langle f(t), \phi(t,.)\rangle_0 \m{t} = \dis{f}{\vphi}.
$$
Moreover, since $Q u \in L^{\infty}([0,T];H^{-1}(\R^n))$ the differential equation implies that
$$
  \d_t u = f - Q u \quad \in L^1([0,T];H^{-1}(\R^n))
$$
and thus $u \in AC([0,T];H^{-1}(\R^n))$. In particular, it makes sense to speak of the initial value $u(0) \in \D'(\R^n)$. Integrating by parts on the left-hand side of \eqref{DE_orig} (now reading \eqref{DE_orig} from right to left, and duality brackets in appropriate dual pairs of spaces) yields for any $v\in C_c^{\infty}(R_T)$
\begin{multline*}
   \int_0^T \langle f(t), v(t)\rangle_0 \m{t} + 
  \langle u_0, v(0) \rangle_0 = 
  \int_0^T \underbrace{\dis{\d_t u (t) + Q u(t)}{v(t)}}_{\langle f(t), v(t,.)\rangle_0} \m{t} - 
  \dis{u(T)}{\underbrace{v(T)}_{= 0}}_0 + \dis{u(0)}{v(0)}_0,
\end{multline*}
hence $u(0) = u_0$

Of course, uniqueness of the solution as well as more precise regularity properties can be deduced in case of $C^{\infty}$ coefficients: For any $s \in \R$, $f \in L^1([0,T]; H^s(\R^n))$, and $u_0 \in H^s(\R^n)$ the solution $u$ is unique in the space $C([0,T];H^s(\R^n))$ (cf.\ \cite[Theorem 23.1.2]{Hoermander:V3}).  

\paragraph{Case of non-smooth symbol:}
The weaker regularity assumptions made above imply $Q^{\ast} v \in L^1([0,T];L^2(\R^n))$ for all $v \in C_c^{\infty}(R_T)$. We may thus define $Q u \in \D'(]0,T[ \times \R^n)$ by putting 
$$
  \dis{Qu}{\vphi} := \inp{u}{Q^{\ast} \bar{\vphi}}_0 
     \qquad \forall \vphi \in C_c^{\infty}(]0,T[ \times \R^n).
$$ 
Then equation \eqref{DE_orig} can be read as an equation in $\D'(]0,T[ \times \R^n)$, namely
%(with the right-most term to be understood in $\D'(\R^n)$)
$$ 
    \dis{\d_t u + Q u}{\vphi} = \dis{f}{\vphi}  
    \qquad \forall \vphi \in C_c^{\infty}(]0,T[ \times \R^n).  
    %+ \dis{\delta \otimes u_0}{1 \otimes \bar{v}(0)}
$$
Furthermore, we can again show that the inital datum is attained: Note that in $Q u = \sum a_j \d_{x_j} u$ each term can be interpreted as a multiplication of functions in $L^1([0,T];H^1(\R^n))$ with distributions in $L^{\infty}([0,T];H^{-1}(\R^n))$ (since $u \in L^{\infty}([0,T];L^2(\R^n))$) in the sense of the duality method (cf.\ \cite[Chapter II, Section 5]{O:92}). Applying Proposition 5.2 in \cite{O:92} to the spatial variables in the products then yields $Qu \in L^1([0,T]; W^{-1,1}(\mathbb{R}^n))$. Reasoning similarly as above, the differential equation then gives
$$
  \d_t u = f - Q u \quad \in L^1([0,T];W^{-1,1}(\R^n)),
$$
which implies $u \in AC([0,T];W^{-1,1}(\R^n))$ and further also that $u(0) = u_0$.    

Again higher regularity of $u$ with respect to the time variable, namely $u \in C([0,T];L^2(\R^n))$ can be shown by means of regularization and passage to the limit (e.g., similarly as in \cite[proof of Theorem 2.8]{BGS:07}).

%%%%%%%%%%%%%%%%%%%%%%%%%%%%%%%%

\subsection{Regularization and energy estimates}

Several advanced theories make use of regularization techniques or concepts at crucial steps in their construction of solutions. Some of these theories succeed by regularization and a careful passage to the limit via energy estimates (as with Hurd-Sattinger and Di Perna-Lions theories presented below). Others even base their solution concept on a further generalization of the weak solution concept beyond distribution and measure spaces and still obtain existence of solutions essentially from asymptotic stability of energy estimates (cf.\ the Lafon-Oberguggenberger theory below). 
 
We introduce the following notation for partial differential operators that will be used in the sequel
\begin{eqnarray} 
  P u & := & \partial_t u + 
    \sum_{k=1}^n a_k \, \partial_{x_k}  u + c\, u 
    \label{pdo1} \\
  L u & := & \partial_t u + 
   \sum_{j=1}^n \partial_{x_j} (a_j\, u) + b\, u.
\end{eqnarray}

\subsubsection{Hurd-Sattinger theory}

We give a brief summary of the results from the first part in Hurd-Sattinger's classic paper \cite{HS:68}. We consider the Cauchy problem for the operator $L$ on the closure of the domain $\Omega := \, ]0,\infty[ \times \R^n$.

\begin{definition} Let $f \in L^2_\loc(\ovl{\Omega})$ and $a_j$ ($j=1,\ldots,n$) as well as $b \in L^2_\loc(\ovl{\Omega})$. A weak solution in the sense of Hurd-Sattinger of the partial differential equation
$$
   L u = f \qquad \text{on } \Omega
$$
with initial condition $u_0 \in L^2_{\rm{loc}}$ is a function  
$u \in L^2_{\rm{loc}}(\ovl{\Omega})$ such that for all $\phi \in C_c^{1}(\mathbb{R}^{n+1})$ we have
\begin{multline}
   \int_{\ovl{\Omega}} \Big(- u(t,x) \phi(t,x) - 
  \sum_{j=1}^n 
  a_j(t,x) u(t,x) \partial_{x_j} \phi(t,x) 
  + b(t,x) u(t,x) \phi(t,x) \Big) \m{(t,x)} \\
  = \int_{\ovl{\Omega}} f(t,x) \phi(t,x) \m{(t,x)}
 + \int_{\mathbb{R}^n} u_0(x) \phi(0,x) \m{x}.
\end{multline}
\end{definition}

Note that if all coefficients are  $C^{\infty}$ functions then a solution in the above sense solves the partial differential equation on $\Omega$ in the sense of distributions.

\begin{theorem}\label{HS_thm} Let $a_j$ ($j=1,\ldots,n$), $b$, and $f$ belong to $L^2_\loc(\ovl{\Omega})$ and $u_0 \in L^2_{\rm{loc}}$.
Assume, in addition, that the following conditions are satisfied:
\begin{enumerate}

\item There exists $c_1 > 0$ such that for almost all $(t,x) \in \ovl{\Omega}$: \enspace $a_k (t,x) \le c_1$ \enspace 
($k=1\ldots,n$).

\item There exists a function $\mu \in L^1_{\rm{loc}}([0,\infty[)$, $\mu \geq 0$, such that $b(t,x) \ge - \mu(t)$ for almost all $(t,x) \in \ovl{\Omega}$.

\item For each $k \in \{1,\ldots,n\}$ there exists $0\le\mu_k \in L^1_{\rm{loc}}([0,\infty[)$ such that for almost all $(t,x) \in \ovl{\Omega}$
$$ 
\frac{a_k(t,x) - a_k(t,x_1,\ldots,x_{k-1},r,x_{k+1}\ldots,x_n)}{x_k - r} \ge - \mu_k(t) \qquad \text{for almost all } r \in \R.
$$
\end{enumerate}
Then there exists a weak solution $u \in L^2_{\rm{loc}}(\ovl{\Omega})$ to $L u = f$ with initial condition $u_0$.
\end{theorem}

Concerning the meaning of condition (iii) in Theorem \ref{HS_thm} we mention two aspects:
\begin{itemize}
\item In one space dimension we obtain
$\frac{a(x,t)-a(t,y)}{x-y} \ge - \mu_1(t)$,
which resembles a one-sided Lipschitz continuity condition in the $x$ variable (apart from the fact that $\mu_1(t)$ need not be finite or defined for all $t$). In particular, it excludes jumps downward (seen when going from smaller to larger values in the $x$ argument).

\item Heuristically --- replacing difference quotients by partial derivatives --- condition (iii) can be read as $\div a(t,x) \geq - \sum \mu_k(t)$, thus giving an $L^1$ lower bound on the divergence of $a$. We observe that upon formally applying the Leibniz rule in the operator $L$ we cast it in the form $P$ as in \eqref{pdo1} with  $c = \div a + b$. In combination with condition (ii) of Theorem \ref{HS_thm}, we obtain that $\frac{1}{2} \div a - c =  -(\frac{1}{2} \div a + b)$ has an $L^1$ upper bound (uniformly in $x$), which can be considered a substitute for condition \eqref{div_cond} used in the derivation of direct energy estimates in Subsection 2.1.
\end{itemize}  

\begin{remark} Hurd-Sattinger (\cite{HS:68}) also give a uniqueness result for first-order systems in case of a single space variable and $b = 0$. For scalar equations the hypotheses require condition (i) to be strengthened to boundedness from above and from below and  condition (iii) to be replaced by a Lipschitz property with an upper bound instead; in particular, no jumps upward are possible.  
\end{remark}

\begin{example} For the operator $L$ in one space dimension and coefficients $a(x) = \rm{sign}(x)$ and $b = 0$, the Poupaud-Rascle theory is not applicable (as mentioned in \cite[Section 1, Example 2]{PoupaudRascle:97}), but Hurd-Sattinger theory ensures existence of weak solutions, if the initial value belongs to $L^2_{\rm{loc}}$.
\end{example}

%%%%%% DiPerna - Lions

\subsubsection{Di Perna-Lions theory}

The weak solution concept introduced by Di Perna-Lions in \cite{DiPernaLions:89} for the Cauchy problem for the operator $P$ on a finite-time domain $[0,T] \times \R^n$ can be interpreted in the following way.

\begin{definition}\label{dPL_sol} Let $T > 0$, $1 \leq p \leq \infty$, $\frac{1}{p} + \frac{1}{q} = 1$, $f \in L^1([0,T]; L^p(\R^n))$, $a_k \in L^{1}([0,T];L^q_{\rm{loc}}(\mathbb{R}^n))$ ($k=1,\ldots,n$), and $c\in L^1([0,T];L^q_{\rm loc}(\mathbb{R}^n))$ such that 
$$
  \div(a) - c  \quad \in 
  L^{1}([0,T];L^q_{\rm{loc}}(\mathbb{R}^n)).
$$ 
A function $u\in L^{\infty}([0,T]; L^{p}(\mathbb{R}^n))$ is called a weak solution in the sense of Di Perna-Lions  of the partial differential equation 
$$
 P u = f \qquad \text{on } ]0,T[\, \times \R^n
$$
with initial value $u_0 \in L^p(\R^n)$, if 
\begin{multline} 
  \int_0^T \int_{\mathbb{R}^n} u(t,x) \Big( -
  \partial_t \varphi(t,x) \m{x}  
  - \sum_{k=1}^n a_k(t,x) 
  \d_{x_k}\varphi(t,x) \Big) \m{x} \m{t} \\ 
  + \int_0^T \int_{\mathbb{R}^n}
  u(t,x) \big( -
   \div a(t,x) + c(t,x) \big) 
  \varphi(t,x) DCm{x} \m{t} \\
  = \int_0^T \int_{\mathbb{R}^n} f(t,x) 
  \varphi(t,x) \m{x} \m{t}  + 
  \int_{\mathbb{R}^n} u_0(x) \varphi(0,x) \m{x} 
\end{multline}
holds for all $\varphi \in C^{\infty}([0,T], \mathbb{R}^n)$ with compact support in $[0,T[ \times \R^n$.
\end{definition}

Clearly, in case of $C^{\infty}$ coefficients we obtain a distributional solution of the partial differential equation in $]0,T[ \times \R^n$.

\begin{theorem}\label{DiPL_thm} Existence of a weak solution $u \in L^{\infty}([0,T]; L^p(\mathbb{R}^n))$ in the sense of and with assumptions as in Definition \ref{dPL_sol} is guaranteed under the additional hypothesis 
\begin{eqnarray*}
  \frac{1}{p} {\rm div}(a) - c \in \quad
     L^1([0,T];L^{\infty}(\R^n)),
  && \rm{if}\ p>1, \\ 
    {\rm div}(a), c \quad \in 
    L^1([0,T];L^{\infty}(\R^n)),
    && \rm{if}\ p=1.
\end{eqnarray*}
\end{theorem}

\begin{remark}\label{DiPL_uni} Uniqueness holds in general under the
 additional hypotheses that $c, {\rm div}(a) \in L^1([0,T]; L^{\infty}(\mathbb{R}^n))$, and for  $j=1,\ldots,n$ also $a_j \in L^1([0,T]; W^{1,q}_{\rm loc}(\mathbb{R}^n))$ as well as
\begin{eqnarray*}
  \frac{a_j}{1+|x|} \quad \in \;
  L^1([0,T]; L^1(\mathbb{R}^n)) + 
  L^1([0,T]; L^{\infty}(\mathbb{R}^n)).
\end{eqnarray*}
\end{remark}

\begin{example}[Hurd-Sattinger applicable, but not Di Perna-Lions] Note that with a single spatial variable boundedness of $\div(a) = a'$ implies Lipschitz continuity. Hence, if $a \in H^1(\R)$ is not Lipschitz continuous but satisfies the one-sided Lipschitz condition in Hurd-Sattinger's existence Theorem \ref{HS_thm} (iii), then a weak solution in the sense of Hurd-Sattinger to the problem
$$
   \d_t u + \d_x(a u) = f \in L^2(\R^2), \quad
   u \mid_{t=0} = u_0 \in L^2(\R)
$$   
is guaranteed to exist, whereas the general statement of DiPerna-Lions' existence theory (Theorem \ref{DiPL_thm} with $p=q=2$) is not applicable to the formally equivalent problem
$$
  \d_t u + a \d_x u + a' u = f \in L^2(\R^2), \quad
   u \mid_{t=0} = u_0 \in L^2(\R).
$$ 
\end{example}

\begin{example}[Di Perna-Lions applicable, but not Hurd-Sattinger] 
Let $0 < \sigma < 1$ and consider the identical coefficient functions $a_1 = a_2 \in C_{*,\mathrm{comp}}^{\sigma}(\R^2)$ (i.e., comsupporteduported functions in $C_*^\sigma(\R^2)$) given by
$$
  a_1(x,y) = a_2(x,y) = 
  -\frac{1}{\sigma} (x-y)_{+}^{\sigma}\, \chi(x,y),
$$
where $\chi \in \mathcal{D}(\mathbb{R}^2)$ such that $\chi = 1$ near $(0,0)$. Note that $a_1$ is not Lipschitz continuous, since for $x > 0$ but $x$ sufficiently small the difference quotient
$$
  \frac{a_1(x,0) - a_1(0,0)}{x} = - \frac{x^{\sigma - 1}}{\sigma}
$$ 
is unbounded as $x \to 0$. In particular, the latter observation shows that the Hurd-Sattinger existence theory is not applicable (condition (iii) in Theorem \ref{HS_thm} is violated) to the Cauchy problem for the operator
$$
   L u = \d_t u + \d_x(a_1 u) + \d_y(a_2 u).
$$

On the other hand, we can show that with $a = (a_1,a_2)$ the DiPerna-Lions existence theory is applicable to the Cauchy problem
\begin{eqnarray*}
  \partial_t u + a_1 \d_x u + a_2 \d_y u  + (\div a)\, u 
  = f \in L^1([0,T]; L^p(\mathbb{R}^2)), 
  \quad u \mid_{t=0} = u_0 \in L^p(\R^2).
\end{eqnarray*} 

To begin with, we observe that
\begin{align*}
  \d_x a_1(x,y) = \d_x a_2(x,y) &=
  - \frac{\chi(x,y)}{(x-y)_+^{1-\sigma}} - 
  \frac{1}{\sigma} (x-y)_+^\sigma \, \d_x \chi(x,y) \\
  \d_y a_1(x,y) = \d_y a_2(x,y) &=
  \frac{\chi(x,y)}{(x-y)_+^{1-\sigma}} - 
  \frac{1}{\sigma} (x-y)_+^\sigma \, \d_y \chi(x,y)  
\end{align*} 
yields
$$
 \div a (x,y) = - 
  \frac{1}{\sigma} (x-y)_+^\sigma \, \div \chi(x,y) 
  \in C_{*,\mathrm{comp}}^\sigma(\R^2).
$$
Note that in the notation of Definition \ref{dPL_sol} and Theorem \ref{DiPL_thm} we have $c = \div a \in L^\infty(\R^2)$ (and time-independent). Therefore, the basic assumptions for the solution concept to make sense as well as the hypotheses of the existence statement are clearly satisfied.

As for uniqueness, we remark that all the conditions mentioned in Remark \ref{DiPL_uni} are met if and only if $\sigma > 1/p$.
\end{example}

\begin{remark} We mention that with coefficients as in the above example, the system of characteristic differential equations has  forward-unique solutions, hence the Poupaud-Rascle solution concept for measures is also applicable. 
\end{remark}

%%%%

\subsubsection{Lafon-Oberguggenberger theory}

The theory for symmetric hyperbolic systems presented in \cite{LO:91} by Lafon-Oberguggenberger allows for Colombeau generalized functions as coefficients as well as inital data and right-hand side. Thus we consider the following hyperbolic Cauchy
problem in $\R^{n+1}$
\begin{align}
  P u = \d_t u + \sum_{j=1}^{n} a_j \d_{x_j} u + c u = f   
\label{hypsys_equ} \\
  u \mid_{t = 0} = u_0,
\label{hypsys_ini}
\end{align}
where $a_j$ ($j=1,\ldots,n$), $c$ are real valued generalized
functions in $\G(\R^{n+1})$ (in the sense that all representatives are
real valued smooth functions), $f \in \G(\R^{n+1})$, and initial value $u_0 \in\G(\R^n)$. 

The coefficients will be subject to some restriction on the allowed
divergence in terms of $\eps$-dependence. A Colombeau function
$v \in\G(\R^d)$ is said to be of \emph{logarithmic type} if it has a
representative $(v_\eps)$ with the following property: there are constants
$N\in\N$, $C>0$, and $1 > \eta>0$ such that
\begin{equation*}%\label{log_growth}
  \sup\limits_{y\in\R^d} |v_\eps(y)| \leq N \log\big(\frac{C}{\eps}\big) 
  \qquad 0 <  \eps < \eta \; .
\end{equation*}
(This property then holds for any representative.) By a suitable modification of \cite[Proposition 1.5]{O:89}
it is always possible to model any finite order distribution as coefficient with such properties (in the sense that the Colombeau coefficient is associated to the original distributional coefficient).

\begin{theorem}\label{LO_thm} Assume that $a_j$ and $c$ are constant for large $|x|$ and that 
$\d_{x_k} a_j$ ($k=1,\ldots,n$) as well as $c$ are of logarithmic type. Then given initial data $u_0 \in\G(\R^n)$ and right-hand side $f \in \G(\R^{n+1})$, the Cauchy problem (\ref{hypsys_equ})-(\ref{hypsys_ini}) has a unique solution $u \in\G(\R^{n+1})$.
\end{theorem}

We also mention the following {\bf consistency result} which shows that
Colombeau theory includes the classically solvable cases: If we assume that the coefficients $a_j$ and $c$ are $C^\infty$ then we have the following consistency with classical and distributional solutions (cf.\ \cite{LO:91})
\begin{itemize}
\item If $f$ and $u_0$ are $C^\infty$ functions then the generalized solution
  $u \in\G(\R^{n+1})$ is equal (in $\G$) to the classical smooth
  solution.
\item If $f \in L^2(\R;H^s(\R^n))$ and $u_0 \in H^s(\R^n)$ for some
  $s\in\R$, then the generalized solution $u \in\G(\R^{n+1})$ is
  associated to the classical solution belonging to $C(\R;
  H^s(\R^n))$.
\end{itemize}

\begin{example} Consider the $(1+1)$-dimensional operator
$$
   L u  = \d_t u  + \d_x (H(-x) u).
$$
Since the coefficient (of the formal principal part) has a jump downward neither Hurd-Sattinger nor Di Perna-Lions theory is applicable. In fact, it has been shown in \cite[Section 2]{HdH:01} that none of the distributional products from the coherent hierarchy (cf.\ \cite{O:92} and the introductionary section) applied to $H(-x) \cdot u$ is capable of allowing for distributional solutions of the homogeneous Cauchy problem for arbitrary smooth initial data.

Recall from Section 1 that measure solutions according to Bouchut-James exist for the corresponding Cauchy problem, if the Heaviside function (usually understood as a class of functions in $L^\infty$) is replaced by the particular Borel measurable representative with value $0$ at $x = 0$.  For example, the initial value $u_0 = 1$ then yields the measure solution $u = 1 + t \delta(x)$ in the sense of Bouchut-James as seen in Example \ref{BJ_jump}. 

However, Colombeau generalized solutions are easily obtained --- even for arbitrary generalized initial data --- if the coefficient $H(-x)$ is regularized by convolution with a delta net of the form $\rho_\eps(x) =  \log(1/\eps) \rho(x \log(1/\eps))$ ($0 < \eps < 1$), where $\rho \in C_c^\infty(\R)$ with $\int \rho = 1$. Let $a$ denote the class of this regularization in the Colombeau algebra $\G$, then the operator $L$ may now be written equivalently in the form
$$
  P u = \d_t u + a \d_x u + a' u,
$$
where $a' \approx \delta$ and $u \in \G$. Due to the logarithmic scale in the regularization the hypotheses of Theorem  \ref{LO_thm} are satisfied and the corresponding Cauchy problem is uniquely solvable. Moreover, for most interesting initial data (e.g.\  Dirac measures or $L^1_{\mathrm{loc}}$) weak limits of the Colombeau solution $u$ are known to exist and can be computed (cf.\ \cite[Section 6]{HdH:01}). In particular, for the initial value $u_0 = 1$ we obtain the measure solution $u = 1 + t \delta(x)$ as such a distributional shadow.

\end{example}

\begin{remark} (i) The basic results of Lafon-Oberguggenberger have been extended to the case of (scalar) pseudodifferential equations with generalized symbols in \cite{GH:04}.  Special cases and very instructive examples can be found in \cite{O:88}, and an application of Colombeau theory to the linear acoustics system is presented in \cite{O:89}).

(ii) Colombeau-theoretic approaches allow for a further flexibility even in interpreting distributional differential equations with smooth coefficients. For example, in \cite{CHO:96} the concept of regularized derivatives is used, where partial differentiation is replaced by convolution with the corresponding derivative of a delta sequence. When acting on distributions this concept produces the usual differential operator actions in the limit. When considered as operators in Colombeau spaces, one can prove (cf.\ \cite[Theorem 4.1]{CHO:96}) that evolution equations with smooth coefficients all whose derivatives are bounded have unique generalized function solutions for initial data and right-hand side in generalized functions.
In particular, famous examples like the Lewy equation become solvable and Zuily's non-uniqueness examples become uniquely solvable then.
\end{remark}

\subsection{Paradifferential techniques}

\subsubsection{Energy estimates}

Bony's paradifferential calculus has been successfully applied  in nonlinear analysis and, in particular, to regularity theory for  nonlinear partial differential equations. An ingredient in such approaches is often a refined regularity assessment of corresponding linearizations of the differential operators involved. A recent account of M\'{e}tivier's methods and results of this type can be found in \cite[Subsection 2.1.3]{BGS:07}, or with more details on microlocal properties in \cite{Hoermander:97}.

Let $s \in \R$ and $H^s_w(\R^n)$ denote the Sobolev space $H^s(\R^n)$ equipped with the weak topology. We consider a differential operator of the form
$$
  \widetilde{P}_v (x,t;\d_t,\d_x) := 
  \d_t + \sum_{j=1}^n a_j(v(x,t)) \, \d_j,
$$
where $a_j \in C^\infty(\R)$ ($j=1\ldots,n$) and $v \in L^\infty([0,T];H^s(\R^n)) \cap C([0,T];H^s_w(\R^n))$ such that $\d_t v \in L^\infty([0,T];H^{s-1}(\R^n)) \cap C([0,T];H^{s-1}_w(\R^n))$. 

\begin{remark}
Not all hyperbolic first-order differential operators with coefficients of regularity as above can be written in the special form of $\widetilde{P}_v$. In fact, this amounts to writing any given list $w_1, \ldots, w_n$ of such functions as $w_j = a_j \circ v$ ($j=1\ldots,n$) with $a_j \in C^\infty(\R)$ and $v$ as above. The latter is, in general, not possible, which can be seen from the following example: consider the Lipschitz continuous functions $w_1(t) = |t|$ and $w_2(t) = t$; if $w_1 = a_1 \circ v$ and $w_2 = a_2 \circ v$ with a Lipschitz continuous function $v$, then $v$ is necessarily non-differentiable at $0$; on the other hand 
$$
  1 = w_2'(0) = \lim_{h \to 0} (a_2(v(h)) - a_2(v(0)))/h 
   = \lim_{h \to 0} a_2'(\xi(h)) (v(h) - v(0)) /h,
$$ 
where $\xi(h)$ lies between $v(0)$ and $v(h)$; hence $a_2'(\xi(h)) \to a_2'(v(0))$ and the second factor $(v(h) - v(0)) /h$ stays bounded, but is not convergent; in case $a_2'(v(0)) = 0$ we obtain the contradiction $1 = 0$, in case $a_2'(v(0)) \neq 0$ we have a contradiction to convergence of the difference quotient for $w_2$.   
\end{remark}

The key technique in analyzing the operator $\widetilde{P}_v$ is to replace all terms $a_j(v)  \d_j$ by $T_{a_j(v)}  \d_j$, i.e., partial differentiation followed by the para-product operator $T_{a_j(v)}$, and then employ estimates of the error terms as well as a paradifferential variant of G\r{a}rding's inequality  (cf.\ \cite[Appendix C.3-4]{BGS:07}). This leads to the following result.
\begin{theorem}[{\cite[Theorem 2.7]{BGS:07}}]
If $s > \frac{n}{2} + 1$, then for any $f \in L^\infty([0,T];H^s(\R^n)) \cap C([0,T];H^s_w(\R^n))$ and $u_0 \in H^s(\R^n)$ the Cauchy problem
$$
  \widetilde{P}_v u = f, \quad u \mid_{t=0} = u_0 
$$
has a unique solution $u \in L^2([0,T];H^s(\R^n))$. Moreover, $u$ belongs to $C([0,T];H^s(\R^n))$ and there are constants $K, \gamma, C \geq 0$ such that $u$ satisfies the energy estimate
$$
  \norm{u(t)}{s}^2 \leq K e^{\gamma t} \norm{u(0)}{s}^2 
    + C \int_0^t e^{\gamma (t - \tau)} 
    \norm{\widetilde{P}_v u(\tau)}{s}^2 \m{\tau}.
$$
\end{theorem}

\subsubsection{Improvement of regularity in one-way wave equations}

We briefly recall some basic notions and properties concerning symbols with certain H\"older regularity in $x$ and smoothness in $\xi$  \`{a} la Taylor (cf.\ \cite{Taylor:91}). 

\begin{definition} Let $r > 0$, $0 < \delta < 1$, and $m \in \R$. A continuous function $p: \R^n \times \R^n \to \C$ belongs to the symbol space $C_\ast^r S^m_{1,\delta}$, if for every fixed $x \in \R^n$ the map $\xi \mapsto p(x,\xi)$ is smooth and 
for all $\al \in \N_0^n$ there exists $C_\al > 0$ such that
$$
  |\d_\xi^\al p(x,\xi)| \leq C_\al (1 + |\xi|)^{m - |\al|}
  \qquad \forall x, \xi \in \R^n
$$
and
$$
  \norm{\d_\xi^\al p(.,\xi)}{C_\ast^r} \leq 
    C_\al (1 + |\xi|)^{m - |\al| + r \delta}
    \qquad \forall \xi \in \R^n.
$$
\end{definition}

Basic examples are, of course, provided by symbols of differential operators $\sum a_\al \d^\al$ with coefficient functions $a_\al \in C_\ast^r$ ($|\al| \leq m$) or any symbol of the form $p(x,\xi) = a(x) h(x,\xi)$, where $a \in C_\ast^r$ and $h$ is a smooth symbol of order $m$.

\paragraph{Symbol smoothing:} By a coupling of a Littlewood-Paley decomposition in $\xi$-space with convolution regularization in $x$-space via a $\delta$-dependent scale one obtains a decomposition of any symbol $p \in C_\ast^r S^m_{1,\delta}$ in the form 
$$
 p = p^\sharp + p^\flat, \quad \text{ where } p^{\sharp} 
     \in S^m_{1,\delta} \text{ and } p^{\flat} 
     \in C_\ast^r S^{m - r \delta}_{1,\delta}.
$$
Observe that $p^\sharp$ is $C^\infty$ and of the same order whereas $p^\flat$ has the same regularity as $p$ but is of lower order.

\paragraph{Mapping properties:}
Let $0 < \delta < 1$ and $-(1-\delta) r < s < r$. Then any symbol  
$p\in C_\ast^r S^{m}_{1,\delta}$ defines a continuous linear operator
$p(x,D): H^{s+m}(\mathbb{R}^n) \rightarrow H^s(\mathbb{R}^n)$.

\paragraph{Elliptic symbols:} $p \in C_\ast^r S^m_{1,\delta}$ is said to be elliptic, if there are constants $C, R > 0$ such that
$$
   |p(x,\xi)| \geq C (1 + |\xi|)^m 
     \qquad \forall \xi \in \R^n, |\xi| \geq R.
$$

One-way wave equations result typically from second-order partial differential equations by a pseudodifferential decoupling into two first-order equations (cf.\ \cite[Section
IX.1]{Taylor:81}). For example, this has become a standard technique in mathematical geophysics for the decoupling of modes in seismic wave propagation (cf.\ \cite{SdH:02}). The corresponding Cauchy problem with seismic source term $f \in C^{\infty}([0,T]; H^{s}(\mathbb{R}^n))$ (with $s \in \R$) and initial value of the displacement $u_0 \in H^{s+1}(\mathbb{R}^n)$ is of the form
\begin{eqnarray}
  \partial_t u + i \, Q(x,D) u &=& f \label{one-way-PDE}\\
  u\mid_{t=0} &=& u_0 \label{one-way-initial},
\end{eqnarray}
where $Q$ has real-valued elliptic symbol $q \in C_\ast^r S^1$ with $r > s$.

\begin{lemma}\label{ell_lemma}
If $q\in C^{r}S^m_{1,0}$ is elliptic, then $q^{\sharp} \in S^{m}_{1,\delta}$ is also elliptic. 
\end{lemma}
\begin{proof}
By ellipticity of $q$ and the symbol properties of $q^\flat$ there are constants $C_1, C_2, R > 0$ such that 
\begin{eqnarray*}
C_1(1+|\xi|)^m \le |q(x,\xi)|\le |q^{\sharp}(x,\xi)| +|q^{\flat}(x,\xi)| \le  |q^{\sharp}(x,\xi)| + C_2 (1+|\xi|)^{m-r\delta}
\end{eqnarray*}
holds for all $x, \xi \in \R^n$ with $|\xi| \geq R > 0$. 
Therefore
$$
|q^{\sharp}(x,\xi)| \ge 
  (C_1 - C_2 (1 + |\xi|)^{-r\delta}) (1+|\xi|)^m 
  \geq C (1 + |\xi|)^m  
  \qquad \forall x, \xi \in \R^n, |\xi| \geq R'
$$
for suitably chosen constants $C$ and $R' > 0$.
\end{proof}

Let $0 < \delta < 1$. We have the decomposition $q = q^{\sharp} + q^{\flat}$, where
$q^{\sharp} \in S^1_{1,\delta}$ and $q^{\flat} \in C^{r} S^{1-\delta r}_{1,\delta}$. By Lemma \ref{ell_lemma} $Q^{\sharp} = q^\sharp(x,D)$ is elliptic and thus possesses a parametrix $E^{\sharp} \in S^{-1}_{1,\delta}$.

We have\
$$
  (\partial_t + i Q) E^{\sharp} f = 
  (\partial_t + i Q^\sharp + i Q^\flat) E^{\sharp} f 
  = \partial_t E^{\sharp} f + i Q^{\sharp} E^{\sharp} f 
    + i Q^{\flat} E^{\sharp} f
  =  \partial_t E^{\sharp} f + f +i R^{\sharp} f +
    i Q^{\flat} E^{\sharp} f,
$$
where $R^{\sharp}$ is a regularizing operator.
Therefore
$$
   (\partial_t + i Q) (u - E^{\sharp} f) = 
   -\partial_t E^{\sharp} f  -i R^{\sharp} f - 
   i Q^{\flat} E^{\sharp} f =: \widetilde{f},
$$
where the regularity of the right-hand side $\widetilde{f}$ can be deduced from the following facts
$$
  \partial_t E^{\sharp} f \in 
    C^{\infty}([0,T]; H^{s+1}(\mathbb{R}^n)), \quad
 R^{\sharp} f \in C^{\infty}([0,T]; H^{\infty}(\mathbb{R}^n)),
 \quad
 Q^{\flat} E^{\sharp} f \in C^{\infty}([0,T]; H^{s +\delta r}(\mathbb{R}^n)).
$$
Hence $\widetilde{f} \in C^{\infty}([0,T]; H^{s + \min{(\delta r,1)}}(\mathbb{R}^n))$.

If we put  
$w= u+ E^{\sharp} f$  and $w_0 := u_0 + E^{\sharp} f(0)$, 
then the original Cauchy problem (\ref{one-way-PDE}-\ref{one-way-initial}) is reduced to solving the Cauchy problem
$$
  \partial_t w + i\, Q(x,D) w = \widetilde{f}, 
  \qquad w \mid_{t=0} = w_0,
$$
where the spatial regularity of the source term on the right-hand side has been raised by $\min (\delta r,1)$.

\begin{remark} In case of a homogeneous ($1+1$)-dimensional partial differential equation the precise H\"older-regularity properties  
of classical as well as generalized solutions
have been determined in \cite[Section 3]{GH:04b}.
\end{remark}

%% file: regularity.tex
\chapter{Microlocal analysis of generalized pullbacks of Colombeau functions}

This chapter is based on the article \cite{HalSim:07}.

The pullback of a general distribution by a $C^{\infty}$-function in classical distribution theory, 
as defined in \cite[Theorem 8.2.4]{Hoermander:V1}, exists if the normal bundle of the $C^{\infty}$-function
intersected with the wave front set of the distribution is empty.
These microlocal restrictions reflect also the well-known fact that in general distribution theory one cannot carry out multiplications unrestrictedly,
since the product of two distributions can formally be written as the pullback of a tensor product of the two factors by the diagonal map $\delta:x \mapsto (x,x)$.

Generalized functions in the sense of Colombeau extend distribution theory in a way that it becomes a differential algebra with a
product that preserves the classical product $\cdot: C^{\infty} \times C^{\infty} \rightarrow C^{\infty}$. 
In addition \cite[Proposition 1.2.8]{GKOS:01} states that the Colombeau algebra of generalized functions allows the definition of
a pullback by any c-bounded generalized function.
The classical concept of a wave front set has been extended to generalized functions of Colombeau type in \cite{DPS:98,Hoermann:99,NPS:98}.

In this Chapter we investigate how the generalized wave front set of a Colombeau function $u$ transforms under the 
pullback by a c-bounded map $f$. Our main theorem is a result corresponding to the classical Theorem \cite[Theorem 8.2.4]{Hoermander:V1}.

%Observe that the singular directions in $\wf{f^{\ast}u}$ are caused by singularity either of $u$
% or of $f$.

%O

\section{Transformation of wave front sets}\label{sectionwfset}

In order to obtain a microlocal inclusion (corresponding to \cite[eq. 8.2.4]{Hoermander:V1}) for generalized pullbacks, 
we need the notion of a generalized normal bundle and find a method to transform wave front set under 
c-bounded generalized maps.

In this section we consider a c-bounded generalized map $f\in \colmap{\Omega_1}{\Omega_2}$, 
where $\Omega_1 \subseteq \mathbb{R}^n,\ \Omega_2 \subseteq \mathbb{R}^m$ are open sets and assume that
$\Gamma \subseteq \Omega_2 \times S^{m-1}$ is a closed set. 

Although the pullback of a Colombeau function $u$ by any c-bounded generalized map $f$ is well-defined, 
we cannot derive a general microlocal inclusion 
for the pullback without requiring further properties for the generalized map $f$.
We define an open subdomain $\domainset{f}$ of $\Omega_1 \times S^{m-1}$, 
where the generalized map $(x,\eta) \mapsto {}^T\! df_{\varepsilon}(x) \eta$ has certain properties 
which are needed to obtain a microlocal inclusion relation. This leads to the notion of a generalized normal bundle.

\begin{definition} \label{unfavdef}
Let $f \in\colmap{\Omega_1}{\Omega_2}$ be a c-bounded generalized map, then 
we define the open set $\domainset{f}$ by 
\begin{equation*} \begin{split}
\domainset{f} := &\{ (x,\eta) \in \Omega_1 \times S^{m-1} \mid \exists\  \rm{\ neighborhood\ }X \times V \subseteq \Omega_1 \times S^{m-1} \rm{\ of\ } (x,\eta) \\
&\rm{\ and\ a\ positive\ net\ of\ slow\ scale\ } 
(\sigma_{\varepsilon})_{\varepsilon}, \exists \alpha, \beta\in ]0,\infty[,\exists \varepsilon'\in ]0,1]:\\
&\inf_{(x,\eta)\in X\times V} 
   |\sigma_{\varepsilon}  {}^T\! df_{\varepsilon}(x) \eta | \ge \alpha  \ \rm{and}\\
&\sup_{(x,\eta) \in X\times V^{\perp}} |\sigma_{\varepsilon}  {}^T\! df_{\varepsilon}(x) \eta | \le \beta \rm{\ for\ all\ } \varepsilon < \varepsilon'  \},
\end{split}
\end{equation*}
where $V^{\perp}:= \{\eta \in S^{m-1} \mid \exists \eta_0 \in V: \langle \eta, \eta_0 \rangle = 0 \}$. 
Then the \emph{generalized normal bundle} \index{normal bundle!generalized} of $f$ is defined by
\begin{equation*}
\gnormalb{f} := \{(y,\eta) \in \Omega_2 \times S^{m-1} \mid (x,y) \in \ggraph{f}, (x,\eta) \not \in \domainset{f} \}.
\end{equation*}
This is in correspondence to the classical normal bundle. 

Furthermore, we define the \emph{wave front unfavorable support} \index{wave front unfavorable support} of 
$f$ with respect to a closed set $\Gamma\subseteq \Omega_2 \times S^{m-1}$ by
\begin{eqnarray*} 
\unfavsupp{f}{\Gamma}:= \{ x\in \mathbb{R}^n \mid  (x,y) \in \ggraph{f}, (x,\eta) \not \in \domainset{f}, 
(y,\eta) \in \Gamma \}. 
\end{eqnarray*}
\end{definition}

\begin{example}\label{mgexample}
Consider the c-bounded generalized map defined by 
\begin{eqnarray*}
f_{\varepsilon}(x,y)=(x+ \gamma_{\varepsilon} y, x-\gamma_{\varepsilon} y),
\end{eqnarray*}
where $(\gamma_{\varepsilon})_{\varepsilon}$ is a positive net converging to zero.
\end{example}
Then the transposed Jacobian is
\begin{eqnarray*}
{}^T\!df_{\varepsilon}(x,y) := \left( \begin{array}{cc} 
1 &  1 \\
\gamma_{\varepsilon}  & - \gamma_{\varepsilon} 
\end{array}\right),
\end{eqnarray*}
which is constant with respect to the $(x,y)$ variable, so we put $M_{\varepsilon} := {}^T\!df_{\varepsilon}(x,y)$. 

We observe that $M_{\varepsilon} \eta= (\eta_1 +\eta_2, \gamma_{\varepsilon}(\eta_1-\eta_2))$.
Setting $\nu_1:=( 1/\sqrt{2}, - 1/\sqrt{2})$ and $\nu_2:=( 1/\sqrt{2}, 1/\sqrt{2})$ we observe 
$$
|\sigma_{\varepsilon} M_{\varepsilon} \nu_1| = \sqrt{2} \sigma_{\varepsilon}  \gamma_{\varepsilon}, \quad
|\sigma_{\varepsilon} M_{\varepsilon} \nu_2| = \sqrt{2} \sigma_{\varepsilon}
$$
and $\langle M_{\varepsilon} \nu_1, M_{\varepsilon} \nu_2 \rangle = 0$.

We can write any $\eta \in S^1$  in terms of the orthonormal basis $(\nu_k)_{k=1,2}$ by
$\eta= \langle \eta, \nu_1 \rangle  \nu_1 + \langle \eta, \nu_2 \rangle \nu_2$. 
Note that $\langle \eta, \nu_1 \rangle^2 +\langle \eta, \nu_2 \rangle^2=1$ since $|\eta|=1$. 

For any $\eta \in S^1$ we have
\begin{multline*}
|\sigma_{\varepsilon} M_{\varepsilon} \eta|^2 = 
|\sigma_{\varepsilon} M_{\varepsilon} ( \langle \eta, \nu_1 \rangle \nu_1 + \langle \eta, \nu_2 \rangle \nu_2) |^2 \\
=  \langle \eta, \nu_1 \rangle^2 |\sigma_{\varepsilon} M_{\varepsilon} \nu_1|^2  +  
\langle \eta, \nu_2 \rangle^2 |\sigma_{\varepsilon} M_{\varepsilon} \nu_2|^2 \\
= 2 \sigma_{\varepsilon}^2 (\langle \eta, \nu_1 \rangle^2 \gamma_{\varepsilon}^2 + 
 \langle \eta, \nu_2 \rangle^2 ) = 2 \sigma_{\varepsilon}^2 (1 -  \langle \eta, \nu_1 \rangle^2  (1-\gamma_{\varepsilon}^2 )). 
\end{multline*}

Now for any $\eta_0 \in S^1$ with $\eta_0 \ne \pm \nu_1$ we can find a neighborhood $V_0$ such that there exists a 
$\delta \in ]0,1]$ with the property that $|\langle \eta, \nu_1 \rangle| \le 1-\delta$ holds for all $\eta \in V_0$. 
Choosing $\sigma_{\varepsilon} := 1$ 
we obtain 
\begin{multline*}
\inf_{(x,\eta)\in \mathbb{R}^n  \times V_0} |\sigma_{\varepsilon}  {}^T\! df_{\varepsilon}(x) \eta | = 
\sqrt{2}\sigma_{\varepsilon} \inf_{\eta \in V_0}  \sqrt{1 -  \langle \eta, \nu_1 \rangle^2  (1-\gamma_{\varepsilon}^2 )} \\
= \sqrt{2} \sigma_{\varepsilon} \sqrt{1- (1-\delta)^2  (1-\gamma_{\varepsilon}^2 )} \rightarrow \sqrt{2}  \sqrt{2 \delta-\delta^2} 
\end{multline*}
and
\begin{multline*}
\sup_{ (x,\eta)\in \mathbb{R}^n  \times V_0^{\perp}}    |\sigma_{\varepsilon}  {}^T\! df_{\varepsilon}(x) \eta | = 
\sqrt{2} \sigma_{\varepsilon} \sup_{\eta \in S^1}  \sqrt{1 -  \langle \eta, \nu_1 \rangle^2  (1-\gamma_{\varepsilon}^2 )} \\
\le  \sqrt{2}\sigma_{\varepsilon}   =  \sqrt{2},
\end{multline*}
so it follows immediately that $\domainset{f} \supseteq \mathbb{R}^2 \times (S^1 /\{\pm \nu_1\})$.

It remains to check whether $\mathbb{R}^2 \times \{\pm \nu_1\}$ belongs to $\domainset{f}$:
Let $V \subseteq S^1$ be a neighborhood of $\nu_1$, then it follows that
$$
\inf_{(x,\eta)\in \mathbb{R}^n  \times V} |\sigma_{\varepsilon}  {}^T\! df_{\varepsilon}(x) \eta | =  
\sqrt{2} \sigma_{\varepsilon} \inf_{\eta \in V}  \sqrt{1 -  \langle \eta, \nu_1 \rangle^2  (1-\gamma_{\varepsilon}^2 )} = 
\sqrt{2} \sigma_{\varepsilon}   \gamma_{\varepsilon} ,
$$
and 
$$
\sup_{ (x,\eta)\in \mathbb{R}^n  \times V^{\perp}}    |\sigma_{\varepsilon}  {}^T\! df_{\varepsilon}(x) \eta | = 
\sqrt{2} \sigma_{\varepsilon} \sup_{\eta \in V^{\perp}}  \sqrt{1 - \langle \eta, \nu_1 \rangle^2  (1-\gamma_{\varepsilon}^2 )} 
= \sqrt{2} \sigma_{\varepsilon}, 
$$
since the set $V^{\perp}$ by definition contains a vector $\nu'$ with $\langle \nu_1, \nu^{\prime}\rangle =0$.
Any slow scaled net $(\sigma_{\varepsilon})_{\varepsilon}$ with the property that $\sqrt{2} \sigma_{\varepsilon} \gamma_{\varepsilon}$ is bounded
away from zero, satisfies $\sigma_{\varepsilon} \rightarrow \infty$ (since $\gamma_{\varepsilon}$ tends to zero).
It immediately follows that $\domainset{f} \cap (\mathbb{R}^2 \times \{\pm \nu_1\}) = \emptyset$ , 
thus $\domainset{f} = \mathbb{R} \times (S^1 \backslash \{\pm \nu\})$.

\begin{remark} \label{closed_set_remark}
Note that any non-empty closed set $\Gamma \subseteq \Omega_2 \times S^{m-1}$ can be
considered as a set-valued map 
$$
\supp{(\Gamma)}:= \pi_1(\Gamma) \rightarrow S^{m-1}, \quad y \mapsto \Gamma_y,
$$
where $\Gamma_y := \pi_2(\Gamma \cap(\{y\}\times S^{m-1}))$. 
Then Theorem \ref{set_valued_map_main_theorem} yields that $y \mapsto \Gamma_y$ is upper semi-continuous, i.e.
for all $y\in \supp{(\Gamma)}$ and $W$ some open neighborhood of $\Gamma_y$, there exists some open neighborhood $Y \subseteq \Omega_2$ 
of $y$ such that $\Gamma_Y = \Gamma_{Y\cap\supp{(\Gamma)}} \subseteq W$, where  $\Gamma_Y = \bigcup_{y\in Y} \Gamma_y$ 
as in Definition \ref{set_valued_map_def}.
Due to Proposition \ref{usc_corr_proposition} the upper semi-continuity of $\Gamma$ yields that $\supp{(\Gamma)}$ is a closed subset of $\Omega_2$, so
we obtain for all $y\not\in \supp{(\Gamma)}$, that there exists some neighborhood $Y$ of $y$ such that $\Gamma_Y =\emptyset$.
  
Finally it is straight-forward to prove that for all $Y_0 \csub \Omega_2$ and $W$ some open neighborhood of 
$\Gamma_{Y_0}$, there exists an open neighborhood $Y \subseteq \Omega_2$ of $Y_0$ such that
$\Gamma_Y \subseteq W$.
\end{remark}

\begin{lemma} \label{normal_bundle_lemma}
The generalized normal bundle $\gnormalb{f}$ and the wave front unfavorable support $\unfavsupp{f}{\Gamma}$ of $f$ (with respect to $\Gamma$) are closed sets.
If $\gnormalb{f} \cap \Gamma =\emptyset$, then $\unfavsupp{f}{\Gamma} = \emptyset$.
\end{lemma}
\begin{proof}

Let $x_0 \not \in \unfavsupp{f}{\Gamma}$ and put $Y_0 := \ggraph{f}_{x_0}$, then since $\domainset{f}$ is open there exist open 
neighborhoods $X_0$  of $x_0$ and  $V_0$ of $\Gamma_{Y_0}$ such that $X_0 \times V_0 \subseteq \domainset{f}$.
As pointed out in Remark \ref{closed_set_remark} there exists a neighborhood $Y$ of $Y_0$, such that
$V_0$ is still a neighborhood of $\Gamma_Y$.  The upper semi-continuity of $\ggraph{f}$ provides a neighborhood $X'$ of
$x_0$ such that $\ggraph{f} \cap (X' \times \Omega_2) \subseteq X' \times Y$.
Now we can choose a smaller neighborhood $X_1 \subset X' \cap X_0$ (such that $X' \cap X_0$ is a neighborhood of $X_1$) 
of $x_0$, 
such that $X_1 \times V_0 \subseteq \domainset{f}$ and $\ggraph{f} \cap (X_1 \times \Omega_2) \subseteq X_1 \times Y$.  
For all $x_1\in X_1$ and $y\in \Omega_2$ with $(x_1,y)\in \ggraph{f}$, it holds that $(x_1,y) \in X_1 \times Y$ and if 
$(y,\eta) \in \Gamma$ it follows
that $\eta \in \Gamma_Y \subseteq V$ and thus $(x_1,\eta) \in \domainset{f}$. 
It follows that $x_1\not\in \unfavsupp{f}{\Gamma}$ for all $x_1\in X_1$, $X_1 \cap \unfavsupp{f}{\Gamma} =\emptyset$ 
and thus $\unfavsupp{f}{\Gamma}^c$ is an open set.
\end{proof}

A substantial step in the proof of the main theorem is the application of a generalized stationary phase theorem (cf. Appendix B). 
In order to obtain a lower bound for the gradient of the occurring phase function we have to consider the map
\begin{eqnarray*}
\Omega_1 \times S^{m-1} \times S^{n-1} & \rightarrow& S^{n-1} \\
(x,\eta,\xi)	&	\mapsto & \left|\frac{{}^T df_{\varepsilon}(x) \eta}{|{}^T df_{\varepsilon}(x) \eta|} -\xi \right|.
\end{eqnarray*}

\begin{lemma}
Let $W \subseteq S^{n-1}$ be a neighborhood of some $\eta_0 \in S^{n-1}$, then we have for any
$\eta_1,\eta_2 \in W$ that
$$
\frac{\eta_1 -\eta_2}{|\eta_1 - \eta_2|} \in W^{\perp} := \{\eta \in S^{n-1} \mid \exists \eta' \in W : \langle \eta', \eta\rangle = 0 \}.
$$
\end{lemma}

\begin{proof}
Assume $W \subseteq C_{\delta}(\eta_0) := \{\eta \in S^{n-1} \mid \langle \eta_0, \eta \rangle \ge 1- \delta\}$ for some 
$\delta \in ]0,1[$. If no such $\delta \in ]0,1[$ exists $C_{0}(\eta_0) \subset W$ implies $W^{\perp} = S^{n-1}$, so the statement
is trivial.
Observe that 
$$
\langle \frac{\eta_1 +\eta_2}{|\eta_1 + \eta_2|} , \eta_0 \rangle =
\frac{1}{|\eta_1 + \eta_2|} \left( \langle \eta_1, \eta_0 \rangle+
\langle \eta_2 , \eta_0 \rangle \right) \ge \frac{2 (1-\delta)}{\sqrt{2(1+ \langle \eta_1 ,\eta_2 \rangle})} \ge 1-\delta
$$
and since
$$
\langle \frac{\eta_1 +\eta_2}{|\eta_1 + \eta_2|} , \frac{\eta_1 -\eta_2}{|\eta_1 - \eta_2|} \rangle 
= \frac{|\eta_1|^2 - |\eta_2|^2}{|\eta_1 + \eta_2||\eta_1 - \eta_2|}=0
$$
which implies $\frac{\eta_1 -\eta_2}{|\eta_1 - \eta_2|} \in W^{\perp}$.
\end{proof}

So we introduce the following notation:
\begin{lemma}\label{notionlemma}
We define the Colombeau map $M \in \colmap{\domainset{f}}{S^{n-1}}$ by
\begin{eqnarray*}
M_{\varepsilon}(x,\eta):=\frac{{}^T\!df_{\varepsilon}(x) \eta }{|{}^T\!df_{\varepsilon}(x) \eta|}
\end{eqnarray*}
on the domain $\domainset{f}$, then $\eta \mapsto M_{\varepsilon}(x,\eta)$ defines an equi-continuous Colombeau function for fixed $x$. 
\end{lemma}
\begin{proof}
The map $g:\xi \rightarrow \frac{\xi}{|\xi|}$
is equi-continuous on $\mathbb{R}^n / B_{\delta}(0)$ for any fixed $\delta>0$ since
\begin{eqnarray*}
\left| \frac{\xi}{|\xi|} - \frac{\eta}{|\eta|} \right| = \frac{1}{|\xi|} \left| \xi - \frac{\eta|\xi|}{|\eta|} \right| =\frac{1}{|\xi|}
\left| \xi -\eta + \frac{\eta}{|\eta|} \left( |\eta|-|\xi| \right) \right| \le \frac{2}{|\xi|} |\xi-\eta|.
\end{eqnarray*}
Let $(x_0,\eta_0) \in D_{f}$, then we can find some neighborhood $X\times V \subseteq \Omega_1 \times S^{m-1}$ of $(x_0,\eta_0)$ such that
\begin{eqnarray*}
\inf_{(x,\eta)\in X\times V}  | \sigma_{\varepsilon}  {}^T\! df_{\varepsilon}(x) \eta|  &\ge& \alpha \quad \text{and } \sup_{(x,\eta) \in X\times V^{\perp}} |  \sigma_{\varepsilon}  {}^T\! df_{\varepsilon}(x) \eta | \le \beta
\end{eqnarray*}
for some slow scaled net $(\sigma_{\varepsilon})_{\varepsilon}$ and $\alpha, \beta \in ]0, \infty[$.
Then we conclude that
\begin{equation*}\begin{split}
&|M_{\varepsilon}(x,\eta)- M_{\varepsilon}(x,\xi)| \le \frac{2}{  | {}^T df_{\varepsilon}(x) \eta |} \left|{}^T\! df_{\varepsilon}(x) \eta-{}^T\! df_{\varepsilon}(x) \xi  \right| \\ &=  \frac{2}{ |  \sigma_{\varepsilon} {}^T\! df_{\varepsilon}(x) \eta| } 
\left| \sigma_{\varepsilon} {}^T\! df_{\varepsilon}(x) (\eta-\xi)  \right| \le
2 \frac{ \sup_{(x,\zeta) \in X\times V^{\perp}} |\sigma_{\varepsilon} {}^T\! df_{\varepsilon}(x) \zeta |}{ 
 \inf_{(x,\eta)\in X\times V}   |\sigma_{\varepsilon} {}^T\! df_{\varepsilon}(x) \eta| } \\
&\cdot |\xi -\eta| \le  2 |\xi- \eta| \frac{\beta}{\alpha}
\end{split}
\end{equation*}
uniformly for all $x\in X_0$ and $\xi,\eta \in V$.
\end{proof}

\begin{remark}
By the preceding Lemma we have obtained that $\left[(\eta \mapsto M_{\varepsilon}(x,\eta)\mid_{\domainset{f}})_{\varepsilon}\right]$ is an equi-continuous Colombeau map, thus it follows by Proposition \ref{generalized_graph_equicontinuous_proposition} that
\begin{eqnarray*}
\ggraph{M}_{x,\eta} &=&
 \left\{ \xi \in S^{n-1} \mid \exists \text{\ a \ net \ }(x_{\varepsilon})_{\varepsilon} 
\text{\ in\ } \Omega_1: \lim_{\varepsilon\rightarrow 0} x_{\varepsilon}=x \in \Omega_1 \right. 
\\ &&\left. \text{ and } y\in\cp{(M_{\tau(\varepsilon)}(x_{\varepsilon},\eta))_{\varepsilon}}  
 \text{for\ some\  map}\ \tau \in \mathcal{T} \right\},
\end{eqnarray*}
which simplifies the determination of the generalized graph $\ggraph{M}$ considerably.
\end{remark}

\begin{definition} \label{transformed_wf_set}
Let $f\in\colmap{\Omega_1}{\Omega_2}$ be a c-bounded generalized map and assume that $\domainset{f} \ne \emptyset$.
We define the pullback of some closed set $\Gamma \subseteq  \Omega_1 \times S^{n-1}$ by
$$
f^{\ast} \Gamma := \ggraph{M} \circ ( \{\bullet\} \times (\Gamma \circ \ggraph{f})_{\bullet} ),
$$
where $\{\bullet\} \times (\Gamma \circ \ggraph{f})_{\bullet}$ denotes the set-valued map defined by
$x \mapsto \{x\} \times (\Gamma \circ \ggraph{f})_{x}$,
using the notation from Definition \ref{set_valued_map_def} and \ref {set_valued_map_inverse}. 

If $\domainset{f} = \emptyset$ we put $(f^{\ast} \Gamma)_x =  S^{n-1}$ for all $x\in\Omega_1$.
\end{definition}

\begin{remark}
Obviously the set-valued map $f^{\ast} \Gamma$ is defined as the composition of upper semi-continuous and 
locally bounded set-valued maps, thus $f^{\ast} \Gamma$ is again an upper semi-continuous and locally bounded set-valued map 
$\Omega_1 \mapsto \mathcal{F}_0(S^{n-1})$. According to Theorem \ref{set_valued_composition_theorem} the pullback is 
defined by 
$$
(f^{\ast} \Gamma)_x = \bigcup_{\eta \in \widetilde{\Gamma}_x}\ggraph{M}_{x,\eta}
$$
where $\widetilde{\Gamma}_x:= \bigcup_{y\in \ggraph{f}_x} \Gamma_y$.
The upper semi-continuity of $f^{\ast} \Gamma$ yields that for any neighborhood $W \subseteq S^{n-1}$ of $(f^{\ast} \Gamma)_x$ there exists
some neighborhood $X$ of $x$ such that $(f^{\ast} \Gamma)_X \subseteq W$.
\end{remark}

\section{Generalized pullbacks of Colombeau functions}

In this section we prove the main result which gives a microlocal inclusion relation for the generalized pullback of a Colombeau function.
The proof of the theorem relies on a generalized stationary phase theorem, the details of which are discussed in the Appendix B.

\begin{definition} \label{specialsets}
Let $f \in \colmap{\Omega_1}{\Omega_2}$ be a c-bounded generalized map. Then we call $f$ slow scaled in all
derivatives on the open set $X_0\subseteq \Omega_1$, if for all $\alpha\in\mathbb{N}^n$ there exists slow scaled nets $(r_{\alpha,\varepsilon})_{\varepsilon}$
such that
\begin{eqnarray}\label{slowscale}
\sup_{x \in X_0} |\partial^{\alpha} f_{\varepsilon}(x)| \le C_{\alpha} r_{\alpha,\varepsilon} \ \rm{as} \ \varepsilon \rightarrow 0
\end{eqnarray}
holds, where $C_{\alpha}$ are constants. Furthermore we call
$f$ \emph{slow scaled in all derivatives} \index{slow scaled!in all derivatives} at $x_0\in\Omega_1$, if there exists a neighborhood $X_0$ of $x_0$ such that (\ref{slowscale}) holds for all $\alpha \in \mathbb{N}^n$.

Define the sets
\begin{eqnarray*}
S^f := \{ x \in \Omega_1 \mid f \rm{\ is\ slow scaled\ in\ all\ derivatives\ at\ } x\}
\end{eqnarray*}
and 
\begin{eqnarray*}
K^f(u) := \bigcap_{k \in  \widetilde{\Omega_2}  } \pi_1(\ggraph{f} \cap \rm{supp}(u-k) \times \Omega_2) .
\end{eqnarray*}
\end{definition}

\begin{remark}
The set $\pi_1(\ggraph{f} \cap \Omega_1 \times \rm{supp}(u-k))$ is closed for all $k \in \widetilde{\Omega_2}$ due to Proposition \ref{usc_corr_proposition} since $\ggraph{f}$ is upper semi-continuous and $\rm{supp}(u-k)$ is closed. This implies that $\ksupp{f}(u)$ is closed in the relative topology of $\Omega_1$.
\end{remark}

\begin{lemma}\label{lemmacool}
Let $f \in \colmap{\Omega_1}{\Omega_2}$ be a c-bounded generalized map. Furthermore let $x_0 \not \in \unfavsupp{f}{\Gamma} \cup (\slsupp{f})^c$ and $Y_0 := \pi_2(\ggraph{f} \cap \{x_0\}\times \Omega_2)$. If $W$ is some open neighborhood of $(f^{\ast} \Gamma)_{x_0}$ , 
then there exist neighborhoods $X$ of $x_0$, $V$ of $\Gamma_{Y_0}$ 
and $Y$ of $Y_0$ 
with the following properties:
\begin{equation*}\begin{split}
&X \times V \subseteq \domainset{f}(\Gamma),\ (f^{\ast} \Gamma)_X \subseteq W,\
f_{\varepsilon}(X) \subseteq Y \ \rm{for\ all\ }\ \varepsilon < \varepsilon', \\ 
&(f_{\varepsilon}) \rm{\ is\ slow scaling\ in\ all\ derivatives\ on\ } X \rm{\ and\ } 
\Gamma_{Y} \subseteq V \end{split}
\end{equation*}
Furthermore there exists a positive constant $c>0$, such that
\begin{equation*}
\inf_{(x,\eta,\xi)\in X \times V \times W^c} \left|M_{\varepsilon}(x,\eta) - \xi\right|>c \ \rm{for\ all\ }\varepsilon< \varepsilon'
\end{equation*}
holds (we are using the notation from Lemma \ref{notionlemma}).
\end{lemma}

\begin{proof}
From $x_0 \not \in \unfavsupp{f}{\Gamma}$ it follows that $\{x_0 \}\times \Gamma_{Y_0}$
is a compact subset of the open set $\domainset{f}$. Thus we can find some
neighborhood $X' \times V' \subseteq \domainset{f}$ of $\{x_0\} \times \Gamma_{Y_0}$.
Since $x_0 \in \slsupp{f}$ we can assume without loss of generality that $f$ is slow scale in all derivatives on
the compact set $X'$.

Since $W$ is a neighborhood of $(f^{\ast}\Gamma)_{x_0}$ we have that
\begin{eqnarray*}
(f^{\ast}\Gamma)_{x_0}  = \pi_2( \ggraph{M} \cap \{x_0\} \times \Gamma_{Y_0} \times S^{n-1}) \csub W
\end{eqnarray*}
where $\ggraph{M}$ denotes the generalized graph of the generalized map defined by $(x,\eta) \mapsto M_{\varepsilon}(x,\eta)$ on the open domain $\domainset{f}$.
By Lemma \ref{generalized_graph_approximation_property} 
there exist neighborhoods $X'', V''$ of $x_0$ resp. $\Gamma_{Y_0}$ such that
\begin{eqnarray} \label{propcool}
M_{\varepsilon}(X'' \times V'') \subseteq W 
\end{eqnarray}
for all $\varepsilon < \varepsilon'$. Let $V:= V' \cap V''$, then by Remark \ref{closed_set_remark} we can find some neighborhood $Y$ of $Y_0$ such that $\Gamma_Y \subseteq V$. By Lemma \ref{generalized_graph_approximation_property} there exists some neighborhood $X'''$ of $x_0$ such 
that $f_{\varepsilon}(X''') \subseteq Y$ for small $\varepsilon$. Let $X:=X' \cap X'' \cap X'''$ and
$Z:= X \times V \times W^c$. Then there exists $(x_{\varepsilon},y_{\varepsilon},\xi_{\varepsilon})\in Z$ (note that $(x,\eta,\xi) \mapsto |M_{\varepsilon}(x,\eta) - \xi |$ is a continuous function for each $\varepsilon \in ]0,1]$ and $Z$ is a compact set) such that
\begin{equation*}
c_{\varepsilon} := \inf_{(x,\eta,\xi) \in Z} |M_{\varepsilon}(x,\eta) - \xi | = |M_{\varepsilon}(x_{\varepsilon},\eta_{\varepsilon}) - \xi_{\varepsilon} |
\end{equation*}
holds for some net $(x_{\varepsilon}, \eta_{\varepsilon}, \xi_{\varepsilon})_{\varepsilon}$.
By (\ref{propcool}) we have that $c_{\varepsilon} > c >0$ holds for all $\varepsilon < \varepsilon'$, where $c$ is a constant.
\end{proof}

\begin{lemma} \label{minilemma}
Let $v,w \in \mathbb{R}^n$ , then the inequality
$$
|\alpha v -(1-\alpha)w|^2 \ge \frac{|v|^2 |w|^2}{|v+w|^2} \left(1- (\langle v/|v|, w/|w|\rangle)^2\right)
$$ 
holds for all $\alpha \in [0,1]$.
\end{lemma}
\begin{proof}
Observe that
\begin{multline*}
|\alpha v -(1-\alpha)w|^2 = \alpha^2 |v|^2 - 2\alpha(1-\alpha) \langle v, w \rangle + (1-\alpha)^2 |w|^2 \\
 = \alpha^2 |v+w|^2  -2 \alpha (|w|^2 + \langle v,w \rangle) + |w|^2
\end{multline*}
and differentiation in $\alpha$ shows that the expression has a local extremum at $\alpha_0:=(\langle v,w \rangle + |w|^2)/|v+w|^2$.
Since the second derivative with respect to $\alpha$ is $|v-w|^2 > 0$ for $v\ne -w$ (the case $v=-w$ is trivial), we have a local minimum at $\alpha_0$. Finally we obtain
$$
|\alpha v -(1-\alpha)w|^2  \ge |\alpha_0 v -(1-\alpha_0)w|^2 = \frac{|v|^2 |w|^2}{|v+w|^2} \left(1- \langle v/|v|, w/|w|\rangle^2\right),
$$
by checking the boundary case $\alpha\in \{0,1\}$ we verify that it is a global lower bound. 
\end{proof}

\begin{theorem} \label{maintheorem}
Let $f \in \colmap{\Omega_1}{\Omega_2}$ be a c-bounded generalized map with representative $(f_{\varepsilon})_{\varepsilon}$.
For $u \in \col{ \Omega_2}$ with representative $(u_{\varepsilon})_{\varepsilon}$, we define the pullback  $f^{\ast} u$  by
\begin{eqnarray*}
f^{\ast} u := (u_{\varepsilon}(f_{\varepsilon}(x)))_{\varepsilon} +\mathcal{N}(\Omega_1).
\end{eqnarray*}
It satisfies the microlocal inclusion relation
\begin{eqnarray*}
\wf{f^{\ast} u} \subseteq f^{\ast} \wf{u}\ \bigcup\ \left(\unfavsupp{f}{\wf{u}} \times S^{n-1}\right)\ \bigcup\ 
\left(\left(\ksupp{f}(u) \cap (\slsupp{f})^c \right) \times S^{n-1}\right).
\end{eqnarray*}
\end{theorem}

\begin{remark}
It is apparent that the microlocal inclusion relation for $\wf{f^{\ast} u}$ is splitted into three different parts:
\begin{itemize}
\item $f^{\ast} \wf{u}$: This part corresponds to its classical counterpart (cf. \cite[eq. 8.2.4.]{Hoermander:V1}).
\item $\unfavsupp{f}{\wf{u}} \times S^{n-1}$: According to Lemma \ref{normal_bundle_lemma} this set is non-empty whenever
 $\gnormalb{f} \cap \wf{u}$ is non-empty. 
 In the classical case the pullback of a distribution $u$ by a smooth map $f$ is not defined if $\gnormalb{f} \cap \wf{u}$
 is non-empty, so there is no classical counterpart to this set.
\item $\left(\ksupp{f}(u) \cap (\slsupp{f})^c \right) \times S^{n-1}$: 
$\left(\ksupp{f}(u) \cap (\slsupp{f})^c \right)$ is the set of all points $x\in \Omega_1$, where the generalized map 
$f$ is not of slow-scale in all derivatives, whenever the Colombeau map $u$ is not constant in a neighborhood of $\ggraph{f}_x$.
This set is required since we allow non-regular Colombeau maps $f$, which are only c-bounded. So the singularities of the map $f$ may cause
singularities of $f^{\ast} u$. There is no classical correspondence, since in \cite[Theorem 8.2.4]{Hoermander:V1} 
the map $f$ is assumed to be smooth.
\end{itemize}
\end{remark}

\begin{proof}[Proof of Theorem \ref{maintheorem}]
According to \cite[Proposition 1.2.8]{GKOS:01} the pullback $f^{\ast} u$ is a well-defined Colombeau function in $\mathcal{G}(\Omega_1)$. Let $(x_0,\xi_0) \in \Omega_1 \times S^{n-1}$ and $Y_0 =\{ y \in \Omega_2 \mid (x_0, y) \in \ggraph{f}(\Omega_1) \}$.
We set $\Gamma:=\wf{ u }$ and $\Gamma_y:=\{\eta \mid (y,\eta) \in \Gamma\}$.
In order to prove the statement we show that if $(x_0,\xi_0) \not \in f^{\ast} \Gamma  \bigcup \unfavsupp{f}{\Gamma} \times S^{n-1} \bigcup \left(\ksupp{f}(u) \cap (\slsupp{f})^c \right) \times S^{n-1}$ it follows that $(x_0, \xi_0)\not \in \wf{f^{\ast} u} $.

So let us assume that $(x_0,\xi_0) \not \in f^{\ast} \Gamma  \bigcup \unfavsupp{f}{\Gamma} \times S^{n-1} \bigcup \left(\ksupp{f}(u) \cap (\slsupp{f})^c \right) \times S^{n-1}$. Then
we have $\xi_0 \not \in (f^{\ast} \Gamma)_{x_0}$ and $x_0 \not \in \unfavsupp{f}{\Gamma} \bigcup \ksupp{f}(u) \cap (\slsupp{f})^c$. This allows splitting the proof in two parts for the cases $x_0 \not\in \ksupp{f}(u) \cup 
\unfavsupp{f}{\Gamma}$ and 
$x_0 \not\in (\slsupp{f})^c \cup \unfavsupp{f}{\Gamma}$.\\

In the first case ($x_0 \not\in \ksupp{f}(u) \cup \unfavsupp{f}{\Gamma}$): Since $x_0 \not\in \ksupp{f}(u)$ 
there exists a generalized constant $k\in \widetilde{\mathbb{R}}^n$ with representative $(k_{\varepsilon})_{\varepsilon}$, such that $ \ggraph{f} \cap \{x_0\}\times \rm{supp}(u-k) = \emptyset$. Hence
for all $y \in Y_0$ it holds that $y \not \in \rm{supp}(u-k)$.
Since $Y_0$ and $\rm{supp}(u-k)$ are two disjoint closed sets,
we can find a closed neighborhood $Y$ of $Y_0$ such that $Y \cap \rm{supp}(u)=\emptyset$.
Let $\chi \in C_{c}^{\infty}(\Omega_2)$ with the property that $\chi \equiv 1$ on some compact neighborhood $Y' \subseteq Y^{\circ}$ of $Y_0$
and $\rm{supp}(\chi) \subseteq Y$. By Lemma \ref{generalized_graph_approximation_property}
there exists a neighborhood $X'$ of $x_0$ and $\varepsilon'\in]0,1]$ such that $f_{\varepsilon}(X') \subseteq Y'$ for $\varepsilon<\varepsilon'$.

In order to show that $\xi_0 \not \in (\wf{f^{\ast}u})_{x_0}$ for $x_0 \not\in \ksupp{f}(u)$ we have to find a smooth function $\varphi$ with support on a neighborhood of $x_0$, such that $\mathcal{F}(f^{\ast}u \varphi)$ is rapidly decreasing on some neighborhood of $\xi_0$.
We choose $\varphi$ to be a smooth function with $\rm{supp}(\varphi) \subseteq X'$. For all $\varepsilon < \varepsilon'$ the
identity $f_{\varepsilon}^{\ast} u \cdot \varphi = f_{\varepsilon}^{\ast}(\chi u) \cdot \varphi$ holds, since the functions $\chi$ and $\varphi$ where chosen
such that $ (\chi \circ f_{\varepsilon})\cdot \varphi \equiv 1$ for $\varepsilon < \varepsilon'$.
We have that $\chi \cdot (u_{\varepsilon}-k_{\varepsilon}) \in \colneg{\Omega_2}$ and it follows that 
$f^{\ast}u\cdot \varphi$ is $\mathcal{G}^{\infty}$, so $\xi_0 \not \in (\wf{f^{\ast}u})_{x_0}$. \\

The second case ($x_0 \not \in \unfavsupp{f}{\Gamma} \cup (\slsupp{f})^c$ and $\xi_0 \not \in (f^{\ast} \Gamma)_{x_0}$) is more difficult to prove:
Again we have to find a smooth cutoff function $\varphi$ with compact support containing $x_0$, such that 
$\mathcal{F}(f^{\ast}u \varphi)$ is rapidly decreasing on some neighborhood of $\xi_0$. The main step will be the application of the generalized
stationary phase theorem (cf. Appendix B), which requires a suitable lower bound on the norm of the gradient 
of the occurring phase function (it will turn out that $x_0 \not \in \unfavsupp{f}{\Gamma}$ is essential) and slow-scaledness of the phase function in all derivatives 
in a neighborhood of $x_0$.

First of all we note that $(f_{\varepsilon})$ is slow scaling in all derivatives in some neighborhood of $x_0$, since $x_0\in \slsupp{f}$.

Now let $W \subseteq S^{n-1}$ be an open neighborhood of $(f^{\ast} \Gamma)_{x_0}$  such that $W^c$ is a neighborhood of $\xi_0$ (this is possible since $\xi_0 \not \in (f^{\ast} \Gamma)_{x_0}$), then Lemma \ref{lemmacool} implies that there exist neighborhoods $X$ of $x_0$, $Y$ of $Y_0$, 
and $V$ of $\Gamma_{Y_0}$ such that $(f^{\ast} \Gamma)_X \subseteq W$, $f_{\varepsilon}(X) \subseteq Y$ for all $\varepsilon < \varepsilon'$, 
$\Gamma_{Y} \subseteq V$, and 
\begin{equation}\label{eq1}
\inf_{(x,\eta_1,\xi_1) \in X\times V \times W^c}  \left|M_{\varepsilon}(x,\eta_1) -  \xi_1 \right| > d >0
\end{equation}
holds for all $\varepsilon < \varepsilon'$, where $d$ is some positive constant. Without loss
of generality $(f_{\varepsilon})$ is slow scaling in all derivatives on $X$. 

Let $\chi \in C_{c}^{\infty}(\Omega_2)$ with the property that $\chi \equiv 1$ on some compact neighborhood $Y' \subseteq Y^{\circ}$ of $Y_0$
and $\rm{supp}(\chi) \subseteq Y$. By Lemma \ref{generalized_graph_approximation_property} there exists a neighborhood $X'$ of $x_0$ and $\varepsilon'\in]0,1]$ such that 
$f_{\varepsilon}(X') \subseteq Y'$ for $\varepsilon<\varepsilon'$. Without loss of generality we can assume that $X' \subseteq X$.

In order to show that $\xi_0 \not \in \wf{f^{\ast} u}_{x_0}$ we are going to prove that 
there exists a smooth function $\varphi$ with support on a neighborhood of $x_0$, such that $\mathcal{F}(f^{\ast}u \varphi)$ is rapidly decreasing on $W^c$ (which is a neighborhood of $\xi_0$), which implies $(x_0, \xi_0) \not \in \wf{f^{\ast} u}$.

We choose $\varphi$ to be a smooth function with $\rm{supp}(\varphi) \subseteq X'$. For all $\varepsilon < \varepsilon'$ the
identity $(f_{\varepsilon}^{\ast}) u_{\varepsilon}\ \varphi = f_{\varepsilon}^{\ast}(\chi u_{\varepsilon})\ \varphi$ holds, since the functions $\chi$ and $\varphi$ where chosen such that $ (\chi \circ f_{\varepsilon})\cdot \varphi \equiv 1$ for $\varepsilon < \varepsilon_1$. 

Set $\widetilde{V}:=\{\eta \in S^{m-1}\mid  \exists \eta_1 \in V, \lambda \in \mathbb{R}_+: \eta=\lambda \cdot \eta_1 \}$ and
$\widetilde{W}:=\{\xi \in S^{m-1}\mid  \exists \xi \in W, \lambda \in \mathbb{R}_+: \xi=\lambda \cdot \xi_1 \}$. 
Obviously $\four{u \chi}$ is rapidly decreasing on $V^c$. We have that
\begin{eqnarray*}
&&
|\four{( f^{\ast} u)_{\varepsilon} \varphi}(\xi)| = |\four{(f^{\ast} \chi u)_{\varepsilon} \varphi } (\xi)| \\
&=&\left| \int_{\mathbb{R}^m} \widehat{\chi u_{\varepsilon}}(\eta) 
\left(\int_{\mathbb{R}^n} \exp{(i \langle f_{\varepsilon} (x), \eta \rangle - i\langle x, \xi \rangle )} \varphi(x) \m{x} \right) \m{\eta}\right|\\
&=&\left|\int_{\mathbb{R}^m} \widehat{\chi u_{\varepsilon}}(\eta) I_{\varepsilon}(\xi, \eta) \m{\eta} \right| 
= \left|\int_{\widetilde{V}} \widehat{\chi u_{\varepsilon}}(\eta) I_{\varepsilon}(\xi, \eta) \m{\eta} \right| + 
\left|\int_{\widetilde{V}^c} \widehat{\chi u_{\varepsilon}}(\eta) I_{\varepsilon}(\xi, \eta) \m{\eta} \right|,
\end{eqnarray*}
where we have set
\begin{eqnarray*}
I_{\varepsilon}(\xi, \eta)  := \int \exp{(i \langle f_{\varepsilon} (x), \eta \rangle - i\langle x, \xi \rangle )} \varphi(x) \m{x}.
\end{eqnarray*}
We intend to apply the stationary phase theorem \ref{statphas} (cf. the Appendix B) with
\begin{eqnarray*}
\omega&:=&|\xi|+|\eta| \\
\phi_{\varepsilon} &:=& \langle f_{\varepsilon} (x), \frac{\eta}{|\eta|+|\xi|} \rangle - \langle x, \frac{\xi}{|\eta|+|\xi|} \rangle. 
\end{eqnarray*}
Thus we have to find a bound for the gradient of the phase function $|\phi_{\varepsilon}'(x)| = \left|  {}^T\!df_{\varepsilon}(x) \alpha \eta_1  -  (1-\alpha)\xi_1 \right|$ from below for all $\alpha\in[0,1]$, $\eta_1 = \eta/|\eta|$, $\xi_1=\xi/|\xi|$ and $x \in X$.
By optimization in the parameter $\alpha$ (according to Lemma \ref{minilemma}) and using the notation of Lemma \ref{notionlemma}  we can bound the expression
\begin{eqnarray*}
\inf_{x \in X,\eta_1 \in V, \xi_1 \in W,\alpha\in[0,1]}  \left| {}^T\!df_{\varepsilon}(x) \alpha \eta_1  -  (1-\alpha) \xi_1 \right|
\end{eqnarray*}
from below by
\begin{eqnarray*}
&& 
\inf_{(x,\eta_1,\xi_1) \in X\times V \times W^c} \frac{| 
{}^{T}\! df_{\varepsilon}(x) \eta_1| }{   |{}^{T}\! df_{\varepsilon}(x) \eta_1 + \xi_1|}
\sqrt{\left|
1- \Big \langle M_{\varepsilon}(x,\eta_1) , \xi_1 
\Big \rangle^2 \right|} \\
&&=
\inf_{(x,\eta_1,\xi_1) \in X\times V \times W^c} \frac{1}{2} \frac{ |  {}^{T}\! df_{\varepsilon}(x) \eta_1| }
{| {}^{T}\! df_{\varepsilon}(x) \eta_1 + \xi_1|}
\left| M_{\varepsilon}(x,\eta_1) -  \xi_1 \right|  \\
&&\ge   \inf_{(x,\eta_1,\xi_1) \in X\times V \times W^c} \frac{1}{4} 
\frac{  |{}^{T}\! df_{\varepsilon}(x) \eta_1| }{
    |{}^T\!df_{\varepsilon}(x) \eta_1| + 1}
\inf_{(x,\eta_1,\xi_1) \in X\times V \times W^c}
  \left|M_{\varepsilon}(x,\eta_1)
 -  \xi_1 \right| \\
 &&\ge   C \sigma_{\varepsilon}^{-1} \inf_{(x,\eta_1,\xi_1) \in X\times V \times W^c}  \left|
M_{\varepsilon}(x,\eta_1)
 -  \xi_1 \right|,  
\end{eqnarray*}
where $\sigma_{\varepsilon}$ is the slow scaling net from the Definition \ref{unfavdef}. 
By (\ref{eq1}) we have
\begin{equation*} \label{supernet}
\inf_{(x,\eta_1,\xi_1) \in X\times V \times W^c}  \left|
M_{\varepsilon}(x,\eta_1)
 -  \xi_1 \right| > d >0 
\end{equation*}
for $\varepsilon< \varepsilon'$, where $d$ is a positive constant.
Thus the gradient of the phase function is uniformly bounded from below by
$|\phi_{\varepsilon}'(x)| \ge  C d \cdot \sigma_{\varepsilon}^{-1}$ 
for all  $(x,\xi ,\eta) \in X \times \widetilde{V} \times \widetilde{W}^c$. The stationary phase theorem \ref{statphas} yields
\begin{eqnarray}\label{superest1}
\left|I_{\varepsilon}(\xi, \eta) \right| &\le& C_q \varepsilon^{-1} (1+|\xi|+|\eta|)^{-q}
\end{eqnarray}
for all $q \in \mathbb{N}_0$. Note that we use $C, C_p, C_q$ and $C_k,l$ as generic constants.
In the case where $(x, \eta, \xi) \in X \times \widetilde{V}^c \times \widetilde{W}^c$ the stationary phase theorem (now with the phase function $\exp{(-i\langle x, \xi \rangle)}$)
gives
\begin{equation} \begin{split} \label{superexpr}
&|\xi|^{k} |I_{\varepsilon}(\xi, \eta)|=| \xi|^{k}  \left|\int  \exp{(- i\langle x, \xi \rangle )}  \exp{(i \langle f_{\varepsilon} (x), \eta \rangle)
 } \varphi(x)  \m{x} \right| \\
&\sum_{|\alpha| \le k} \left|D^{\alpha}_x \left( \exp{(i |\eta| \langle f_{\varepsilon} (x), \eta/|\eta| \rangle )} \varphi \right)\right|  \\ 
\end{split}
\end{equation}
and by repeated use of the chain rule we obtain the estimate
\begin{equation*}
\begin{split}
&\le\sum_{|\alpha| \le k} \sum_{\beta \le \alpha} c_1(\beta) \sup_{x\in X_0} \left|\partial^{\alpha-\beta} \varphi(x)\right| \sum_{l=1}^{|\beta|} |\eta|^l \cdot \\
&\sum_{\gamma_1+...+\gamma_l = \beta} d(\gamma_1, ... ,\gamma_l)  \prod_{1\le i \le l} 
\sup_{x\in X_0} \left|\partial^{\gamma_i} \langle  f_{\varepsilon}(x), \eta/|\eta| \rangle\right|
\end{split}
\end{equation*}
where $\gamma_1 + \gamma_2 + \cdots + \gamma_l$ denotes a partition of the multi-index $\beta$ in exactly
$l$ multi-indices, that add up componentwise to $\beta$. Using the
notation $|g|_{k} := \max_{|\alpha|=k} |\partial^{\beta} f(x)|$, we can bound the expression by  
\begin{equation*}
\begin{split}
&\le
C_{k,2} \sum_{|\alpha| \le k} \sum_{\beta \le \alpha}  \sum_{l=0}^{|\beta|} |\eta|^l \max_{\gamma_1+...+\gamma_l = \beta}{\left( \prod_{1\le i \le l} 
\sup_{x\in X, \eta \in \widetilde{V}^c} \left|\langle  f_{\varepsilon}(x), \eta/|\eta| \rangle\right|_{|\gamma_i|} \right)} \\
&\le 
C_{k,2} \sum_{|\alpha| \le k} \sum_{\beta \le \alpha}  \sum_{l=0}^{|\beta|} |\eta|^l \max_{0\le j \le |\beta| - l + 1}{    
\left( \sup_{x\in X, \eta \in \widetilde{V}^c} \left|\langle  f_{\varepsilon}(x), \eta/|\eta| \rangle\right|_{j}  \right)^l } \\
&\le
C_{k,3}  (1+|\eta|)^k \max_{l=1}^{k} \max_{0\le j \le k - l + 1}{    
\left( \sup_{x\in X, \eta \in \widetilde{V}^c} \left|\langle  f_{\varepsilon}(x), \eta/|\eta| \rangle\right|_{j}  \right)^l }
.\end{split}
\end{equation*}
Since $f_{\varepsilon}(x)$ is of slow scale in all derivatives at $x_0$, it holds that for all $j \in \mathbb{N}_0$ there exists 
constants $C_j$ and slow scaled nets $r_{\varepsilon,j}$ such that
\begin{equation*}
\sup_{(x,\eta_1) \in X\times V^c}{\left|\langle  f_{\varepsilon}(x), \eta_1  \rangle\right|_{j}} \le C_j r_{\varepsilon,j} 
\end{equation*}
holds for small $\varepsilon$.
Summing up the estimates of the form (\ref{superexpr}) give that there exists constants $C_p$ such that
\begin{eqnarray}\label{superest2}
|I_{\varepsilon}(\xi, \eta)| \le C_p (1+|\eta|)^{-p} \varepsilon^{-1} (1+|\xi|)^{p}
\end{eqnarray} 
holds for all $p\in \mathbb{N}_0$ and $(\eta,\xi) \in \widetilde{V}^c \times \widetilde{W}$.  
We observe that
\begin{equation*}\label{last}
\begin{split}
|\four{( f^{\ast} u)_{\varepsilon} \varphi}(\xi)| =& \left|\int_{\mathbb{R}^m} \widehat{\chi u_{\varepsilon}}(\eta) I_{\varepsilon}(\xi, \eta) \m{\eta} \right|\\ 
\le&
C_q  \varepsilon^{-1} \int_{\widetilde{V}} |\widehat{\chi u_{\varepsilon}}(\eta)|  (1+|\xi|+|\eta|)^{-q} \m{\eta}
\\ &+ 
C_p (1+|\xi|)^{-p} \varepsilon^{-1}  \int_{\widetilde{V}^c} |\widehat{\chi u_{\varepsilon}}(\eta)|  (1+|\eta|)^{p} \m{\eta} 
\end{split}
\end{equation*}
holds and using (\ref{superest1}) and (\ref{superest2}) leads to the upper bound
\begin{equation*}
\begin{split}
& C_q \varepsilon^{-1} \sup_{\eta \in \tilde{V} }\left|(1+|\eta|)^{-k} \widehat{\chi u_{\varepsilon}}(\eta) \right|  \int_{\tilde{V}}    (1+|\xi|+|\eta|)^{k-q} \m{\eta}  \\
& + C_p (1+|\xi|)^{-p} \varepsilon^{-1-n} \int_{\widetilde{V}^c} (1+|\eta|)^{p-l}  \m{\eta}. 
\end{split}
\end{equation*}
Finally we set $k:=q-p-m$ and $l:=p+n-1$ and obtain
\begin{equation*}
\begin{split}
& C_q \varepsilon^{-1} \sup_{\eta \in \tilde{V} }\left|(1+|\eta|)^{-k} \widehat{\chi u_{\varepsilon}}(\eta) \right|  \int_{\tilde{V}}    (1+|\xi|+|\eta|)^{k-q} \m{\eta}  \\
& + C_p (1+|\xi|)^{-p} \varepsilon^{-1-n} \int_{\widetilde{V}^c} (1+|\eta|)^{p-l}  \m{\eta} \\
& \le C_{p,q,m,n} \varepsilon^{-1-n} (1+|\xi|)^{-p} 
\end{split}
\end{equation*}
for all $\xi \in \widetilde{W}^c$ and $C_{p,q,m,n}$ some constant depending on $p,q,m$ and $n$.
It follows that $(x_0,\xi_0)  \not \in  \wf{f^{\ast}u}$. 
\end{proof}

\section{Examples}
\begin{example}[Multiplication of Colombeau functions]
This example was presented in \cite[Example 4.2]{HK:01}  in order to show that 
an inclusion relation like in \cite[Theorem 8.2.10]{Hoermander:V1} for the wave front set of a product of distributions,
cannot be extended to Colombeau function with wave front sets in unfavorable position. 

Consider the Colombeau functions $u$ and $v$ defined by
\begin{eqnarray*}
u_{\varepsilon}&:=& {\varepsilon}^{-1} \rho({\varepsilon}^{-1} (x + \gamma_{\varepsilon} y))\\
v_{\varepsilon}&:=& {\varepsilon}^{-1} \rho({\varepsilon}^{-1} (x - \gamma_{\varepsilon} y)),
\end{eqnarray*}
where $\gamma_{\varepsilon}$ is some net with $\lim_{\varepsilon \rightarrow 0} \gamma_{\varepsilon} = 0$. Note that 
in \cite[Example 4.2]{HK:01} $\gamma_{\varepsilon} := \varepsilon^{1/2}$.
These Colombeau functions are both associated to $\delta(x) \otimes 1(y)$ and the wave front sets $\wf{u}=\wf{v} = 
\{0\}\times\mathbb{R} \times \{(\pm 1,0)\}$ are in an unfavorable position.
\end{example}
We are going to apply Theorem (\ref{maintheorem}) in order to calculate $WF(u\cdot v)$. First we rewrite $u \cdot v = f^{\ast} \iota(\delta)$,
where $f_{\varepsilon}(x,y)=(x+ \gamma_{\varepsilon} y, x-\gamma_{\varepsilon} y)$
with $\gamma_{\varepsilon}$ some net tending to zero. In Example (\ref{mgexample}) we already showed that $\domainset{f}= \mathbb{R}^2 \times S^1 / (\pm 1, \mp1)$.
From \cite[Theorem 15]{HdH:01} it follows that $WF(\iota(\delta)) = \{(0,0)\} \times S^{1}$.
The wave front unfavorable support of $f$ with respect to $WF(\iota(\delta))$ according to Definition \ref{unfavdef} is
\begin{eqnarray*}
\unfavsupp{f}{(0,0)\times S^1} = \{(0,0)\}.
\end{eqnarray*}
Since $f_{\varepsilon}(x,y)$ is of slow scale in all derivatives at all $x\in\mathbb{R}^2$, it follows that $(\slsupp{f})^c=\emptyset$.
Now Theorem \ref{maintheorem} gives that
\begin{eqnarray*}
\wf{u \cdot v} \subseteq (0,0) \times S^{1}, 
\end{eqnarray*}
which is consistent with the result in \cite[Example 4.2]{HK:01}.

\begin{example}[Hurd-Sattinger]
Let us consider the initial value problem
\begin{equation} \label{ivp} \begin{split}
&\partial_t u + \Theta \partial_x u + \Theta' u  = 0   \\
&u(0,x)  = u_0 \in \col{\mathbb{R}},
\end{split} 
\end{equation}
where $\Theta \in \col{\mathbb{R}^2}$  is defined by $\Theta_{\varepsilon}(x) =  \rho_{\gamma_{\varepsilon}} \ast H(-\cdot)$
 with $\rho_{\gamma_{\varepsilon}} = \frac{1}{\gamma_{\varepsilon}}\rho (\cdot /  \gamma_{\varepsilon})$  
where $\rho \in \mathcal{S}(\mathbb{R})$ and $\int \rho(x) \m{x}=1$ and $\gamma_{\varepsilon}= 1/\log{(1/\varepsilon)}$ 
is a net of slow scale. For the initial value we choose $u_0 := \iota(\delta_{-s_0})$ a delta like singularity at $-s_0$ 
(for a positive $s_0>0$).
The Hurd-Sattinger example was first given in \cite{HS:68} (it was shown that it is not solvable in $L^1_{\rm{loc}}$,
 when distributional products are employed). 
It was further investigated in \cite{HdH:01} with methods from Colombeau theory. In \cite{GH:05b} 
the wave front set $\rm{WF}_{\gamma}$  (with respect to the slow scale net $\gamma$) of the Colombeau solution was calculated. 
For sake of simplicity we do only consider the standard generalized wave front set $\rm{WF}$, which is smaller since
it neglects the singularities coming from the coefficient $\Theta$.
\end{example}
We can write
\begin{eqnarray*}
\Theta_{\varepsilon}(x) = g(x/\gamma_{\varepsilon}),
\end{eqnarray*}
where  $g(x):=\int_{x}^{\infty} \rho(z) \m{z}$ with $\lim_{x \rightarrow -\infty} g(x)=0$ and $\lim_{x \rightarrow +\infty} g(x)=1$ 
(this implies that $\overline{\ran{g}}=[\alpha_{+},\alpha_{-}]$ for some $\alpha_{-}\le 0, \alpha_{+}\ge 1$).

We have already considered (cf. Example \ref{hs_flow_example}) the ordinary differential equation for the characteristic curves
$$
\partial_s \xi_{\varepsilon}(s) = \Theta_{\varepsilon}(\xi_{\varepsilon}(s)),\quad\xi_{\varepsilon}(t)=x,
$$
and we obtained
$$
\ggraph{\xi(0;\bullet)}_{t,x} \subseteq \left\{
\begin{array}{ll}
\{\min{(x-t,0)}\} & x < 0 \\
\left[-t,0\right] & x = 0, 0 \le  t \\  
\{x\} & x > 0.
\end{array}\right.
$$

If we set $f_{\varepsilon}(t,x):=(t, \xi_{\varepsilon}(0;t,x))$ and $f := [(f_{\varepsilon})_{\varepsilon}]$
then
\begin{equation*}
u = f^{\ast} (1 \otimes u_{0})  \cdot \partial_x f(t,x)
\end{equation*}
is a solution of the initial value problem (\ref{ivp}).

We observe that the derivatives of the characteristic flow satisfies
\begin{eqnarray*}
\partial_t \xi_{\varepsilon}(0;t,x) &=& - \Theta_{\varepsilon}(\xi_\varepsilon(0;t,x)) \ \text{and} \
\partial_x \xi_{\varepsilon}(0;t,x) = \frac{\Theta_{\varepsilon}(\xi_\varepsilon(0;t,x))}{\Theta_\varepsilon(x) },
\end{eqnarray*}
which follows from Remark \ref{ode_derivatives}.
So $\partial_x f_{\varepsilon}(t,x)$ is slow scaled in all derivatives, 
due to the fact that $\Theta_{\varepsilon}(x) = g(x/\gamma_{\varepsilon})$ where
$\gamma_{\varepsilon}$ is a net of slow scale. We may conclude that 
$\wf{u} \subseteq \wf{f^{\ast} (1 \otimes u_{0})}$, which enable us the apply Theorem \ref{maintheorem}
in order to estimate the wave front set of $u$.

We put $M=[(M_{\varepsilon})_{\varepsilon}] \in \colmap{\domainset{f}}{S^1}$ defined by
$$
M_{\varepsilon}(t,x,\eta) = \frac{{}^T df_{\varepsilon}(t,x)\eta}{|{}^T df_{\varepsilon}(t,x)\eta|}
$$
as in Lemma \ref{notionlemma}. Note that
$$
 {}^T df_{\varepsilon}(t,x)
= \left( \begin{array}{rr} 
1 &  - \Theta_{\varepsilon}(\xi_\varepsilon(0;t,x))\\ 
0 & \frac{\Theta_{\varepsilon}(\xi_\varepsilon(0;t,x))}{\Theta_\varepsilon(x) }  
\end{array} \right).
$$
For $\eta_{\pm} := (0, \pm 1) \in S^{1}$ 
(which are the only irregular directions coming from the wave front set of $1\otimes u_0$),
we obtain 
\begin{eqnarray*}
M_{\varepsilon}(t,x,\eta_{\pm}) &=&\frac{df_{\varepsilon}(t,x)^T \eta_{\pm}}{|df_{\varepsilon}(t,x)^T \eta_{\pm}|}\\
&=&\pm \left( \frac{-1}{\sqrt{1+ \Theta_\varepsilon(x)^{-2}}}, \frac{1}{\sqrt{1+\Theta_\varepsilon(x)^2}} \right).
\end{eqnarray*}
Remarkably the result does only depend on the coefficient $\Theta_{\varepsilon}(x)$. In Example 
\ref{heaviside_ggraph_example} we obtained 
$$
\ggraph{(\Theta_{\varepsilon})_{\varepsilon}}
=
\left(\{x\in\mathbb{R}\mid x<0\} \times \{1\}\right) \cup \left(\{0\} \times \overline{\ran{g}} \right) \cup \left(\{x\in\mathbb{R}\mid x>0\} \times \{0\}\right),
$$
so Theorem \ref{generalized_graph_composition} yields
\begin{multline*}
\ggraph{M(\cdot,\eta_{\pm})}_{t,x} \subseteq 
 \left( \left( \frac{\mp 1}{\sqrt{1+ \bullet^{-2}}}, \frac{\pm 1}{\sqrt{1+\bullet^2}} \right) \circ \ggraph{\Theta} \right)_{t,x} =
\\  \left\{  
\begin{array}{ll} 
\{\frac{1}{\sqrt{2}}(\mp 1, \pm 1)\}   &  \text{for}\ x<0\\
\{(0,\pm 1)\} &  \text{for}\ x>0\\ 
\bigcup_{\alpha\in[\alpha_{-},\alpha_{+}]} \{\left(\mp \frac{\alpha}{\sqrt{1+\alpha^2}} ,\pm  \frac{1}{\sqrt{1+\alpha^2}} \right)\}  
 ,\ &\text{for}\ x=0.
\end{array} \right.   
\end{multline*}
Furthermore we observe that 
\begin{multline*}
(\wf{1\otimes u_0} \circ \ggraph{f})_{t,x} \\ \subseteq
 \left\{
\begin{array}{ll}
\wf{1\otimes u_0}_{t,\min{(x-t,0)}}  & x < 0 \\
\bigcup_{z\in \left[-t,0\right]}  \wf{1\otimes u_0}_{t,z}   & x = 0 \\  
\wf{1\otimes u_0}_{t,x} = \emptyset & x > 0
\end{array}\right\} 
=\left\{
\begin{array}{ll}
\{(0, \pm 1)\}  & x < 0, t=x-s_0 \\
\{(0, \pm 1)\}  &x = 0, t \ge  -s_0 \\  
\emptyset & x > 0
\end{array}\right. \\
\end{multline*}
and according to Definition \ref{transformed_wf_set} we obtain
\begin{multline*}
(f^{\ast}(\wf{1 \otimes u_0}))_{t,x}= \ggraph{M} \circ ( \{(t,x)\} \times (\wf{1\otimes u_0} \circ \ggraph{f})_{t,x}) 
\\
=
\left\{ \begin{array}{ll}
 \{ (\mp 1, \pm 1) \}  & t<s_0, x=t-s_0 \\
\bigcup_{\alpha \in [\alpha_{-}, \alpha_{+}]}  \{ (\frac{\mp \alpha}{\sqrt{1+\alpha^2}}, \frac{\pm 1}{\sqrt{1+\alpha^2}} ) \} & t=s_0,x=0  \\
\{ (0, \pm 1 )\}  & t>s_0, x=0 
\end{array}
\right.
\end{multline*}
and Theorem \ref{maintheorem} yields $\wf{u} \subseteq f^{\ast}(\wf{1 \otimes u_0})$.

%% file: appendix.tex
\setcounter{secnumdepth}{0}
\setcounter{theorem}{0}

\renewcommand{\thetheorem}{\thechapter.\arabic{theorem}}
\setcounter{equation}{0}
\renewcommand{\theequation}{\arabic{equation}}

\begin{appendix}

\chapter{Sets and measure theory}

In the Appendix A we have collected a few auxillary measure- and set-theoretic results.
Most of the presented material can be considered common knowledge,
nevertheless for sake of completeness we decided to include a proof for all results, 
adapted to the special needs and notation of earlier sections.

\begin{definition}
Let $\mathcal{I}:=]0,\iota_0]$ with $\iota_0 \in ]0,1]$ and 
$(M_{\iota})_{\iota \in \mathcal{I}}$ a family of subset of $\R^n$. 
\begin{trivlist}
\item
$(M_{\iota})_{\iota \in \mathcal{I}}$ it an ascending family of sets, 
when $M_{\alpha} \supseteq M_{\beta}$ for $\alpha \le \beta$.
\item
$(M_{\iota})_{\iota \in \mathcal{I}}$ it an descending family of sets, 
when $M_{\alpha} \subseteq M_{\beta}$ for $\alpha \le \beta$.
\end{trivlist}
Furthermore we define the limes superior by
$$
\limsup_{\iota \rightarrow 0}M_{\iota} := \bigcap_{\iota_1 \in \mathcal{I}} \bigcup_{\iota \in ]0,\iota_1]} M_{\iota}
$$
and the limes inferior by
$$
\liminf_{\iota \rightarrow 0}M_{\iota} := \bigcup_{\iota_1 \in \mathcal{I}} \bigcap_{\iota \in ]0,\iota_1]} M_{\iota}.
$$
In general it holds that $\liminf_{\iota \rightarrow 0}M_{\iota} \subseteq \limsup_{\iota \rightarrow 0}M_{\iota}$.
When we have $M:=\liminf_{\iota \rightarrow 0}M_{\iota} = \limsup_{\iota \rightarrow 0}M_{\iota}$, we
say the the net $(M_{\iota})_{\iota \in \mathcal{I}}$ converges to the set $M$ and use the notation 
$\lim_{\iota \rightarrow 0}M_{\iota}:=M$. 
\end{definition}

\begin{lemma} \label{limit_of_sets_lemma}
If $(M_{\iota})_{\iota \in \mathcal{I}}$ is an ascending (resp. descending) family of sets, it converges 
$$
\lim_{\iota \rightarrow 0}M_{\iota} = \bigcup_{\iota \in \mathcal{I}} M_{\iota} ,
$$
resp.
$$
\lim_{\iota \rightarrow 0}M_{\iota} = \bigcap_{\iota \in \mathcal{I}} M_{\iota}.
$$
\end{lemma}
\begin{proof}
If $(M_{\iota})_{\iota \in \mathcal{I}}$ is an ascending family of sets, we have that 
$\bigcup_{\iota\in ]0,\iota_1]} = \bigcup_{\iota\in I} M_{\iota}$ and 
$M_{\iota_1} \subseteq \bigcap_{\iota \in ]0,\iota_1]} M_{\iota}$, thus 
$\limsup_{\iota \rightarrow 0} M_{\iota} = \bigcup_{\iota\in \mathcal{I}} M_{\iota} \subseteq \liminf_{\iota \rightarrow 0}M_{\iota}$
which implies the statement $\lim_{\iota \rightarrow 0} M_{\iota} = \bigcup_{\iota\in \mathcal{I}} M_{\iota} $.
The proof for the descending family is analogous.
\end{proof}

\begin{lemma} \label{supremum_ascending_family_of_sets}
Let $(M_{\iota})_{\iota \in \mathcal{I}}$ be an ascending family of sets such that 
$M:= \lim_{\iota \rightarrow 0} M_{\iota}=\bigcup_{\iota \in\mathcal{I}}M_{\iota}$ exists.  If $f: \R^n \rightarrow \R$ is some function,
then it follows that 
%Furthmore we assume that $M$ is a bounded set.
$$
\sup_{x \in M} f(x) = \lim_{\iota \rightarrow 0}{ (\sup_{x \in M_{\iota}} f(x)) } = \sup_{\iota \in \mathcal{I}}{ (\sup_{x \in M_{\iota}} f(x)) }
$$ 
and
$$
\inf_{x \in M} f(x) = \lim_{\iota \rightarrow 0}{ (\inf_{x \in M_{\iota}} f(x)) } = \inf_{\iota \in \mathcal{I}}{ (\inf_{x \in M_{\iota}} f(x)) }.
$$ 
\end{lemma}
\begin{proof}
Note that by Lemma \ref{limit_of_sets_lemma} it holds that  
$M = \bigcup_{\iota \in\mathcal{I}} M_{\iota}$. 

We set $\alpha_{\iota} := \sup_{x\in M_{\iota}} f(x)$ and since $(M_{\iota})_{\iota \in \mathcal{I}}$ is an ascending family we have that
it is a monotone increasing net. As $M= \bigcup_{\iota \in \mathcal{I}} M_{\iota}$ we have that 
$\alpha_{\iota} \le \sup_{x\in M} f(x)$.
It follows immediately that 
$\alpha:= \lim_{\iota\rightarrow 0}{\alpha_\iota}= \sup \{ \alpha_\iota \mid \iota \in \mathcal{I}\}$.
For all $x\in M$ there exists some $\iota \in \mathcal{I}$ such that $x\in M_{\iota}$, thus $f(x) \le \alpha_{\iota}$ which implies that $\sup{ \{ \alpha_{\iota} \mid \iota \in \mathcal{I} \} }$ is an upper bound of the set $f(M)$. As $\sup_{x\in M} f(x)$ is the smallest upper bound of $f$ we have $\alpha = \sup_{x\in M} f(x)$.

We set $\beta_{\iota} := \inf_{x\in M_{\iota}} f(x)$ and since $(M_{\iota})_{\iota \in \mathcal{I}}$ is an ascending family we have that
it is a monotone decreasing net. As $M= \bigcup_{\iota \in \mathcal{I}} M_{\iota}$ we have that 
$\beta_{\iota} \ge \inf_{x\in M}f(x)$.
It follows immediately that 
$\beta:= \lim_{\iota\rightarrow 0}{\beta_\iota} = \inf_{\iota \in \mathcal{I}} \beta_\iota $.
For all $x\in M$ there exists some $\iota \in \mathcal{I}$ such that $x\in M_{\iota}$, thus $f(x) \ge \beta_{\iota}$ which implies that $\inf_{\iota \in \mathcal{I}} \beta_{\iota} $
is an lower bound of the set $f(M)$. As $\inf_{x\in M} f(x)$ is the greatest lower bound of $f(M)$ we have $\beta = \inf{M}$.
\end{proof}

\begin{lemma} \label{inf_descending_family_of_sets}
Let $(M_{\iota})_{\iota \in \mathcal{I}}$ be an descending family of sets in $\R^n$ such that
$M:= \lim_{\iota \rightarrow 0} M_{\iota}=\bigcap_{\iota \in\mathcal{I}}M_{\iota}$ exists. If $f: \R^n \rightarrow \R$ is some function,
then it follows that
$$
\sup_{x\in M} f(x) = \lim_{\iota \rightarrow 0}{ (\sup_{x\in M_{\iota}} f(x) ) } = \inf_{\iota \in \mathcal{I}}{ (\sup_{x\in M_{\iota}} f(x) ) }.
$$  
and
$$
\inf_{x\in M} f(x) = \lim_{\iota \rightarrow 0}{ (\inf_{x\in M_{\iota}}f(x) ) } = \sup_{\iota \in \mathcal{I}}{ (\inf_{x\in M_{\iota}}f(x)) }.
$$ 
\end{lemma}

\begin{proof}
Note that by Lemma \ref{limit_of_sets_lemma} it holds that  
$M = \bigcap_{\iota \in\mathcal{I}} M_{\iota}$. 

We set $\alpha_{\iota} := \sup{M_{\iota}}$ and since $(M_{\iota})_{\iota \in \mathcal{I}}$ is a descending family we have that
it is a monotone decreasing net. As $M= \bigcap_{\iota \in \mathcal{I}} M_{\iota}$ we have that 
$\alpha_{\iota} \ge \sup{M}$.
It follows immediately that 
$\alpha:= \lim_{\iota\rightarrow 0}{\alpha_\iota}= \inf \{ \alpha_\iota \mid \iota \in \mathcal{I}\}$.
For all $x\in M$ we have that $x\in M_{\iota}$ for all $\iota \in \mathcal{I}$, thus $x \le \alpha_{\iota}$ which implies that $\inf{ \{ \alpha_{\iota} \mid \iota \in \mathcal{I} \} }$
is an upper bound of $M$. As $\sup{M}$ is the smallest upper bound of $M$ we have $\alpha = \sup{M}$.

We set $\beta_{\iota} := \inf{M_{\iota}}$ and since $(M_{\iota})_{\iota \in \mathcal{I}}$ is a descending family we have that
it is a monotone increasing net. As $M= \bigcap_{\iota \in \mathcal{I}} M_{\iota}$ we have that 
$\beta_{\iota} \le \inf{M}$.
It follows immediately that 
$\beta:= \lim_{\iota\rightarrow 0}{\beta_\iota}= \sup \{ \beta_\iota \mid \iota \in \mathcal{I}\}$.
For all $x\in M$ we have that $x\in M_{\iota}$ for all $\iota \in \mathcal{I}$, thus $x \ge \beta_{\iota}$ which implies that $\sup{ \{ \beta_{\iota} \mid \iota \in \mathcal{I} \} }$
is an lower bound of $M$. As $\inf{M}$ is the greatest lower bound of $M$ we have $\beta = \inf{M}$.
\end{proof}

\begin{lemma} \label{supremum_approximating_sequence}
Let $f : \Omega \rightarrow \R$ be a function, then for any bounded set $C \subseteq \Omega$ there exists a sequence $(x_k)_{k\in\N}$ in $C$ such that
$$
\sup_{y\in C} f(y) = \limsup_{k\rightarrow \infty} f(x_k).
$$
\end{lemma}
\begin{proof}
Denote $M:=\sup_{y\in C} f(y)$.

If $M=\infty$, we can pick find $x_k \in C$ such that $f(x_k)\ge 1/k$ for all $k\in \N$. It is now trivial that
$\limsup_{k \rightarrow \infty} f(x_k) = \infty$.

In the case $M< \infty$:
Let $x_k\in C$, such that $f(x_k) < M$, then we can find some 
$x_{k+1}$ such that $M- f(x_{k+1}) \le 1/2 (M- f(x_{k}))$.

The last step followed by contradiction: Assuming $M- f(x) > 1/2 (M- f(x_{k}))$ for all $x\in C$ holds, implies 
$f(x) < 1/2 (M +  f(x_{k}))$. The right-hand side is obviously an upper bound, smaller than $M$, contradicting the fact
that $M$ is the smallest upper bound of the set $\{f(x) \mid x\in C\}$. 

Starting with any $x_1 \in C$, we iteratively obtain a sequence with $M- f(x_k) \le 2^{-k}$, thus
$$
\sup_{y\in C} f(y) - \limsup_{k\rightarrow \infty} f(x_k) =0.
$$ 
\end{proof}

\begin{lemma} \label{closed_map_of_neighboorhood}
Let $f: \Omega_1 \rightarrow \Omega_2$ be a  continuous map.
If $(M_{\iota})_{\iota \in \mathcal{I}}$ be an descending family of sets, such that 
$\lim_{\iota \rightarrow 0} M_{\iota} = \bigcap_{\iota \in \mathcal{I}} M_{\iota}$ is a compact subset of $\Omega_1$, then it follows that
$$
\lim_{\iota \rightarrow 0} f(M_{\iota}) = f(\lim_{\iota \rightarrow 0}  M_{\iota}). 
$$
\end{lemma}
\begin{proof}
First of all we note that $f(\bigcap_{\iota \in \mathcal{I}} M_{\iota}) \subseteq \bigcap_{\iota \in \mathcal{I}} f( M_{\iota})$.
We proove by contradiction: Assume there exists some $y\in \bigcap_{\iota \in \mathcal{I}} f( M_{\iota})$ such 
that $y\not \in f(M)$. As $f(M)$ is a compact subset of $\Omega_2$ we have that there exists some $\varepsilon >0$, with
$B_{\varepsilon}(f(M)) \bigcap \{y\}=\emptyset$. 
The continuity of $f$ implies that there exists some $\delta>0$ such that 
$f(B_{\delta}(M)) \subseteq B_{\varepsilon}(f(M))$. 
Now there exist some $\iota_0 \in \mathcal{I}$ such that $M_{\iota_0} \subseteq B_{\delta}(M)$.
Assuming the contrary, that $B_{\delta}(M) \subset M_{\iota}$ for all $\iota\in \mathcal{I}$, would lead to
the contradiction $B_{\delta}(M) \subseteq \lim_{\iota \rightarrow 0} M_{\iota}$. 
Finally we obtain $\bigcap_{\iota \in \mathcal{I}} f(M_{\iota})  \subseteq f(M_{\iota_0}) \subseteq f(B_{\delta}(M))$
$$
y \not \in B_{\varepsilon}(f(M))  \supseteq f(B_{\delta}(M)) \supseteq f(M_{\iota_0}) \supseteq 
\bigcap_{\iota \in \mathcal{I}} f(M_{\iota}),
$$
which contradicts the inital choice of $y$, prooving the statement.
\end{proof}

\begin{lemma} \label{supremum_uppersemicontinuous_functions}
Let $h: \R^n \rightarrow \R$ be a upper semi-continuous function and $K\csub \R^n$, then the
supremum $\sup_{z\in K} h(z)$ is attained, i.e. there exists a $z_0 \in K$ such that 
$\sup_{z\in K} h(z) = h(z_0)<\infty$.
\end{lemma}

\begin{proof}
Let $(x_k)_{k \in \N}$ the sequence from the proof of Lemma \ref{supremum_approximating_sequence}.
By the upper semi-continuity we get that
$$
\sup_{y\in K} f(y)  =\limsup_{k \rightarrow \infty} h(x_k) \le h( \limsup_{k \rightarrow \infty}x_k )  
$$
and since ${(x_k)_{k\in \N}} \subseteq K$ we have that $x:=\limsup_{k \rightarrow \infty}x_k $ is a point in $K$.
Furthermore the semi-continuity of $h$ implies pointwise boundedness, thus $h(x)<\infty$.
\end{proof}

% \begin{lemma} \label{upper_semicontinuous_lemma}
% Let $g\in L^{\infty}(\Omega)$, then the function defined by
% $$
% h_{\delta}(x) := \sup_{y\in B_{\delta}(x)} g(y)
% $$ 
% is upper semi-continuous.
% \end{lemma}
% 
% \begin{proof}
% Observe that 
% \begin{multline*}
% \limsup_{x \rightarrow x_0} h_\delta(x) = 
% \lim_{\rho \rightarrow 0} \sup{ \{ h_{\delta}(x) \mid x\in B_{\rho}(x_0) / \{ x_0 \} \} } =
% \lim_{\rho \rightarrow 0} \sup_{x\in B_{\rho}(x_0) /\{ x_0\} }{\sup_{y \in B_{\delta}(y)}{ g(y)}}\\ \le 
% \lim_{\rho \rightarrow 0}  \sup_{y \in B_{\delta+\rho}(x_0)}{ g(y)}
% = h_{\delta}(x_0)
% \end{multline*}
% holds, thus $h_\delta(x)$ is upper semi-continuous.
% \end{proof}

\begin{lemma} \label{L1_Linfty_norm}
Let $g: J \times \R^n \rightarrow \R$ be a function such that
\begin{enumerate}
\item $t \mapsto g(t,x)$ is Lebesgue measurable for all $x\in \R$, 
\item $x \mapsto g(t,x)$ is upper semi-continuous for almost all $t\in J$, and 
\item there exists a positive function $\beta\in L^1_{\rm loc}(J)$ such that
$\sup_{z\in \R^n} |g(t,z)| \le \beta(t)$ for almost all $t\in J$,
\end{enumerate}
then $g \in L^1_{\rm loc}(J; L^{\infty}(\R^n))$.
\end{lemma}

\begin{proof}
It has to be shown that  $t \mapsto \sup_{z \in \R^n} |g(t,z)|$ is a Lebesgue measurable function. 
Let $(K_l)_{l \in \N}$ be an ascending family of compact sets such that $\bigcup_{l=1}^\infty K_l = \R^n$, then 
we have by Lemma \ref{supremum_ascending_family_of_sets} that 
$$
\sup_{z\in \R^n}{h(t,z)} = \liminf_{l\rightarrow \infty} \left( \sup_{z\in K_l} |h(t,z)|\right).
$$
Since $g$ satisfies property (ii), i.e. the upper semi-continuity of $x \mapsto h(t,x)$ for almost all $t\in J$, we conclude by
Lemma \ref{supremum_uppersemicontinuous_functions} that 
there exists a sequence $(z_l)_{l\in\N}$ with $z_l \in K_l$ such that
$$
\sup_{z\in \R^n}{h(t,z)} = \liminf_{l\rightarrow \infty}  |h(t,z_l)|
$$
for almost all $t\in J$. Property (i) yields that $(|h(t,z_l)|)_{l \in \N}$ is a family of Lebesgue measurable functions. It follows by
\cite[Theorem 11.12]{HewStr:65} that $\liminf_{l\rightarrow \infty}  |h(t,z_l)|$ is Lebesgue measurable.
Since $|h(t,z_{l})| \le \beta(t)$ for almost all $t\in J$ as assumed by (iii), we obtain 
due to \cite[Theorem 12.24]{HewStr:65} 
$$
\int_M \sup_{z\in \R^n}{h(\tau,z)} \m{\tau} =  \int_M  \liminf_{l\rightarrow \infty} \left( \sup_{z\in K_l} |h(\tau,z)|\right)\m{\tau} \le \liminf_{l\rightarrow \infty} \int_M \left( \sup_{z\in K_l} |h(\tau,z)|\right)\m{\tau} \le
\int_M \beta(\tau) \m{\tau} 
$$ 
for all $M \csub \R^n$ which implies $g\in  L^1_{\rm loc}(J; L^{\infty}(\R^n))$.
\end{proof}

\begin{lemma} \label{measurability_of_composition_lemma}
Let $h:J \times \R^n \rightarrow \R$ be a function satisfying properties (i) - (ii) of Lemma \ref{L1_Linfty_norm}, then for any Lebesgue measurable map  
$\xi: J \rightarrow \R^n, t \mapsto \xi(t)$, the composition
$$
t \mapsto h(t,\xi(t))
$$
is a Lebesgue measurable function from $J \rightarrow \R$.
\end{lemma}

\begin{proof}
First we proove the simpler case where $x\mapsto h(t,x)$ is continuous for almost $t\in J$ instead of property (ii).
We approximate $\xi$ by simple functions $s_n(t):=\sum_{k=1}^n a_{n,k} 1_{A_{n,k}}(t)$  (see \cite[Theorem 11.35]{HewStr:65}),  such that
$\xi$ is the pointwise limit of $(s_n)_{n\in \N}$.
Note $A_{n,k}$ are Lebesgue measurable subsets of $J$. We have 
$$
f(t, s_n(t)) = \sum_{k=1}^n F(t,a_{n,k})  1_{A_{n,k}}(t),
$$
so $t \mapsto f(t,s_n(t))$ is Lebesgue measurable as the finite sum of measurable functions.
From the continuity of $x\mapsto f(t,x)$ it follows that
$\lim_{n \rightarrow \infty} f(t,s_n(t)) = f(t, \xi(t))$ holds pointwise for almost all $t\in J$, thus $t \mapsto f(t, \xi(t))$ is Lebesgue measurable as 
the limit of measurable functions (see \cite[Theorem 11.12]{HewStr:65}).

The case where $x \mapsto f(t,x)$ is only upper semi-continuous is much more difficult. The idea is to approximate this function by functions
$f(t,x) = \inf_{k\in \N} \varphi_k(t,x)$, where $(\varphi_k)_{k\in \N}$ satisfy (i) and $x\mapsto \varphi_k(t,x)$ is continuous.
First of all we set $h_k(t,x) := \sup_{y\in B_{1/k}(x)} f(t,y)$, which is again upper semi-continuous and fulfills $h_k(t,y) \ge f(t,x)$ 
for all $y\in B_{1/k}(x)$. 
By the upper semi-continuity of $x\mapsto f(t,x)$ we have that for all $l\in N$ there exists a constant $\delta_l$ of $x$ such that
$f(t,y) \le f(t,x) + 1/l$ for all $y \in B_{\delta_l}(x)$. Choose $k_0 \in \N$ such that $B_{1/k_0}(x) \subseteq B_{\delta_l}(X)$, then
$h_{k}(t,x) \le f(t,x) + 1/l$ for all $k \ge k_0$. It immediately follows that $\inf_{k\in\N} h_k(t,x) = f(t,x)$.

Let $\rho \in C_c^{\infty}(\R^n)$ be positive with $\supp(\rho) \subseteq B_1(0)$ and $\int \rho(x) \m{x} =1$.
Then we define $\varphi_k(t,x) := \int_{B_1(0)} h_{2k}(t,x- (2k)^{-1} y) \rho(y) \m{y}$ by convolution, so $\varphi_k \in L^1_{\rm loc}(J;C^{\infty}(\R^n))$.
Observe that $|x- (x-(2k)^{-1} y)| \le (2k)^{-1} |y|$, thus $x-(2k)^{-1}y \in B_{1/2k}(x)$  and  $h_{2k}(t,x-(2k)^{-1} y) \ge f(t,x)$ which implies
$$
\varphi_k(t,x) = \int_{B_{1}(0)} h_k(t,x-k^{-1}y) \rho(y) \m{y} \ge \int_{B_{1}(0)} f(t,x) \rho(y) \m{y} = f(t,x)
$$
for all $(t,x) \in J \times \R^n$ and $k\in \N$.

Setting
$z_k =(2k)^{-1}y$ and $w \in B_{1/2k}(x-z_k)$ we obtain $|w -(x- z_k)| \le 1/(2k)$, so $||x-w| -|z_k|| \le 1/(2k)$ which implies $|x-w| \le 1/k$ for all $w\in B_{1/2k}(x-z_k)$.
Conclude that $h_{2k}(t,x-z_k) := \sup_{w \in B_{1/(2k)}(x-z_k)} f(t,w) \le \sup_{w \in B_{1/k}(x)} f(t,w) = h_{k}(t,x)$. 
It follows that 
$$
f(t,x) = \inf_{k \in \N} h_k(t,x) \ge \inf_{k \in \N} \varphi_k(t,x) \ge f(t,x),
$$
thus $f(t,x) = \inf_{k \in \N} \varphi_k(t,x)$. It follows by the first part of the proof that $t \mapsto \varphi_k(t,\xi(t))$ is Lebesgue measurable
which yields that $t \mapsto f(t,\xi(t))$ is Lebegue measurable as the infimum of Lebesgue measurable functions  (see \cite[Theorem 11.12]{HewStr:65}).
\end{proof}

\chapter{A generalized Stationary Phase theorem}

Appendix B contains a straight-forward generalization of the stationary phase theorem
for nets of smooth functions. 
We essentially follow the proof of the classical stationary phase theorem as presented in 
\cite[Theorem 7.7.5]{Hoermander:V1}, but with a pedantic book-keeping on the $\varepsilon$-dependence of the estimates.
Due to the technical nature of this result, we decided to put it in the appendix.

\setcounter{theorem}{0}

\begin{definition} \label{growthnot}
Suppose that $g\in \col{\Omega}$ and let $(g_{\varepsilon})_{\varepsilon}$ be some representative, then we
introduce the notation
\begin{eqnarray*}
|g_{\varepsilon}|_k := \sum_{|\alpha|=k} |\partial^{\alpha} g_{\varepsilon} |
\end{eqnarray*}
and
\begin{eqnarray*}
\mu_{K, k,\varepsilon}(g) &:=& \sup_{K} |g_{\varepsilon}|_k \\
\mu_{K, k,\varepsilon}^{\ast}(g)&:=& \max_{l\le k}{\mu_{K, l, \varepsilon}(g)}.
\end{eqnarray*}
\end{definition}

\begin{lemma} \label{boundderivlemma}
Suppose that $g\in \col{\Omega}$ with an non-negative representative $(g_{\varepsilon})_{\varepsilon}$, such that there exists
$\varepsilon'>0$, with $g_{\varepsilon}(x) \ge 0$ for $\varepsilon <\varepsilon'$ and all $x\in\Omega$. Then it follows
that for any $K \csub \Omega$ there exists some constant $C$ such that
\begin{eqnarray*}
\sum_{|\alpha|=1} |\partial^{\alpha} g_{\varepsilon}(x)|  
&\le& C \sqrt{g_{\varepsilon}(x)} \sqrt{\mu_{M,2,\varepsilon}^{\ast}(g)}
\end{eqnarray*}
holds for all $x\in K$ and $M$ is some compact set with $K \csub M^{\circ}$.
\end{lemma}
\begin{proof}
Let $K$ be some compact set and $M \csub \Omega$ such that $K \csub M^{\circ}$. 
Then we can find some $\delta >0$ such that $B_{\delta}(K) \subseteq M$ holds. 
%For all $x\in K$ that $B_{\delta}(x_0) \subseteq M$.
Choose some $x_0\in K$.
By Taylor's formula we obtain for any $x= x_0+ \alpha e_k$ with $k=1,...,n$ and $\alpha \in [-\delta,\delta]$  (note that $x_0+\alpha e_k \in B_{\delta}(x_0) \subseteq M$) that
\begin{eqnarray*}
0\le g_{\varepsilon}(x) \le g_{\varepsilon}(x_0) + \partial_{k} g_{\varepsilon}(x_0) \alpha + \frac{1}{2} m_{\varepsilon}(k) \alpha^2
\end{eqnarray*}
holds with 
\begin{eqnarray*}
m_{\varepsilon}(k) :=  \max{( \sup_{x\in M} |\partial_{k}^2 g_{\varepsilon}(x)|, 2 g_{\varepsilon}(x_0) \delta^{-2} )}.
\end{eqnarray*}
%So we obtain the estimate
%\begin{eqnarray*}
%|\partial_{v} g_{\varepsilon}(x_0)| \le \alpha^{-1} g_{\varepsilon} (x_0) + \frac{\alpha}{2} m_{\varepsilon}(v).
%\end{eqnarray*}
Now we distinguish two cases: If $ \delta^2 m_{\varepsilon}(k) \le 4 g_{\varepsilon}(x_0)$, then we put $\alpha=\pm \delta$ 
and get
\begin{eqnarray*}
\delta|\partial_{k} g_{\varepsilon}(x_0)| &\le& g_{\varepsilon} (x_0) + \frac{\delta^2}{2} m_{\varepsilon}(k) \\
&=&  \sqrt{(g_{\varepsilon} (x_0) + \frac{\delta^2}{2} m_{\varepsilon}(k))^2 } \\
&=&  \sqrt{(g_{\varepsilon}(x_0)^2 + \delta^2 g_{\varepsilon}(x_0) m_{\varepsilon}(k)+ \frac{\delta^4}{4} m_{\varepsilon}(k)^2 } \\
&\le& \sqrt{(g_{\varepsilon}(x_0)^2 +\delta^2 g_{\varepsilon}(x_0) m_{\varepsilon}(v)+ {\delta^2} m_{\varepsilon}(k) g_{\varepsilon}(x_0)} 
\\&=&\sqrt{g_{\varepsilon}(x_0)} \sqrt{ g_{\varepsilon}(x_0) +2 \delta^2 m_{\varepsilon}(k) }.
\end{eqnarray*}
In the case where $ \delta^2 m_{\varepsilon}(k) > 4 g_{\varepsilon}(x_0)$ we set $\alpha= \pm \sqrt{2 g_{\varepsilon}(x_0)/m_{\varepsilon}(k)}\le\delta$
to obtain
\begin{eqnarray*}
|\partial_{k} g_{\varepsilon}(x_0)| &\le& \sqrt{m_{\varepsilon}(k)/2} \sqrt{g_{\varepsilon} (x_0)} + \frac{ \sqrt{2 g_{\varepsilon} (x_0)} m_{\varepsilon}(k) } {2 \sqrt{m_{\varepsilon }(k)} }\\   
&=& \sqrt{2 m_{\varepsilon}(k)}\sqrt{g_{\varepsilon}(x_0)}
\end{eqnarray*}
We can finally conclude that the estimate 
\begin{multline*}
|\partial_{k} g_{\varepsilon}(x_0)|^2
\le \delta^{-2} g_{\varepsilon}(x_0) ( g_{\varepsilon}(x_0) +2 \delta^2 m_{\varepsilon}(k) ) \\
\le \delta^{-2} g_{\varepsilon} (x_0) ( \sup_{x\in K} g_{\varepsilon}(x) +2 \delta^2 m_{\varepsilon}(k) )
\le \max{(\delta^{-2}, 2, 2 \delta^{-4})} g_{\varepsilon} (x_0) \mu_{M,2,\varepsilon}^{\ast}(g)
\end{multline*}
holds for all $x_0\in K$ where $\delta$ and $\mu_{M,2,\varepsilon}^{\ast}(g)$ are independent of $x_0$. The statement follows
by adding up the estimates $\sum_{|\alpha|=1} |\partial^{\alpha} g_{\varepsilon}(x)| = \sum_{k=1}^n |\partial_k g_{\varepsilon}(x)|$ and putting $C=n \cdot \sqrt{\max{(\delta^{-2}, 2, 2 \delta^{-4})}}$.
\end{proof}

\begin{theorem}[Stationary phase theorem] \label{statphas}
Let $u\in \colcomp{\Omega}$ with support $K \csub \Omega$ and ${\phi}_\varepsilon \in \mathcal{E}(\Omega)$ with the property that
there exists an $\varepsilon_0 >0$ and $m\in \mathbb{N}$ such that
\begin{eqnarray}\label{stphprop}
\inf_{x\in K} (|{\phi}_{\varepsilon}'(x) |) \ge \lambda_{\varepsilon} \ \ 
{\rm for \ all \ }\varepsilon \le \varepsilon_0
\end{eqnarray}
holds, with $\lambda_{\varepsilon}$ some net tending to zero. Then we have that
\begin{eqnarray} \label{statexpr1}
(v_{\varepsilon})_{\varepsilon}:= \left(\int u_{\varepsilon}(x) 
\exp{(i \omega {\phi}_{\varepsilon} (x))} d\!x\right)_{\varepsilon}+\colneg{\Omega}
\end{eqnarray}
is a Colombeau function in the $\omega$ variable and it is bounded by
\begin{eqnarray} \label{statexpr2}
\omega^k |v_{\varepsilon}(\omega)| \le L_{k,\varepsilon} \lambda_{\varepsilon}^{-k} 
\sum_{|\alpha| \le k} \sup_K |D^{\alpha}u_{\varepsilon}| ,
\end{eqnarray}
where  
\begin{eqnarray*}
L_{k,\varepsilon}:=C_k \max{ \{1, \mu^{\ast}_{M,k,\varepsilon}(\phi_{\varepsilon})^{2k^2}   \}}
\end{eqnarray*}
and $M$ is some compact set with $K\subset M^{\circ}$.
\end{theorem}

\begin{proof}
It is obvious that (\ref{statexpr1}) is a well-defined Colombeau function.
The proof follows closely the proof of classical stationary phase theorem in \cite[Theorem 7.7.1]{Hoermander:V1}.
By $N_{\varepsilon}(x):=|\phi_{\varepsilon}'(x)|^2$ we denote the square of the norm of the gradient of the phase function. Let
\begin{eqnarray*}
u_{\nu, \varepsilon} := N^{-1}_{\varepsilon} \frac{\partial \phi_{\varepsilon}}{\partial x_{\nu}} u_{\varepsilon}
\end{eqnarray*}
and since
\begin{eqnarray*}
i \omega \frac{\partial \phi_{\varepsilon}}{\partial x_{\nu}} \exp{(i \omega \phi_{\varepsilon})} = \partial_{\nu} 
\exp{(i \omega \phi_{\varepsilon})}
\end{eqnarray*}
we obtain after an integration by parts 
\begin{eqnarray*}
\int u_{\varepsilon} \exp{(i \omega \phi_{\varepsilon})} dx =
\frac{i}{\omega} \sum_{\nu} \int (\partial_{\nu} u_{\nu, \varepsilon}) \exp{(i \omega \phi_{\varepsilon})} dx
\end{eqnarray*}
(using the notation introduced in Definition \ref{growthnot}).
We prove by induction:
For $k=0$ we have the obvious bound
\begin{eqnarray*}
 |\int u_{\varepsilon} \exp{(i \omega \phi_{\varepsilon})} dx | \le C  \sup_K |u_{\varepsilon}(x)|.
\end{eqnarray*}

Assume that the bound (\ref{statexpr2}) holds for power $k-1$, then we have that
\begin{equation*}\begin{split}
\omega^k | \int u_{\varepsilon} \exp{(i \omega \phi_{\varepsilon})} dx | = &
\omega^{k-1} |\sum_{\nu} \int (\partial_{\nu} u_{\nu, \varepsilon}) \exp{(i \omega \phi_{\varepsilon})} dx|\\
\le& 
L_{k-1,\varepsilon}\sum_{m=0}^{k-1}   
\sup_K  \left(\sum_{\nu=1}^n |u_{\nu, \varepsilon}|_{\mu+1}  N^{m/2-k+1}_{\varepsilon} \right) \label{statexpr8}
\end{split}
\end{equation*}
holds. In the next step we are going to show that 
\begin{eqnarray} \label{statexpr3}
N^{\frac{1}{2}} \sum_{\nu} |u_{\nu, \varepsilon}|_{m}  \le M_{m,\varepsilon} 
\sum_{r=0}^m |u_{\varepsilon}|_{r} N_{\varepsilon}^{\frac{r-m}{2}} 
\end{eqnarray} holds. Again we are prooving by induction: For $m=0$ we have
\begin{eqnarray} \label{statexpr66}
N^{\frac{1}{2}} \sum_{\nu} |u_{\nu, \varepsilon}| = |u_{\varepsilon}| \sum_{\nu}  
N^{-\frac{1}{2}} |\frac{\partial \phi_{\varepsilon} }{\partial x_{\nu} } | \le n |u_{\varepsilon}|
\end{eqnarray}
and let us now assume that (\ref{statexpr3}) holds up to $m-1$. Let $\alpha$ be any multi-index with $|\alpha|=m$.
We apply $\partial^{\alpha}$ on
\begin{eqnarray*}
N_{\varepsilon} u_{\nu, \varepsilon}= u_{\varepsilon} \frac{\partial \phi_{\varepsilon} }{\partial x_{\nu} }
\end{eqnarray*}
and obtain
$$
\partial^{\alpha} (N_{\varepsilon} u_{\nu, \varepsilon})=
\sum_{\beta \le \alpha} {\alpha \choose \beta} (\partial^{\beta}u_{\varepsilon}) 
(\partial^{\alpha-\beta+ e_{\nu}} \phi_{\varepsilon}).
$$
It follows that
\begin{multline}  \label{statexpr5}
\left|N_{\varepsilon} \partial^{\alpha} u^{\varepsilon}_{\nu}\right| = 
\left|- \sum_{\beta\ne 0, \beta \le \alpha} {\alpha \choose \beta} \partial^{\beta} N_{\varepsilon}
\partial^{\alpha-\beta} u_{\nu, \varepsilon} + 
\sum_{ \beta \le \alpha} {\alpha \choose \beta} \partial^{\beta} u_{\varepsilon}
\partial^{\alpha-\beta+e_{\nu}} \phi_{\nu}\right|\\
\le C \left( \sum_{l=1}^m |N_{\varepsilon}|_{l} |u_{\nu, \varepsilon}|_{m-l} + 
\sum_{l=0}^{m} |\phi_{\varepsilon}|_{m-l+1} |u_{\varepsilon}|_{l} \right) 
\le C \left( |N_{\varepsilon}|_{1} |u_{\nu, \varepsilon}|_{m-1} + 
|\phi_{\varepsilon}|_{1} |u_{\varepsilon}|_{m} \right.\\ 
+ \max_{l=0}^{m-2} {(\sup{|N_{\varepsilon}|_{m-l}}, 
\sup{|\phi_{\varepsilon}|_{m-l+1})}} \sum_{l=0}^{m-2}  (|u_{\nu, \varepsilon}|_{l} 
+  |u_{\varepsilon}|_{l}) 
+\left. |u_{\varepsilon}|_{m-1} |\phi_{\varepsilon}|_{2}\right). 
\end{multline} 
Now we can apply Lemma \ref{boundderivlemma} to bound $|N_{\varepsilon}|_1$. This yields 
\begin{eqnarray*} 
|N_{\varepsilon}|_1 &=& \sum_{|\alpha|=1} |\partial^{\alpha} N_{\varepsilon}| 
\le  C_1 \sqrt{N_{\varepsilon}} \sqrt{\mu^{\ast}_{L,2,\varepsilon}(N_{\varepsilon})} \\  
|\phi_{\varepsilon}|_1 &\le& C_2 \sqrt{N_{\varepsilon}}
\end{eqnarray*}
and we can verify that
\begin{multline*}
|\partial^{\alpha} N_{\varepsilon}| = |\sum_{i=1}^n \partial^{\alpha} (\partial^{e_i} \phi_{\varepsilon}(x))^2|
\le K_1  \sum_{i=1}^n \sum_{\beta \le \alpha} |\partial^{\beta + e_i}  \phi_{\varepsilon}(x)| 
|\partial^{\alpha-\beta + e_i}  \phi_{\varepsilon}(x)| \\
\le K_2 \sum_{k\le |\alpha|} |\phi_{\varepsilon}(x)|_{|\alpha|-k+1} |\phi_{\varepsilon}(x)|_{k+1}
\end{multline*}
holds. This yields
\begin{equation*}\begin{split}
& \sup_K |N_{\varepsilon}|_l \le K_3 \sum_{k\le l} \sup_K |\phi_{\varepsilon}(x)|_{l-k+1} \sup_K |\phi_{\varepsilon}(x)|_{k+1}\\ 
&\le K_4 \mu^{\ast}_{K,l+1,\varepsilon}(\phi_{\varepsilon})^2.
\end{split}
\end{equation*}
We introduce
\begin{eqnarray*}
\sigma_{\varepsilon,m}&:=&\max_{l=2}^{m} {(\sup{|N_{\varepsilon}|_{l}}, 
\sup{|\phi_{\varepsilon}|_{l+1})}} \\
&\le&\max{(K_4 \mu^{\ast}_{K, m+1,\varepsilon}(\phi_{\varepsilon})^2, \mu^{\ast}_{K, m+1,\varepsilon}(\phi_{\varepsilon}) )}
\end{eqnarray*}
in order to simplify the notation.
Using these bounds we can estimate (\ref{statexpr5}) by
\begin{multline*} 
 C \left( |N_{\varepsilon}|^{1/2} C_1 \sqrt{ \mu^{\ast}_{L,2,\varepsilon}(N_{\varepsilon})} |u_{\nu, \varepsilon}|_{m-1} + 
C_2 |N_{\varepsilon}|^{1/2}  
|u_{\varepsilon}|_{m} \right. \\
 \left.+ C_3 \sigma_{\varepsilon}^{(m-2)}  \left( \sum_{l=1}^{m-2}  (|u_{m, \varepsilon}|_{l}  
+  |u_{\varepsilon}|_{l}) \right)+   |\phi_{\varepsilon}|_2 |u_{\varepsilon}|_{m-1} \right) \\
\le 
C_5 \sigma_{\varepsilon}^{(m-2)}  \left( 
|N_{\varepsilon}|^{1/2} |u_{\nu, \varepsilon}|_{m-1} 
+ |N_{\varepsilon}|^{1/2}  |u_{\varepsilon}|_{m} 
+ \sum_{l=1}^{m-2}  |u_{\nu, \varepsilon}|_{l}     +  \sum_{l=1}^{m-1} |u_{\varepsilon}|_{m-1} \right) 
\end{multline*}
and the induction hypothesis (\ref{statexpr3}) for $m-1$ gives
\begin{multline*}
\le 
C_5  \sigma_{\varepsilon}^{(m)}  \cdot 
\left( 
  M_{m-1,\varepsilon} \sum_{r=0}^{m-1} |u_{\varepsilon}|_{r} N_{\varepsilon}^{\frac{r-m+1}{2}} 
+ |N_{\varepsilon}|^{1/2}  |u_{\varepsilon}|_{m} \right.
 \left.
+ \sum_{l=1}^{m-2}  M_{l,\varepsilon} \sum_{r=0}^l |u_{\varepsilon}|_{r} N_{\varepsilon}^{\frac{r-l-1}{2}} 
  +  \sum_{l=1}^{m-1} |u_{\varepsilon}|_{l} \right) \\
\le C_5  \sigma_{\varepsilon}^{(m)} \cdot 
\left( 
M_{m-1,\varepsilon} \sum_{r=0}^{m} |u_{\varepsilon}|_{r} N_{\varepsilon}^{\frac{r-m+1}{2}}  
+ \sum_{l=0}^{m-2}  \max{\{1,M_{l,\varepsilon}\}} \sum_{r=0}^{l+1} |u_{\varepsilon}|_{r} N_{\varepsilon}^{\frac{r-l-1}{2}}    
\right) \\
\le
C_5 \sigma_{\varepsilon}^{(m)}
\left(\sum_{l=0}^{m-1}  \max{\{1,M_{l,\varepsilon}\}} \right) \sum_{r=0}^{l+1} |u_{\varepsilon}|_{r} N_{\varepsilon}^{\frac{r-l-1}{2}} 
\le
C_6  \sigma_{\varepsilon}^{(m)} 
\max_{l=0}^{m-1}  \max\{1,M_{l,\varepsilon}\}  \sum_{r=0}^{m} |u_{\varepsilon}|_{r} N_{\varepsilon}^{\frac{r-m-1}{2}}. 
\end{multline*}
It follows immediately that 
\begin{eqnarray}
\sum_{\nu} |u_{\nu, \varepsilon}|_{m}  \le M_{m,\varepsilon} \sum_{r=0}^m |u|_{r} N^{\frac{r-m}{2}} 
\end{eqnarray} 
holds, where $M_{m,\varepsilon}$ is defined by
$$
M_{m,\varepsilon} := \sigma_{\varepsilon}^{(m)}  \max_{l=0}^{m-1}  \max{\{1,M_{l,\varepsilon}\}}=C \sigma_{\varepsilon}^{(m)} \Pi_{i=1}^{m-1} \max{\{\sigma_{\varepsilon}^{(i)},1\}}
$$
is a generalized number. We are finally able to estimate (\ref{statexpr8}) from above by 
\begin{eqnarray*}
&& L_{k-1,\varepsilon} \sum_{m=0}^{k-1}   
\sup_K  \left(  M_{m+1,\varepsilon} \sum_{r=0}^{m+1} |u_{\varepsilon}|_{r} N_{\varepsilon}^{\frac{r-m-1}{2}}  
N^{\frac{m+1}{2}-k}_{\varepsilon} \right)\\
&\le &
L_{k-1,\varepsilon}M_{k,\varepsilon}  (k-1) \sum_{r=0}^{k}  
\sup_K  \left(   |u_{\varepsilon}|_{r} N_{\varepsilon}^{\frac{r}{2}-k} \right).
\end{eqnarray*}
The generalized constant $L_{k,\varepsilon}$ is recursivly defined by
\begin{eqnarray*}
L_{k,\varepsilon}&:=& L_{k-1,\varepsilon}M_{k,\varepsilon}  (k-1) = (k-1)! \prod_{l=1}^{k} M_{l,\varepsilon}  \\
&=& C (k-1)! \prod_{l=1}^{k}  \sigma_{\varepsilon}^{(l)}  \prod_{i=1}^{l-1} \max{\{\sigma_{\varepsilon}^{(i)},1\}} \\
&\le&
C_k  \left(\prod_{l=1}^{k}  \max{\{1,\mu^{\ast}_{L, l,\varepsilon}(\phi_{\varepsilon})^2\}}    \right) \prod_{l=1}^{k} \prod_{i=1}^{l-1}  \max{\{1, \mu^{\ast}_{L, i,\varepsilon}(\phi_{\varepsilon})^2   \}} \\
&\le& C_k \max{ \{1, \mu^{\ast}_{L, k,\varepsilon}(\phi_{\varepsilon})^{2k^2}   \}} 
\end{eqnarray*}
Note that for $0\le r\le k$ we have that
\begin{eqnarray*}
N_{\varepsilon}^{\frac{r}{2} -k } \le \min{ \{1,\inf_{K} |\phi_{\varepsilon}'(x)|^{-k}\} } \le \lambda_{\varepsilon}^{-k}
\end{eqnarray*}
holds.
\end{proof}

\end{appendix}

%% file: abstract.tex
\newpage

\ihead{} 
\chead{}
\ohead{}

\begin{center}
\huge{\textbf{Abstract} }\\*[1cm]
\end{center}

\normalsize

As the title ``Generalized regularity and solution concepts for differential equations'' suggests, the main topic of my thesis
is the investigation of generalized solution concepts for differential equations, in particular 
first order hyperbolic partial differential equations with real-valued, non-smooth coefficients and their characteristic system of ordinary
differential equations.

In Colombeau theory there have been developed existence results that yield solutions for ordinary and partial differential equations beyond the scope of classical approaches. Nevertheless this comes at the price of sacrificing regularity (in general a Colombeau solution may even lack a distributional shadow). It is prevailing in the Colombeau setting that the question of mere existence of solutions is much easier to answer than to determine their regularity properties (i.e. if a distributional shadow exists and how regular it is).
In order order to address these regularity question and encouraged by the fact that the solution of a (homogeneous) first order partial differential equation can be written as a pullback of the initial condition by the characteristic backward flow, a main topic of my thesis deals with the microlocal analysis of pullbacks of c-bounded Colombeau generalized functions. 

Another topic is the comparsion of Colombeau techniques for solving ordinary and partial differential equations to other generalized solution concepts, which has led to a joint article with G\"unther H\"ormann. A useful tool for this  purpose is the concept of a generalized graph, which has been developed in the thesis. 
\newline

Wie der Titel ``Verallgemeinerte Regularit\"at und L\"osungskonzepte f\"ur Differentialgleichungen'' andeutet, ist das Hauptthema die Untersuchung von verallgemeinerten L\"osungskonzepten f\"ur Differentialgleichungen, insbesondere f\"ur hyperbolische, partielle Differentialgleichungen erster Ordnung mit reellen, nicht-glatten Koeffizienten und deren charakteristisches System von gew\"ohnlichen Differentialgleichungen.

 In der Colombeau Theorie gibt es Resultate die L\"osungen f\"ur gew\"ohnliche und partielle
Differentialgleichungen liefern, die ausserhalb der Reichweite g\"angiger klassischer L\"osungskonzepte liegen. Nichts destotrotz werden diese Resultate auf Kosten der Regularit\"at solcher L\"osungen erzielt (im Allgemeinen folgt nicht, dass Colombeau L\"osungen einen distributionellen Schatten besitzen m\"ussen). Es ist meist einfacher im Rahmen der Colombeau Theorie die Existenz einer L\"osung zu zeigen, als der deren Regularit\"atseigenschaften zu bestimmen (d.h. ob ein distributioneller Schatten existiert und wie regul\"ar dieser ist). Um solche Regularit\"atsfragen zu behandeln und durch die Tatsache best\"arkt, dass die L\"osung einer (homogenen) partiellen Differentialgleichung erster Ordnung als Pullback der Anfangsbedingung durch den charakteristischen R\"uckw\"artsfluss bestimmt werden kann, ist ein Teil meiner Dissertation mit der mikrolokalen Analyse von Pullbacks kompakt beschr\"ankter Colombeau verallgemeinerter Funktionen befasst.

 Ein weiteres Thema ist der Vergleich der Colombeau Techniken zum L\"osen gew\"ohnlicher und partieller Differentialgleichungen mit anderen verallgemeinerten L\"osungskonzepten. Ein n\"utzliches Werkzeug f\"ur diese Zwecke ist das Konzept eines verallgemeinerten Graphen, das im Rahmen dieser Arbeit ebenfalls entwickelt worden ist.

%% file: cv.tex
\newpage

\ihead{} 
\chead{}
\ohead{}

\huge \center{\textbf{Curriculum Vitae}}\\*[1cm]
\normalsize
\textbf{Simon L. Haller}, born 1978, February 22th in Salzburg.
\\*[1.0cm]

\begin{tabular}
[t]{l  l}
&Volkschule/Hauptschule Friedburg,\\
6/1996&BORG Neumarkt a. Wallersee.\\ \\
10/1996&Military service: EF-course.\\ \\
10/1997&Medicine at the University of Innsbruck.\\ \\
10/1999&Finished first part of Medicine.\\ \\
10/1999&Physics at the University of Innsbruck.\\ \\
3/2000&Mathematics  at the University of Innsbruck.\\ \\
10/2000&Proceeding my studies at the University of Vienna.\\ \\
04/2001 - 12/2003& Working part time as Java software developer.\\ \\
06/2001&Finished first part of Physics with distinction. \\ \\
09/2001&Finished first part of Mathematics with distinction.\\ \\
03/2005&Master degree in Physics with distinction.\\ \\
04/2005 - 06/2008&Research assistant at the University of Vienna (DIANA research group).\\ \\
\end{tabular}